\documentclass{amsart}
\usepackage[utf8]{inputenc}
\usepackage{amsmath,amsfonts,amssymb,amscd,amsthm,amsbsy,upref, siunitx}
\usepackage[dvipsnames]{xcolor}
\usepackage{amsmath}
\usepackage{hyperref,cleveref,amssymb,mathtools}
\usepackage{bbm}
\usepackage{enumitem}
\definecolor{darkgreen}{rgb}{0.0, 0.27, 0.13}

\numberwithin{equation}{section}
\def\hangbox to #1 #2{\vskip3pt\hangindent #1\noindent \hbox to #1{#2}$\!\!$}

\usepackage{hyperref,cleveref,amssymb,mathtools}

\theoremstyle{plain}
\newtheorem{theorem}{Theorem}[section]
\newtheorem{proposition}[theorem]{Proposition}
\newtheorem{corollary}[theorem]{Corollary}
\newtheorem{lemma}[theorem]{Lemma}
\newtheorem{conjecture}[theorem]{Conjecture}

\theoremstyle{definition}
\newtheorem{remark}[theorem]{Remark}

\newtheorem{definition}[theorem]{Definition}
\newtheorem{regularization bounds}[theorem]{Regularization bounds}

\DeclareSymbolFont{bbold}{U}{bbold}{m}{n}
\DeclareSymbolFontAlphabet{\mathbbold}{bbold}

\def\N{{\mathbb N}}

\DeclareMathAlphabet\mathbfcal{OMS}{cmsy}{b}{n}

\def\sfrac#1#2{\kern.1em\raise.5ex\hbox{$#1$}
	\kern-.1em/\kern-.05em\lower.25ex\hbox{$#2$}}

\newcommand{\fw}{\text{\fw}}

\newcommand{\dist}{{\rm dist}}

\DeclareMathOperator{\curl}{curl}

\title[Sharp well-posedness for the free boundary MHD equations]{Sharp well-posedness for the free boundary MHD equations}

\begin{document}

\author{Mihaela Ifrim}
\address{Department of Mathematics\\
University of Wisconsin - Madison
} \email{ifrim@wisc.edu}
\author{Ben Pineau}
\address{Courant Institute for Mathematical Sciences\\
New York University
} \email{brp305@nyu.edu}
\author{Daniel Tataru}
\address{Department of Mathematics\\
University of California at Berkeley
} \email{tataru@math.berkeley.edu}
\author{Mitchell~A.\ Taylor}
\address{Department of Mathematics\\
ETH Z\"urich, Ramistrasse 101, 8092 Z\"urich, Switzerland.
} \email{mitchell.taylor@math.ethz.ch}

 \begin{abstract}
In this article, we provide a definitive well-posedness theory for the free boundary problem in incompressible magnetohyrodynamics. 
Despite the clear physical interest in this system and the remarkable progress in the study of the free boundary Euler equations in recent decades, the low regularity well-posedness of the free boundary MHD equations has remained completely open. This is due, in large part, to the highly nonlinear wave-type coupling between the velocity, magnetic field and  free boundary, which has forced  previous works to impose  restrictive geometric constraints on the data. To address this problem, we introduce  a novel Eulerian approach and  an entirely new functional setting, which better captures the wave equation structure of the MHD equations and permits a complete  Hadamard well-posedness theory in low-regularity Sobolev spaces. In particular, we give the first proofs of existence, uniqueness and continuous dependence on the data at the sharp  $s>\frac{d}{2}+1$ Sobolev regularity, in addition to a blowup criterion for smooth solutions at the same low regularity scale. Moreover, we provide a completely new method for constructing smooth solutions which, to our knowledge, gives the first proof of existence (at any regularity) in our new functional setting. All of our results hold in arbitrary dimensions and in general, not necessarily simply connected, domains. By taking the magnetic field to be zero, they also recover the corresponding sharp well-posedness theorems for the free boundary Euler equations. The methodology and tools that we employ here can likely be fruitfully implemented in other free boundary models.
\end{abstract} 


\maketitle

\setcounter{tocdepth}{1}
\tableofcontents

\section{Introduction}
In this article, we study the dynamics of an electrically conducting fluid droplet. At time $t$, our fluid occupies a compact, connected region $\overline{\Omega}_t\subseteq \mathbb{R}^d$ with $d\geq 2$, and its motion is governed by the incompressible, inviscid MHD equations:
\begin{equation}\label{Euler}
  \begin{cases}
    &\partial_tv+v\cdot \nabla v-B\cdot\nabla B+\nabla P=0,
    \\  
    &\partial_tB+v\cdot\nabla B-B\cdot \nabla v=0,
    \\
    &\nabla\cdot v=0,
    \\
    &\nabla\cdot B=0.
    \end{cases}
\end{equation}
Here, $v:\overline{\Omega}_t\to \mathbb{R}^d$ is the fluid velocity, $P:\overline{\Omega}_t\to \mathbb{R}$ is the total pressure, and $B:\overline{\Omega}_t\to \mathbb{R}^d$ is the magnetic field. The evolution of the free boundary $\Gamma_t:=\partial\Omega_t$ is coupled to the     interior dynamics by three natural boundary conditions. The first is the \emph{kinematic boundary condition}, 
\begin{equation}\label{bb}
    D_t:=\partial_t+v\cdot \nabla \ \ \text{is tangent to} \ \bigcup_t \{t\}\times\partial\Omega_t\subseteq\mathbb{R}^{d+1},
\end{equation} 
which states that the free hypersurface $\Gamma_t$  flows with normal velocity $v\cdot n_{\Gamma_t}$. The second is the  \emph{dynamic boundary condition}, which states that 
\begin{equation}\label{BC1}
\begin{split}
&P=0 \ \text{on} \ \Gamma_t.
\end{split}
\end{equation}
In the physical derivation of \eqref{Euler}, the total pressure $P$ arises as the combination $P:=p+\frac{1}{2}|B|^2$ of the fluid and magnetic pressures, and \eqref{BC1} represents the balance of forces at the fluid-vacuum interface in the absence of surface tension. 
As a final boundary condition, we require that 
\begin{equation}\label{Tangent B}
B\cdot n_{\Gamma_t}=0 \ \text{on} \ \Gamma_t,
\end{equation}
which says that the fluid is a perfect conductor. This closes the above system and ensures that the total energy
\begin{equation*}
    E:=\frac{1}{2}\int_{\Omega_t}|v|^2dx+\frac{1}{2}\int_{\Omega_t}|B|^2dx
\end{equation*}
is formally conserved. Throughout the article, we will refer to the system \eqref{Euler}-\eqref{Tangent B}  as the \emph{free boundary MHD equations}. 
\medskip

By taking the divergence of \eqref{Euler}, we obtain the following elliptic equation for the total pressure $P$:
\begin{equation}\label{Euler-pressure}
  \begin{cases}
    &\Delta P=\text{tr}(\nabla B)^2-\text{tr}(\nabla v)^2\ \ \text{in} \ \Omega_t,
    \\  
    &P=0 \ \ \text{on} \ \Gamma_t.
    \end{cases}
\end{equation}
Assuming sufficient regularity on $(v(t),B(t),\Gamma_t)$, the equation \eqref{Euler-pressure} uniquely determines the pressure from the velocity, magnetic field and domain at any given time $t$.
\medskip

When $B\equiv 0$, the free boundary MHD equations reduce to the free boundary Euler equations.  By a classical result of Ebin \cite{MR886344}, these latter equations are  ill-posed unless the normal derivative of the pressure points into the fluid. Therefore, in the well-posedness theory of the free boundary Euler equations, it is generally assumed that the pressure $p_0$ associated to the initial data satisfies $-\nabla p_0\cdot n_{\Gamma_0}>c_0$ on $\Gamma_0$ for some $c_0>0$. Note that for irrotational data on compact simply connected domains, such a $c_0$ can always be found,  by the strong maximum principle.
\medskip

For the free boundary MHD equations, the analogous assumption to avoid ill-posedness is that
\begin{equation}\label{TS}
    a:=-\nabla P\cdot n_{\Gamma_t}> c_0>0.
\end{equation}
Indeed, ill-posedness results when \eqref{TS} is violated can be found in \cite{MR4093862}. Therefore, we will always assume that our initial data satisfies the following:
\medskip

{\bf Taylor sign condition.} There is a $c_0>0$ such that $a_0:=-\nabla P_0\cdot n_{\Gamma_0}>c_0$ on $\Gamma_0$.  
\medskip

 Geometrically, enforcing $a_0>0$ ensures that the initial pressure $P_0$ is a non-degenerate defining function for the initial boundary hypersurface $\Gamma_0$, and thus can be used to describe the regularity of the boundary. As part of our well-posedness theorem below, we will prove that -- under minimal regularity requirements -- the  positivity of the Taylor coefficient persists  on some non-trivial time interval.
\medskip

Finally, we remark that for the incompressible Euler flow, if the initial 
data is irrotational (i.e.~$\omega_0=0$) then the solution 
remains irrotational at later times. Hence, it is meaningful and indeed interesting to study irrotational 
flows, or, as they are commonly referred to, \emph{water waves}. By contrast, in the incompressible MHD case the 
Lorentz force generates vorticity and the irrotationality 
condition does not propagate.
\subsection{ The structure of the MHD equations} In order to justify the setup for our main results as well as the function space framework in the upcoming subsections, it is instructive to begin with a heuristic description  of the 
leading part of the free boundary MHD equations, also comparing  it with the free boundary incompressible Euler 
equations.
\medskip

We begin our discussion with the boundaryless case, where 
at leading order the incompressible Euler equations may be seen as  a transport equation of the form
\begin{equation*}
D_t v = f,
\end{equation*}
where we view $f$ as a perturbative error. In contrast, by applying $D_t$ to the $v$ and $B$ equations in \eqref{Euler} and observing that $D_t$ and $\nabla_B := B \cdot \nabla$ commute, the coupled system structure of MHD naturally leads to a second order evolution,
\begin{equation}\label{wave equ intro}
D_t^2v-\nabla_B^2v=f_1,\hspace{5mm} D_t^2B-\nabla_B^2B=f_2,
\end{equation}
which is akin to a one dimensional, possibly degenerate,  wave equation relative to the distinguished direction of the magnetic field.
\medskip

Turning our attention now to  free boundary problems,
the leading-order description requires an additional ``good variable" 
in order to capture the motion of the free boundary. In the incompressible Euler case, a convenient geometric choice 
is exactly the Taylor coefficient, $a$, which at leading order 
solves a second-order time evolution
\begin{equation*}
D_t^2 a + a \mathcal{N}a = f.
\end{equation*}
Here, $\mathcal{N}$ represents the Dirichlet-to-Neumann map associated 
to the fluid domain, which should be thought of as a first order elliptic pseudodifferential operator on the free boundary
at fixed time.  
\medskip

Moving on to the MHD counterpart, the Taylor coefficient remains the good variable, but its evolution acquires a wave 
component in the $B$ direction. More specifically, the resulting equation can be written at leading order as
\begin{equation}\label{schematicDta}
D_t^2 a - \nabla_B^2 a +  a \mathcal{N} a = f.
\end{equation}
To complete our heuristic description of the equations we also need to consider the coupling between the interior and the boundary component. In the incompressible Euler case, 
at leading order this can be roughly interpreted in terms of the rotational/irrotational decomposition of the velocity field,
\[
v = v^{irr} + v^{rot} 
\]
where the two components satisfy
\[
\nabla \times v^{irr} = 0, \qquad v^{rot} \cdot n_\Gamma = 0.
\]
The rotational component is essentially described by the vorticity
$\omega = \nabla \times v$, which carries the transport equation in the interior evolution,
\[
D_t \omega = f. 
\]
The irrotational component, on the other hand, can be seen as part of the boundary evolution, via the coupling
\[
  \mathcal{N}(v  \cdot n_{\Gamma_t}) \approx   D_t a.
\]
\begin{remark}
The above decomposition is at this point crudely stated and should be taken with a grain of salt, as it is  not  entirely compatible with the regularity of the good variables $(a,\omega)$. This will be elaborated on later.
\end{remark}
On the other hand, for the free boundary MHD system we have an additional rotational component to contend with; namely, $\omega_B := \nabla \times B$, which is the electric current. 
Correspondingly, after a diagonalization, we have two transport equations, 
\[
(D_t \pm \nabla_B)(\omega \mp \omega_B) = f
\]
and two associated coupling conditions
\[
  \mathcal{N}((v \pm B) \cdot n_{\Gamma_t}) \approx   (D_t \mp \nabla_B) a.
\]
Given the above heuristic discussion, we are now ready to  describe the functional setting that we will use to study the dynamics of solutions to these equations.
\subsection{Scaling and function spaces}
A state for the free boundary MHD equations consists 
of a domain $\Omega$ together with a velocity field $v$ and a magnetic field $B$ on $\Omega$.
A  bounded connected domain $\Omega$ can be equally described by its
boundary $\Gamma$. Hence, in the sequel, by a state we mean a triple $(v,B,\Gamma)$. 
\medskip

To understand the correct functional setting for analyzing the free boundary MHD equations, it is imperative to look at the scaling of the equations:
\begin{equation*}
    \begin{split}
        v_\lambda(t,x)&=\lambda^{-\frac{1}{2}}v\left(\lambda^\frac{1}{2}t,\lambda x\right),
        \\
         B_\lambda(t,x)&=\lambda^{-\frac{1}{2}}B\left(\lambda^\frac{1}{2}t,\lambda x\right),
        \\
         P_\lambda(t,x)&=\lambda^{-1}P\left(\lambda^\frac{1}{2}t,\lambda x\right),
         \\
          a_\lambda(t,x)&=a\left(\lambda^\frac{1}{2}t,\lambda x\right),
         \\
          \left(\Gamma_{\lambda}\right)_t&=\{\lambda^{-1}x : x\in \Gamma_{\lambda^\frac{1}{2}t}\}.
    \end{split}
\end{equation*}
Just as in the incompressible Euler case, here we remark that in the boundaryless case the problem admits a two parameter family of scaling laws, but the presence of the free boundary eliminates one parameter by the additional requirement that the Taylor coefficient has  dimensionless scaling.
\medskip

The above scaling  suggests that one should place the velocity, magnetic field and domain at the same $H^s$-based Sobolev regularity.  On the other hand, one may check  that both the material derivative $D_t$ and the operator $\nabla_B:=B\cdot \nabla$ -- which play balanced roles in the equations -- scale like half derivatives. In the incompressible Euler case we only have the material derivative, and  the 
$H^{s-\frac12}$ regularity of $D_t(v,\Gamma)$ can be recovered dynamically from the $H^s$ regularity of the 
initial data. In the MHD case, however, it also becomes natural to consider the $H^{s-\frac12}$ regularity of $\nabla_B(v,B)$, with the difference being that this regularity should now be part of the initial data assumptions.  We remark that, taken together, the $H^{s-\frac12}$ regularity of $D_t(v,B)$ and $\nabla_B(v,B)$
comprise the energy associated to the wave operator 
$D_t^2-\nabla_B^2$ arising in the description of  the incompressible MHD equations in \eqref{wave equ intro}. Based on this discussion, it is very natural to incorporate the $\nabla_B$ regularity 
into our state space; this will play a fundamental  role in the low regularity analysis. 
\medskip

Concerning the potential choices for $s$, scaling 
provides the universal critical threshold $s_c = \frac{d+1}2$. However, this turns out to be far from the actual local well-posedness threshold. Indeed, even if $B\equiv 0$ (the incompressible Euler case) and without a free boundary, 
the results of \cite{MR3359050} show that local well-posedness in $H^s$ holds if and only if $s > \frac{d+2}2$,
one half-unit above scaling. Our recent results in \cite{Euler}
show that the same range is valid also with a free boundary. So, the best we could hope for in the case of MHD is also  
$s > \frac{d+2}2$, which motivates the definition below.
\begin{definition}[State space]\label{D state space} Let $s>\frac{d}{2}+1$. The
\emph{state space}  $\mathbf{H}^s$ is the set of all triples $(v,B,\Gamma)$ such that $\Gamma$ is the boundary of a bounded, connected domain $\Omega$ and such that the following properties are satisfied:
\begin{enumerate}
    \item (Regularity). $v,B\in H_{div}^s(\Omega)$ and $\Gamma\in H^s$, where $H_{div}^s(\Omega)$ denotes the space of divergence-free vector fields in $H^s(\Omega)$.
    \item (Taylor sign condition). $a:=-\nabla P\cdot n_{\Gamma}>c_0>0$, where $c_0$ may depend on the choice of $(v,B,\Gamma)$, and the pressure $P$ is obtained from $(v,B,\Gamma)$ by solving the  elliptic equation \eqref{Euler-pressure} associated to \eqref{Euler} and \eqref{BC1}.
    \item (Tangency of $B$).   $B\cdot n_\Gamma=0$ on $\Gamma.$
    \item \label{iiii}(Wave-type regularity condition).  $\nabla_Bv, \nabla_BB\in H^{s-\frac{1}{2}}(\Omega)$.
\end{enumerate}
\end{definition}
\begin{remark}
To our knowledge, our results are the first to incorporate the natural wave regularity condition (iv) into the functional setting for this problem. Therefore, the main results we describe below actually represent the first results (at any regularity) in this state space. All previous approaches have only incorporated properties (i)-(iii), and many of them impose further highly restrictive geometric constraints on the data and work in high regularity. While condition (iv) is entirely natural in view of the wave-type structure of the MHD equations, and at the same scaling level as (i), it is far from straightforward to propagate this in the requisite energy estimates and even more difficult to construct (even at high regularity) solutions actually satisfying this condition. In our opinion, the sharp results that we present below seem to be substantially out of reach of contemporary approaches.
\end{remark}
\begin{remark}
At first glance, it may seem like we are omitting a data condition of the form $\nabla_B \Gamma\in H^{s-\frac{1}{2}}$ (or more precisely, $\nabla_Ba\in H^{s-\frac{3}{2}}(\Gamma)$). It turns out that this is automatic in view of condition (iii), which one can interpret as saying that the free surface is more regular in the direction of $B$, since it directly leads to the formula $\nabla_Bn_\Gamma=-\nabla^\top B\cdot n_{\Gamma}$, as will be shown in \Cref{commutatorremark}. In fact, the $H^{s-\frac{1}{2}}(\Omega)$ regularity of $\nabla_B B$ will also ensure that we have the second-order bounds $\nabla_B^2\Gamma\in H^{s-1}$ (or more precisely, $\nabla_B^2a\in H^{s-2}(\Gamma)$). This enhanced regularity for the free surface will be crucial for enforcing several subtle energy cancellations which will be necessary for well-posedness at low regularity. By slight abuse of language, we will sometimes refer to this observation as a ``regularizing effect" although, really, it should be viewed as a property of the state itself. A precise quantitative description of this property will be given in \Cref{partitionofBa}. 
\end{remark}
\begin{remark}\label{Time zero propagate}
It is well-known that, for any incompressible Euler-like system, the pressure enforces the  divergence-free condition on the velocity. For the free boundary MHD equations, the divergence-free condition on $B$ and the tangency of $B$ on the boundary should be thought of as constraints on the initial data rather than as dynamical requirements, which is why we incorporate them into the state space. Indeed,  by taking the divergence of the second equation in \eqref{Euler}, it follows that $D_t(\nabla\cdot B)=0$. Hence, if $\nabla\cdot B_0 =0$ in $\Omega_0$, then $\nabla \cdot B=0$ in $\Omega_t$ for all $t$. On the other hand, propagating  the boundary condition \eqref{Tangent B} is slightly more subtle. To see that, recall from \cite{MR2388661} that for a one-parameter family of domains flowing with velocity $v$, we have $D_tn_{\Gamma_t}=-((\nabla v)^*(n_{\Gamma_t}))^\top$. Here, $(\nabla v)_{ij}=\partial_j v_i$, the operation $\cdot{}^\top$ denotes the projection onto the tangent space, and the operation $\cdot{}^*$ denotes the adjoint, so that $(\nabla v)^*_{ij}=\partial_iv_j.$ An elementary computation then shows that 
\begin{equation*}
    D_t(B\cdot n_{\Gamma_t})=B\cdot n_{\Gamma_t}\left[n_{\Gamma_t}\cdot \nabla v\cdot n_{\Gamma_t}\right].
\end{equation*}
Hence, as long as $v\in L^1([0,T];C^1)$, the magnetic boundary condition is propagated by the flow.   Note that, on a physical level, the condition $B\cdot n_{\Gamma_t}=0$ states that the fluid is a perfect conductor.
\end{remark}
For states $(v,B,\Gamma)$ as in \Cref{D state space}, we may quantify their regularity by
\[
\|(v,B,\Gamma)\|_{\mathbf{H}^s}^2:=\|\Gamma\|_{H^s}^2+\|v\|^2_{H^s(\Omega)}+\|B\|_{H^s(\Omega)}^2+\|\nabla_Bv\|_{H^{s-\frac{1}{2}}(\Omega)}^2+\|\nabla_BB\|_{H^{s-\frac{1}{2}}(\Omega)}^2.
\]
Notice, however, that  $\mathbf{H}^s$ is not  a linear space, so the above formula does not define a norm. 
Nevertheless, $\mathbf{H}^s$  does come equipped with a compatible topology, so may be informally viewed  as an infinite dimensional manifold. Using this structure, we may rigorously define the space $C([0,T];\mathbf{H}^s)$  of continuous functions with values in $\mathbf{H}^s$, as well as an appropriate notion of $\mathbf{H}^s$ continuity of the data-to-solution map $(v_0,B_0,\Gamma_0)\mapsto (v(t),B(t),\Gamma_t)$. The main goal of this article  is to study the following:
\medskip

\noindent
\textbf{Cauchy problem for the free boundary MHD equations}: \emph{ Given  an initial state $(v_0,B_0,\Gamma_0) \in \mathbf H^s$, find the unique solution $(v,B,\Gamma) \in 
C([0,T];\mathbf{H}^s)$ in some time interval $[0,T]$ and show that the data-to-solution map is continuous.}
\subsection{Historical remarks}
The free boundary MHD equations arise as a coupling of the free boundary Euler equations and the Maxwell equations. They have significant physical interest, and  model  a wide variety of electrically conducting fluids and plasma. Despite this, little is known about the mathematical foundations of these equations. 
Indeed, the first step towards nonlinear well-posedness did not occur until 2014, \cite{MR3187678}, where a priori estimates were established for $H^4$ initial data in three dimensions. However, instead of using the natural boundary condition \eqref{BC1}, the authors in \cite{MR3187678} assume that both the fluid pressure $p$ and the magnetic pressure $\frac{1}{2}|B|^2$ are identically constant on the boundary. To the best of our knowledge, there is no general local-existence result that preserves these boundary conditions.
\medskip

In \cite{MR4072383}, low regularity a priori estimates for the 3D free boundary MHD equations are  proven under the same restrictive boundary condition, assuming, in addition, that  the fluid domain has a smooth boundary and satisfies an  appropriate smallness condition.  In terms of regularity, this latter result places the initial velocity and magnetic field at the natural $H^{2.5+\delta}$ threshold but needs an extra half degree of regularity on the Lagrangian flow map $\eta$. As was originally pointed out in \cite[Section 5]{MR2388661}, the Lagrangian flow map for the free boundary Euler equations is, in general, only as smooth as the velocity -- a characterization of when it has $H^{3+\delta}$ regularity was recently given in \cite{aydin2024construction}.
 \medskip

Under the natural boundary condition \eqref{BC1}, existence and uniqueness of solutions to the free boundary MHD equations was  proven in Lagrangian coordinates in \cite{MR3980843}, in the specific case that  $\Omega_0=\mathbb{T}^2\times (0,1)$ and $v_0,B_0\in H^4(\Omega_0)$. Since the domain $\Omega$ is one of the dynamical variables of the problem, the assumption that $\Omega_0=\mathbb{T}^2\times (0,1)$ is quite restrictive.  Moreover, in general, the assumption of a flat initial domain somewhat obscures the precise manner in which the regularity of the boundary is tracked -- see the above discussion on the regularity of the Lagrangian flow map and  \cite[p.~4013]{disconzi2019lagrangian} for a more extensive discussion of this issue. In principle, one may be able to transfer  results on $\mathbb{T}^2\times (0,1)$ to more general domains by using a partition of unity and localizing to each coordinate patch, but the highly nonlinear and nonlocal character of the problem substantially increases the technical difficulty due to the need to obtain estimates for the transition maps.  
\medskip

Perhaps the only existing result for our problem on a general class of domains is due to Liu and Xin in \cite{liu2023free}.  However, in contrast to the present article, they primarily study the capillary problem (i.e.~with surface tension) and their approach follows more closely the global geometric method of Shatah and Zeng developed in \cite{MR2400608} to study the free boundary Euler equations with surface tension.  One relevant consequence of the analysis in \cite{liu2023free} is to obtain local existence (but not uniqueness nor continuous-dependence) on simply connected domains in three dimensions in a high regularity Sobolev regime by studying the zero-surface tension limit. The solution that they obtain satisfies properties (i)-(iii) in \Cref{D state space} but not (iv). It also seems  that their method needs  the initial interface to have $H^6$ regularity (in contrast to $H^{2.5+\delta}$ in our setting). Nevertheless, their result is substantial in that it applies to a broader range of physically relevant situations than the other aforementioned works. For instance, it applies to situations where the free interface does not have graph geometry, which is of great physical interest. We do remark, however, that the method of proof in  \cite{liu2023free} seems to crucially rely on the assumption that the domain is simply connected, as it uses the well-posedness of certain div-curl systems (which are known to be ill-posed without certain topological assumptions) in order  to recover the velocity and magnetic fields in their iteration scheme. We will completely avoid such assumptions in our work.
\medskip


 The objective of the present article is to provide a complete well-posedness theory for the free boundary MHD equations (along with several new, strong auxiliary results) at optimal regularity levels and for arbitrary  domains in general dimensions. To accomplish this, we introduce several novel techniques and also significantly extend the reach of the methods developed in our previous work  \cite{Euler}, which served as a proof of concept to establish several new sharp results for the free boundary Euler equations. One key novelty in \cite{Euler} was the introduction of a robust and fully Eulerian framework for studying free boundary problems in the incompressible setting. Although the free boundary MHD equations form a far more complicated model with a host of additional difficulties,  the basic framework developed in \cite{Euler} will serve as a powerful means for side-stepping many of the well-known technical difficulties encountered in free boundary problems. In particular, our scheme will efficiently address the intricate coupling between the elliptic estimates and  the regularity of the dynamic variables, while only  needing minimal  assumptions on the data and avoiding any specialized assumptions regarding the domain geometry. 
\medskip

We remark that the present article focuses exclusively on the dynamics of the incompressible free boundary MHD equations without surface tension. Although this is a fundamental, benchmark model of magnetohydrodynamics,  several other variants of the free boundary  MHD equations have attracted recent attention. We refer the reader to \cite{gu2021zero, gu2021local,hao2023motion,MR4462752, MR4250567} for the study of the incompressible  problem with surface tension, \cite{lindblad2021anisotropic,MR4444136,zhang2020local} for the compressible problem,  \cite{MR3614754,liu2023local,liu2023free,MR3745155,MR3981394,MR4101311} for additional information on incompressible plasma-vacuum interface problems,   and \cite{MR3151094,MR3503661,MR4201624} for compressible analogues. Although the proof that we present in this article is fine-tuned to the incompressible problem without surface tension, we believe that our general methodology will have a  wide range of applicability in the mathematical theory of magnetohydrodynamics. In particular, we hope that it will inspire the development of complete, refined well-posedness theories for  these other MHD models.

\subsection{Overview of the main results}\label{OOTP} 
Our primary objective  is to prove that the free boundary  MHD equations are well-posed in $\mathbf H^s$ 
for $s > \frac{d}2+1$. However, our well-posedness theory includes substantially more than just existence, uniqueness, and continuity of the data-to-solution map in $\mathbf{H}^s$. Therefore, we find it prudent to divide our results into multiple intrinsically interesting components.
\medskip

We begin by setting  notation. Let $\Omega_*$ be a bounded, connected domain with smooth boundary $\Gamma_*$. Given $\epsilon,\delta>0$,  consider the collar neighborhood $\Lambda_*:=\Lambda(\Gamma_*,\epsilon,\delta)$ consisting of all hypersurfaces $\Gamma$ which are  $\delta$-close to $\Gamma_*$ in the $C^{1,\epsilon}$ topology. As long as $\delta>0$ is sufficiently small,  elements of $\Lambda_*$ can be written as graphs over $\Gamma_*$. As a consequence,  Sobolev and H\"older norms  can be defined in a consistent fashion. To state our results, we will  assume  that a collar neighborhood $\Lambda_*$ has been fixed, and consider solutions with initial data $(v_0,B_0,\Gamma_0)$ satisfying $\Gamma_0 \in \Lambda_*$.  A more precise description of this functional setting will be given  in \Cref{AOMD}. 

\subsubsection{Enhanced uniqueness} 
We begin by stating our main uniqueness result, which requires the least in terms of regularity.  Here, of critical importance are the  control parameters 
\begin{equation}\label{ACONT}
A:=A_\epsilon:=\|(v,B)\|_{C^{\frac{1}{2}+\epsilon}_x(\Omega_t)}+\|\Gamma_t\|_{C_x^{1,\epsilon}}, \qquad \epsilon>0,
\end{equation}
which can be thought of as  almost scale-invariant pointwise quantities, and
\begin{equation}\label{BCONT}
A^{\frac{1}{2}}:=\|(v,B)\|_{W^{1,\infty}_x(\Omega_t)}+\|(D_t^+P,D_t^-P)\|_{W^{1,\infty}_x(\Omega_t)}+\|\Gamma_t\|_{C_x^{1,\frac{1}{2}}},
\end{equation}
which is at $\frac{1}{2}$ derivatives above scaling. Here, $D_t^\pm:=D_t\pm\nabla_B$.
We note that the homogeneous part of the $L^1_T$ norm of $A^{\frac{1}{2}}$ is invariant with respect to the natural scaling symmetry for the free boundary MHD equations. Our first main result states that uniqueness holds in the class where these  quantities remain finite.
\begin{theorem}[Uniqueness]\label{Uniqueness intro} Let $\epsilon, T>0$ and let $\Omega_0$ be a domain with boundary $\Gamma_0\in \Lambda_*$ of $C^{1,\frac{1}{2}}$ regularity. Then for every divergence-free initial data $v_0,B_0\in W^{1,\infty}(\Omega_0)$, the free boundary MHD equations with the Taylor sign condition admit at most one solution $(v,B,\Gamma_t)$  with $\Gamma_t\in \Lambda_*$  and
\begin{equation*}
\sup_{0\leq t\leq T}A_\epsilon(t)+\int_{0}^{T}A^{\frac{1}{2}}(t) \, dt<\infty. 
\end{equation*}
\end{theorem}
\Cref{Uniqueness intro} is a considerable improvement in terms of regularity over all other known uniqueness results for this problem. Moreover, to our knowledge, \Cref{Uniqueness intro} is the first uniqueness result for the free boundary MHD equations at any regularity which holds for general initial data (i.e.~without any of the restrictions mentioned above, such as the initial domain being flat). We also emphasize that \Cref{Uniqueness intro}   applies in arbitrary dimensions and to domains with (from most physical perspectives) arbitrary geometries.
\medskip

There is one other remarkable and surprising feature about \Cref{Uniqueness intro}. From the Laplace equation for $D_t^\pm P$ (see \Cref{Dtpdef}), Sobolev embeddings and product estimates, it is straightforward to control (when $B$ is tangent to $\Gamma$) the parameter $A^{\frac{1}{2}}$ entirely in terms of $\|(v,B,\Gamma)\|_{H^s}$ for any $s>\frac{d}{2}+1$ (in fact, a better pointwise bound is also possible, but more delicate). In particular, the norm needed to ensure uniqueness of solutions is substantially weaker than the norm we will use to obtain local well-posedness, as the latter will also incorporate the wave-type regularity condition $(\nabla_B v,\nabla_B B)\in H^{s-\frac{1}{2}}(\Omega)$. 
\begin{remark}\label{slight improvement}
When $B\equiv 0$, \Cref{Uniqueness intro} recovers the main uniqueness theorem in \cite{Euler} for the free boundary Euler equations. Moreover, for the full free boundary MHD equations, \Cref{Uniqueness intro} (as well as \Cref{Diff es thm intro} below) remains valid even if one uses the slightly weaker control parameter obtained by replacing the second term in \eqref{BCONT} by $\|D_tP\|_{W^{1,\infty}_x(\Omega_t)}$. We will  remark on the structural properties of the MHD equations which allow for  this minor improvement in Sections~\ref{LR} and \ref{DF}.
\end{remark}
\subsubsection{Stability estimates}
While uniqueness is a foundational property in its own right, here
we view it as a corollary of a  far more powerful  stability theorem. To explain the setting, let $(v,B,\Gamma_t)$ and $(v_h,B_h,\Gamma_{t,h})$ be two solutions to the free boundary MHD equations  with corresponding domains $\Omega_t$ and $\Omega_{t,h}$. We intend to show that if $(v,B,\Gamma_t)$ and $(v_h,B_h,\Gamma_{t,h})$ are ``close" at time zero, then they remain close on a suitable time-scale. However, since the domains $\Omega_t$ and $\Omega_{t,h}$ are evolving in time, it is impossible to compare $(v,B,\Gamma_t)$ and $(v_h,B_h,\Gamma_{t,h})$ in a linear fashion. To resolve this issue, we construct a nonlinear functional  which measures the distance between solutions at the $L^2$ level and which is  propagated by the flow. 
\medskip

To avoid comparing solutions with entirely different domains, we harmlessly restrict our attention 
to solutions $(v,B,\Gamma_t)$ and $(v_h,B_h,\Gamma_{t,h})$ evolving in the same collar neighborhood $\Lambda_*$. For such solutions, we wish to define a nonlinear distance functional that  is propagated by the flow.  Although in our actual analysis we will work with a symmetrized version of the free boundary MHD equations, our distance functional will, in spirit, take the form
\begin{equation}\label{diff functional candidate1}
\begin{split}
   D((v,B,\Gamma),(v_h,B_h,\Gamma_h)):= \frac{1}{2}\int_{\tilde{\Omega}_t}|v-v_h|^2dx+\frac{1}{2}\int_{\tilde{\Omega}_t}|B-B_h|^2dx+\frac{1}{2}\int_{\tilde{\Gamma}_t}b|P-P_h|^2\, dS.
\end{split}
\end{equation}
Here, $P$ and $P_h$ are the  pressures, $\tilde{\Gamma}_t$ is the boundary of $\tilde{\Omega}_t:=\Omega_t\cap\Omega_{t,h}$ and $b$ is a well-chosen weight function. Heuristically, the first two terms on the right-hand side of \eqref{diff functional candidate1} measure the $L^2$ distance between $v$ and $v_h$ and $B$ and $B_h$, respectively. On the other hand, by the Taylor sign condition, the third term measures the distance between the free hypersurfaces.  The objective of \Cref{DF} is to  prove the following theorem.
\begin{theorem}[Stability]\label{Diff es thm intro}
 Let $0<\epsilon,\delta\ll 1$ and let $\Lambda_*=\Lambda(\Gamma_*,\epsilon,\delta)$ be a collar neighborhood. Suppose that $(v,B,\Gamma_t)$ and $(v_h,B_h,\Gamma_{t,h})$ are solutions to the free boundary MHD equations that evolve in the collar in a time interval $[0,T]$ and satisfy  $a$,$a_h>c_0>0$.  Then we have the estimate
\begin{equation*}\label{Stab est1}
\frac{d}{dt} D((v,B,\Gamma),(v_h,B_h,\Gamma_h)) \lesssim_{A,A_h} (A^{\frac{1}{2}}+A_h^{\frac{1}{2}}) D((v,B,\Gamma),(v_h,B_h\Gamma_h)),
\end{equation*}
where $A_h$ and $A_h^{\frac{1}{2}}$ are  the control parameters \eqref{ACONT} and \eqref{BCONT} corresponding to the solution $(v_h, B_h,\Gamma_{t,h})$.
\end{theorem}
As an immediate corollary of  \Cref{Diff es thm intro}, we obtain \Cref{Uniqueness intro}. However, \Cref{Diff es thm intro}  will also prove to be useful for several other purposes. For example, we will use it in our proof of the continuity of the data-to-solution map as well as in our construction of rough solutions. 
\subsubsection{Well-posedness}
We now turn our attention to the well-posedness 
 problem for the free boundary MHD equations. Our main result proves  sharp well-posedness in $\mathbf H^s$.
\begin{theorem}[Hadamard local well-posedness]\label{MT}
    Fix $s>\frac{d}{2}+1$ and a collar $\Lambda_*$. For any $(v_0,B_0,\Gamma_0)$ in $\mathbf{H}^s$ 
    with $\Gamma_0 \in \Lambda_*$ there exists a time $T>0$,  depending only on $\|(v_0,B_0,\Gamma_0)\|_{\mathbf{H}^s}$ and the lower bound in the Taylor sign condition, for which there exists a unique  solution $(v(t),B(t),\Gamma_t)\in C([0,T];\mathbf{H}^s)$ to the free boundary MHD equations satisfying a proportional  uniform lower bound in the Taylor sign condition. Moreover, the data-to-solution map is continuous with respect to the $\mathbf{H}^s$ topology.
\end{theorem}
\begin{remark}
The proof of \Cref{MT}  also  gives rise to a natural continuation criterion at low regularity, which we will explain in \Cref{HOEE INTRO DIS}.
\end{remark}
In the special case $B\equiv0$, \Cref{MT} recovers the sharp well-posedness theorem for the free boundary Euler equations in \cite{Euler}. For non-zero $B$,  this theorem proves that one can dynamically propagate the natural wave-type condition  $\nabla_Bv,\ \nabla_BB\in H^{s-\frac{1}{2}}$, which is the first result to do so in any Sobolev-type space (let alone at the sharp low-regularity scale). While such a condition is natural, establishing such bounds is far from trivial and will require us to construct a suitable nonlinear energy functional that is not only sufficiently coercive and conforming to the boundary conditions of the state space but also can be dynamically propagated at low regularity. In particular, our construction will require (among many other things) a normal form correction to the equation \eqref{schematicDta} to exploit  certain subtle energy cancellations in order to estimate various 
``perturbative" source terms at low regularity.  

\subsection{Outline of the paper}
The rest of the paper is divided into six sections and an appendix. To a large extent, sections \ref{DF}, \ref{HEB} and \ref{Existence section} can be read independently.
\subsubsection{The linearized equations} We begin in \Cref{LR} by noting a simple change of dependent variables that reformulates the free boundary MHD equations as a system of two free boundary Euler equations, but with skewed transport velocities. We then derive the linearization of our problem in Eulerian coordinates and establish $L^2$ bounds for solutions in terms of our control parameters $A$ and $A^\frac{1}{2}$.
Although the linearization does not play a direct role in our nonlinear analysis, it will serve to motivate many of the choices we make later on. 
\subsubsection{Notation and function spaces} In \Cref{AOMD}, we formally define the state space for our problem and set our notation. The notation that we use in this paper is entirely consistent with our previous work \cite{Euler}, and the material in \Cref{AOMD} is largely presented for the reader's convenience. Supplementing \Cref{AOMD} is \Cref{BEE}, which collects various tools developed in \cite{Euler} that will be needed for our analysis. This includes regularization operators, nonlinear inequalities, ``balanced" elliptic estimates, and function space theory. The so-called balanced elliptic estimates in \Cref{BEE} (many of which were developed in \cite{Euler}) 
will serve as a powerful tool for efficiently obtaining energy estimates at low regularity, where one of the major technical difficulties is in dealing with the dependence of the elliptic regularity on the geometry of the free surface.

\subsubsection{Difference estimates and uniqueness} \Cref{DF} is devoted to proving stability estimates and enhanced uniqueness for the free boundary MHD equations. The main challenge is to construct a nonlinear distance functional that  measures the distance between solutions at low regularity and is also dynamically propagated by the flow. Interestingly, for our specific choice of distance functional, the magnetic field cancels in the most problematic higher-order term in the difference estimates. For this reason, the analysis from \cite{Euler} carries over rather seamlessly, though the pressure now has to be treated slightly differently.  Apart from yielding enhanced uniqueness, the stability estimates that we prove in \Cref{DF}  will also be crucial for constructing rough solutions and proving the continuity of the data-to-solution map.
\medskip

Although omitted from this manuscript, we remark that the proof from  \cite{Euler} that the higher order term in the difference estimates can be estimated is far from trivial, as it requires one to carry out a subtle boundary layer analysis on the intersection of two domain states, which in general has only Lipschitz regularity. 
\medskip

\subsubsection{Higher order energy estimates}\label{HOEE INTRO DIS} In \Cref{HEB}, we prove energy estimates for the free boundary MHD equations within the $\mathbf{H}^k$ scale of spaces for integer $k>\frac{d}{2}+1$. Here, the analysis of the free boundary Euler and MHD equations significantly diverge. To roughly explain the main difficulties, we write
\begin{equation*}
M_{\sigma}:=\|(v,B,\Gamma)\|_{\mathbf{H}^{\sigma}},\hspace{5mm}\sigma\geq 1,
\end{equation*}
and we let $s$ be any real number (not necessarily an integer) with $s>\frac{d}{2}+1$. 
\medskip

The primary task of \Cref{HEB} is to construct and estimate  a family of  energy functionals $(v,B,\Gamma)\mapsto E^k(v,B,\Gamma)$ by identifying Alinhac style ``good variables” which solve the linearized equation to leading order. The major difficulty here is to identify  energy functionals  which are \emph{coercive}, in the sense that
\begin{equation*}
E^k(v,B,\Gamma)\approx_{M_{s-\frac{1}{2}}} \|(v,B,\Gamma)\|_{\mathbf{H}^k}^2,
\end{equation*}
but also satisfies the energy estimate
\begin{equation*}
\frac{d}{dt}E^k(v,B,\Gamma)\lesssim_{M_s} E^k(v,B,\Gamma).
\end{equation*}
Here, the implicit dependence on $M_{s-\frac{1}{2}}$ and $M_s$ is polynomial. Achieving these bounds at optimal regularity levels (i.e.~with implicit constants depending only on $M_{s-\frac{1}{2}}$ and $M_s$) is a tall order, which will require identifying and exploiting various hidden structures of the free boundary MHD equations. In particular, we will have to carry out a variant of a normal form correction to ensure that the above implicit constants depend only on low regularity norms. Moreover, we will have to take great care to understand  the dependence of the elliptic estimates on the free surface regularity. To this end, we will rely on the balanced elliptic estimates in \Cref{BEE} to deal with rather complicated expressions involving iterated applications of the Dirichlet-to-Neumann and   other geometric and elliptic operators.
\medskip

The outcome of \Cref{HEB} is a family of energy estimates which recover the sharp $s>\frac{d}{2}+1$ well-posedness threshold for the free boundary Euler equations in the special case $B\equiv 0$. For general $B\not\equiv 0$ in three dimensions, our estimates constitute a $\frac{3}{2}$-derivatives improvement in scale over all previous results. Moreover, unlike in previous works, our results apply in arbitrary dimensions and to domains with very general geometries.  We remark that although we only directly prove our  energy estimates for integer $k>\frac{d}{2}+1$, in \Cref{RS} we will remove this condition  by carefully interpolating against the difference estimates outlined above. 
\medskip

One immediate corollary of the above energy estimates (and the well-posedness established later in \Cref{RS}) is the following low-regularity continuation criterion which we roughly state as follows.
\begin{theorem}\label{continuationcrit}
Let $\frac{d}{2}+1<s\leq\sigma<\infty$. Moreover, let $(v,B,\Gamma)\in C([0,T);\mathbf{H}^{\sigma})$ be a solution to the free boundary  MHD equations. Then $(v,B,\Gamma)$ can be continued past time $T>0$ for as long as it stays in the collar and the following properties hold:
\begin{enumerate}[label=\alph*)]
\item\label{a2} (Uniform bound from below for the Taylor coefficient). There is a $c>0$ such that 
\[
a(t) \geq c > 0,\hspace{5mm} 0\leq t<T.
\]
\item (Low regularity bound). There holds
\[
\sup_{0\leq t<T}M_s(t)<\infty.
\]
\end{enumerate}
\end{theorem}

One might ask if it is possible to further improve the energy estimates so as to only involve $L^{\infty}$ based norms of the natural variables $(v,B,\Gamma)$. More precisely, does an analogue of the famous Beale-Kato-Majda criterion (see \cite{MR763762} and \cite{MR1462753})  hold for the free boundary MHD equations? Proving such a result is likely to be possible by heavily optimizing the estimates established in \Cref{HEB}.  However, this would involve numerous lengthy and delicate calculations as well as a far more careful application of the balanced elliptic estimates. Since many of the basic questions involving the well-posedness of the free boundary MHD equations were open until now, we have chosen not to attempt such refinements in the present paper. Nevertheless, we conjecture that the following continuation criterion for the free boundary MHD equations is valid. 
\begin{conjecture}\label{Conj BKM}
Let $\frac{d}{2}+1<\sigma<\infty$. Moreover, let $(v,B,\Gamma)\in C([0,T);\mathbf{H}^{\sigma})$ be a solution to the free boundary  MHD equations. There exists $1\leq p\leq\infty$ such that $(v,B,\Gamma)$ can be continued past time $T>0$ for as long as it stays in the collar and the following properties hold:
\begin{enumerate}[label=\alph*)]
\item\label{a3} (Uniform bound from below for the Taylor coefficient). There is a $c>0$ such that 
\[
a(t) \geq c > 0,\hspace{5mm} 0\leq t<T.
\]
\item (Low regularity bound). There holds
\[
\|(v,B,\Gamma)\|_{L^{\infty}([0,T);\mathbf{C}^{\frac{1}{2}+\epsilon})}+\|(v,B,\Gamma)\|_{L^{p}([0,T);\mathbf{C}^{1})}<\infty,
\]
\end{enumerate}
where $\mathbf{C}^{\alpha}$ denotes the natural H\"older space analogue of $\mathbf{H}^{\alpha}$.
\end{conjecture}
The best result -- which is (almost) scale-invariant -- would correspond to the case $p=1$, which is exactly what we achieved in \cite{Euler} when $B\equiv 0$ (i.e.~for the free boundary Euler equations). We suspect that the estimates in \Cref{HEB} could be optimized using the balanced elliptic estimates to establish \Cref{Conj BKM} for some $p<\infty$, though we stress that this appears to be quite non-trivial in its own right. The case $p=1$ seems like it would require additional new ideas, if  true. We remark that \Cref{Conj BKM} is ``localized" in the sense that to prove it, it suffices to improve the control parameters in the energy estimates in \Cref{Energy est. thm} and then implement  an adaptation of the bootstrap argument  in \cite[Section 9]{Euler} which we used  to obtain the analogous sharp continuation criterion for the free boundary Euler equations.
\subsubsection{Construction of regular solutions} \Cref{Existence section} is devoted to the construction of regular solutions to the free boundary MHD equations in the state space $\mathbf{H}^s$. The question of existence of solutions to these equations on general domains was largely open, and the construction that we present is one of the central novelties of the  paper. On a high level, our overarching scheme utilizes a time discretization via an Euler type method together with a separate transport step to produce good approximate solutions. To avoid derivative loss, we also include a carefully designed  regularization of each iterate which respects the uniform energy bounds and the boundary conditions for the problem. Constructing this regularization is the main difficulty encountered in this section. 
\medskip

The strategy that we employ takes some mild inspiration from the time discretization approach carried out in the case of a compressible gas in \cite{ifrim2020compressible}. We stress, however, that aside from the basic logical structure of the argument (i.e.~the need for carrying out a regularization plus an Euler type iteration), the central difficulties here are entirely different. For instance, in our setting, the surface of a liquid carries a non-trivial energy,  so the geometry of the free boundary hypersurface $\Gamma$ plays a significant role in preserving the energy bounds through each iteration. Moreover, the matched regularity of the magnetic field and the free surface makes it a very delicate matter to understand what types of regularizations of the hypersurface $\Gamma$ are compatible with the boundary condition $B\cdot n_\Gamma=0$. Dealing with such issues will require several novel technical innovations; we refer the reader to the start of \Cref{Existence section} for a detailed outline of the existence scheme.  
\subsubsection{Rough solutions and continuous dependence}
In the last section of the paper, we construct rough solutions as strong limits of smooth solutions and prove the continuity of the data-to-solution map in $\mathbf{H}^s$. The construction of rough solutions is achieved by considering a family of dyadic regularizations of the initial data, which, by the results of \Cref{Existence section}, produce corresponding smooth solutions. From our energy estimates in \Cref{HEB}, we obtain control over the higher  $\mathbf{H}^k$ norms of these smooth solutions. On the other hand, from the difference estimates in \Cref{DF} (and also a milder but still non-trivial difference type bound for the variables $\nabla_Bv$ and $\nabla_B B$), we obtain control over the distance between consecutive solutions in a weaker topology. Consequently, we obtain rapid convergence in all  $\mathbf{H}^l$  spaces with $l < k$. Using interpolation, frequency envelopes, and similar arguments, we may then conclude strong convergence in  $\mathbf{H}^k$, prove local existence in fractional regularity $\mathbf{H}^s$ spaces, and deduce the continuity of the data-to-solution map. 
\medskip

We remark  that when compared with our previous article \cite{Euler}, there are two main additional difficulties in this stage of the argument.
The first stems from the need to construct and propagate a new distance functional incorporating the $\nabla_B(v,B)$ variables in order to ensure that the regularized solutions converge in the state space $\mathbf{H}^s$. We remark that although propagating the functional \eqref{diff functional candidate1} results in a powerful uniqueness theorem, it does not suffice for this latter purpose, as it does not control the distance between regularizations of  $\nabla_B(v,B)$. Although propagating distance bounds for these latter variables is non-trivial, the benefit we have in \Cref{RS} is that we may work with regularized states and stronger control parameters, which makes the analysis feasible.
\medskip

The second additional difficulty we encounter is in the construction of frequency envelopes. Due to conditions (i) and (iii) in \Cref{D state space}, whatever regularization operators we choose to employ must preserve both the divergence-free condition and the tangency of the magnetic field on the new, regularized domain. Constructing an appropriate dyadic family of regularization operators enforcing such conditions is somewhat delicate, and will be established in \Cref{SSRO}.
\subsection{Acknowledgments}
The first author was supported  by the  NSF CAREER grant DMS-1845037, by the Sloan Foundation, by the Miller Foundation and by a Simons Fellowship. The second author was supported by a fellowship in the Simons Society of Fellows. The latter three authors were further supported by the NSF grant DMS-2054975 as well as by a Simons Investigator grant from the Simons Foundation.  
Some of this work was carried out while all four authors were in residence at the Simons Laufer Mathematical Sciences Institute (formerly MSRI) in Berkeley, California, during the summer of 2023, 
participating in the program ``Mathematical problems in fluid dynamics, Part II", which
 was supported by the National Science Foundation under Grant No.~DMS-1928930.
\section{A reformulation of MHD and its linearization}\label{LR}
In this section, we  formally derive the linearization of our problem, working entirely in Eulerian coordinates. The outcome is the system of equations \eqref{DM2}, together with the linearized energy \eqref{new Linearized energy}, and the basic energy estimate \eqref{EE for DMlin2}. 
\medskip

To begin, we reformulate the free boundary MHD equations as a system of Euler-like equations  using the so-called \textit{Els\"asser variables}.  For this, we define $W^+:=v+B$ and $W^-:=v-B.$ The associated material derivatives are then 
\begin{equation}\label{New D_t}
    \begin{split}
        D_t^-&:=\partial_t+(v-B)\cdot\nabla,
        \\
        D_t^+&:=\partial_t+(v+B)\cdot\nabla.
    \end{split}
\end{equation}
 Note that both of the vector fields in \eqref{New D_t} are tangent to $\Gamma_t$ since $B\cdot n_{\Gamma_t}=0$ by \eqref{Tangent B} and $D_t$ is tangent to $\Gamma_t$ by \eqref{bb}. Re-writing \eqref{Euler} in terms of these new variables, we obtain 
\begin{equation}\label{pm equations}
    \begin{cases}
        D_t^{\pm}W^{\mp}&=-\nabla P,
        \\
        \nabla \cdot W^\pm&=0,
    \end{cases}
\end{equation}
in $\Omega_t$. The key benefit of using the variables $(W^+,W^-)$ is that the above equations are now, almost, in a standard   symmetric hyperbolic form. 
\medskip

The equations  \eqref{pm equations}  are supplemented with the Taylor sign condition \eqref{TS}, the boundary condition \eqref{BC1}, and the boundary conditions
\begin{equation}\label{bb2}
    D_t^\pm \  \text{are both tangent to} \ \bigcup_t \{t\}\times\partial\Omega_t\subseteq\mathbb{R}^{d+1}.
\end{equation} 
 The resulting system \eqref{BC1}-\eqref{TS}-\eqref{pm equations}-\eqref{bb2} is what we will analyze in this paper. Clearly,  it is both algebraically and analytically equivalent to the free boundary MHD equations.
\subsection{The linearized equations} To derive the linearized system, we take a one parameter family of solutions $(W^+_h,W^-_h,P_h)$ defined on domains $\Omega_{t,h}$, with $(W^+_0,W^-_0,P_0):=(W^+,W^-,P)$ and $\Omega_{t,0}:=\Omega_t$. We  define $w^\pm=\partial_hW_h^\pm |_{h=0}$ and $P_{lin}=\partial_hP_h|_{h=0}$.
\medskip

In $\Omega_t$, the linearized equations are obtained by  standard means: 
\begin{equation*}\label{Euler linearized}
  \begin{cases}
    &D_t^\pm w^\mp+\nabla P_{lin}=-w^\pm\cdot\nabla W^\mp,
\\
    &\nabla\cdot w^\pm=0.
    \end{cases}
\end{equation*}
However, this is not the full story, since we also need to linearize the   boundary conditions on the hypersurface $\Gamma_t$. For this, we denote by $\Gamma_{t,h}$ the free hypersurface at time $t$ for the solution $(W^+_h,W^-_h,P_h)$, so that $\Gamma_{t,0}:=\Gamma_{t}$. We then fix a one parameter family of diffeomorphisms $\phi_h(t):\Gamma_t\to \Gamma_{t,h}$ with $\phi_0(t)=Id_{\Gamma_t}$. The dynamic boundary condition \eqref{BC1} states that for every point $x\in \Gamma_t$,
$$P_h(t,\phi_h(t)(x))=0.$$
Differentiating in $h$ and evaluating at $h=0$ yields
$$P_{lin}|_{\Gamma_t}=-\nabla P|_{\Gamma_t}\cdot \psi(t),$$
where $\psi(t):=\frac{\partial \phi_h(t)}{\partial h}|_{h=0}$.
By \eqref{BC1}, $\nabla P|_{\Gamma_t}$ is normal to $\Gamma_{t}$. Thus,
\begin{equation}\label{Def of s}
P_{lin}|_{\Gamma_t}=-\nabla P|_{\Gamma_t}\cdot n_{\Gamma_t} \psi(t)\cdot n_{\Gamma_t}=:as.\end{equation}
Here, $s:=\psi(t)\cdot n_{\Gamma_t}$ does not depend on any choice of diffeomorphism, since the Taylor sign condition asserts that $a:=-\nabla P|_{\Gamma_t}\cdot n_{\Gamma_t}$ is strictly positive. We will call $s$ the \emph{normal displacement function}, and use it as one of our linearized variables.
\medskip


Next, we must find a suitable expression for the linearization of the kinematic boundary condition. Since $D_t^\pm$ are both tangent to $\Gamma_t$ by \eqref{bb2}, we may apply these vector fields to \eqref{BC1} to obtain
\begin{equation}\label{Mat pressure}
    D_t^\pm P=0 \ \ \ \text{on} \ \Gamma_t.
\end{equation}
Given the dynamic boundary condition and the Taylor sign condition, \eqref{Mat pressure} is equivalent to the kinematic boundary condition. To linearize \eqref{Mat pressure}, we let $\phi_h(t):\Gamma_t\to\Gamma_{t,h}$ be a diffeomorphism, as before. We then have for $x\in \Gamma_t,$
\begin{equation*}
    [(\partial_t+W^\pm_h\cdot \nabla)P_h](t,\phi_h(t)(x))=0.
\end{equation*}
Taking $h$ derivative and evaluating at $h=0$ we see that
\begin{equation}\label{Deriving 2.3}
    w^\pm\cdot\nabla P+(\partial_t+W^\pm \cdot\nabla)(P_{lin})+\nabla(D_t^\pm P)\cdot \psi=0 \ \ \text{on} \ \Gamma_t.
\end{equation}
 We now write $\nabla P|_{\Gamma_t}=\nabla P|_{\Gamma_t}\cdot n_{\Gamma_t}n_{\Gamma_t}=-an_{\Gamma_t}$, split $D_t^\pm(P_{lin})=aD_t^\pm s+sD_t^\pm a$, and use \eqref{Def of s} together with the fact that $\nabla(D_t^\pm P)$ is normal to $\Gamma_t$ by \eqref{Mat pressure}. This reduces \eqref{Deriving 2.3} to
\begin{equation}\label{Deriving 2.3 2}
    -aw^\pm\cdot n_{\Gamma_t}+aD_t^\pm s+sD_t^\pm a+s\nabla(D_t^\pm P)\cdot n_{\Gamma_t}=0.
\end{equation}
After division by $a$ and some algebraic manipulation, we arrive at the transport equations
\begin{equation}\label{Transport equation 1}
    D_t^\pm s-w^\pm\cdot n_{\Gamma_t}=s(n_{\Gamma_t}\cdot \nabla W^\pm)\cdot n_{\Gamma_t}\ \ \text{on} \ \Gamma_t.
\end{equation}
Indeed, the right-hand side of \eqref{Transport equation 1} follows by writing $sD_t^\pm a=-sD_t^\pm (\nabla P\cdot n_{\Gamma_t})=-sD_t^\pm(\nabla P)\cdot n_{\Gamma_t}$ and commuting the gradient with the material derivative in the last term of \eqref{Deriving 2.3 2}. The reason that we have $-sD_t^\pm (\nabla P\cdot n_{\Gamma_t})=-sD_t^\pm (\nabla P)\cdot n_{\Gamma_t}$ is because $\nabla P$ is normal to $\Gamma_t$, while $D_t^\pm n_{\Gamma_t}$ is tangent. The reason that $D_t^\pm n_{\Gamma_t}$ is tangent to $\Gamma_t$ is because $n_{\Gamma_t}$ is unit length.
\medskip

Putting everything together, the linearized system takes the form:
\begin{equation}\label{DM2}
  \begin{cases}
     &D_t^\pm w^\mp+\nabla P_{lin}=-w^\pm\cdot\nabla W^\mp\  \ \text{in} \ \Omega_t,
    \\
    &\nabla\cdot w^\pm=0 \  \ \text{in} \ \Omega_t,
    \\
    & D_t^\pm s-w^\pm\cdot n_{\Gamma_t}=s(n_{\Gamma_t}\cdot \nabla W^\pm)\cdot n_{\Gamma_t}\ \ \text{on} \ \Gamma_t,
    \\
    &P_{lin}|_{\Gamma_t}=as  \ \ \text{on} \ \Gamma_t.
    \end{cases}
\end{equation}
Here we need to clarify why we have two apparent equations for $s$. Subtracting them, we obtain
\begin{equation}\label{lin-constraint}
\nabla_B s -  b \cdot n_{\Gamma_t} = s(n_{\Gamma_t}\cdot \nabla B)\cdot n_{\Gamma_t}
\end{equation}
where $b = \frac12( w^+-w^-)$ is the linearization of $B$. 
This should be seen as a constraint on the space of  linearized states, which is nothing but the linearization of the tangency condition $B \cdot n_{\Gamma_t}=0$.
\medskip

 The natural energy associated to \eqref{DM2} is
\begin{equation}\label{new Linearized energy}
E_{lin}(w^\pm,s)(t)=\frac{1}{2}\int_{\Omega_t}|w^+|^2dx+\frac{1}{2}\int_{\Omega_t}|w^-|^2dx+\int_{\Gamma_t}as^2dS.
\end{equation}
Correspondingly, the linear system can be viewed as a linear evolution on the space of functions
\begin{equation*}
\mathbfcal H^0_{lin} = \{ (w^\pm,s) \in L^2(\Omega_t) \times L^2(\Gamma_t) :\ \nabla \cdot w^\pm = 0, \text{  \eqref{lin-constraint} holds}\}.
\end{equation*}
Here, we remark that the trace of $w^{\pm} \cdot n_{\Gamma_t}$ on $\Gamma_t$ is well defined as an $H^{-\frac{1}{2}}$ distribution
due to the divergence-free condition. 
\medskip

We have the following fundamental energy estimate for the linearized system \eqref{DM2} which will help to inform our choice of higher-order energy and distance functionals later on.
\begin{proposition}\label{linearizedenergy1}
Suppose that $(w^\pm,s)\in\mathbfcal{H}^0_{lin}$ is a (sufficiently regular) solution to the linearized equation \eqref{DM2} with
\begin{equation}\label{control parameter}
A^{\frac{1}{2}}_{lin}(t):=\|(W^+,W^-)\|_{W^{1,\infty}(\Omega_t)}+\|a^{-1}(D_t^+a, D_t^-a)\|_{L^{\infty}(\Gamma_t)}<\infty
\end{equation}
uniformly in time. Then there holds
\begin{equation}\label{EE for DMlin2}
  \left|  \frac{d}{dt}E_{lin}(w^\pm,s)(t)\right|\lesssim A^{\frac{1}{2}}_{lin}(t)E_{lin}(w^\pm,s)(t).
\end{equation}
\end{proposition}
\begin{remark}
It is possible to prove a full well-posedness type theorem for the linearized equation in $\mathbfcal{H}_{lin}^0$ by establishing suitable energy estimates for the adjoint linearized system. However, we will not pursue this  here as we will only need the bound \eqref{EE for DMlin2} in our analysis later on.
\end{remark}
\begin{remark}\label{remark on control parameters}
It is common practice to use $A$ to denote the scale invariant control parameter and $B$ to denote the leading control parameter. However, to avoid confusion with the magnetic field, we will use $A^{\frac{1}{2}}$ for leading control parameters, where the superscript $\frac{1}{2}$ indicates that such quantities are one half derivatives above scale. We also remark that \eqref{EE for DMlin2} can be established with a weakening of  the control parameter $A^{\frac{1}{2}}_{lin}(t)$. More specifically, for the first term in \eqref{control parameter} it suffices to use the homogeneous $\dot{W}^{1,\infty}$ norm. Moreover, the second term may be replaced by $\|a^{-1}D_ta\|_{L^{\infty}(\Gamma_t)}$ as only the symmetrized operator $D^+_t+D_t^-=2D_t$ will fall on the Taylor coefficient $a$. The reason for using the expression in \eqref{control parameter} is simply to avoid introducing additional notation, as in the general linearized system below we will partially decouple the roles of the $\pm$ variables, which will make the less symmetrized control parameter \eqref{control parameter} appear.
\end{remark}

To control the energy in \Cref{linearizedenergy1}, we use the following Leibniz-type formulas.
\begin{proposition}\label{Leibniz}
\begin{enumerate}
\item Assume that the time-dependent domain $\Omega_t$ flows with Lipschitz velocity $v$. Then the time derivative of the time-dependent volume integral is given by
\begin{equation*}
    \frac{d}{dt}\int_{\Omega_t}f(t,x)dx=\int_{\Omega_t}D_tf+f\nabla\cdot vdx.
\end{equation*}
\item Assume that the time-dependent hypersurface $\Gamma_t$ flows with divergence-free velocity $v$. Then the time derivative of the time-dependent surface integral is given by
\begin{equation*}\label{Leibniz derivative under int general}
    \frac{d}{dt}\int_{\Gamma_t}g(t,x)dS=\int_{\Gamma_t}D_tg-g(n_{\Gamma_t}\cdot\nabla v)\cdot n_{\Gamma_t}dS.
\end{equation*}

\end{enumerate}

\end{proposition}
\begin{proof}[Proof of \Cref{linearizedenergy1}]
By the divergence theorem and the fact that $B$ is tangent to $\Gamma_t$,  we may replace the pair $(D_t,v)$ in \Cref{Leibniz} with either $(D_t^+,W^+)$ or $(D_t^-, W^-)$, whenever convenient. A simple computation  then shows that
\begin{equation*}\label{ELIN ESTIMATE2}
\begin{split}
    \frac{d}{dt}E_{lin}(w^\pm,s)(t)=&-\int_{\Omega_t} w^+\cdot (w^-\cdot\nabla W^+)dx-\int_{\Omega_t} w^-\cdot (w^+\cdot\nabla W^-)dx
    \\
    &+\frac{1}{2}\int_{\Gamma_t}s^2(D_t^++D_t^-)adS+\frac{1}{2}\int_{\Gamma_t}as^2n_{\Gamma_t}\cdot\nabla \left[W^++W^-\right]\cdot n_{\Gamma_t}dS,
    \end{split}
\end{equation*}
which implies \eqref{EE for DMlin2}.
\end{proof}
To prove estimates for our higher order energy functionals in \Cref{HEB}, we will need to work also with the \emph{generalized linear system},
\begin{equation}\label{DM gen}
  \begin{cases}
     &D_t^\pm w^\mp+\nabla P_{lin}^\mp=f^\pm\  \ \text{in} \ \Omega_t,
    \\
    &\nabla\cdot w^\pm=0 \  \ \text{in} \ \Omega_t,
    \\
    & D_t^\pm s^\pm-w^\pm\cdot n_{\Gamma_t}=g^\pm\ \ \text{on} \ \Gamma_t,
   \\
   &P_{lin}^\pm|_{\Gamma_t}=as^\pm  \ \ \text{on} \ \Gamma_t.
   \end{cases}
\end{equation}
The above formulation allows for arbitrary source terms $f^\pm$, $g^\pm$ and slightly more general variables $s^{\pm}$.
This eliminates the need for the constraint condition
\eqref{lin-constraint} and uncouples the pairs of variables $(w^+,s^+)$ and $(w^-,s^-)$. By contrast, these 
variables are coupled in \eqref{DM2}, though only in a weak, perturbative fashion.  
\medskip

The natural energy associated to \eqref{DM gen} is
\begin{equation*}\label{new Linearized energy2}
E_{glin}(w^\pm,s^\pm)(t)=\frac{1}{2}\int_{\Omega_t}|w^+|^2dx+\frac{1}{2}\int_{\Omega_t}|w^-|^2dx+\frac{1}{2}\int_{\Gamma_t}a|s^+|^2dS+\frac{1}{2}\int_{\Gamma_t}a|s^-|^2dS.
\end{equation*}
In analogy with \Cref{linearizedenergy1}, we  have the following estimates.
\begin{proposition}\label{linearizedenergy2}
Suppose $w^\pm\in L^2(\Omega_t)$ and $s^\pm\in L^2(\Gamma_t)$ give a (sufficiently regular) solution to the generalized linearized equation \eqref{DM gen} with $A^{\frac{1}{2}}_{lin}(t)<\infty$
uniformly in time. Then there holds
\begin{equation}\label{Linear EE MHD}
\begin{split}
 \frac{d}{dt}E_{glin}(w^\pm,s^\pm)(t)\lesssim  A^{\frac{1}{2}}_{glin}(t)E_{glin}(w^\pm,s^\pm)(t)&+\langle as^+,g^+\rangle_{L^2(\Gamma_t)}+\langle as^-,g^-\rangle_{L^2(\Gamma_t)}
 \\
 &+\langle w^+,f^-\rangle_{L^2(\Omega_t)}+\langle w^-,f^+\rangle_{L^2(\Omega_t)}.
 \end{split}
\end{equation}
\end{proposition}
\begin{proof}
This is a straightforward computation that follows along similar lines to \Cref{linearizedenergy1}. We omit the details.
\end{proof}
In \Cref{HEB}, we will construct Alinhac style good variables which solve \eqref{DM gen} for an appropriate choice of perturbative source terms $f^\pm$, $g^\pm$.

\section{Notation and function spaces}\label{AOMD} In this section, we recall the function space framework from \cite{Euler} and set notation. 
\subsection{Function spaces} Throughout the paper, $\Omega\subseteq \mathbb{R}^d$ will denote a bounded,  connected domain. We define $H^s(\Omega)$ as the set of all $f\in L^2(\Omega)$ such that
\begin{equation*}\label{Hs norm}
\|f\|_{H^s(\Omega)}:=\inf \left\{\|F\|_{H^s(\mathbb{R}^d)} :F\in H^s(\mathbb{R}^d), \ F|_{\Omega}=f \right\}
\end{equation*}
is finite. Here, $\|\cdot\|_{H^s(\mathbb{R}^d)}$ is defined in the standard way, via the Fourier transform. We let $H^s_0(\Omega)$ denote the closure of $C_0^\infty(\Omega)$ in $H^s(\Omega)$. Evidently, with these definitions, the constants in Sobolev embedding theorems are independent of $\Omega$. 
\medskip

The regularity of a bounded connected domain $\Omega$ is defined in terms of the regularity of  local coordinate parameterizations of $\partial\Omega$. More precisely, an $m$-dimensional manifold $\mathcal{M}\subseteq \mathbb{R}^d$ is said to be of class $C^{k,\alpha}$ or $H^s$, $s>\frac{d}{2}$, if,  locally in linear frames, $\mathcal{M}$ can be represented by graphs with the same regularity. 
\medskip

Suppose now that $s>\frac{d+1}{2}$ and  $\Omega$ has a boundary of class $H^{s}$. Given $r\in [-s,s]$,  the Sobolev space $H^r(\partial\Omega)$ consists of all functions $f:\partial\Omega\to \mathbb{R}$ whose coordinate representatives are locally in $H^r(\mathbb{R}^{d-1})$. For $s$, $r$ and $\Omega$ as above, it is  easy to see that $H^r(\partial\Omega)$ is a Banach space. Indeed, a norm can be chosen by selecting a covering of $\partial\Omega$ by a finite number of coordinate patches and an adapted partition of unity. Although such a norm is well-defined up to equivalence, the precise value of the norm is dependent on the choice of  local coordinates. Since we will be dealing with a  family of moving domains, we need to make sure that we define norms on their boundaries in a consistent fashion.
\subsection{Collar coordinates}\label{Collarcoords} Consider a bounded, connected reference domain $\Omega_*$ with a smooth boundary $\Gamma_*$ and define $H^s$ and $C^{k,\alpha}$ based norms on $\Gamma_*$ by selecting  local coordinates. Our objective will be to use this fixed set of coordinates on $\Gamma_*$ to define norms on a family of nearby hypersurfaces. To make this precise, we let $\delta>0$ and define $N(\Gamma_*,\delta)$ as the collection of all $C^1$ hypersurfaces $\Gamma$ for which there exists a $C^1$ diffeomorphism $\Phi_\Gamma: \Gamma_*\to \Gamma$ satisfying 
\begin{equation*}
\|\Phi_\Gamma -Id_{\Gamma_*}\|_{C^1(\Gamma_*)}<\delta.
\end{equation*}
If $\delta>0$ is sufficiently small, hypersurfaces $\Gamma\in N(\Gamma_*,\delta)$ may be viewed as graphs over $\Gamma_*$. Indeed,  one may select a smooth unit vector field $\nu:\Gamma_*\to\mathbb{S}^{d-1}$ which is suitably transversal to $\Gamma_*$ and then use an  implicit function theorem argument (see~\cite[Section 2.1]{MR2763036} for details) to prove the  existence of a $\delta>0$ such that the map 
 \begin{equation*}\label{diff}
     \varphi:\Gamma_*\times[-\delta,\delta]\to\mathbb{R}^d, \ \ \varphi(x,\mu)=x+\mu\nu(x)
 \end{equation*}
 is a $C^1$ diffeomorphism from its domain to a neighborhood of $\Gamma_*$. If  $\delta>0$ is small enough,  the above coordinate system pairs each hypersurface $\Gamma\in N(\Gamma_*, \delta)$ with a unique function $\eta_\Gamma:\Gamma_*\to\mathbb{R}$ such that 
 \begin{equation*}\label{PhiGamma}
    \Phi_\Gamma(x):=\varphi(x,\eta_\Gamma(x))=x+\eta_\Gamma(x)\nu(x)
 \end{equation*}
 is a diffeomorphism in $C^1(\Gamma_*,\Gamma\subseteq \mathbb{R}^d)$. Using the above framework, we may consistently define   Sobolev and H\"older norms on hypersurfaces close to $\Gamma_*$.
 \begin{definition}
Fix $0<\delta\ll 1.$ Given $s\geq 0$ and a hypersurface $\Gamma\in N(\Gamma_*,\delta)$ with associated map $\eta_\Gamma:\Gamma_*\to\mathbb{R}$ satisfying $\eta_\Gamma\in H^s(\Gamma_*)$  we  define the $H^s$ norm of $\Gamma$ by 
\begin{equation*}
\|\Gamma\|_{H^s}:=\|\eta_\Gamma\|_{H^s(\Gamma_*)}.
\end{equation*}
 \end{definition}
 In an analogous way, we define for $\alpha\in [0,1)$ and integers $k\geq 0$ the  norm $\|\Gamma\|_{C^{k,\alpha}}$.  Our analysis will take place in  the following control neighborhoods.
\begin{definition}
For $\delta>0$ small enough and $\alpha\in [0,1)$, we define the control neighborhood $\Lambda (\Gamma_*,\alpha,\delta)$ as the collection of all hypersurfaces $\Gamma\in N(\Gamma_*,\delta)$ which have an associated map $\eta_\Gamma:\Gamma_*\to\mathbb{R}$ satisfying
$$\|\eta_\Gamma\|_{C^{1,\alpha}(\Gamma_*)}<\delta.$$ 
\end{definition}
In practice, we will want our control neighborhood to be as weak as possible. We therefore commonly abbreviate $\Lambda_*:=\Lambda (\Gamma_*,\epsilon,\delta)$, where $0<\epsilon,\delta\ll 1$ are understood to be small but universal positive constants. As noted in \cite[Section 2.1]{MR2763036}, when $0<\delta\ll 1$ we may associate each $\Gamma\in \Lambda(\Gamma_*,\alpha,\delta)$ with a well-defined domain $\Omega$. 

\begin{remark}
Using the above functional setting, one may consistently define Sobolev norms for functions on $\partial\Omega$ and prove that the implicit constants in various fundamental estimates are uniformly bounded for domains in the collar. Precise details can be found in \Cref{BEE}.
\end{remark}

\subsection{The state space}\label{statespace} Given a collar neighborhood $\Lambda_*$ and $s>\frac{d}{2}+1$ the
\emph{state space}  $\mathbf{H}^s$ is the set of all triples $(v,B,\Gamma)$ such that $\Gamma\in \Lambda_*$ is the boundary of a bounded, connected domain $\Omega$ and such that the following conditions are satisfied:
\begin{enumerate}
    \item (Regularity). $v,B\in H_{div}^s(\Omega)$ and $\Gamma\in H^s$, where $H_{div}^s(\Omega)$ denotes the space of divergence-free vector fields in $H^s(\Omega)$.
    \item (Taylor sign condition). $a:=-\nabla P\cdot n_{\Gamma}>c_0>0$, where $c_0$ may depend on the choice of $(v,B,\Gamma)$, and the pressure $P$ is obtained from $(v,B,\Gamma)$ by solving the  elliptic equation \eqref{Euler-pressure} associated to \eqref{Euler} and \eqref{BC1}.
    \item (Tangency of $B$).   $B\cdot n_\Gamma=0$ on $\Gamma.$
    \item (Wave-type regularity condition).  $\nabla_Bv, \nabla_BB\in H^{s-\frac{1}{2}}(\Omega)$.
\end{enumerate}
For  data $(v_0,B_0,\Gamma_0)$ in the state space $\mathbf{H}^s$, our objective will be to construct local solutions $(v(t),B(t),\Gamma_t)$ to the free boundary MHD equations which evolve continuously in $\mathbf{H}^s$. However, in order to consider the continuity of solutions with values in $\mathbf{H}^s$ and the continuous dependence of solutions $(v(t),B(t),\Gamma_t)$ as functions of the initial data $(v_0,B_0,\Gamma_0)$, we must first define a suitable notion of topology on $\mathbf{H}^s$.
\medskip

 To measure the size of individual states $(v,B,\Gamma)\in \mathbf{H}^s$, we define 
 \begin{equation}\label{Size of states.}
 \|(v,B,\Gamma)\|_{\mathbf{H}^s}^2:=\|\Gamma\|_{H^s}^2+\|v\|^2_{H^s(\Omega)}+\|B\|_{H^s(\Omega)}^2+\|\nabla_Bv\|_{H^{s-\frac{1}{2}}(\Omega)}^2+\|\nabla_BB\|_{H^{s-\frac{1}{2}}(\Omega)}^2.
 \end{equation}
 However, since  $\mathbf{H}^s$ is not a linear space, \eqref{Size of states.} does not define a norm.  To remain consistent with the literature \cite{disconzi2020relativistic,Euler,ifrim2020compressible}, we  use \eqref{Size of states.} to define a convergence in $\mathbf{H}^s$.
\begin{definition}\label{Def of convergence} We say that a sequence $(v_n,B_n,\Gamma_n)\in \mathbf{H}^s$ converges to $(v,B,\Gamma)\in \mathbf{H}^s$ if 
\begin{enumerate}
\item (Uniform Taylor sign condition). For some $c_0>0$ independent of $n$, we have
\begin{equation*}
a_n, a>c_0>0.     
\end{equation*}
\item\label{def domain conver} (Domain convergence). $\Gamma_n\to \Gamma$ 
 in $H^s$. That is, $\eta_{\Gamma_n}\to \eta_{\Gamma}$ in $H^s(\Gamma_*)$ where $\eta_{\Gamma_n}$ and $\eta_\Gamma$ correspond to the collar coordinate representations of $\Gamma_n$ and $\Gamma$, respectively. 
 \item\label{def norm conver} (Norm convergence). For every $\epsilon>0$ there exist smooth divergence-free functions $\tilde{v},\tilde{B}$ and smooth functions $f,g$ defined on a neighborhood $\tilde{\Omega}$ of $\overline{\Omega}$ with   $$\|\tilde{v}\|_{H^s(\tilde\Omega)}+\|\tilde{B}\|_{H^s(\tilde\Omega)}+\|f\|_{H^{s-\frac{1}{2}}(\tilde\Omega)}+\|g\|_{H^{s-\frac{1}{2}}(\tilde\Omega)}<\infty$$ and satisfying 
 \begin{equation*}
\|v-\tilde{v}\|_{H^s(\Omega)}+\|B-\tilde{B}\|_{H^s(\Omega)}+\|\nabla_Bv-f\|_{H^{s-\frac{1}{2}}(\Omega)}+\|\nabla_BB-g\|_{H^{s-\frac{1}{2}}(\Omega)}\leq \epsilon
 \end{equation*}
 and
 \begin{equation*}
 \limsup_{n\to \infty}\left(\|v_n-\tilde{v}\|_{H^s(\Omega_n)}+\|B_n-\tilde{B}\|_{H^s(\Omega_n)}+\|\nabla_{B_n}v_n-f\|_{H^{s-\frac{1}{2}}(\Omega_n)}+\|\nabla_{B_n}B_n-g\|_{H^{s-\frac{1}{2}}(\Omega_n)}\right)\leq\epsilon.    
 \end{equation*}
 \end{enumerate}
With the above notion of convergence, the meaning of $C([0,T];\mathbf{H}^s)$ is now clear. We remark that the equicontinuity-type condition in  property (iii) above is natural to ensure that the $H^s$ mass of the sequence $(v_n,B_n)$ does not concentrate in thin layers near the boundary.
\begin{remark}
At certain points in the paper, it will be convenient to consider convergence more generally: Given $(f,\partial\Omega)$ and a sequence $(f_n,\partial\Omega_n)$ with $f:\Omega\to \mathbb{R}$ and $f_n:\Omega_n\to \mathbb{R}$, we  define the convergence $(f_n,\partial\Omega_n)\to (f,\partial\Omega)$ by requiring both the domain convergence in \eqref{def domain conver} and the convergence of $f_n$ to $f$ via intermediate functions $\tilde{f}$ as in \eqref{def norm conver}. 
\end{remark}


\end{definition}
\section{Difference estimates and enhanced uniqueness}\label{DF}
 The  objective of  this section is to establish a  Lipschitz bound  for the $L^2$
distance between  solutions to the free boundary MHD equations. The fundamental difficulty in achieving these  bounds is the need to compare states which live on different domains. To overcome this difficulty, we  construct a ``distance functional" which measures the distance between (functions on) different domains and is compatible with the MHD flow. Notably, our distance bounds propagate at the level  of the control parameters \eqref{ACONT} and \eqref{BCONT}, which require very little in terms of a priori regularity. This is what will allow us to establish uniqueness of solutions in a very large class. Moreover, as we shall see,  the distance bounds that we prove in this section will also serve as an essential ingredient in our construction of rough solutions  as well as in our proof of the continuity of the data-to-solution map. 
\subsection{Difference estimates}
To set the stage,  we begin by fixing a collar neighborhood $\Lambda_*:=\Lambda(\Gamma_*,\epsilon,\delta)$, where  $0<\epsilon,\delta\ll 1$. Given  states $W:=(W^\pm,\Gamma)$ and $W_h:=(W^\pm_h,\Gamma_h)$ with respective domains $\Omega$, $\Omega_h$, we let $\eta_{\Gamma}$ and $\eta_{\Gamma_{h}}$ be the corresponding representations of $\Gamma$ and $\Gamma_{h}$ as graphs over $\Gamma_*$. Following the linearized energy estimate, we aim to construct analogues of the linearized variables $w^\pm$ and $s$. However, as we shall see, this is not a completely straightforward task.
\medskip

We define $\widetilde{\Omega}:=\Omega\cap\Omega_{h}$ and represent the free boundary $\widetilde{\Gamma}$ for $\widetilde{\Omega}$ as a graph over $\Gamma_*$ via the function $\eta_{\widetilde{\Gamma}}=\eta_{\Gamma}\wedge \eta_{\Gamma_h}$. Note that although the graph parameterization $\eta_{\widetilde{\Gamma}}$ is well-defined,  $\widetilde{\Gamma}$ is only Lipschitz in general, so will not be in $\Lambda_*$. 
To measure the (signed) distance between $\Gamma$ and $\Gamma_h$, we define $s_h^*:\Gamma_*\to \mathbb{R}$ by
\begin{equation*}\label{distancefunction}
    s_h^*(x)=\eta_{\Gamma_h}(x)-\eta_{\Gamma}(x).
\end{equation*}
We then consider the variable $s_h:\widetilde{\Gamma}\to \mathbb{R}$ obtained by pushing  $s_h^*$ forward to the hypersurface $\widetilde{\Gamma}$. More precisely, for $x\in \widetilde{\Gamma},$ we define $s_h(x):=s_h^*(\pi(x))$, where $\pi$ denotes the canonical projection, mapping the image of $\Gamma_*\times [-\delta,\delta]$ under $\varphi$ back to $\Gamma_*$. We similarly extend $\nu$ to a vector field $X$ defined  on the image of $\varphi$ via $X(x)=\nu(\pi(x)).$ 
\medskip

Although the displacement function $s_h$ directly measures the distance between the hypersurfaces $\Gamma$ and $\Gamma_h$ in collar coordinates, it is not tailored to the MHD flow. Therefore, we will not use it as our analogue of the linearized variable $s$. Instead, we will measure the distance between $\Gamma$ and $\Gamma_h$ via the pressure difference $P-P_h$. To motivate this, recall that \eqref{BC1} and the Taylor sign condition ensure that  $P$ and $P_h$ are non-degenerate defining functions for $\Gamma$ and $\Gamma_{h}$ within a suitable collar neighborhood. Therefore, on $\widetilde{\Gamma}$, $P-P_h$ is proportional to the displacement function $s_h$. More precisely, letting $\overline{F}$ denote the average of $F$ along the flow $\varphi$ between the free hypersurfaces, the  fundamental theorem of calculus implies that for $x\in \widetilde{\Gamma}$,
\begin{equation}\label{FTC pressure average}
P_h(x)-P(x)=
\begin{cases}
&-\overline{\nabla P_h\cdot X}s_h(x) \hspace{10mm}\text{if $x\in \mathcal{A}$}:=\tilde{\Gamma}\cap\Gamma-\Gamma\cap\Gamma_{h},
\\
&-\overline{\nabla P\cdot X}s_h(x)\hspace{12mm}\text{if $x\in \mathcal{A}_{h}$}:=\tilde{\Gamma}\cap\Gamma_{h}-\Gamma\cap\Gamma_{h}.
\end{cases}
\end{equation}
Therefore, assuming the Taylor sign condition and the regularity $P,P_h\in C^{1,\epsilon}$, we have the relation $$|P-P_h|\approx |s_h| \hspace{3mm} \text{on}\hspace{2mm} \tilde{\Gamma},$$ within a tight enough collar neighborhood. The key point here is that although  $P-P_h$ and $s_h$ measure  distance equally well, the former has much more natural dynamics under the MHD flow.
\medskip

Motivated by the above, we define
\begin{equation}\label{diff functional candidate pm}
\begin{split}
D^{\pm}(W,W_h):=\frac{1}{2}\int_{\tilde{\Omega}}|W_h^\pm-W^\pm|^2\,dx+\frac{1}{2}\int_{\mathcal{A}}a^{-1}|P_h-P|^2\, dS+\frac{1}{2}\int_{\mathcal{A}_{h}}a_h^{-1}|P_h-P|^2\, dS   
\end{split}
\end{equation}
and
\begin{equation}\label{diff functional candidate}
D(W,W_h):=D^+(W,W_h)+D^-(W,W_h).    
\end{equation}
 Note that the latter two terms in \eqref{diff functional candidate pm} may be written as
\begin{equation*}\label{diff functional candidate2}
\begin{split}
\int_{\mathcal{A}}a^{-1}|P_h-P|^2\, dS+\int_{\mathcal{A}_{h}}a_h^{-1}|P_h-P|^2\, dS  =\int_{\tilde{\Gamma}}b|P-P_h|^2\, dS,
\end{split}
\end{equation*}
where the weight function
\begin{equation*}
b:=a^{-1}1_{\tilde{\Gamma}\cap\Gamma}+a_h^{-1}1_{\tilde{\Gamma}\cap\Gamma_h}
\end{equation*}
is chosen so that \eqref{diff functional candidate}  recovers the linearized energy \eqref{new Linearized energy} in the formal limit. Our objective is to use the above distance functional to propagate difference bounds for solutions to the free boundary MHD equations.  
\begin{theorem}[Difference Bounds]\label{Difference} 
 Let $0<\epsilon,\delta\ll 1$ and let $\Lambda_*=\Lambda(\Gamma_*,\epsilon,\delta)$ be a collar neighborhood. Suppose that $(W^\pm,\Gamma_t)$ and $(W^\pm_h,\Gamma_{t,h})$ are solutions to the free boundary MHD equations that evolve in the collar in a time interval $[0,T]$ and satisfy  $a$,$a_h>c_0>0$.  Then we have the estimate
\begin{equation*}
\frac{d}{dt}D(W,W_h)\lesssim_{A,A_h} (A^{\frac{1}{2}}+A^{\frac{1}{2}}_h)D(W,W_h)
\end{equation*}
where
\begin{equation*}
\begin{split}
A^{\frac{1}{2}}&:=\|(W^+,W^-)\|_{W^{1,\infty}(\Omega_t)}+\|\Gamma_t\|_{C^{1,\frac{1}{2}}}+\|(D_t^+,D_t^-)P\|_{W^{1,\infty}(\Omega_t)},
\\
A&:=\|(W^+,W^-)\|_{C^{\frac{1}{2}+\epsilon}(\Omega_t)}+\|\Gamma_t\|_{C^{1,\epsilon}},
\end{split}
\end{equation*}
 $A^{\frac{1}{2}}_h$ and $A_h$ are the analogous quantities corresponding to $W_h^\pm$, $P_h$, $D_t^hP_h$ and $\Gamma_{t,h}$, and we  have implicitly assumed that our solutions have regularity $A^{\frac{1}{2}},A^{\frac{1}{2}}_h\in L^1_T$ and $A,A_h\in L^\infty_T$.
\end{theorem}
\begin{remark}\label{remark on cp 2}
It will be clear from the proof below that the control parameter $A^{\frac{1}{2}}$ in \Cref{Difference} may be replaced with the control parameter  mentioned in \Cref{slight improvement}. 
\end{remark}
\begin{proof}
Clearly, it suffices to show that
\begin{equation*}
\frac{d}{dt}D^{\pm}(W,W_h)\lesssim_{A,A_h}(A^{\frac{1}{2}}+A_h^{\frac{1}{2}})D(W,W_h) 
\end{equation*}
for  $D^+$ and $D^-$ separately. However, some interesting cancellations will occur when we view these terms together.
\medskip

For  notational simplicity, we drop the $t$ subscript on the domains below. We also use $\lesssim_A$ as a shorthand for $\lesssim_{A,A_h}$. To estimate expressions involving the pressure in terms of the control parameters $A$ and $A^\frac{1}{2}$ above, we require the  bounds
\begin{equation}\label{pressurebounddiff}
\|P\|_{C^{1,\epsilon}(\Omega)}\lesssim_A 1,\hspace{5mm}\|P\|_{C^{1,\frac{1}{2}}(\Omega)}\lesssim_A A^\frac{1}{2}   ,
\end{equation}
as well as the analogous bounds for $P_h$. Proofs that these bounds hold will be presented later; see \Cref{Linfest} for details. 
\medskip

To proceed with the difference estimate, we recall the identity
\begin{equation}\label{dt-dist} 
\begin{split}
\frac{d}{dt}D^\pm(W,W_h)&=\frac{1}{2}\frac{d}{dt}\int_{\tilde{\Omega}}|W^\pm_h-W^\pm|^2\, dx+\frac{1}{2}\frac{d}{dt}\int_{\mathcal{A}}a^{-1}|P-P_h|^2\,dS+\frac{1}{2}\frac{d}{dt}\int_{\mathcal{A}_h}a_h^{-1}|P-P_h|^2\,dS.
\end{split}
\end{equation}
To compute the first term in \eqref{dt-dist}, we use \cite[Proposition 4.4]{Euler} with  velocity $W^\mp$  to  obtain the estimate
\begin{equation}\label{first term}
\begin{split}
\frac{1}{2}\frac{d}{dt}\int_{\tilde{\Omega}}|W^\pm_h-W^\pm|^2\, dx\leq \frac{1}{2}\int_{\tilde{\Omega}}D_t^\mp|W^\pm_h-W^\pm|^2\, dx +\frac{1}{2}\int_{\tilde{\Gamma}}|W^\pm_h-W^\pm|^2|W^\mp_h-W^\mp|\,dS.
\end{split}    
\end{equation}
Note that the latter term in  \eqref{first term} arises as a consequence of  working on the intersected domain $\widetilde{\Omega}$ and estimating the additional boundary weight in \cite[Proposition 4.4]{Euler} by  $|W^\mp_h-W^\mp|$. 
Since this new term is cubic in $|W^\pm_h-W^\pm|$ and $|W^\mp_h-W^\mp|$, it is straightforward to handle. Indeed, as $\Gamma,\Gamma_{h}\in\Lambda_*$, we may find a smooth vector field $X$ defined on $\mathbb{R}^d$ with $C^k$ bounds uniform in $\Lambda_*$ which is also uniformly transverse to $\tilde{\Gamma}$. By the  divergence theorem, we then have
\begin{equation}\label{trace1}
\begin{split}
\frac{1}{2}\int_{\tilde{\Gamma}}|W^\pm_h-W^\pm|^2|W^\mp_h-W^\mp|\,dS&\lesssim \int_{\tilde{\Gamma}}X\cdot n_{\tilde{\Gamma}}|W^\pm_h-W^\pm|^2|W^\mp_h-W^\mp|\,dS
\\
&\lesssim (A^\frac{1}{2}+A_h^\frac{1}{2})(\|W^\pm_h-W^\pm\|_{L_x^2(\tilde{\Omega})}^2+\|W^\mp_h-W^\mp\|_{L_x^2(\tilde{\Omega})}^2)
\\
&\lesssim (A^\frac{1}{2}+A_h^\frac{1}{2})D(W,W_h).
\end{split}
\end{equation}
To estimate the remaining term in \eqref{first term}, we recall the following equations for $W^\pm_h-W^\pm$ in $\tilde{\Omega}_t$:
\begin{equation*}\label{veqn}
\begin{cases}
&D_t^\mp(W^\pm_h-W^\pm)=-\nabla(P_h-P)-(W^\mp_h-W^\mp)\cdot\nabla W_h^\pm,
\\
&\nabla\cdot (W^\pm_h-W^\pm)=0.
\end{cases}
\end{equation*}
Using these equations, it follows that
\begin{equation}\label{interior1}
\begin{split}
\frac{1}{2}\int_{\tilde{\Omega}}D_t^\mp|W^\pm_h-&W^\pm|^2\,dx=\int_{\tilde{\Omega}}(W^\pm_h-W^\pm)\cdot D_t^\mp(W^\pm_h-W^\pm)\,dx
\\
&=-\int_{\tilde{\Gamma}}(P_h-P)(W^\pm_h-W^\pm)\cdot n_{\tilde{\Gamma}}\,dS-\int_{\tilde{\Omega}}(W^\pm_h-W^\pm)\cdot [(W^\mp_h-W^\mp)\cdot\nabla W_h^\pm]\,dx   
\\
&\leq -\int_{\tilde{\Gamma}}(P-P_h)(W^\pm-W^\pm_h)\cdot n_{\tilde{\Gamma}}\,dS+(A^\frac{1}{2}+A_h^\frac{1}{2})D(W,W_h).
\end{split}
\end{equation}
From the decomposition $\tilde{\Gamma}=\mathcal{A}\cup \mathcal{A}_h\cup (\Gamma\cap \Gamma_{h})$ and the fact that $P-P_h=0$ on $\Gamma\cap \Gamma_{h}$ by the dynamic boundary condition \eqref{BC1}, we may write
\begin{equation*}
\begin{split}
-\int_{\tilde{\Gamma}}(P-P_h)(W^\pm-W^\pm_h &)\cdot n_{\tilde{\Gamma}}\,dS=-\int_{\mathcal{A}}(P-P_h)(W^\pm-W_h^\pm)\cdot n_{\Gamma}\,dS-\int_{\mathcal{A}_h}(P-P_h)(W^\pm-W_h^\pm)\cdot n_{\Gamma_{h}}\,dS
\\
&=\int_{\mathcal{A}}a^{-1}(P-P_h)(W^\pm-W_h^\pm)\cdot\nabla P\,dS+\int_{\mathcal{A}_h}a_h^{-1}(P-P_h)(W^\pm-W_h^\pm)\cdot\nabla P_h \,dS.
\end{split}
\end{equation*}
We now define
\begin{equation*}
J^\pm:=\int_{\mathcal{A}}a^{-1}(P-P_h)(W^\pm-W_h^\pm)\cdot\nabla P\,dS+\frac{1}{2}\frac{d}{dt}\int_{\mathcal{A}}a^{-1}|P-P_h|^2\,dS,    
\end{equation*}
and
\begin{equation*}
J_{h}^\pm:=\int_{\mathcal{A}_h}a_h^{-1}(P-P_h)(W^\pm-W_h^\pm)\cdot\nabla P_h \,dS+\frac{1}{2}\frac{d}{dt}\int_{\mathcal{A}_h}a_h^{-1}|P-P_h|^2\,dS.    
\end{equation*}
Combining \eqref{trace1} with \eqref{interior1}, we  obtain
\begin{equation*}
\frac{d}{dt}D^\pm(W,W_h)\lesssim (A^\frac{1}{2}+A_h^\frac{1}{2})D(W,W_h)+J^\pm+J_{h}^\pm.    
\end{equation*}
It remains  to show that
\begin{equation*}
J^++J^-+J_h^++J_h^-\lesssim_A (A^\frac{1}{2}+A_h^\frac{1}{2})D(W,W_h). 
\end{equation*}
We will only show the details for $J:=J^++J^-$ as the treatment of $J_{h}^++J_h^-$ is virtually identical. 
\medskip

Notice that when $J^+$ and $J^-$ are combined we obtain the identity
\begin{equation}\label{B cancels out}
J=2\int_{\mathcal{A}}a^{-1}(P-P_h)(v-v_h)\cdot\nabla P\,dS+\frac{d}{dt}\int_{\mathcal{A}}a^{-1}|P-P_h|^2\,dS.   
\end{equation}
Interestingly,  \eqref{B cancels out} makes no reference to the magnetic field $B$, except implicitly through the pressure $P$. Therefore, to estimate the second term in \eqref{B cancels out}, it is best to use \Cref{Leibniz} with the pair $(D_t,v)$ rather than $(D_t^\pm, W^\pm)$. In this case, we have
\begin{equation}\label{Int on s_h>0}
\begin{split}
\frac{d}{dt}\int_{\mathcal{A}}a^{-1}|P-P_h|^2\,dS&=-\int_{\mathcal{A}}a^{-2}D_ta|P-P_h|^2\,dS-\int_{\mathcal{A}}a^{-1}|P-P_h|^2 [n_{\Gamma}\cdot\nabla v\cdot n_{\Gamma}]\,dS
\\
&+2\int_{\mathcal{A}}a^{-1}(P-P_h)D_t(P-P_h)\,dS.
\end{split}    
\end{equation}
The validity of the identity \eqref{Int on s_h>0} is justified by noting that $|P-P_h|^2$ vanishes to second order on $\Gamma\cap \Gamma_h$, so one can extend by zero to  write the integral on the left-hand side of \eqref{Int on s_h>0}  as an integral over $\Gamma$, apply standard identities there, and then return to an integral over $\mathcal{A}$. Note that in \eqref{Int on s_h>0} only the combination $D_t=\frac{D_t^++D_t^-}{2}$ falls on the pressure. This is consistent with the linearized estimates; see \Cref{remark on control parameters}.  It is also the structural reason why we may use the improved control parameter from \Cref{slight improvement} in this portion of our analysis.
\medskip

Inserting  \eqref{Int on s_h>0} into \eqref{B cancels out} and recalling that $D_tP=0$ on $\mathcal{A}$ by the kinematic and dynamic boundary conditions, we have 
\begin{equation*}
J\lesssim_A -2\int_{\mathcal{A}}a^{-1}(P-P_h)D_{t}^{h}P_hdS+2\int_{\mathcal{A}}a^{-1}(P-P_h)(v-v_h)\cdot\nabla (P-P_h)dS+(A^\frac{1}{2}+A_h^\frac{1}{2})D(W,W_h).
\end{equation*}
Here, we used the identity \eqref{Moving normal} and also $2D_t=D_t^++D_t^-$ to control $D_ta$. For our next estimate, we combine the fact that $D_{t}^hP_h=0$ on $\Gamma_{h}$ with \eqref{FTC pressure average}, the fundamental theorem of calculus, the Taylor sign condition and \eqref{pressurebounddiff} to obtain
\begin{equation*}
\begin{split}
    |D_{t}^hP_h|&\lesssim_A \|\nabla D_{t}^hP_h\|_{L^{\infty}}|s_h|\approx_A \|\nabla D_{t}^hP_h\|_{L^{\infty}}|P-P_h|\lesssim_A (A^\frac{1}{2}+A_h^\frac{1}{2})|P-P_h|.
\end{split}
\end{equation*}
As a consequence, we deduce a good bound on the first component of the estimate for $J$:
\begin{equation*}
\left|\int_{\mathcal{A}}a^{-1}(P-P_h)D_{t}^hP_hdS\right|\lesssim_A (A^\frac{1}{2}+A_h^\frac{1}{2})D(W,W_h).
\end{equation*}
The final task is to show that
\begin{equation}\label{last-cubic}
\begin{split}
\left| \int_{\mathcal{A}}a^{-1}(P-P_h)(v-v_h)\cdot\nabla (P-P_h)\, dS \right|
\lesssim_A (A^\frac{1}{2}+A_h^\frac{1}{2})D(W,W_h).
\end{split}
\end{equation}
 The estimate \eqref{last-cubic} is far from trivial. However, other than substituting $p-p_h$ with $P-P_h$, \eqref{last-cubic} has  exactly the same structure as the delicate cubic term in our previous work \cite[Equation (4.11)]{Euler}. Since the analysis from \cite{Euler} carries over rather directly,  we leave the verification of \eqref{last-cubic} to the reader.
\end{proof}
One of several consequences of the above difference bounds is the following uniqueness result. Note that, when written in terms of the variables $(v,B,\Gamma)$, \Cref{t:unique} proves \Cref{Uniqueness intro}.
\begin{theorem}[Uniqueness] \label{t:unique} Let $\epsilon>0$ and let $\Omega_0$ be a bounded domain with boundary $\Gamma_0\in\Lambda(\Gamma_*,\epsilon,\delta)$. Then for $\Gamma_0\in C^{1,\frac{1}{2}}$ and divergence-free $W_0^\pm\in W^{1,\infty}(\Omega_0)$ satisfying the Taylor sign condition, the free boundary MHD equations admit at most one solution $(W^\pm,\Gamma_t)$ on a time interval $[0,T]$ with  $\Gamma_t\in \Lambda(\Gamma_*,\epsilon,\delta)$ and 
\begin{equation*}
\sup_{0\leq t\leq T}\|(W^+,W^-)\|_{C^{\frac{1}{2}+\epsilon}_x(\Omega_t)}+\int_{0}^{T}\|(W^+,W^-)\|_{W^{1,\infty}_x(\Omega_t)}+\|D_tP\|_{W^{1,\infty}_x(\Omega_t)}+\|\Gamma_t\|_{C_x^{1,\frac{1}{2}}}\,dt<\infty.    
\end{equation*}
\end{theorem}
\begin{proof}
This follows immediately from \Cref{Difference}, \Cref{remark on cp 2} and the same reasoning as in \cite[Theorem 4.6]{Euler}.
\end{proof}
\section{Higher energy bounds}\label{HEB}
Let $k>\frac{d}{2}+1$ be an integer. The objective of this section is to establish control over the $\mathbf{H}^k$ norms of solutions $(v, B, \Gamma)$ to the free boundary MHD equations. We accomplish this by constructing a family of coercive energy functionals $(v, B, \Gamma)\mapsto E^k(v, B, \Gamma)$ which are propagated by the MHD flow. 
\begin{theorem}[Low regularity energy estimates]\label{Energy est. thm}
Let $s\in\mathbb{R}$ with $s>\frac{d}{2}+1$ and let $k> \frac{d}{2}+1$ be an integer. Fix a collar neighborhood $\Lambda(\Gamma_*, \epsilon,\delta)$ with $\delta>0$ sufficiently small. Then for $\Gamma$ restricted to $\Lambda_*$ there exists an energy functional $(v, B, \Gamma)\mapsto E^k(v, B, \Gamma)$ such that
\begin{enumerate}
\item (Energy coercivity).
\begin{equation}\label{Coercivity bound on integers}
E^k(v, B, \Gamma)\approx_{M_{s-\frac{1}{2}}} M_k^2.
\end{equation}
\item  (Energy propagation). If, in addition to the above, $(v, B, \Gamma)=(v(t), B(t), \Gamma_t)$ is a solution to the free boundary MHD equations, then $E^k(t):=E^k(v(t), B(t), \Gamma_t)$ satisfies 
\begin{equation*}
\frac{d}{dt}E^k\lesssim_{M_s} E^k.
\end{equation*}
\end{enumerate}
Here, $M_{\sigma}:=1+|\Omega|+\|(v, B,\Gamma)\|_{\mathbf{H}^{\sigma}}$ for $\sigma\geq s-\frac{1}{2}$. Moreover, the implicit constants in the above estimates are sub-polynomial in $M_{\sigma}$.

\end{theorem}
\begin{remark}
It is essential to note that  the first statement in \Cref{Energy est. thm} is valid for general states $(v,B,\Gamma)\in \mathbf{H}^k$.
For solutions $(v(t), B(t), \Gamma_t)$ to the free boundary MHD equations, \Cref{Energy est. thm} may be combined with Gr\"onwall's inequality to obtain the bound
\begin{equation*}\label{doubleexpbound}
\begin{split}
\|(v(t), B(t), \Gamma_t)\|^2_{\mathbf{H}^k}&\lesssim_{M_{s-\frac{1}{2}}}\exp\left(\int_{0}^{t}\mathcal{P}(M_s)d\tau\right)(1+\|(v_0, B_0, \Gamma_0)\|_{\mathbf{H}^k}^2)
\end{split}
\end{equation*}
for some polynomial $\mathcal{P}$. 
\end{remark}
\subsubsection{Notation} In the sequel, we will need to efficiently estimate many multilinear expressions. To make the notation more compact, we will  use $\mathcal{M}_n$ to denote a generic $n$-fold multilinear expression in its arguments. For instance, $\mathcal{M}_2(\nabla P,\nabla^2 P)$ will denote a bilinear expression in $\nabla P$ and $\nabla^2 P$. If the order of the multilinear expression $\mathcal{M}_n$ is not important, we will simply write $\mathcal{M}$ instead of $\mathcal{M}_n$. To further simplify the notation, we will  omit the $\pm$ superscript on $W^\pm$ when it is not important to keep track of. For instance, $\mathcal{M}_2(\nabla W,\nabla W)$ will denote a bilinear expression in any combination of $\nabla W^+$ and $\nabla W^-$. Moreover,  we will  write $(W^\pm,\Gamma)$ as a shorthand for $(W^+,W^-,\Gamma)$.
\medskip

In this section, we will refer to \Cref{BEE} quite frequently, so the reader may wish to consult \Cref{BEE} for additional notation and conventions. In particular, familiarity with the  dyadic regularization operators from \Cref{SSRO}    will be assumed.
\subsection{Constructing the energy functional}\label{CTEF}
In order to establish \Cref{Energy est. thm} we will need to control the $H^k$ norms of $v$, $B$ and $\Gamma$ as well as the $H^{k-\frac{1}{2}}$ norms of $\nabla_B v$ and $\nabla_B B$. Equivalently, we must   control  $W^\pm$ and $\Gamma$ in $H^k$ and $\nabla_{B}W^\pm$ in $H^{k-\frac{1}{2}}$. These latter ``diagonalized" variables are easier to work with, so we opt to phrase our estimates in this language. 
\medskip

To prove our desired  energy estimates, our strategy will be to construct Alinhac style \emph{good variables} which solve the linearized system to leading order. Our choice of good variables is as follows.
\begin{enumerate}[label=\roman*)]
    \item The ``vorticities"  $\omega^\pm:=\nabla\times W^\pm$.  If $V$ is a divergence-free vector field on $\Omega$ and $\omega:=\nabla\times V$  then we have the relation
\begin{equation*}
\Delta V_i=-\partial_j\omega_{ij}.
\end{equation*}
Therefore, $V$ is controlled by $\omega$ and a suitable boundary value. Alternatively, one may view  $V$ as a solution to a div-curl  system,  
 with a boundary condition which we will discuss in detail below.
 
\item The Taylor coefficient $a$. This variable is responsible for describing the boundary regularity. Indeed, as we will see later, we have the proportionality 
\[
\mathcal N a \approx a \kappa 
\]
where $\kappa$ denotes the mean curvature of $\Gamma$. Thus, the $H^k$ norm of $\Gamma$ is  comparable at leading order to the $H^{k-1}(\Gamma)$ norm of $a$, as long as the Taylor sign condition is satisfied. 
\item The variables $\mathcal{G}^\pm:=D_t^\pm a- \nabla_n\Delta^{-1}D_t^\pm\Delta P$. At leading order these variables provide information about $W^\pm$ via the approximate
paradifferential relation
\begin{equation*}
\mathcal{G}^\pm\approx\mathcal{N}T_nW^\pm.
\end{equation*}
We will use this relation to obtain the desired boundary condition for the div-curl system for $W^\pm$. 
\item The variables $\nabla_B\mathcal{G}^\pm$ and $\nabla_B\omega^\pm$. By similar heuristics, these variables  will be used to control $\nabla_BW^\pm$ in $H^{k-\frac{1}{2}}(\Omega)$. 
\end{enumerate}
The above discussion  suggests that at  the principal level we have the correspondences
\[
W^\pm \leftrightarrow (\omega^\pm, \mathcal{G}^\pm),\qquad \nabla_BW^\pm\leftrightarrow(\nabla_B\omega^\pm,\nabla_B\mathcal{G}^\pm), \qquad  \Gamma \leftrightarrow a,
\]
which will serve as the basis for our coercivity property. To obtain the heuristic identification $W^\pm \leftrightarrow (\omega^\pm, \mathcal{G}^\pm)$ we will view $W^\pm$ as solving a div-curl system.
To construct such a system, a natural first idea would be to use the rotational/irrotational decomposition  $W^\pm = W^\pm_{rot}+ W^\pm_{ir}$, where 
\[
\left\{
\begin{aligned}
&\curl W^\pm_{rot} = \omega^\pm, \\ & \nabla \cdot W^\pm_{rot} = 0, \\ & W^\pm_{rot} \cdot n_\Gamma = 0 \ \ \text{on } \Gamma ,  
\end{aligned}
\right.
\qquad 
\left\{
\begin{aligned}
& \curl W^\pm_{ir} = 0, \\ & \nabla \cdot W^\pm_{ir} = 0, \\ & W^\pm_{ir} \cdot n_\Gamma = W^\pm \cdot n_\Gamma \ \ \text{on } \Gamma . 
\end{aligned}
\right.
\]
However, this system cannot be directly used to analyze the free boundary MHD equations, as  $n_\Gamma$ has only $H^{k-1}(\Gamma)$ regularity whereas $W^\pm\in H^{k-\frac12}(\Gamma)$.  Instead, we will  associate the $\mathcal{G}^\pm$ variables with $\nabla^\top W^\pm \cdot n_\Gamma$, the normal component of the tangential derivatives of $W^\pm$   on the boundary; one may think of this as a proxy for the paraproduct $\nabla^\top (T_{n_\Gamma} W^\pm)$.  Then, we will design a div-curl system for $W^\pm$ which will be suitable for obtaining the $W^\pm$ part of the coercivity bound.
\medskip

We  now  discuss the dynamical problem, which is what truly dictates the choice of good variables.  Note that by taking the curl of \eqref{pm equations} we obtain the equations
\begin{equation}\label{vorteq}
D_t^\mp\omega^\pm_{ij}=\partial_jW_k^\mp\partial_kW_i^\pm-\partial_iW_k^\mp\partial_kW_j^\pm=\mathcal{M}_2(\nabla W,\nabla W).
\end{equation}
Based on this transport structure, it is natural to include  $\|\omega^\pm\|_{H^{k-1}(\Omega)}^2$ as part of the energy.
\medskip

In order to identify the other components of the energy, we must make several key observations. The first is that $\|(a,\mathcal{G}^\pm)\|_{H^{k-1}(\Gamma)\times H^{k-\frac{3}{2}}(\Gamma)}^2$ is controlled by the linearized energy $E_{lin}(w^\pm,s)$, where 
\begin{equation*}
\begin{cases}
&w^\pm:=\nabla\mathcal{H}\mathcal{N}^{k-2}\mathcal{G}^\pm,
\\
&s\hspace{1mm}:=\mathcal{N}^{k-1}a,
\end{cases}
\end{equation*}
solve the linearized system to leading order. This suggests including $E_{lin}(w^\pm,s)$ as  part of the energy $E^k(W^\pm,\Gamma)$. The next step is to identify the portion of the energy corresponding to the variables $\nabla_BW^\pm$. 
\medskip

Using \eqref{Euler}, it is straightforward to verify  that $\nabla_B$ commutes with $D_t^\pm$. Therefore, like $\omega^\pm$, the variables $\nabla_B\omega^\pm$ are approximately transported by  $W^\mp$. Hence, it is natural to include the terms $\|\nabla_B\omega^\pm\|_{H^{k-\frac{3}{2}}(\Omega)}^2$ in the energy $E^k(W^\pm,\Gamma)$. The final key observation is that the variables
\begin{equation*}
\begin{cases}
&w_B^\pm:=-\nabla\mathcal{H}\mathcal{N}^{k-2}\nabla_Ba,
\\
&s_B^\pm\hspace{1mm}:=a^{-1}\mathcal{N}^{k-2}\nabla_B\mathcal{G}^\mp,
\end{cases}
\end{equation*}
also solve the linearized system modulo perturbative error terms.  This yields the last piece of the energy; namely, $E_{lin}(w_B^\pm,s_B^\pm)$.
\medskip

Motivated by the above, we define our energy $E^k(W^+,W^-,\Gamma)$ by taking $E^k(W^+,W^-,\Gamma):=E^k_++E^k_-$ with
\begin{equation}\label{ENERGYE1}
E^k_{\pm}:=1+\|W^\pm\|_{L^2(\Omega)}^2+\|\omega^\pm\|_{H^{k-1}(\Omega)}^2+\|\nabla_B\omega^\pm\|_{H^{k-\frac{3}{2}}(\Omega)}^2+E_{lin}(w^\pm,s)+E_{lin}(w_B^\pm,s_B^\pm).
\end{equation}
In the sequel, we will sometimes refer to the  components of \eqref{ENERGYE1} involving $\omega^\pm$ and $\nabla_B\omega^\pm$ as the rotational part of the energy and the components involving $E_{lin}$ as the irrotational part of the energy. We will denote the former by $E^k_{r}$ and the latter by  $E^k_{i}$. More explicitly, we have
\begin{equation*}
\begin{split}
&E^k_r:=\sum_{\alpha\in\{+,-\}}\|\omega^\alpha\|_{H^{k-1}(\Omega)}^2+\|\nabla_B\omega^\alpha\|_{H^{k-\frac{3}{2}}(\Omega)}^2
\end{split}
\end{equation*}
and 
\begin{equation*}
\begin{split}
E^k_i:=&\sum_{\alpha\in\{+,-\}}\left(\|a^{\frac{1}{2}}\mathcal{N}^{k-1}a\|_{L^2(\Gamma)}^2+\|\nabla\mathcal{H}\mathcal{N}^{k-2}\mathcal{G}^\alpha\|_{L^2(\Omega)}^2\right)
\\
+&\sum_{\alpha\in\{+,-\}}\left(\|\nabla\mathcal{H}\mathcal{N}^{k-2}\nabla_Ba\|_{L^2(\Omega)}^2+\|a^{-\frac{1}{2}}\mathcal{N}^{k-2}\nabla_B\mathcal{G}^\alpha\|_{L^2(\Gamma)}^2\right).
\end{split}
\end{equation*}
\begin{remark}\label{remarkBdefcorrection} It is important that, a priori, the definition of the energy functional does not depend on the dynamics of the problem. 
Therefore, we need a way to interpret the expression \eqref{ENERGYE1} when $W^\pm$ and $\Gamma$ do not solve the free boundary MHD equations. To make this precise, we consider a bounded connected domain $\Omega$ with $(W^\pm,\Gamma)\in\mathbf{H}^k$. We define the pressure $P$ through the Laplace equation
\begin{equation*}
\Delta P=-\partial_iW^+_j\partial_jW^-_i 
\end{equation*}
and the boundary condition $P_{|\Gamma}=0$. We then define the Taylor term $a$ as
\begin{equation*}\label{adef}
a:=-n_{\Gamma}\cdot \nabla P_{|\Gamma}.
\end{equation*}
To define $D_t^\pm P$, $D_t^\pm a$ and $\mathcal{G}^\pm$, we begin by observing that for the dynamic problem we have $D_t^\pm P:=\Delta^{-1}F^\pm$ where
\begin{equation*}\label{Dtpdefold}
\begin{split}
F^\pm &:=\Delta W^\pm\cdot\nabla P+2\nabla W^\pm\cdot\nabla^2 P+D_t^\pm\Delta P
\end{split}
\end{equation*}
and where we define
\begin{equation}\label{correctiontermdef}
\begin{split}
D_t^\pm\Delta P&:=\partial_iW_k^\pm\partial_kW_j^+\partial_jW_i^-+\partial_iW_k^\pm\partial_kW_j^-\partial_jW_i^+-\partial_iD_t^\pm W_j^+\partial_jW_i^--\partial_iW_j^+\partial_jD_t^\pm W_i^-
\end{split}
\end{equation}
and
\begin{equation}\label{materialofW}
D_t^\pm W^\mp:=-\nabla P,\hspace{5mm} D_t^\pm W^\pm:=-\nabla P\pm 2\nabla_B W^\pm.
\end{equation}
These expressions are equivalent to the expansions of $D_t^\pm \Delta P$, $D_t^\pm W^\pm$ and $D_t^\mp W^\pm$ one would obtain for the dynamic problem. Using that $W^\pm$ is divergence-free, it is easy to see that we can write
\begin{equation*}
\Delta W^\pm\cdot\nabla P+2\nabla W^\pm\cdot\nabla^2 P=\nabla\cdot\mathcal{M}_2(\nabla W^\pm,\nabla P)
\end{equation*}
where $\mathcal{M}_2$ is an $\mathbb{R}^d$-valued bilinear expression. Therefore, we have
\begin{equation}\label{Dtpdef}
F^\pm=\nabla\cdot\mathcal{M}_2(\nabla W^\pm,\nabla P)+D_t^\pm\Delta P.
\end{equation}
We then simply declare $D_t^\pm P:=\Delta^{-1}F^\pm$ with $F^\pm$ as above. Having settled on a definition for  $D_t^\pm P$, we may   define $D_t^\pm \nabla P$ by
\begin{equation*}\label{DDpdef}
D_t^\pm \nabla P:=-\nabla W^\pm\cdot\nabla P+\nabla D_t^\pm P
\end{equation*}
and then $D_t^\pm a$ by
\begin{equation*}\label{Dadef}
D_t^\pm a:=-n_{\Gamma}\cdot D_t^\pm \nabla P_{|\Gamma}.
\end{equation*}
For our low regularity energy estimates, we will need to correct the variables $D_t^\pm a$ as follows. We first define the auxiliary variables
\begin{equation*}\label{Bdefcorrection}
\mathcal{B}^\pm:=D_t^\pm P-\Delta^{-1}D_t^\pm\Delta P=[D_t^\pm,\Delta^{-1}]\Delta P.
\end{equation*}
We then note importantly that $\mathcal{B}^\pm=\Delta^{-1}H^\pm$ where $H^\pm$ satisfies the simple identity
\begin{equation*}
H^\pm:=\nabla\cdot\mathcal{M}_2(\nabla W^\pm,\nabla P).
\end{equation*}
That is, $H^\pm$ agrees with $F^\pm$ up to correcting by $D_t^\pm \Delta P$ (which one can think of as a type of normal form correction).
We then define the preliminary corrected variables $\mathcal{A}^\pm$ by
\begin{equation*}\label{Adefcorrection}
\mathcal{A}^\pm:=D_t^\pm\nabla P-\nabla\Delta^{-1}D_t^\pm\Delta P,
\end{equation*}
so that $\mathcal{A}^\pm=\nabla\mathcal{B}^\pm-\nabla W^\pm\cdot \nabla P$. The good variables $\mathcal{G}^\pm$ may then be written as  
\begin{equation*}\label{Gcorrection}
\begin{split}
\mathcal{G^\pm}&:=-n_\Gamma\cdot\mathcal{A}^\pm=D_t^\pm a+\nabla_n\Delta^{-1}D_t^\pm\Delta P,
\end{split}
\end{equation*}
or more concretely,
\begin{equation}\label{Gdeftrue}
\begin{split}
\mathcal{G^\pm}&=\nabla_nW^\pm\cdot\nabla P-\nabla_n\Delta^{-1}(\Delta W^\pm\cdot\nabla P+2\nabla W^\pm\cdot\nabla^2 P)
\\
&=\nabla_nW^\pm\cdot\nabla P-\nabla_n\Delta^{-1}\nabla\cdot\mathcal{M}_2(\nabla W^\pm,\nabla P).
\end{split}
\end{equation}
With these interpretations, the energy functional \eqref{ENERGYE1} is well-defined, regardless of whether the state $(W^\pm,\Gamma)$ evolves in time.
\end{remark}
\subsection{A brief heuristic discussion on the normal form correction} One might ask why in the definition of $\mathcal{G}^\pm$ (and $\nabla_B\mathcal{G}^\pm)$ we perform the normal form correction to the variables $D_t^\pm a$ instead of working directly with $D_t^\pm a$ (or $\nabla_BD_t^\pm a)$. The basic reason for this is two-fold. First, $\mathcal{G}^\pm$ and $\nabla_B\mathcal{G}^\pm$ will better capture the leading part of the ``irrotational components" of $W^\pm$ and $\nabla_B W^\pm$, respectively. This leads to a simpler and more natural coercive energy functional for propagating the regularity of the dynamic quantities $W^\pm$ and $\Gamma$. The second reason is the crucial one, in that one cannot treat the error term $D_t^\pm\Delta P$ perturbatively in the energy estimates in the low regularity regime (although, it is lower order in the high regularity regime). To very briefly describe the issue, it turns out that by carrying out the analogue of the energy propagation estimates in \Cref{energyprop} below but with the variables $D_t^\pm a$, one would eventually (after expanding out all of the relevant terms) have to estimate  $\nabla_B\nabla_n\Delta^{-1}\mathcal{M}_2(\nabla\nabla_B W,\nabla\nabla_B W)=:\nabla_B f$ in $H^{k-2}(\Gamma)$. Since we only have $\nabla_B^2 W\in H^{k-\frac{3}{2}}(\Omega)$ rather than $H^{k-1}(\Omega)$, we would be forced to place $\nabla\nabla_B W$ in $L^{\infty}(\Omega)$. This would lead to the regularity restriction $s>\frac{d}{2}+\frac{3}{2}$,  since we would need to ensure that $H^{s-\frac{1}{2}}$ embeds into $C^1$. The point of using the good variables $\mathcal{G}^\pm$ is that they eliminate this term from the corresponding wave-type equation for $a$. The price to pay, however, is that in the energy propagation there will be an additional error term essentially of the form (see \Cref{RB2} for more details) 
\begin{equation*}
\langle\nabla\mathcal{H}\mathcal{N}^{k-2}\nabla_Ba, \nabla\mathcal{H}\mathcal{N}^{k-2}\nabla_B(\mathcal{G}^\pm-D_t^\pm a)\rangle_{L^2(\Omega)}.
\end{equation*}
At first glance, this expression appears to lose derivatives.  However, by carefully commuting the vector field $\nabla_B$ with $\nabla\mathcal{H}\mathcal{N}^{k-2}$, we can essentially integrate by parts twice to obtain 
\begin{equation*}
\langle\nabla\mathcal{H}\mathcal{N}^{k-2}\nabla_Ba, \nabla\mathcal{H}\mathcal{N}^{k-2}\nabla_B(\mathcal{G}^\pm-D_t^\pm a)\rangle_{L^2(\Omega)}\approx \langle \mathcal{N}^{k-2}\nabla_B^2a, \mathcal{N}^{k-1}(\mathcal{G}^\pm-D_t^\pm a)\rangle_{L^2(\Gamma)}.
\end{equation*}
At this point, we will be able to close our estimates by relying on the improved regularity of $a$ in the direction of $\nabla_B$, as alluded to in the introduction. Namely, in sharp contrast to not having $\nabla_B^2 W^\pm\in H^{k-1}(\Omega)$, we do have $\nabla_B^2a\in H^{k-2}(\Gamma)$. See \Cref{partitionofBa} for more details.  
\subsection{Coercivity of the energy functional}
We begin by establishing the coercivity part of \Cref{Energy est. thm}; namely, the relation
\begin{equation*}
E^k\approx_{M_{s-\frac{1}{2}}} M_k^2.
\end{equation*}
The proof will be split into several components. 
\subsubsection{\texorpdfstring{$L^{\infty}$}{} estimates for the pressure} Here we will establish some $L^{\infty}$ based estimates for $P$ in terms of the control parameters $\tilde{A}:=\|(W^+,W^-)\|_{C^{\frac{1}{2}+\epsilon}(\Omega)}+\|\Gamma\|_{C^{1,\epsilon}}$ and $\tilde{A}^{\frac{1}{2}}:=\|(W^+,W^-)\|_{W^{1,\infty}(\Omega)}+\|\Gamma\|_{C^{1,\frac{1}{2}}}$. Such estimates will, in particular, complete the proof of the difference bounds in \Cref{DF}. Moreover, they will imply pointwise estimates for the pressure in terms of the stronger control parameters $M_s$ and $M_{s-\frac{1}{2}}$, which will suffice for our energy estimates. 
\begin{lemma} \label{Linfest}
Given the assumptions of \Cref{Energy est. thm}, the following pointwise estimates for  $P$ hold.
\begin{enumerate}
\item ($C^{1,\epsilon}$ estimate for $P$).
\begin{equation*} 
\|P\|_{C^{1,\epsilon}(\Omega)}\lesssim_{\tilde{A}} 1,\hspace{5mm}\text{and}\hspace{5mm}\|P\|_{C^{1,\epsilon}(\Omega)}\lesssim_{M_{s-\frac{1}{2}}} 1.
\end{equation*}
\item ($C^{1,\frac{1}{2}}$ estimate for $P$). 
\begin{equation*}
\|P\|_{C^{1,\frac{1}{2}}(\Omega)}\lesssim_{\tilde{A}} \tilde{A}^{\frac{1}{2}},\hspace{5mm}\text{and}\hspace{5mm}\|P\|_{C^{1,\frac{1}{2}}(\Omega)}\lesssim_{M_{s-\frac{1}{2}}} M_s.
\end{equation*}
\end{enumerate}
\end{lemma}
\begin{proof}
Using the paradifferential bookkeeping devices developed in \Cref{SSRO}, the proof is entirely similar to \cite[Lemmas 7.5 and 7.9]{Euler}, so we leave the details to the reader.
\end{proof}
\subsubsection{Preliminary \texorpdfstring{$H^s$}{}  estimates}
We next establish some important preliminary $L^2$-based estimates for  the quantities that will  appear frequently in our analysis later on. Let us define for the rest of this section 
\begin{equation*}
\Lambda_{\sigma}:=\|(W^\pm,\Gamma)\|_{\mathbf{H}^{\sigma}}+\|P\|_{H^{\sigma+\frac{1}{2}}(\Omega)}+\|\nabla_B P\|_{H^{\sigma}(\Omega)}+\|\mathcal{A}^\pm\|_{H^{\sigma-1}(\Omega)}+\|\nabla_B\mathcal{A}^\pm\|_{H^{\sigma-\frac{3}{2}}(\Omega)},
\end{equation*}
where  $\mathcal{A}^\pm$ is defined as in \Cref{remarkBdefcorrection}. To ensure that the implicit constants in our coercivity estimates depend only on $M_{s-\frac{1}{2}}$, we will need the following lemma which will allow us to estimate $\Lambda_{s-\frac{1}{2}}$. 
\begin{lemma}\label{coerciveprelim}For $s>\frac{d}{2}+1$ we have
\begin{equation*}\label{coerciveprelim1}
\Lambda_{s-\frac{1}{2}}+\|\nabla_Bn_\Gamma\|_{H^{s-2}(\Gamma)}+\|\mathcal{B}^\pm\|_{H^{s-\frac{1}{2}}(\Omega)}+\|\nabla_B\mathcal{B}^\pm\|_{H^{s-1}(\Omega)}\lesssim_{M_{s-\frac{1}{2}}}1.
\end{equation*}
\end{lemma}
\begin{proof}
To control $\Lambda_{s-\frac{1}{2}}$, we  must estimate the latter four terms in its definition. 
\medskip

\noindent
\textbf{Control of $P$}. First, we have by \Cref{direst},
\begin{equation}\label{Pestlow}
\|P\|_{H^s(\Omega)}\lesssim_{M_{s-\frac{1}{2}}}\|\partial_iW_j^+\partial_jW_i^-\|_{H^{s-2}(\Omega)}+\|P\|_{C^1(\Omega)}.
\end{equation}
By \Cref{Linfest}, \Cref{Multilinearest} and the embedding $H^{s-\frac{1}{2}}(\Omega)\subset C^{\frac{1}{2}+\epsilon}(\Omega)$ we see that  $\|P\|_{H^s(\Omega)}\lesssim_{M_{s-\frac{1}{2}}}1$.
\medskip

\noindent
\textbf{Control of $\nabla_B n_\Gamma$}. We recall from \Cref{Movingsurfid} and \Cref{commutatorremark} that $\nabla_B n_\Gamma=-((\nabla B)^*(n_\Gamma))^{\top}$. Hence, from \Cref{boundaryest} and \Cref{baltrace} we have
\begin{equation*}
\|\nabla_B n_\Gamma\|_{H^{s-2}(\Gamma)}\lesssim_{M_{s-\frac{1}{2}}}\|\nabla B\|_{H^{s-\frac{3}{2}}(\Omega)}+\sup_{j>0}2^{-\frac{j}{2}}\|\nabla \Phi_{\leq j} B\|_{L^{\infty}(\Omega)}+\sup_{j>0}2^{j(s-1-2\epsilon)}\|\nabla\Phi_{\geq j} B\|_{H^{\frac{1}{2}+\epsilon}(\Omega)}\lesssim_{M_{s-\frac{1}{2}}}1.
\end{equation*}
Above, $\Phi_{\leq j}$ and $\Phi_{\geq j}$ denote the regularizing kernels outlined in \Cref{SSRO}. 
\medskip 

\noindent
\textbf{Control of $\nabla_B P$}. Since $\nabla_B P_{|\Gamma}=0$, we may again apply \Cref{direst} (or \Cref{direstlow}, depending on whether $s-\frac{5}{2}\geq 0$ or not) to estimate
\begin{equation}\label{BPestlow}
\|\nabla_B P\|_{H^{s-\frac{1}{2}}(\Omega)}\lesssim_{M_{s-\frac{1}{2}}}\|\Delta(\nabla_B P)\|_{H^{s-\frac{5}{2}}(\Omega)}+\|\nabla_BP\|_{C^{\frac{1}{2}}(\Omega)}.
\end{equation}
By Sobolev embedding, we have  $\|\nabla_B P\|_{C^{\frac{1}{2}}(\Omega)}\lesssim_{M_{s-\frac{1}{2}}}\|\nabla_B P\|_{H^{s-\frac{1}{2}-\epsilon}(\Omega)}$. Therefore, by interpolating and using the above bound for $P$, we conclude that
\begin{equation*}
\begin{split}
\|\nabla_B P\|_{H^{s-\frac{1}{2}}(\Omega)}&\lesssim_{M_{s-\frac{1}{2}}}\|\Delta(\nabla_B P)\|_{H^{s-\frac{5}{2}}(\Omega)}\lesssim \|[\Delta,\nabla_B]P\|_{H^{s-\frac{5}{2}}(\Omega)}+\|\nabla_B(\partial_iW_j^+\partial_jW_i^-)\|_{H^{s-\frac{5}{2}}(\Omega)}.
\end{split}
\end{equation*}
As $B$ is divergence-free, we may write $[\Delta,\nabla_B]P$ in the form $\nabla\cdot\mathcal{M}_2(\nabla B,\nabla P)$. Hence,
\begin{equation*}
\|\nabla_B P\|_{H^{s-\frac{1}{2}}(\Omega)}\lesssim_{M_{s-\frac{1}{2}}}\|\mathcal{M}_2(\nabla B\cdot\nabla P)\|_{H^{s-\frac{3}{2}}(\Omega)}+\|\nabla_B(\partial_iW_j^+\partial_jW_i^-)\|_{H^{s-\frac{5}{2}}(\Omega)}.
\end{equation*}
The first term is easily estimated using \Cref{Multilinearest} and the estimates for $P$ above. The latter term can be similarly estimated using \Cref{Multilinearest} if $s\geq \frac{5}{2}$. Otherwise, we observe that for $p=\frac{2d}{d-2s+5}$  we have the embedding $L^p(\Omega)\subset H^{s-\frac{5}{2}}(\Omega)$. Moreover, we may write $\nabla_B(\partial_i W_j^+\partial_jW_i^-)=\mathcal{M}_2(\nabla\nabla_B W,\nabla W)+\mathcal{M}_3(\nabla B, \nabla W,\nabla W)$. From this, H\"older's inequality and Sobolev embeddings, we see that
\begin{equation*}
\begin{split}
\|\nabla_B(\partial_i W_j^+\partial_jW_i^-)\|_{H^{s-\frac{5}{2}}(\Omega)}\lesssim_{M_{s-\frac{1}{2}}}& \|\nabla B\|_{L^{\frac{2d}{d-2s+3}}(\Omega)}\|\nabla W\|_{L^{2d}(\Omega)}\|\nabla W\|_{L^{2d}(\Omega)}
\\
+&\|\nabla\nabla_B W\|_{L^{\frac{2d}{d-2s+4}}(\Omega)}\|\nabla W\|_{L^{2d}(\Omega)}
\\
\lesssim_{M_{s-\frac{1}{2}}}& 1.
\end{split}
\end{equation*}
This gives the desired control of $\nabla_B P$. 
\medskip

\noindent
\textbf{Control of $\mathcal{A}^\pm$ and $\mathcal{B}^\pm$}. Now, we focus on estimating $\mathcal{A}^\pm$. We first recall  that $\mathcal{A}^\pm=\nabla\mathcal{B}^\pm-\nabla W^\pm\cdot\nabla P$ where
\begin{equation}\label{Bdef2}
\begin{split}
\mathcal{B}^\pm &=\Delta^{-1}\nabla\cdot\mathcal{M}_2(\nabla W^\pm,\nabla P).
\end{split}
\end{equation}
The term $\nabla W^\pm\cdot \nabla P$ is estimated as above. On the other hand, by \Cref{direst} we have
\begin{equation*}
\|\mathcal{B}^\pm\|_{H^{s-\frac{1}{2}}(\Omega)}\lesssim_{M_{s-\frac{1}{2}}}\|\Delta\mathcal{B}^\pm\|_{H^{s-\frac{5}{2}}(\Omega)}+\|\mathcal{B}^\pm\|_{C^{\frac{1}{2}}(\Omega)}.
\end{equation*}
Recall that there is an $\epsilon>0$ such that $H^{s-\frac{1}{2}-\epsilon}(\Omega)$ embeds into $C^{\frac{1}{2}}(\Omega)$. Interpolating between $H^{s-\frac{1}{2}}(\Omega)$ and $H^1(\Omega)$ and using the $H^{-1}\to H_0^1$ bound for $\Delta^{-1}$, we conclude that $\|\mathcal{B}^\pm\|_{H^{s-\frac{1}{2}}(\Omega)}\lesssim_{M_{s-\frac{1}{2}}}\|\Delta\mathcal{B}^\pm\|_{H^{s-\frac{5}{2}}(\Omega)}.$ Then using \eqref{Bdef2} and estimating similarly to the above, we observe that
\begin{equation}\label{Bestlow}
\|\mathcal{B}^\pm\|_{H^{s-\frac{1}{2}}(\Omega)}\lesssim_{M_{s-\frac{1}{2}}} 1,
\end{equation}
which gives the corresponding estimate for $\mathcal{A}^\pm$ and $\mathcal{B}^\pm$.
\medskip

\noindent
\textbf{Control of $\nabla_B\mathcal{A}^\pm$ and $\nabla_B \mathcal{B}^\pm$}. We begin by noting that
\begin{equation*}
\nabla_B\mathcal{A}^\pm=\nabla\nabla_B\mathcal{B}^\pm-\nabla B\cdot\nabla \mathcal{B}^\pm-\nabla_B(\nabla W^\pm\cdot \nabla P).
\end{equation*}
Using \Cref{Multilinearest} and \eqref{Bestlow} we obtain
\begin{equation*}
\|\nabla B\cdot\nabla\mathcal{B}^\pm\|_{H^{s-2}(\Omega)}\lesssim \|B\|_{C^{\frac{1}{2}}(\Omega)}\|\mathcal{B}^\pm\|_{H^{s-\frac{1}{2}}(\Omega)}+\|B\|_{H^{s-\frac{1}{2}}(\Omega)}\|\mathcal{B}^\pm\|_{C^{\frac{1}{2}}(\Omega)}\lesssim_{M_{s-\frac{1}{2}}} 1.
\end{equation*}
Moreover, by expanding $\nabla_B (\nabla W^\pm\cdot\nabla P)$ and using \Cref{Multilinearest}, \eqref{Pestlow} and \eqref{BPestlow}, we see that 
\begin{equation*}
\|\nabla_B(\nabla W^\pm\cdot \nabla P)\|_{H^{s-2}(\Omega)}\lesssim_{M_{s-\frac{1}{2}}}1.
\end{equation*}
To estimate $\nabla\nabla_B\mathcal{B}^\pm$, we note that $\nabla_B\mathcal{B}^\pm_{|\Gamma}=0$ since $B$ is tangent to $\Gamma$. Therefore, by \Cref{direst} and an argument similar to that in the proof of \eqref{Bestlow}, we have
\begin{equation*}
\|\nabla\nabla_B\mathcal{B}^\pm\|_{H^{s-2}(\Omega)}\lesssim \|\nabla_B\mathcal{B}^\pm\|_{H^{s-1}(\Omega)}\lesssim_{M_{s-\frac{1}{2}}}\|\Delta\nabla_B \mathcal{B}^\pm\|_{H^{s-3}(\Omega)}. 
\end{equation*}
Analogously to the estimate for $\nabla_B\nabla P$ above, we may expand $\Delta\nabla_B \mathcal{B}^\pm$ and use that $B$ is divergence-free to estimate
\begin{equation*}
\|\Delta\nabla_B \mathcal{B}^\pm\|_{H^{s-3}(\Omega)}\lesssim_{M_{s-\frac{1}{2}}}\|\nabla B\cdot\nabla\mathcal{B}^\pm\|_{H^{s-2}(\Omega)}+\|\nabla_B\Delta\mathcal{B}^\pm\|_{H^{s-3}(\Omega)}\lesssim_{M_{s-\frac{1}{2}}} 1+\|\nabla_B\Delta\mathcal{B}^\pm\|_{H^{s-3}(\Omega)}.
\end{equation*}
Obtaining control over the latter term in the above inequality is a mostly straightforward application of \Cref{Multilinearest} if $s\geq 3$. To allow for the case $2<s<3$, we need to carefully expand $\nabla_B\Delta \mathcal{B}^\pm$. Using that $B$ is divergence-free, we have the commutator identity $[\partial_i,\nabla_B]f=\partial_j(\partial_i B_j f)$. From this and straightforward manipulations, it is easy to see that we may write $\nabla_B\Delta \mathcal{B}^\pm$ in the form
\begin{equation*}
\begin{split}
\nabla_B\Delta\mathcal{B}^\pm&=\nabla\cdot[\nabla_B\mathcal{M}_2(\nabla W^\pm,\nabla P)]+\nabla\cdot\mathcal{M}_3(\nabla B,\nabla W^\pm,\nabla P).
\end{split}
\end{equation*}
Noting that $s-2\geq 0$ and using \Cref{Multilinearest}, it is straightforward to then estimate 
\begin{equation*}
\begin{split}
\|\nabla_B\Delta\mathcal{B}^\pm\|_{H^{s-3}(\Omega)}\lesssim_{M_{s-\frac{1}{2}}}1.
\end{split}
\end{equation*}
This gives the desired estimates for $\nabla_B\mathcal{A}^\pm$ and $\nabla_B\mathcal{B}^\pm$.
\end{proof}
\begin{corollary}\label{Lambdakest}
For $\sigma>\frac{d}{2}+1$ we have
\begin{equation*}
\Lambda_\sigma+\|\mathcal{B}^\pm\|_{H^{\sigma}(\Omega)}+\|\nabla_B\mathcal{B}^\pm\|_{H^{\sigma-\frac{1}{2}}(\Omega)}+\|\nabla_B n_\Gamma\|_{H^{\sigma-\frac{3}{2}}(\Gamma)}\lesssim_{M_{s-\frac{1}{2}}} M_\sigma.
\end{equation*}
Moreover, 
\begin{equation*}
\|D_t^\pm P\|_{H^{\sigma}(\Omega)}+\|\nabla_B D_t^\pm P\|_{H^{\sigma-\frac{1}{2}}(\Omega)}\lesssim_{M_s} M_{\sigma}.
\end{equation*}
\end{corollary}
\begin{proof}
From \Cref{direst}, Sobolev embeddings, \Cref{coerciveprelim} and \Cref{Multilinearest}, we have
\begin{equation*}
\|P\|_{H^{\sigma+\frac{1}{2}}(\Omega)}\lesssim_{M_{s-\frac{1}{2}}} \|\partial_iW_j^+\partial_jW^-_i\|_{H^{\sigma-\frac{3}{2}}(\Omega)}+\|\Gamma\|_{H^\sigma}\|P\|_{C^1(\Omega)}\lesssim_{M_{s-\frac{1}{2}}} M_\sigma.
\end{equation*}
Similarly,
\begin{equation*}
\|\nabla_B P\|_{H^\sigma (\Omega)}\lesssim_{M_{s-\frac{1}{2}}} M_\sigma.
\end{equation*}
Writing $\mathcal{A}^\pm=\nabla\mathcal{B}^\pm-\nabla W^\pm\cdot\nabla P$, we  obtain
\begin{equation*}
\|\mathcal{A}^\pm\|_{H^{\sigma-1}(\Omega)}\lesssim_{M_{s-\frac{1}{2}}} \|\mathcal{B}^\pm\|_{H^{\sigma}(\Omega)}+M_\sigma.
\end{equation*}
By using \Cref{direst}, Sobolev embedding, \Cref{coerciveprelim} and \Cref{Multilinearest}, it follows that
\begin{equation*}
\|\mathcal{B}^\pm\|_{H^\sigma(\Omega)}\lesssim_{M_{s-\frac{1}{2}}} \|\Delta\mathcal{B}^\pm\|_{H^{\sigma-2}(\Omega)}+\|\Gamma\|_{H^\sigma}\|\mathcal{B}^\pm\|_{C^{\frac{1}{2}}(\Omega)}\lesssim_{M_{s-\frac{1}{2}}} M_\sigma.
\end{equation*}
A similar but more involved  analysis gives
\begin{equation*}
\|\nabla_B n_\Gamma\|_{H^{\sigma-\frac{3}{2}}(\Gamma)}+\|\nabla_B\mathcal{A}^\pm\|_{H^{\sigma-\frac{3}{2}}(\Omega)}+\|\nabla_B\mathcal{B}^\pm\|_{H^{\sigma-\frac{1}{2}}(\Omega)}\lesssim_{M_{s-\frac{1}{2}}} M_\sigma.
\end{equation*}
It remains to estimate $D_t^\pm P$ and $\nabla_B D_t^\pm P$. We recall that, by definition, we have
\begin{equation*}
\mathcal{B}^\pm:=D_t^\pm P-\Delta^{-1}D_t^\pm\Delta P=:D_t^\pm P- \mathcal{C}^\pm.
\end{equation*}
From \eqref{materialofW} and \eqref{correctiontermdef}, we can write $\mathcal{C}^\pm$ in the form
\begin{equation*}\label{Cform}
\begin{split}
\mathcal{C}^\pm &=\Delta^{-1}\mathcal{M}_3(\nabla W,\nabla W,\nabla W)+\Delta^{-1}\mathcal{M}_2(\nabla^2 P,\nabla W)+\Delta^{-1}\mathcal{M}_2(\nabla W,\nabla\nabla_B W).
\end{split}
\end{equation*}
Thanks to the above estimates for $\mathcal{B}^\pm$ and $\nabla_B\mathcal{B}^\pm$, we  only need to estimate $\mathcal{C}^\pm$ and $\nabla_B \mathcal{C}^\pm$ in $H^{\sigma}(\Omega)$ and $H^{\sigma-\frac{1}{2}}(\Omega)$, respectively. Our first observation is that by the algebra property \eqref{algebrapropproduct}, we have
\begin{equation*}
\|\nabla_B\mathcal{C}^\pm\|_{H^{\sigma-\frac{1}{2}}(\Omega)}\lesssim_{M_{s}} \|B\|_{H^{\sigma-\frac{1}{2}}(\Omega)}\|\mathcal{C}^\pm\|_{C^1(\Omega)}+\|B\|_{L^{\infty}(\Omega)}\|\mathcal{C}^\pm\|_{H^{\sigma+\frac{1}{2}}(\Omega)}.
\end{equation*}
Therefore, to establish both bounds, it suffices to show that 
\begin{equation*}
\|\mathcal{C}^\pm\|_{C^1(\Omega)}\lesssim_{M_s}1,\hspace{5mm}\|\mathcal{C}^\pm\|_{H^{\sigma+\frac{1}{2}}(\Omega)}\lesssim_{M_{s}} M_{\sigma}.
\end{equation*}
First, we observe that by Sobolev embedding, \Cref{direst}, \Cref{Multilinearest} and the estimates for the pressure, we have the $C^1$ bound
\begin{equation*}
\|\mathcal{C}^\pm\|_{C^1(\Omega)}\lesssim_{M_{s-\frac{1}{2}}}\|\Delta \mathcal{C}^\pm\|_{H^{s-2}(\Omega)}\lesssim_{M_s} 1.
\end{equation*}
Using \Cref{direst} and \Cref{Multilinearest}, we conclude that
\begin{equation}\label{correctionsestimates}
\|\mathcal{C}^\pm\|_{H^{\sigma+\frac{1}{2}}(\Omega)}\lesssim_{M_s}\|\Delta\mathcal{C}^\pm\|_{H^{\sigma-\frac{3}{2}}(\Omega)}+\|\mathcal{C}^\pm\|_{C^1(\Omega)}\|\Gamma\|_{H^{\sigma}(\Omega)}\lesssim_{M_s} M_{\sigma}.
\end{equation}
This completes the proof.
\end{proof}
\begin{remark}
With a more careful analysis, the estimate for $D_t^\pm P$ can be improved to 
\[
\|D_t^\pm P\|_{H^{\sigma}(\Omega)}\lesssim_{M_{s-\frac{1}{2}}}M_{\sigma},
\]
but we will not actually need this in the sequel.
\end{remark}
\subsubsection{\texorpdfstring{$H^s$}{}  estimates for the surface energy variables}
Our next objective is to control the surface energy variables $(a,\mathcal{G}^\pm)$ in $H^{k-1}(\Gamma)\times H^{k-\frac{3}{2}}(\Gamma)$ and $(\nabla_Ba,\nabla_B\mathcal{G}^\pm)$ in $H^{k-\frac{3}{2}}(\Gamma)\times H^{k-2}(\Gamma)$ by the energy plus some lower order terms. 
\begin{proposition}\label{aest}
We have
\begin{equation}\label{aest11}
\|a\|_{H^{k-1}(\Gamma)}+\|\mathcal{G}^\pm\|_{H^{k-\frac{3}{2}}(\Gamma)}\lesssim_{M_{s-\frac{1}{2}}} (E^k)^{\frac{1}{2}}+\Lambda_{k-\epsilon}
\end{equation}
and
\begin{equation}\label{aest12}
\|\nabla_B a\|_{H^{k-\frac{3}{2}}(\Gamma)}+\|\nabla_B\mathcal{G}^\pm\|_{H^{k-2}(\Gamma)}\lesssim_{M_{s-\frac{1}{2}}} (E^k)^{\frac{1}{2}}+\Lambda_{k-\epsilon}.
\end{equation}
\end{proposition}
\begin{proof}
We begin with the first estimate. To control $a$ in $H^{k-1}(\Gamma)$, as in \cite{Euler}, we use the ellipticity estimate for the Dirichlet-to-Neumann operator from \Cref{ellipticity} to obtain
\begin{equation*}
\|a\|_{H^{k-1}(\Gamma)}\lesssim_{M_{s-\frac{1}{2}}} 
 \|a\|_{L^2(\Gamma)}+\|\mathcal{N}^{k-1}a\|_{L^2(\Gamma)}+\|\Gamma\|_{H^{k-\epsilon}}\|a\|_{C^{\epsilon}(\Gamma)}\lesssim_{M_{s-\frac{1}{2}}} (E^k)^{\frac{1}{2}}+\Lambda_{k-\epsilon}.
\end{equation*}
To control $\mathcal{G}^\pm$ in $H^{k-\frac{3}{2}}(\Gamma)$, we use \Cref{ellipticity}, Sobolev embedding, \Cref{coerciveprelim}, \Cref{Lambdakest} and \Cref{baltrace} to estimate
\begin{equation*}
\begin{split}
\|\mathcal{G}^\pm\|_{H^{k-\frac{3}{2}}(\Gamma)}&\lesssim_{M_{s-\frac{1}{2}}} \|\mathcal{N}^{k-2}\mathcal{G}^\pm\|_{H^{\frac{1}{2}}(\Gamma)}+\|\Gamma\|_{H^{k-\epsilon}}\sup_{j>0}2^{-j(\frac{1}{2}-\epsilon)}\|n_\Gamma\cdot \Phi_{\leq j}\mathcal{A}^\pm\|_{L^{\infty}(\Omega)}
\\
&+\sup_{j>0}2^{j(k-\frac{3}{2}-2\epsilon)}\|n_\Gamma\cdot \Phi_{\geq j}\mathcal{A}^\pm\|_{H^{\epsilon}(\Gamma)}+\Lambda_{k-\epsilon}
\\
&\lesssim_{M_{s-\frac{1}{2}}} \|\mathcal{N}^{k-2}\mathcal{G}^\pm\|_{H^{\frac{1}{2}}(\Gamma)}+\Lambda_{k-\epsilon}.
\end{split}
\end{equation*}
From the trace theorem,  we have
\begin{equation*}
\|\mathcal{N}^{k-2}\mathcal{G}^\pm\|_{H^{\frac{1}{2}}(\Gamma)}\lesssim_{M_{s-\frac{1}{2}}} \|\mathcal{H}\mathcal{N}^{k-2}\mathcal{G}^\pm\|_{H^1(\Omega)}.
\end{equation*}
Since $k\geq 3$ and
\begin{equation*}
\int_{\Gamma}\mathcal{N}^{k-2}\mathcal{G}^\pm\, dS=\int_{\Gamma}n_\Gamma\cdot\nabla \mathcal{H}\mathcal{N}^{k-3}\mathcal{G}^\pm\, dS=0,
\end{equation*}
we conclude by a Poincare type inequality that 
\begin{equation*}
\|\mathcal{H}\mathcal{N}^{k-2}\mathcal{G}^\pm\|_{H^1(\Omega)}\lesssim_{M_{s-\frac{1}{2}}} \|\nabla\mathcal{H}\mathcal{N}^{k-2}\mathcal{G}^\pm\|_{L^2(\Omega)}\lesssim_{M_{s-\frac{1}{2}}}(E^k)^{\frac{1}{2}}.
\end{equation*}
This gives \eqref{aest11}. Now, we move to \eqref{aest12}. First, using that $\nabla_B a=-n_\Gamma\cdot\nabla_B\nabla P$ and the partition $\nabla_B\nabla P=\Phi_{\leq j}\nabla_B\nabla P+\Phi_{\geq j}\nabla_B\nabla P$, we obtain from \Cref{ellipticity} in a similar way to the above that 
\begin{equation*}
\|\nabla_Ba\|_{H^{k-\frac{3}{2}}(\Gamma)}\lesssim_{M_{s-\frac{1}{2}}}\|\nabla\mathcal{H}\mathcal{N}^{k-2}\nabla_Ba\|_{L^2(\Omega)}+\Lambda_{k-\epsilon}\lesssim (E^k)^{\frac{1}{2}}+\Lambda_{k-\epsilon}.
\end{equation*}
To estimate $\nabla_B\mathcal{G}^\pm$, we write
\begin{equation*}
\nabla_B\mathcal{G}^\pm=-\nabla_B n_\Gamma\cdot\mathcal{A}^\pm-n_\Gamma\cdot\nabla_B\mathcal{A}^\pm.
\end{equation*}
From \Cref{Movingsurfid} and \Cref{commutatorremark}, we have $\nabla_B n_\Gamma=-((\nabla B)^*n_\Gamma)^{\top}=-(\nabla B)^*n_\Gamma+(n_\Gamma\cdot((\nabla B)^*n_\Gamma))n_\Gamma$. Therefore, we may expand $\nabla_B n_\Gamma\cdot\mathcal{A}^\pm$ as a sum of terms of the form $\mathcal{M}(n_\Gamma)\mathcal{M}_2(\nabla B,\mathcal{A}^\pm)$ where $\mathcal{M}(n_\Gamma)$ is some multilinear expression in $n_\Gamma$. Hence, we obtain from \Cref{ellipticity} and \Cref{boundaryest} the bounds
\begin{equation*}
\begin{split}
\|\nabla_B\mathcal{G}^\pm\|_{H^{k-2}(\Gamma)}\lesssim_{M_{s-\frac{1}{2}}} &\|\nabla_B\mathcal{G}^\pm\|_{L^2(\Gamma)}+\|\mathcal{N}^{k-2}\nabla_B\mathcal{G}^\pm\|_{L^2(\Gamma)}
\\
+&\|\Gamma\|_{H^{k-\epsilon}}\sup_{j>0}2^{-j(1-\epsilon)}\left[\|\Phi_{\leq j}\mathcal{M}_2\|_{L^{\infty}(\Omega)}+\|\Phi_{\leq j}(\nabla_B\mathcal{A}^\pm)\|_{L^\infty(\Omega)}\right]
\\
+&\sup_{j>0}2^{j(k-2-2\epsilon)}\left[\|\Phi_{\geq j}\mathcal{M}_2\|_{H^{\epsilon}(\Gamma)}+\|\Phi_{\geq j}(\nabla_B\mathcal{A}^\pm)\|_{H^\epsilon(\Gamma)}\right].
\end{split}
\end{equation*}
Using the trace theorem, the regularization properties of $\Phi_{\leq j}$, \Cref{coerciveprelim}, Sobolev embeddings and \Cref{Multilinearest}, the terms in the latter  two  lines may be estimated by $\Lambda_{k-\epsilon}$. Hence,
\begin{equation*}
\|\nabla_B\mathcal{G}^\pm\|_{H^{k-2}(\Gamma)}\lesssim_{M_{s-\frac{1}{2}}} \|\nabla_B\mathcal{G}^\pm\|_{L^2(\Gamma)}+\|\mathcal{N}^{k-2}\nabla_B\mathcal{G}^\pm\|_{L^2(\Gamma)}+\Lambda_{k-\epsilon}.
\end{equation*}
For the first term on the right, we use \Cref{baselineDN},  \Cref{baselineDN2} and the fact that $B$ is tangent to $\Gamma$ to estimate
\begin{equation*}
\|\nabla_B\mathcal{G}^\pm\|_{L^2(\Gamma)}\lesssim_{M_{s-\frac{1}{2}}}\|\nabla^{\top}\mathcal{G}^\pm\|_{L^2(\Gamma)}\lesssim_{M_{s-\frac{1}{2}}}\|\mathcal{G}^\pm\|_{H^1(\Gamma)}.
\end{equation*}
Since $k-\frac{3}{2}\geq 1$, we may use the $H^{k-\frac{3}{2}}(\Gamma)$ estimate for $\mathcal{G}^\pm$ from above to control the last term on the right by the energy. This concludes the proof.
\end{proof}
With our preliminary estimates in hand, let us proceed with the proof of the first (and harder) half of the coercivity estimate; namely,
\begin{equation}\label{coercivityboundmain}
\|(W^\pm,\Gamma)\|_{\mathbf{H}^k}\lesssim_{M_{s-\frac{1}{2}}} (E^k)^{\frac{1}{2}}.
\end{equation}
The crux of the proof of \eqref{coercivityboundmain} is establishing the  intermediate bound
\begin{equation*}\label{intermediatebound}
\Lambda_k\lesssim_{M_{s-\frac{1}{2}}} (E^k)^{\frac{1}{2}}+\Lambda_{k-\epsilon}
\end{equation*}
for some small $\epsilon>0$. The coercivity bound \eqref{coercivityboundmain}  then follows from a simple interpolation argument, which we will outline later.
\subsubsection{Control of the pressure and surface regularity} As a first step, we  control the pressure and surface regularity via
\begin{equation*}\label{surfcontrol}
\|P\|_{H^{k+\frac{1}{2}}(\Omega)}+\|\Gamma\|_{H^k}\lesssim_{M_{s-\frac{1}{2}}} (E^k)^{\frac{1}{2}}+\Lambda_{k-\epsilon}.
\end{equation*}
The proof of this is virtually identical  to the corresponding bound in Section 7.4 of \cite{Euler},  so we omit the proof for brevity. 
\subsubsection{Control of \texorpdfstring{$W^\pm$}{}  and \texorpdfstring{$\mathcal{A}^\pm$}{}} We next estimate $W^\pm$ in terms of $\mathcal{A}^\pm$. To achieve this, we  use the div-curl estimate in \Cref{Balanced div-curl} and the fact that $W^\pm$ is divergence-free to control 
\begin{equation*}\label{div-curl-est1}
\begin{split}
\|W^{\pm}\|_{H^{k}(\Omega)}&\lesssim_{M_{s-\frac{1}{2}}} \|\omega^\pm\|_{H^{k-1}(\Omega)}+\|\nabla^{\top}W^{\pm}\cdot n_\Gamma\|_{H^{k-\frac{3}{2}}(\Gamma)}+\|\Gamma\|_{H^{k-\epsilon}}\|W^\pm\|_{C^{\frac{1}{2}+\epsilon}(\Omega)}
\\
&\lesssim_{M_{s-\frac{1}{2}}} (E^k)^{\frac{1}{2}}+\Lambda_{k-\epsilon}+\|\nabla^{\top}W^\pm\cdot n_\Gamma\|_{H^{k-\frac{3}{2}}(\Gamma)}.
\end{split}
\end{equation*}
It remains to study the boundary term $\nabla^{\top}W^\pm\cdot n_\Gamma$. Similarly to \cite{Euler}, our starting point is  the commutator identity
\begin{equation*}
D_t^{\pm}\nabla P=-\nabla W^{\pm}\cdot\nabla P+\nabla D_t^{\pm}P.
\end{equation*}
As $P_{|\Gamma}=0$, we have $\nabla^{\top}D_t^{\pm}P=0$. Therefore, we obtain the identity
\begin{equation*}
\nabla^{\top}W^{\pm}\cdot n_\Gamma=a^{-1}(D_t^{\pm}\nabla P)^{\top}=a^{-1}(D_t^{\pm}\nabla P-\nabla \Delta^{-1}D_t^\pm\Delta P)^{\top}=a^{-1}(\mathcal{A}^\pm)^{\top},
\end{equation*}
where in the second equality we used that $\nabla \Delta^{-1}D_t^\pm\Delta P$ is normal to $\Gamma$. 
The next step is to  estimate $a^{-1}(\mathcal{A}^\pm)^{\top}$ in terms of $\mathcal{A}^\pm$. For this, we use the balanced product and trace estimates in \Cref{boundaryest} and \Cref{baltrace} as well as \Cref{coerciveprelim} to obtain
\begin{equation*}\label{div-curD_tp1}
\begin{split}
\|a^{-1}(\mathcal{A}^\pm)^{\top}\|_{H^{k-\frac{3}{2}}(\Gamma)}&\lesssim_{M_{s-\frac{1}{2}}} \|\mathcal{A}^\pm\|_{H^{k-1}(\Omega)}+(\|a^{-1}\|_{H^{k-1-\epsilon}(\Gamma)}+\|\Gamma\|_{H^{k-\epsilon}})\sup_{j>0}2^{-j(\frac{1}{2}-\epsilon)}\|\Phi_{<j}\mathcal{A}^\pm\|_{L^{\infty}(\Omega)}
\\
&+\sup_{j>0}2^{j(k-\frac{3}{2}-2\epsilon)}\|\Phi_{\geq j}\mathcal{A}^\pm\|_{H^{\frac{1}{2}+\epsilon}(\Omega)}\lesssim_{M_{s-\frac{1}{2}}} \|\mathcal{A}^\pm\|_{H^{k-1}(\Omega)}+\Lambda_{k-\epsilon}.
\end{split}
\end{equation*}
It now suffices to show that
\begin{equation*}
\|\mathcal{A}^\pm\|_{H^{k-1}(\Omega)}\lesssim_{M_{s-\frac{1}{2}}} (E^k)^{\frac{1}{2}}+\Lambda_{k-\epsilon}.    
\end{equation*}
For this we study an appropriate div-curl decomposition for $\mathcal{A}^\pm$. We have 
\begin{equation}\label{divcurlA}
\begin{cases}
&\nabla\cdot \mathcal{A}^\pm=\nabla^2 P\cdot\nabla W^\pm\hspace{2mm}\text{in $\Omega$},
\\
&\nabla\times \mathcal{A}^\pm=\nabla^2 P\cdot\nabla W^\pm-(\nabla W^\pm)^*\cdot\nabla^2 P\hspace{2mm}\text{in $\Omega$},
\\
&\mathcal{A}^\pm\cdot n_\Gamma=-\mathcal{G}^\pm\hspace{2mm}\text{on $\Gamma$}.
\end{cases}
\end{equation}
Hence, using \Cref{Balanced div-curl}, \Cref{Multilinearest} and the partition $\mathcal{A}^\pm=\Phi_{\leq j}\mathcal{A}^\pm+\Phi_{\geq j}\mathcal{A}^\pm$ as above, we obtain
\begin{equation*}
\begin{split}
\|\mathcal{A}^\pm\|_{H^{k-1}(\Omega)}&\lesssim_{M_{s-\frac{1}{2}}} \|\nabla^{\top}\mathcal{A}^\pm\cdot n_{\Gamma}\|_{H^{k-\frac{5}{2}}(\Gamma)}+\Lambda_{k-\epsilon}.
\end{split}
\end{equation*}
It remains to estimate the boundary term.  We compute
\begin{equation}\label{top1}
\nabla^{\top}\mathcal{A}^\pm\cdot n_\Gamma=-\nabla^{\top}\mathcal{G}^\pm-\mathcal{A}^\pm\cdot \nabla^{\top}n_\Gamma.
\end{equation}
By \Cref{boundaryest}, \Cref{tangradientbound} and  the decomposition $\mathcal{A}^\pm=\Phi_{\leq j}\mathcal{A}^\pm+\Phi_{\geq j}\mathcal{A}^\pm$, the first term on the right in \eqref{top1} may be controlled in a similar fashion to  the above. Indeed, we have
\begin{equation*}
\|\nabla^{\top}\mathcal{G}^\pm\|_{H^{k-\frac{5}{2}}(\Gamma)}\lesssim_{M_{s-\frac{1}{2}}} \|\mathcal{G}^\pm\|_{H^{k-\frac{3}{2}}(\Gamma)}+\Lambda_{k-\epsilon}\lesssim_{M_{s-\frac{1}{2}}} (E^k)^{\frac{1}{2}}+\Lambda_{k-\epsilon}.
\end{equation*}
Estimating the latter term is more tedious. Using that $\nabla_{\tau}n_\Gamma$ is tangent to $\Gamma$ for any tangent vector $\tau$ and the fact that $\mathcal{B}^\pm_{|\Gamma}=0$, we have 
\begin{equation*}
\mathcal{A}^\pm\cdot\nabla^{\top}n_{\Gamma}=-\nabla W^\pm\cdot\nabla P\cdot\nabla^{\top}n_{\Gamma}=-(\nabla W^\pm\cdot\nabla P\cdot\nabla\mathcal{H}n_{\Gamma})^{\top}=:-\mathcal{U}^{\top}.
\end{equation*}
Using the decomposition $\mathcal{U}=\Phi_{\leq j}\mathcal{U}+\Phi_{> j}\mathcal{U}$ together with \Cref{boundaryest} and \Cref{baltrace}, we have 
\begin{equation*}
\|\mathcal{U}^{\top}\|_{H^{k-\frac{5}{2}}(\Gamma)}\lesssim_{M_{s-\frac{1}{2}}} \|\Gamma\|_{H^{k-\epsilon}}\sup_{j>0}2^{-(\frac{3}{2}-\epsilon)j}\|\Phi_{\leq j}\mathcal{U}\|_{L^{\infty}(\Omega)}+\|\mathcal{U}\|_{H^{k-2}(\Omega)}.
\end{equation*}
If $d=2$, we can crudely Sobolev embed and use the regularization properties of $\Phi_{\leq j}$ to estimate 
\begin{equation*}
2^{-(\frac{3}{2}-\epsilon)j}\|\Phi_{\leq j}\mathcal{U}\|_{L^{\infty}(\Omega)}\lesssim_{M_{s-\frac{1}{2}}} \|\mathcal{U}\|_{L^1(\Omega)}\lesssim_{M_{s-\frac{1}{2}}} \|\nabla P\|_{H^{s-1}(\Omega)}\|\nabla W^\pm\|_{L^2(\Omega)}\|\nabla\mathcal{H}n_{\Gamma}\|_{L^2(\Omega)}\lesssim_{M_{s-\frac{1}{2}}} 1 ,
\end{equation*}
where we used that $s-\frac{1}{2}>\frac{3}{2}$. If $d\geq 3$,  we can Sobolev embed to obtain 
\begin{equation*}
2^{-(\frac{3}{2}-\epsilon)j}\|\Phi_{\leq j}\mathcal{U}\|_{L^{\infty}(\Omega)}\lesssim_{M_{s-\frac{1}{2}}} \|\mathcal{U}\|_{H^{s-\frac{5}{2}}(\Omega)}.
\end{equation*}
Since in this case we have $s-\frac{5}{2}\geq 0$, we can apply \Cref{Multilinearest}, Sobolev embeddings and \Cref{Hbounds} to obtain
\begin{equation*}
\|\mathcal{U}\|_{H^{s-\frac{5}{2}}(\Omega)}\lesssim_{M_{s-\frac{1}{2}}}\|W^\pm\|_{H^{s-\frac{1}{2}}(\Omega)}\|\nabla P\|_{H^{s-1}(\Omega)}\|\mathcal{H}n_\Gamma\|_{H^{s-1}(\Omega)}\lesssim_{M_{s-\frac{1}{2}}}1.
\end{equation*}
By \Cref{Multilinearest} and \Cref{Hbounds}, we also obtain
\begin{equation*}
\|\mathcal{U}\|_{H^{k-2}(\Omega)}\lesssim_{M_{s-\frac{1}{2}}} \Lambda_{k-\epsilon}.
\end{equation*}
Combining the above, we finally obtain
\begin{equation*}
\|\nabla^{\top}\mathcal{A}^\pm\cdot n_{\Gamma}\|_{H^{k-\frac{5}{2}}(\Gamma)}\lesssim_{M_{s-\frac{1}{2}}} (E^k)^{\frac{1}{2}}+\Lambda_{k-\epsilon}.
\end{equation*}
Consequently, we have the bound
\begin{equation*}
\|\mathcal{A}^\pm\|_{H^{k-1}(\Omega)}+\|W^\pm\|_{H^k(\Omega)}\lesssim_{M_{s-\frac{1}{2}}} (E^k)^{\frac{1}{2}}+\Lambda_{k-\epsilon}.
\end{equation*}
\subsubsection{Control of \texorpdfstring{$\nabla_B P$}{}}
To control $\nabla_B P$, we use \Cref{direst}, \Cref{coerciveprelim} and the fact that $\nabla_BP$ vanishes on $\Gamma$ to estimate
\begin{equation*}
\|\nabla_BP\|_{H^{k}(\Omega)}\lesssim_{M_{s-\frac{1}{2}}}\|\Delta\nabla_BP\|_{H^{k-2}(\Omega)}+\|\Gamma\|_{H^{k-\epsilon}}\|\nabla_B P\|_{C^{\frac{1}{2}+\epsilon}(\Omega)}\lesssim_{M_{s-\frac{1}{2}}} \|\Delta\nabla_BP\|_{H^{k-2}(\Omega)}+\Lambda_{k-\epsilon}.
\end{equation*}
Then using \Cref{Multilinearest}, \Cref{coerciveprelim} and expanding out $\Delta\nabla_B P$, we obtain
\begin{equation*}
\|\nabla_BP\|_{H^k(\Omega)}\lesssim_{M_{s-\frac{1}{2}}}\|W^\pm\|_{H^k(\Omega)}+\|P\|_{H^{k+\frac{1}{2}}(\Omega)}+\Lambda_{k-\epsilon}.
\end{equation*}
When combined with the above estimates for $W^\pm$ and $P$, this gives the desired bound
\begin{equation*}
\|\nabla_BP\|_{H^k(\Omega)}\lesssim_{M_{s-\frac{1}{2}}}(E^k)^{\frac{1}{2}}+\Lambda_{k-\epsilon}.
\end{equation*}
\subsubsection{Control of \texorpdfstring{$\nabla_B W^{\pm}$}{} and \texorpdfstring{$\nabla_B\mathcal{A}^\pm$}{}} It remains to estimate $\nabla_BW^\pm$ and $\nabla_B \mathcal{A}^\pm$. Our strategy is to proceed similarly to the estimate for $W^\pm$ by performing a div-curl decomposition. Using \Cref{Balanced div-curl}, we see that
\begin{equation*}
\begin{split}
\|\nabla_B W^{\pm}\|_{H^{k-\frac{1}{2}}(\Omega)}&\lesssim_{M_{s-\frac{1}{2}}} \|\nabla\times (\nabla_B W^\pm)\|_{H^{k-\frac{3}{2}}(\Omega)}+\|\nabla\cdot (\nabla_B W^\pm)\|_{H^{k-\frac{3}{2}}(\Omega)}
\\
&+\|\nabla^{\top}(\nabla_B W^{\pm})\cdot n_\Gamma\|_{H^{k-2}(\Gamma)}+\|\Gamma\|_{H^{k-\epsilon}}\|\nabla_BW^\pm\|_{C^{\epsilon}(\Omega)}.
\end{split}
\end{equation*}
By Sobolev embeddings, the last term on the right can be controlled by $\Lambda_{k-\epsilon}$. Moreover, using \Cref{Multilinearest} and the fact that $W^\pm$ is divergence-free, it is straightforward to obtain the bounds
\begin{equation*}
\|\nabla\times (\nabla_BW^\pm)\|_{H^{k-\frac{3}{2}}(\Omega)}+\|\nabla\cdot (\nabla_B W^\pm)\|_{H^{k-\frac{3}{2}}(\Omega)}\lesssim_{M_{s-\frac{1}{2}}}\|\nabla_B\omega^\pm\|_{H^{k-\frac{3}{2}}(\Omega)}+\Lambda_{k-\epsilon}.
\end{equation*}
Therefore,
\begin{equation*}
\|\nabla_B W^{\pm}\|_{H^{k-\frac{1}{2}}(\Omega)}\lesssim_{M_{s-\frac{1}{2}}} (E^k)^{\frac{1}{2}}+\|\nabla^{\top}(\nabla_BW^\pm)\cdot n_\Gamma\|_{H^{k-2}(\Gamma)}+\Lambda_{k-\epsilon}.
\end{equation*}
It remains to estimate the boundary data $\nabla^{\top}(\nabla_B W^{\pm})\cdot n_\Gamma  $ in terms of $\nabla_B\mathcal{G}^\pm$.
Since $D_t^\pm$ and $\nabla_B$ commute (this is a slight abuse of language since at this stage our variables do not yet come from solutions to the dynamic problem; instead, the fact that they commute follows directly from their definitions), we have
\begin{equation*}
\begin{split}
D_t^\pm\nabla_B\nabla P&=-\nabla_B(\nabla W^\pm\cdot\nabla P)+\nabla_B\nabla D_t^\pm P
\\
&= -\nabla\nabla_B W^\pm\cdot \nabla P-\nabla W^\pm\cdot\nabla_B\nabla P+\nabla\nabla_B D_t^\pm P-\nabla B\cdot\nabla D_t^\pm P+\nabla B\cdot\nabla W^\pm\cdot \nabla P.
\end{split}    
\end{equation*}
Combining this with the identity
\begin{equation*}
(D_t^\pm\nabla_B\nabla P)^{\top}=(\nabla_B\mathcal{A}^\pm)^{\top}\pm2\nabla^{\top}B\cdot\nabla\Delta^{-1}(\partial_i\nabla_B W^\pm_j\partial_j W_i^\mp)
\end{equation*}
and the definition of $\mathcal{B}^\pm$ we obtain the formula
\begin{equation}\label{bdotgradidentity1}
\nabla^{\top}(\nabla_BW^\pm)\cdot n_\Gamma= a^{-1}(\nabla_B\mathcal{A}^\pm)^{\top}+a^{-1}(\nabla B\cdot\nabla \mathcal{B}^\pm)^{\top}+a^{-1}(\nabla W^\pm\cdot \nabla_B\nabla P)^{\top}-a^{-1}(\nabla B\cdot\nabla W^\pm\cdot \nabla P)^{\top}.   
\end{equation}
We begin by estimating the latter three terms in \eqref{bdotgradidentity1}, which are essentially lower order. By \Cref{boundaryest} and \Cref{baltrace}, we have
\begin{equation*}
\begin{split}
\|a^{-1}(\nabla B\cdot \nabla\mathcal{B}^\pm)^\top\|_{H^{k-2}(\Gamma)}&\lesssim_{M_{s-\frac{1}{2}}} \|\nabla B\cdot \nabla\mathcal{B}^\pm\|_{H^{k-\frac{3}{2}}(\Omega)}+\|\Gamma\|_{H^{k-\epsilon}}\sup_{j>0}2^{-j(1-\epsilon)}\|\Phi_{<j}(\nabla B\cdot \nabla\mathcal{B}^\pm)\|_{L^{\infty}(\Omega)}
\\
&+\sup_{j>0}2^{j(k-2-\epsilon)}\|\Phi_{\geq j}(\nabla B\cdot \nabla\mathcal{B}^\pm)\|_{H^{\frac{1}{2}+\epsilon}(\Omega)}.
\end{split}
\end{equation*}
By Sobolev embeddings, \Cref{Multilinearest} and \Cref{coerciveprelim} we then conclude that
\begin{equation*}
\|a^{-1}(\nabla B\cdot \nabla\mathcal{B}^\pm)^\top\|_{H^{k-2}(\Gamma)}\lesssim_{M_{s-\frac{1}{2}}}\Lambda_{k-\epsilon}.
\end{equation*}
By a similar analysis, we obtain the same estimates for the last two terms in \eqref{bdotgradidentity1}, so it remains to control $\nabla_B\mathcal{A}^\pm$ in $H^{k-\frac{3}{2}}(\Omega)$ by the energy. By \Cref{Balanced div-curl}, we have
\begin{equation}\label{divcurlgradba}
\begin{split}
\|\nabla_B \mathcal{A}^\pm\|_{H^{k-\frac{3}{2}}(\Omega)}&\lesssim_{M_{s-\frac{1}{2}}} \|\nabla\cdot(\nabla_B\mathcal{A}^\pm)\|_{H^{k-\frac{5}{2}}(\Omega)}+\|\nabla\times\nabla_B\mathcal{A}^\pm\|_{H^{k-\frac{5}{2}}(\Omega)}+\|\nabla_B\mathcal{A}^\pm\cdot n_\Gamma\|_{H^{k-2}(\Gamma)}
\\
&+\|\Gamma\|_{H^{k-\epsilon}}2^{-j(1-\epsilon)}\|\Phi_{\leq j}\nabla_B\mathcal{A}^\pm\|_{L^{\infty}(\Omega)}+\sup_{j>0}2^{j(k-\frac{3}{2}-\epsilon)}\|\Phi_{\geq j}\nabla_B\mathcal{A}^\pm\|_{L^2(\Omega)}+\Lambda_{k-\epsilon}.
\end{split}
\end{equation}
From \Cref{coerciveprelim}, Sobolev embeddings and the regularization bounds for $\Phi_{\leq j}$, we can control the terms in the second line by $\Lambda_{k-\epsilon}$. Furthermore, using the div-curl decomposition for $\mathcal{A}^\pm$, \eqref{divcurlA}, \Cref{Multilinearest} and \Cref{coerciveprelim}, it is easy to estimate
\begin{equation*}
\|\nabla\cdot(\nabla_B\mathcal{A}^\pm)\|_{H^{k-\frac{5}{2}}(\Omega)}+\|\nabla\times(\nabla_B\mathcal{A}^\pm)\|_{H^{k-\frac{5}{2}}(\Omega)}\lesssim_{M_{s-\frac{1}{2}}}\Lambda_{k-\epsilon}.
\end{equation*}
We then notice that
\begin{equation*}
\|\nabla_B\mathcal{A}^\pm\cdot n_\Gamma\|_{H^{k-2}(\Gamma)}\lesssim \|\nabla_B\mathcal{G}^\pm\|_{H^{k-2}(\Gamma)}+\|\mathcal{A}^\pm\cdot\nabla_B n_\Gamma\|_{H^{k-2}(\Gamma)}.
\end{equation*}
Arguing as in the estimate for $\nabla_B \mathcal{G}^\pm$ in \Cref{aest} we obtain
\begin{equation}\label{crossterms}
\|\mathcal{A}^\pm\cdot\nabla_B n_\Gamma\|_{H^{k-2}(\Gamma)}\lesssim_{M_{s-\frac{1}{2}}}(E^k)^\frac{1}{2}+\Lambda_{k-\epsilon}.
\end{equation}
This gives 
\begin{equation*}
\|\nabla_B\mathcal{A}^\pm\|_{H^{k-\frac{3}{2}}(\Omega)}+\|\nabla_B W^\pm\|_{H^{k-\frac{1}{2}}(\Omega)}\lesssim_{M_{s-\frac{1}{2}}} (E^k)^{\frac{1}{2}}+\Lambda_{k-\epsilon}.
\end{equation*}
By combining all of the above estimates, we conclude that
\begin{equation*}
\Lambda_k\lesssim_{M_{s-\frac{1}{2}}}(E^k)^{\frac{1}{2}}+\Lambda_{k-\epsilon}.
\end{equation*}
Using interpolation and the definition of $\Lambda_k$, it follows that
\begin{equation*}
\Lambda_k\lesssim_{M_{s-\frac{1}{2}}} (E^k)^{\frac{1}{2}}+\|P\|_{H^1(\Omega)}+\|\mathcal{A}^\pm\|_{L^2(\Omega)}.
\end{equation*}
By the $H^{-1}\to H_0^1$ bound for $\Delta^{-1}$ and \Cref{Multilinearest} we have
\begin{equation*}
\|P\|_{H^1(\Omega)}+\|\mathcal{A}^\pm\|_{L^2(\Omega)}\lesssim_{M_{s-\frac{1}{2}}}\|(W^+,W^-)\|_{H^{k-\epsilon}(\Omega)}.
\end{equation*}
Interpolating again, we may finally conclude that
\begin{equation*}
\Lambda_k\lesssim_{M_{s-\frac{1}{2}}} (E^k)^{\frac{1}{2}}+\|(W^+,W^-)\|_{L^2(\Omega)}\lesssim_{M_{s-\frac{1}{2}}}(E^k)^{\frac{1}{2}}.
\end{equation*}
This completes the proof of \eqref{coercivityboundmain}. To complete the proof of the energy coercivity property in \Cref{Energy est. thm}, it remains to establish the easier bound
\begin{equation*}
(E^k)^{\frac{1}{2}}\lesssim_{M_{s-\frac{1}{2}}}M_k.
\end{equation*}
Clearly, the only issue is to control the irrotational energy. More precisely, we have to show that the expression
\begin{equation}\label{reversecoercive}
\begin{split}
\|\nabla\mathcal{H}\mathcal{N}^{k-2}\mathcal{G}^\pm\|_{L^2(\Omega)}&+\|a^{\frac{1}{2}}\mathcal{N}^{k-1}a\|_{L^2(\Gamma)}+\|\nabla\mathcal{H}\mathcal{N}^{k-2}\nabla_B a\|_{L^2(\Omega)}+\|a^{-\frac{1}{2}}\mathcal{N}^{k-2}\nabla_B\mathcal{G}^\pm\|_{L^2(\Gamma)}
\end{split}
\end{equation}
is $\lesssim_{M_{s-\frac{1}{2}}} M_k$. For the term $\|a^{\frac{1}{2}}\mathcal{N}^{k-1}a\|_{L^2(\Gamma)}$, we have from \Cref{baselineDN2}  and   \Cref{higherpowers},
\begin{equation*}
\|a^{\frac{1}{2}}\mathcal{N}^{k-1}a\|_{L^2(\Gamma)}\lesssim_{M_{s-\frac{1}{2}}} \|a\|_{H^{k-1}(\Gamma)}+\|a\|_{L^{\infty}(\Gamma)}\|\Gamma\|_{H^{k}}\lesssim_{M_{s-\frac{1}{2}}} \|a\|_{H^{k-1}(\Gamma)}+\|\Gamma\|_{H^{k}}.   
\end{equation*}
Then from \Cref{boundaryest}, \Cref{baltrace} and \Cref{Lambdakest}, we have
\begin{equation*}
\|a\|_{H^{k-1}(\Gamma)}\lesssim_{M_{s-\frac{1}{2}}} \|P\|_{H^{k+\frac{1}{2}}(\Omega)}+\|\Gamma\|_{H^{k}}\lesssim_{M_{s-\frac{1}{2}}} M_k.   
\end{equation*}
Using similar analysis, the $H^{\frac{1}{2}}\to H^1$ bound for $\mathcal{H}$ and the identity $\nabla_Ba=-n_\Gamma\cdot\nabla_B\nabla P$, we may estimate
\begin{equation*}
\begin{split}
\|\nabla&\mathcal{H}\mathcal{N}^{k-2}\nabla_Ba\|_{L^2(\Omega)}\lesssim_{M_{s-\frac{1}{2}}}\|\mathcal{N}^{k-2}\nabla_Ba\|_{H^{\frac{1}{2}}(\Gamma)}
\\
&\lesssim_{M_{s-\frac{1}{2}}}\|\nabla_Ba\|_{H^{k-\frac{3}{2}}(\Gamma)}+\|\Gamma\|_{H^k}\sup_{j>0}2^{-\frac{j}{2}}\|n_\Gamma\cdot\Phi_{\leq j}\nabla_B\nabla P\|_{L^{\infty}(\Gamma)}+\sup_{j>0}2^{j(k-\frac{3}{2}-\epsilon)}\|n_\Gamma\cdot\Phi_{\geq j}\nabla_B\nabla P\|_{H^{\epsilon}(\Gamma)}
\\
&\lesssim_{M_{s-\frac{1}{2}}}\|\nabla_B \nabla P\|_{H^{k-1}(\Omega)}+M_k\lesssim_{M_{s-\frac{1}{2}}} M_k.
\end{split}
\end{equation*}
To control $\nabla\mathcal{H}\mathcal{N}^{k-2}\mathcal{G}^\pm$, we can argue in a similar fashion by using the partition 
\begin{equation*}
\mathcal{G}^\pm=-\mathcal{A}^\pm\cdot n_\Gamma=-\Phi_{\leq j}\mathcal{A}^\pm\cdot n_\Gamma-\Phi_{\geq j}\mathcal{A}^\pm\cdot n_\Gamma
\end{equation*}
to obtain
\begin{equation*}
\|\nabla\mathcal{H}\mathcal{N}^{k-2}\mathcal{G}^\pm\|_{L^2(\Omega)}\lesssim_{M_{s-\frac{1}{2}}}M_k.
\end{equation*}
Likewise, to estimate $\mathcal{N}^{k-2}\nabla_B\mathcal{G}^\pm$ we can write $\nabla_B\mathcal{G}^\pm=-\nabla_B n_\Gamma\cdot\mathcal{A}^\pm-n_\Gamma\cdot\nabla_B\mathcal{A}^\pm$. The latter term is estimated similarly to the above by partitioning $\nabla_B\mathcal{A}^\pm=\Phi_{\leq j}(\nabla_B\mathcal{A}^\pm)+\Phi_{\geq j}(\nabla_B\mathcal{A}^\pm)$, so that
\begin{equation*}
\|\mathcal{N}^{k-2}(n_\Gamma\cdot\nabla_B\mathcal{A}^\pm)\|_{L^2(\Gamma)}\lesssim_{M_{s-\frac{1}{2}}}M_k.
\end{equation*}
We can use \Cref{Movingsurfid} to write $\nabla_B n_{\Gamma}=-\nabla^{\top} B\cdot n_{\Gamma}$, and thus, since $\mathcal{B}^\pm_{|\Gamma}=0$, there holds
\begin{equation*}
-\nabla_B n_{\Gamma}\cdot\mathcal{A}^\pm=-\nabla W^\pm\cdot\nabla P\cdot \nabla^{\top} B\cdot n_\Gamma.
\end{equation*}
Using \Cref{higherpowers}, \Cref{baselineDN2} and arguing somewhat similarly to the estimate for the latter term in \eqref{top1}, we obtain
\begin{equation*}
\|\mathcal{N}^{k-2}(\nabla_B n_{\Gamma}\cdot\mathcal{A}^\pm)\|_{L^2(\Gamma)}\lesssim_{M_{s-\frac{1}{2}}} M_k.
\end{equation*}
We omit the above straightforward (albeit slightly tedious) computation, which completes the estimation of each term in \eqref{reversecoercive} and thus the proof of part (i) of \Cref{Energy est. thm}. Next, we turn to part (ii), which is the energy propagation bound. 
\subsection{Proof of energy propagation}\label{energyprop}
Here we prove the second part of \Cref{Energy est. thm}. Using \eqref{vorteq} and the coercivity bound \eqref{Coercivity bound on integers} it is straightforward to verify the following energy estimate for the rotational component of the energy:
\begin{equation*}\label{vortprop}
\frac{d}{dt}E^k_r\lesssim_{M_{s}} E^k.
\end{equation*}
Hence, the main objective of the work will be to establish a propagation bound for the irrotational part of the energy. More specifically, we intend to show that
\begin{equation*}\label{irrEE}
\frac{d}{dt}E_{i}^k\lesssim_{M_{s}} E^k.
\end{equation*}
To achieve this propagation bound, we start by deriving a wave-type equation for $a$. This equation will govern, at leading order, the dynamics of the free surface as well as the ``irrotational" good variables $\mathcal{G}^\pm$.
\medskip

We begin our derivation with the simple commutator identity
\begin{equation*}
D_t^\pm\nabla P=-\nabla W^\pm\cdot\nabla P+\nabla D_t^\pm P
\end{equation*}
which gives
\begin{equation*}
\mathcal{A}^\pm=-\nabla W^\pm\cdot\nabla P+\nabla\mathcal{B}^\pm.
\end{equation*}
Applying $D_t^\mp$ and performing some elementary algebraic manipulations, we see that
\begin{equation*}
\begin{split}
D_t^\mp \mathcal{A}^\pm &=-\nabla D_t^\mp W^\pm\cdot\nabla P+D_t^\mp\nabla\mathcal{B}^\pm+\nabla W^\mp\cdot (\nabla W^\pm\cdot \nabla P)-\nabla W^\pm\cdot D_t^\mp\nabla P
\\
&=\frac{1}{2}\nabla|\nabla P|^2+D_t^\mp\nabla\mathcal{B}^\pm+\nabla W^\mp\cdot (\nabla W^\pm\cdot \nabla P)-\nabla W^\pm\cdot D_t^\mp\nabla P,
\end{split}
\end{equation*}
where in the last line, we have used \eqref{pm equations} to write $-\nabla D_t^\mp W^\pm\cdot\nabla P=\frac{1}{2}\nabla |\nabla P|^2$. As $\Delta P=-\partial_iW_j^+\partial_jW_i^-$ is lower order, it is natural to further split $\nabla |\nabla P|^2$  as
\begin{equation*}
\frac{1}{2}\nabla |\nabla P|^2=\frac{1}{2}\nabla \mathcal{H}|\nabla P|^2+\frac{1}{2}\nabla \Delta^{-1}\Delta |\nabla P|^2.
\end{equation*}
This yields the equation
\begin{equation}\label{aeqndom}
D_t^\mp \mathcal{A}^\pm-\frac{1}{2}\nabla \mathcal{H}|\nabla P|^2=\frac{1}{2}\nabla \Delta^{-1}\Delta |\nabla P|^2+D_t^\mp\nabla\mathcal{B}^\pm+\nabla W^\mp\cdot (\nabla W^\pm\cdot \nabla P)-\nabla W^\pm\cdot D_t^\mp\nabla P=:g.
\end{equation}
We will later see that $g$ may be thought of as a perturbative source term (although this will require a considerable amount of effort). In order to convert  \eqref{aeqndom} into an equation for the good variables, we take the normal component of the trace on $\Gamma_t$ to obtain
\begin{equation}\label{qeqdom2}
D_t^\mp\mathcal{A}^\pm\cdot n_{\Gamma_t}-\frac{1}{2}\mathcal{N}(a^2)=g\cdot n_{\Gamma_t},
\end{equation}
where we used the dynamic boundary condition $P_{|{\Gamma_t}}=0$ to write $|\nabla P_{|\Gamma_t}|^2=a^2$. Since $D_t^\mp $ is tangent to $\Gamma_t$ and  $\mathcal{G}^\pm=-n_{\Gamma_t}\cdot\mathcal{A}^\pm$, we have
\begin{equation}\label{qeqdom3}
\begin{split}
D_t^\mp \mathcal{G}^\pm &=-D_t^\mp \mathcal{A}^\pm \cdot n_{\Gamma_t}-\mathcal{A}^\pm\cdot D_t^\mp n_{\Gamma_t}.
\end{split}
\end{equation}
Since $D_t^\mp n_{\Gamma_t}$ is also tangent to the boundary, we have by definition of $\mathcal{A}^\pm $ the identities $\mathcal{A}^\pm\cdot D_t^\mp n_{\Gamma_t}=D_t^\mp n_{\Gamma_t}\cdot D_t^\pm (\nabla P)=-aD_t^\mp n_{\Gamma_t}\cdot D_t^\pm n_{\Gamma_t}.$ Therefore, 
\begin{equation*}
D_t^\mp\mathcal{G}^\pm=-D_t^\mp\mathcal{A}^\pm\cdot n_{\Gamma_t}+aD_t^\mp n_{\Gamma_t}\cdot D_t^\pm n_{\Gamma_t}.
\end{equation*}
Combining \eqref{qeqdom2} and \eqref{qeqdom3}, we arrive at the equations
\begin{equation*}\label{adynamics1}
D_t^\mp\mathcal{G}^\pm+\frac{1}{2}\mathcal{N}(a^2)=-g\cdot n_{\Gamma_t}+aD_t^\mp n_{\Gamma_t}\cdot D_t^\pm n_{\Gamma_t},
\end{equation*}
which can be further reduced using the Leibniz type formula for $\mathcal{N}$ from \eqref{DNLeibniz} to 
\begin{equation}\label{adynamics}
D_t^\mp \mathcal{G}^\pm+a\mathcal{N}a=f,
\end{equation}
where
\begin{equation*}
f:=-g\cdot n_{\Gamma_t}+aD_t^\mp n_{\Gamma_t}\cdot D_t^\pm n_{\Gamma_t}+n_{\Gamma_t}\cdot\nabla\Delta^{-1}(|\nabla\mathcal{H}a|^2).
\end{equation*}
Since $\nabla_B$ and $D_t^\mp$ commute, we also obtain the identity
\begin{equation*}\label{Badynamics}
D_t^\mp\nabla_B\mathcal{G}^\pm+a\mathcal{N}\nabla_Ba=\nabla_Bf-[\nabla_B,a]\mathcal{N}a-a[\nabla_B,\mathcal{N}]a=:f_{B}.
\end{equation*}
To propagate $(a,\mathcal{G}^\pm)$ in $H^{k-1}(\Gamma_t)\times   H^{k-\frac{3}{2}}(\Gamma_t)$ and ($\nabla_Ba,\nabla_B\mathcal{G}^\pm)$ in $H^{k-\frac{3}{2}}(\Gamma_t)\times H^{k-2}(\Gamma_t)$, our strategy will be to  identify  solutions to the linearized system \eqref{DM gen} with perturbative source terms and then invoke the linearized energy estimates. With this goal in mind, we focus first on  $(a,\mathcal{G}^\pm)$. We define the good variables
\begin{equation*}
\begin{split}
&w^\pm:=\nabla\mathcal{H} \mathcal{N}^{k-2}\mathcal{G}^\pm,
\\
&s^\pm:=\mathcal{N}^{k-1}a,
\\
&q^\pm:=\mathcal{H} (  a\mathcal{N}^{k-1}a).
\end{split}    
\end{equation*}
Note that we clearly have $\nabla\cdot w^\pm=0$. Moreover, it is easy to see that $q^\pm_{|{\Gamma_t}}=as^\pm$
 and  $w_{|{\Gamma_t}}^\pm\cdot n_{\Gamma_t}=\mathcal{N}^{k-1}\mathcal{G}^\pm.$ Hence, 
\begin{equation*}
\begin{split}
D_t^\pm s^\pm-w^\pm_{|{\Gamma_t}}\cdot n_{\Gamma_t}=[D_t^\pm,\mathcal{N}^{k-1}]a+\mathcal{N}^{k-1}(D_t^\pm a-\mathcal{G}^\pm)=:\mathcal{R}.
\end{split}
\end{equation*}
By using the equation \eqref{adynamics} for $a$  and the Leibniz formula for $\mathcal{N}$,  we see that
\begin{equation*}\label{Qdef}
\begin{split}
D_t^\mp w^\pm +\nabla q^\pm&=\mathcal{Q}\hspace{3mm}\text{in}\ \Omega_t,
\end{split}
\end{equation*}
where 
\begin{equation*}\label{def of Q}
\mathcal{Q}:=-\nabla W^\mp\cdot w^\pm+\nabla [D_t^\mp,\mathcal{H}](\mathcal{N}^{k-2}\mathcal{G}^\pm)+\nabla\mathcal{H}[D_t^\mp,\mathcal{N}^{k-2}]\mathcal{G}^\pm+\nabla\mathcal{H}\mathcal{N}^{k-2}f-\nabla\mathcal{H}[\mathcal{N}^{k-2},a]\mathcal{N}a.    
\end{equation*}
To summarize the above in a compact form, we have
\begin{equation*}
\begin{cases}
&D_t^\mp w^\pm+\nabla q^\pm=\mathcal{Q} \ \text{in} \ \Omega_t,
\\
&\nabla\cdot w^\pm=0 \ \text{in} \ \Omega_t,
\\
&D_t^\pm s^\pm-w^\pm\cdot n_{\Gamma_t}=\mathcal{R} \ \text{on} \ \Gamma_t,
\\
&q^\pm=as^\pm \ \text{on}\ \Gamma_t.
\end{cases}    
\end{equation*}
Next, we similarly phrase the estimate for $(\nabla_Ba,\nabla_B\mathcal{G}^\pm)$ in terms of the linearized equations. We define the good variables $(w_B^\pm, s_B^\pm, q_B^\pm)$ via
\begin{equation*}
\begin{split}
&w_B^\pm:=-\nabla\mathcal{H} \mathcal{N}^{k-2}\nabla_Ba,
\\
&s_B^\pm:=a^{-1}\mathcal{N}^{k-2}\nabla_B\mathcal{G}^\mp,
\\
&q_B^\pm:=\mathcal{H}\mathcal{N}^{k-2}\nabla_B\mathcal{G}^\mp,
\end{split}    
\end{equation*}
which similarly to the above satisfy the equations
\begin{equation*}
\begin{cases}
&D_t^\mp w_B^\pm+\nabla q_B^\pm=\mathcal{R}_{B,1}+\mathcal{R}_{B,2} \ \text{in} \ \Omega_t,
\\
&\nabla\cdot w_B^\pm=0 \ \text{in} \ \Omega_t,
\\
&D_t^\pm s_B^\pm-w_B^\pm\cdot n_{\Gamma_t}=\mathcal{Q}_B \ \text{on} \ \Gamma_t,
\\
&q_B^\pm=as_B^\pm \ \text{on}\ \Gamma_t,
\end{cases}    
\end{equation*}
where
\begin{equation*}
\begin{split}
\mathcal{Q}_B&:=-a^{-2}D_t^\pm a\mathcal{N}^{k-2}\nabla_B\mathcal{G}^\mp+a^{-1}[D_t^\pm,\mathcal{N}^{k-2}]\nabla_B\mathcal{G}^\mp+a^{-1}\mathcal{N}^{k-2}f_B-a^{-1}[\mathcal{N}^{k-2},a]\mathcal{N}\nabla_Ba
\end{split}
\end{equation*}
and
\begin{equation*}
\mathcal{R}_{B,1}:=-\nabla W^\mp\cdot w_B^\pm-\nabla[D_t^\mp,\mathcal{H}]\mathcal{N}^{k-2}\nabla_Ba-\nabla\mathcal{H}[D_t^\mp,\mathcal{N}^{k-2}]\nabla_Ba,\hspace{5mm}\mathcal{R}_{B,2}:=\nabla\mathcal{H}\mathcal{N}^{k-2}\nabla_B(\mathcal{G}^\mp-D_t^\mp a).
\end{equation*}
\begin{remark}
We briefly remark that the reason we write $\mathcal{Q}$ as the source term for the $w^\pm$ equation (which is posed on $\Omega_t$) but write $\mathcal{Q}_B$ instead of $\mathcal{R}_B$ as the source term for the $s_B^\pm$ equation (which is posed on $\Gamma_t$) is because the linearized variables $w^\pm$ and $s_B^\pm$ correspond to the good variables $\mathcal{G}^\pm$ and $\nabla_B\mathcal{G}^\pm$, respectively, and thus, the corresponding source terms will have a similar structure in the estimates below. There is an identical motivation for denoting the other source terms by $\mathcal{R}$, $\mathcal{R}_{B,1}$ and $\mathcal{R}_{B,2}$ as these come from the equations for the good variables $a$ and $\nabla_Ba$, respectively.
\end{remark}
The linearized energy estimate  \eqref{Linear EE MHD} together with \Cref{Lambdakest} and Cauchy-Schwarz immediately gives the preliminary bound 
\begin{equation}\label{preliminaryirrest}
\begin{split}
\frac{d}{dt}E_{i,\pm}^k&\lesssim_{M_{s}} E^k+(\|\mathcal{R}\|_{L^2(\Gamma_t)}+\|\mathcal{R}_{B,1}\|_{L^2(\Omega_t)})(E^k)^{\frac{1}{2}}+(\|\mathcal{Q}\|_{L^2(\Omega_t)}+\|\mathcal{Q}_B\|_{L^2(\Gamma_t)})(E^k)^{\frac{1}{2}}
\\
&-\langle \nabla\mathcal{H}\mathcal{N}^{k-2}\nabla_B a, \mathcal{R}_{B,2}\rangle_{L^2(\Omega_t)}. 
\end{split}
\end{equation}
We are left to control the remaining terms on the right-hand side of \eqref{preliminaryirrest}. This will be where the bulk of the work is concentrated. 
\subsubsection{Control of \texorpdfstring{$\mathcal{R}$}{} and \texorpdfstring{$\mathcal{R}_{B,1}$}{}} Our goal is to show that
\begin{equation*}\label{Rest}
\|\mathcal{R}\|_{L^2(\Gamma_t)}+\|\mathcal{R}_{B,1}\|_{L^2(\Omega_t)}\lesssim_{M_{s}} (E^k)^{\frac{1}{2}}.
\end{equation*}
We begin with $\mathcal{R}$. The analysis of this term is almost identical to the analogous term in \cite{Euler}, but we include the short proof as a convenience to the reader. One ingredient we need is the following commutator estimate which is a consequence of Proposition 7.14 in \cite{Euler}.
\begin{proposition}\label{lowregcomm}
Let $s>\frac{d}{2}+1$ and let $f\in H^1(\Gamma_t)$. Then 
\begin{equation*}
\|[\mathcal{N},D_t^\pm]f\|_{L^2(\Gamma_t)}\lesssim_{M_s} \|f\|_{H^1(\Gamma_t)}.
\end{equation*}
\end{proposition}
\begin{remark}
Technically, the above proposition as stated in \cite{Euler} applies to the material derivative $D_t$, but the proof there applies almost verbatim to handle the case $D_t^\pm$. 
\end{remark}
Returning to  the estimate for $\mathcal{R}_1:=[D_t^\pm,\mathcal{N}^{k-1}]a$, we begin by writing
\begin{equation*}
[D_t^\pm,\mathcal{N}^{k-1}]a=[D_t^\pm,\mathcal{N}]\mathcal{N}^{k-2}a+\mathcal{N}[D_t^\pm,\mathcal{N}^{k-2}]a.
\end{equation*}
By \Cref{lowregcomm}, we have
\begin{equation*}
\|[D_t^\pm,\mathcal{N}]\mathcal{N}^{k-2}a\|_{L^2(\Gamma_t)}\lesssim_{M_s}\|\mathcal{N}^{k-2}a\|_{H^{1}(\Gamma_t)}.
\end{equation*}
Using \Cref{higherpowers}, \Cref{Linfest} and \Cref{aest}, we obtain the estimate
\begin{equation*}
\|\mathcal{N}^{k-2}a\|_{H^{1}(\Gamma_t)}\lesssim_{M_s} \|a\|_{H^{k-1}(\Gamma_t)}+\|\Gamma_t\|_{H^k}\|a\|_{L^{\infty}(\Gamma_t)}\lesssim_{M_s} (E^k)^{\frac{1}{2}}.
\end{equation*}
On the other hand, using \Cref{materialcom} and the coercivity bound, we may estimate
\begin{equation*}
\begin{split}
\|[D_t^\pm,\mathcal{N}^{k-2}]a\|_{H^1(\Gamma_t)}\lesssim_{M_s} \|a\|_{H^{k-1}(\Gamma_t)}+\|a\|_{C^{\frac{1}{2}}(\Gamma_t)}(\|\Gamma_t\|_{H^k}+\|W^\pm\|_{H^k(\Omega_t)})\lesssim_{M_s} (E^k)^{\frac{1}{2}}.    
\end{split}
\end{equation*}
Next, we turn to the estimate for $\mathcal{N}^{k-1}(D_t^\pm a-\mathcal{G}^\pm)$. We begin by recalling that
\begin{equation*}
D_t^\pm a-\mathcal{G}^\pm=-\nabla_n\Delta^{-1}D_t^\pm\Delta P=:-\nabla_n\mathcal{C}^\pm.
\end{equation*}
Hence, by \Cref{EEcorollary} and \Cref{direst} we have
\begin{equation}\label{correctionest}
\begin{split}
\|\mathcal{N}^{k-1}\nabla_n\mathcal{C}^\pm\|_{L^2(\Gamma_t)}&\lesssim_{M_{s}} \|\mathcal{N}^{k-2}\nabla_n\mathcal{C}^\pm\|_{H^1(\Gamma_t)}\lesssim_{M_s} \|\mathcal{C}^\pm\|_{H^{k+\frac{1}{2}}(\Omega_t)}+\|\Gamma_t\|_{H^k}\|\mathcal{C}^\pm\|_{C^1(\Omega_t)}.
\end{split}
\end{equation}
By Sobolev embedding, the estimate \eqref{correctionsestimates} from  
\Cref{Lambdakest} and the energy coercivity, we see that
\begin{equation*}
\|\mathcal{N}^{k-1}(D_t^\pm a-\mathcal{G}^\pm)\|_{L^2(\Gamma_t)}\lesssim_{M_s} (E^k)^{\frac{1}{2}}.
\end{equation*}
This concludes the estimate for $\mathcal{R}$. Now, we turn to $\mathcal{R}_{B,1}$. Clearly, we have
\begin{equation*}
\|\nabla W^\mp\cdot w_B^\pm\|_{L^2(\Omega_t)}\lesssim_{M_s} (E^k)^{\frac{1}{2}}.
\end{equation*}
To handle the second term in the definition of $\mathcal{R}_{B,1}$, we begin by recalling the simple commutator identity
\begin{equation*}\label{Hcomm}
[D_t^{\pm},\mathcal{H}]\psi=\Delta^{-1}\nabla\cdot\mathcal{B}(\nabla W^\pm,\nabla \mathcal{H}\psi)
\end{equation*}
 from \eqref{Sid}. Invoking the $H^{-1}\to H_0^1$ bound for $\Delta^{-1}$, we conclude that
\begin{equation*}
\|\nabla [D_t^\pm,\mathcal{H}]\mathcal{N}^{k-2}\nabla_B a\|_{L^2(\Omega_t)}\lesssim_{M_s}\|\nabla\mathcal{H}\mathcal{N}^{k-2}\nabla_B a\|_{L^2(\Omega_t)}\lesssim_{M_s} (E^k)^{\frac{1}{2}}.
\end{equation*}
To estimate the last term in the definition of $\mathcal{R}_{B,1}$, we use the $H^{\frac{1}{2}}\to H^1$ bound for $\mathcal{H}$ and \Cref{materialcom} to obtain 
\begin{equation*}
\begin{split}
\|\nabla\mathcal{H}[D_t^\pm,\mathcal{N}^{k-2}]\nabla_Ba\|_{L^2(\Omega_t)}\lesssim_{M_s}\|\nabla_B a\|_{H^{k-\frac{3}{2}}(\Gamma_t)}+\left(\|W^\pm\|_{H^k(\Omega_t)}+\|\Gamma_t\|_{H^k}\right)\|\nabla_Ba\|_{L^{\infty}(\Gamma_t)}.
\end{split}
\end{equation*}
Then writing $\nabla_B a=-n_{\Gamma_t}\cdot\nabla_B\nabla P$ and using \Cref{Lambdakest}, we see that $\|\nabla_Ba\|_{L^{\infty}(\Gamma_t)}\lesssim_{M_s} 1$. Combining this with \Cref{aest}, the energy coercivity estimate and Sobolev embedding, we have
\begin{equation*}
\|\mathcal{R}_{B,1}\|_{L^2(\Omega_t)}\lesssim_{M_s} (E^k)^{\frac{1}{2}}.
\end{equation*}
\subsubsection{Control of \texorpdfstring{$\mathcal{Q}$}{}} Next, we estimate $\mathcal{Q}$ and some terms in $\mathcal{Q}_{B}$. First, by applying similar arguments as in the estimates for $\mathcal{R}_1$ and $\mathcal{R}_{B,1}$, we may easily control the first three terms in the definition of $\mathcal{Q}$. More precisely, we have 
\begin{equation*}
\|\nabla W^\mp\cdot w^\pm\|_{L^2(\Omega_t)}+\|\nabla [D_t^\mp,\mathcal{H}]\mathcal{N}^{k-2}\mathcal{G}^\pm\|_{L^2(\Omega_t)}+\|\nabla\mathcal{H}[D_t^\mp,\mathcal{N}^{k-2}]\mathcal{G}^\pm\|_{L^2(\Omega_t)}\lesssim_{M_s} (E^k)^{\frac{1}{2}}.
\end{equation*}
Thanks to the bound $\|D_t^\pm a\|_{L^{\infty}(\Gamma_t)}\lesssim_{M_s}1$ and the definition of $E^k$, the first term in the definition of $\mathcal{Q}_B$ may be estimated immediately by
\begin{equation*}
\|a^{-2}D_t^\pm a\mathcal{N}^{k-2}\nabla_B\mathcal{G}^\mp\|_{L^2(\Gamma_t)}\lesssim_{M_s}(E^k)^{\frac{1}{2}}.
\end{equation*}
The second term $a^{-1}[D_t^\pm,\mathcal{N}^{k-2}]\nabla_B\mathcal{G}^\mp$ is a bit more delicate. As in the estimate for $[D^\pm,\mathcal{N}^{k-1}]a$, we can first bound 
\begin{equation}\label{commutatoresidual}
\|a^{-1}[D_t^\pm,\mathcal{N}^{k-2}]\nabla_B\mathcal{G}^\mp\|_{L^2(\Gamma_t)}\lesssim_{M_s} \|[D_t^\pm,\mathcal{N}^{k-3}]\nabla_B\mathcal{G}^\mp\|_{H^1(\Gamma_t)}+\|\mathcal{N}^{k-3}\nabla_B\mathcal{G}^\mp\|_{H^1(\Gamma_t)}.
\end{equation}
To estimate the first term in \eqref{commutatoresidual}, we  recall that we can write
\begin{equation*}
\nabla_B\mathcal{G}^\mp=-\nabla_B n_{\Gamma_t}\cdot\mathcal{A}^\mp-n_{\Gamma_t}\cdot\nabla_B\mathcal{A}^\mp.
\end{equation*}
Using the fact that $\|\nabla_B n_{\Gamma_t}\cdot\mathcal{A}^\mp\|_{L^{\infty}(\Gamma_t)}\lesssim_{M_s}1$, the bound \eqref{crossterms}, \Cref{materialcom} and the energy coercivity, we can estimate 
\begin{equation*}
\|[D_t^\pm,\mathcal{N}^{k-3}](\nabla_B n_{\Gamma_t}\cdot\mathcal{A}^\mp)\|_{H^1(\Gamma_t)}\lesssim_{M_s} \|\nabla_B n_{\Gamma_t}\cdot\mathcal{A}^\mp\|_{H^{k-2}(\Gamma_t)}+\|W^\pm\|_{H^k(\Omega_t)}+\|\Gamma_t\|_{H^k}\lesssim_{M_s}(E^k)^{\frac{1}{2}}.
\end{equation*}
On the other hand, using the partition $n_{\Gamma_t}\cdot\nabla_B\mathcal{A}^\mp=n_{\Gamma_t}\cdot\Phi_{\leq j}\nabla_B\mathcal{A}^\mp+n_{\Gamma_t}\cdot\Phi_{\geq j}\nabla_B\mathcal{A}^\mp$, \Cref{materialcom}, Sobolev embeddings and the properties of $\Phi_{\leq j}$, we have
\begin{equation*}
\begin{split}
\|[D_t^\pm,\mathcal{N}^{k-3}](n_{\Gamma_t}\cdot\nabla_B\mathcal{A}^\mp)\|_{H^1(\Gamma_t)}\lesssim_{M_s} &\|n_{\Gamma_t}\cdot\nabla_B\mathcal{A}^\mp\|_{H^{k-2}(\Gamma_t)}+(\|W^\pm\|_{H^k(\Omega_t)}+\|\Gamma_t\|_{H^k})\|\nabla_B\mathcal{A}^\mp\|_{H^{s-\frac{3}{2}}(\Omega_t)}
\\
+&\|\nabla_B\mathcal{A}^\mp\|_{H^{k-\frac{3}{2}}(\Omega_t)}.
\end{split}
\end{equation*}
As in the estimates following the div-curl analysis in \eqref{divcurlgradba} and the energy coercivity, we then obtain
\begin{equation*}
\|[D_t^\pm,\mathcal{N}^{k-3}](n_{\Gamma_t}\cdot\nabla_B\mathcal{A}^\mp)\|_{H^1(\Gamma_t)}\lesssim_{M_s} (E^k)^{\frac{1}{2}}.
\end{equation*}
Combining the above estimates and using a similar analysis to deal with $\mathcal{N}^{k-3}\nabla_B\mathcal{G}^\mp$ (except using \Cref{higherpowers} in place of \Cref{materialcom}), we obtain
\begin{equation*}
\|a^{-1}[D_t^\pm,\mathcal{N}^{k-2}]\nabla_B\mathcal{G}^\mp\|_{L^2(\Gamma_t)}\lesssim_{M_s} (E^k)^{\frac{1}{2}}
\end{equation*}
as desired.
\subsubsection{Control of \texorpdfstring{$\nabla\mathcal{H}\mathcal{N}^{k-2}f$}{} and \texorpdfstring{$\mathcal{N}^{k-2}f_B$}{}}
Next, we turn to the estimates for $\nabla\mathcal{H}\mathcal{N}^{k-2}f$ and $\mathcal{N}^{k-2}f_B$. We  recall that 
\begin{equation*}
f:=-g\cdot n_{\Gamma_t}+aD_t^\mp n_{\Gamma_t}\cdot D_t^\pm n_{\Gamma_t}+n_{\Gamma_t}\cdot\nabla\Delta^{-1}(|\nabla\mathcal{H}a|^2)
\end{equation*}
and
\begin{equation*}
f_B:=\nabla_B f-[\nabla_B, a]\mathcal{N}a-a[\nabla_B,\mathcal{N}]a,
\end{equation*}
where $g$ is defined as in \eqref{aeqndom}. We first dispense with the commutators in the definition of $f_B$. We have
\begin{equation*}
[\nabla_B,a]\mathcal{N}a=\nabla_Ba\mathcal{N}a.
\end{equation*}
Writing $\mathcal{N}a=n_{\Gamma_t}\cdot\nabla\Phi_{<j}\mathcal{H}a+n_{\Gamma_t}\cdot\nabla\Phi_{\geq j}\mathcal{H}a=:\mathcal{N}_{<j}a+\mathcal{N}_{\geq j}a$, we conclude from \Cref{baselineDN2}, \Cref{higherpowers}, \Cref{boundaryest}, the trace theorem and the bound $\|\nabla_Ba\|_{C^{\epsilon}(\Gamma_t)}\lesssim_{M_s}1$ that
\begin{equation*}
\begin{split}
\|\mathcal{N}^{k-2}(\nabla_Ba\mathcal{N}a)\|_{L^2(\Gamma_t)}\lesssim_{M_s}&\|\mathcal{N}a\|_{H^{k-2}(\Gamma_t)}+(\|\Gamma_t\|_{H^k}+\|\nabla_Ba\|_{H^{k-\frac{3}{2}}(\Gamma_t)})\sup_{j>0}2^{-\frac{j}{2}}\|\mathcal{N}_{<j}a\|_{L^{\infty}(\Gamma_t)}
\\
+&\sup_{j>0}2^{j(k-2-\epsilon)}\|\mathcal{N}_{\geq j}a\|_{H^{\epsilon}(\Gamma_t)}
\\
\lesssim_{M_s}&\|\mathcal{N}a\|_{H^{k-2}(\Gamma_t)}+\|\mathcal{H}a\|_{H^{k-\frac{1}{2}}(\Omega_t)}+(\|\Gamma_t\|_{H^k}+\|\nabla_Ba\|_{H^{k-\frac{3}{2}}(\Gamma_t)})\|\mathcal{H}a\|_{H^{s-\frac{1}{2}}(\Omega_t)}.
\end{split}
\end{equation*}
Using \Cref{DNpower1} and \Cref{aest}, it is easy to see that the first term on the right-hand side above is controlled by the energy. Using \Cref{Hbounds} and the bound $\|a\|_{L^{\infty}(\Gamma_t)}\lesssim_{M_s}1$, the same is true for the second term on the right. Moreover, by \Cref{Hbounds}, we have $\|\mathcal{H}a\|_{H^{s-\frac{1}{2}}(\Omega_t)}\lesssim_{M_s}1$. Therefore, from \Cref{aest} and the energy coercivity, we have
\begin{equation*}
\|\mathcal{N}^{k-2}(\nabla_Ba\mathcal{N}a)\|_{L^2(\Gamma_t)}\lesssim_{M_s}(E^k)^{\frac{1}{2}}.
\end{equation*}
To handle the other commutator, we recall from \Cref{Movingsurfid} and \Cref{commutatorremark} that we may write
\begin{equation}\label{NBcom}
[\nabla_B,\mathcal{N}]a=\nabla_B n_{\Gamma_t}\cdot\nabla\mathcal{H}a-n_{\Gamma_t}\cdot ((\nabla B)^*(\nabla\mathcal{H}a))+\nabla_n\Delta^{-1}\nabla\cdot \mathcal{M}_2(\nabla B,\nabla\mathcal{H}a).
\end{equation}
We consider the partition $[\nabla_B,\mathcal{N}]a=T_j^1+T_j^2$  of the above commutator where $T_j^1$ is defined by replacing all instances of the term $\mathcal{H}a$ in \eqref{NBcom} with $\Phi_{\leq j}\mathcal{H}a$. It is easy to verify using a similar analysis to the above that
\begin{equation*}
\|T_j^1\|_{L^{\infty}(\Gamma_t)}\lesssim_{M_s}2^{\frac{j}{2}},\hspace{5mm}\|T_j^2\|_{H^{\epsilon}(\Gamma_t)}\lesssim_{M_s}2^{-j(k-2-\epsilon)}(E^k)^{\frac{1}{2}}.
\end{equation*}
Therefore, from \Cref{baselineDN2}, \Cref{higherpowers}, \Cref{boundaryest} and the energy coercivity, we have
\begin{equation*}
\|\mathcal{N}^{k-2}(a[\nabla_B,\mathcal{N}]a)\|_{L^2(\Gamma_t)}\lesssim_{M_s} \|[\nabla_B,\mathcal{N}]a\|_{H^{k-2}(\Gamma_t)}+(E^k)^{\frac{1}{2}}.
\end{equation*}
Then, using \Cref{boundaryest}, the partition $\mathcal{H}a=\Phi_{\leq j}\mathcal{H}a+\Phi_{\geq j}\mathcal{H}a$ and arguing similarly to the above, we see that
\begin{equation*}
\|\nabla_B n_{\Gamma_t}\cdot\nabla\mathcal{H}a-n_{\Gamma_t}\cdot ((\nabla B)^*(\nabla\mathcal{H}a))\|_{H^{k-2}(\Gamma_t)}\lesssim_{M_s}(E^k)^{\frac{1}{2}}.
\end{equation*}
Moreover, from \Cref{L1bound}, Sobolev embedding, \Cref{direst} and \Cref{Multilinearest}, we  have
\begin{equation*}
\|\nabla_n\Delta^{-1}\nabla\cdot\mathcal{M}_2(\nabla B,\nabla\mathcal{H}a)\|_{H^{k-2}(\Gamma_t)}\lesssim_{M_s} (E^k)^{\frac{1}{2}}.
\end{equation*}
Combining everything gives
\begin{equation*}
\|\mathcal{N}^{k-2}(a[\nabla_B,\mathcal{N}]a)\|_{L^2(\Gamma_t)}\lesssim_{M_s}(E^k)^{\frac{1}{2}}.
\end{equation*}
Now, we turn to the estimates involving $f$ and $\nabla_Bf$. Using the identities 
\[
D_t^\pm n_{\Gamma_t}=-((\nabla W^\pm)^*n_{\Gamma_t})^{\top}=-(\nabla W^\pm)^*n_{\Gamma_t}+n_{\Gamma_t}(n_{\Gamma_t}\cdot (\nabla W^\pm)^*n_{\Gamma_t}), \qquad |\nabla\mathcal{H}a|^2=\frac{1}{2}\Delta |\mathcal{H}a|^2,
\]
we may reorganize $f$ as
\begin{equation}\label{fsimp}
f=\frac{1}{2}\nabla_n \Delta^{-1}\Delta (\mathcal{H}a)^2-\frac{1}{2}\nabla_n \Delta^{-1}\Delta |\nabla P|^2+\mathcal{M}-\nabla_n D_t^\mp \mathcal{B}^\pm,
\end{equation}
where $\mathcal{M}$ is multilinear in $n_{\Gamma_t}$, $\nabla P$, $\nabla W^\pm$, $\nabla\mathcal{B}^\pm$ and $\nabla D_t^\mp P$. Next, we estimate each term in $\nabla\mathcal{H}\mathcal{N}^{k-2}f$, with the  expression \eqref{fsimp} for $f$ substituted in. We begin with $F_1:=\nabla \mathcal{H}\mathcal{N}^{k-2}\nabla_n\Delta^{-1}\Delta(\mathcal{H}a)^2$ and $F_{B,1}:=\mathcal{N}^{k-2}\nabla_B\nabla_n\Delta^{-1}\Delta(\mathcal{H}a)^2$. We first note the pointwise bounds
\begin{equation}\label{F1bound}
\|\Delta^{-1}\Delta(\mathcal{H}a)^2\|_{C^{1}(\Omega_t)}\lesssim_{M_s}\|\Delta (\mathcal{H}a)^2\|_{H^{s-2}(\Omega_t)}\lesssim_{M_s}\|a\|_{C^{\frac{1}{2}}(\Gamma_t)}\|\mathcal{H}a\|_{H^{s-\frac{1}{2}}(\Omega_t)}\lesssim_{M_s}1,
\end{equation}
\begin{equation}\label{F1Bbound}
\|[\nabla_B,\nabla_n]\Delta^{-1}\Delta(\mathcal{H}a)^2\|_{L^{\infty}(\Gamma_t)}+\|[\nabla_B,\Delta^{-1}]\Delta(\mathcal{H}a)^2\|_{L^{\infty}(\Omega_t)}\lesssim_{M_s}1
\end{equation}
and
\begin{equation}\label{F1Bbound2}
\|\Delta^{-1}\nabla_B|\nabla\mathcal{H}a|^2\|_{L^{\infty}(\Omega_t)}\lesssim_{M_s}\|\nabla_B|\nabla\mathcal{H}a|^2\|_{H^{s-3}(\Omega_t)}\lesssim_{M_s}\||\nabla\mathcal{H}a|^2\|_{H^{s-2}(\Omega_t)}\lesssim_{M_s}1.
\end{equation}
From the $H^{\frac{1}{2}}\to H^1$ bound for $\mathcal{H}$, \Cref{EEcorollary}, \eqref{F1bound}, \Cref{direst}, \Cref{Multilinearest}, \Cref{Hbounds} and the energy coercivity, we have
\begin{equation*}
\begin{split}
\|F_1\|_{L^2(\Omega_t)}&\lesssim_{M_{s}} \|\Gamma_t\|_{H^k}\|\Delta^{-1}\Delta(\mathcal{H}a)^2\|_{C^{\frac{1}{2}}(\Omega_t)}+\||\nabla \mathcal{H}a|^2\|_{H^{k-2}(\Omega_t)}\lesssim_{M_s}(E^k)^{\frac{1}{2}}.
\end{split}
\end{equation*}
 Using \Cref{baselineDN2}, \Cref{higherpowers} and \eqref{F1bound}-\eqref{F1Bbound2} where relevant, a similar argument gives
\begin{equation*}
\|F_{B,1}\|_{L^2(\Gamma_t)}\lesssim_{M_s} (E^k)^{\frac{1}{2}}.
\end{equation*}
Next, we turn to the estimate for $F_2:=\nabla\mathcal{H}\mathcal{N}^{k-2}\nabla_n\Delta^{-1}\Delta|\nabla P|^2$ and $F_{B,2}:=\mathcal{N}^{k-2}\nabla_B\nabla_n\Delta^{-1}\Delta|\nabla P|^2$. Writing $\Delta |\nabla P|^2=2|\nabla ^2 P|^2-2\nabla P\cdot (\partial_iW_j^+\partial_jW_i^-)$, we obtain in an analogous fashion to the pointwise estimates above 
\begin{equation*}\label{F2estimate}
\|\Delta^{-1}\Delta |\nabla P|^2\|_{C^1(\Omega_t)}\lesssim_{M_s} 1
\end{equation*}
and
\begin{equation*}\label{F2Bestimate}
\|[\nabla_B,\nabla_n]\Delta^{-1}\Delta |\nabla P|^2\|_{L^{\infty}(\Gamma_t)}+\|[\nabla_B,\Delta^{-1}]\Delta |\nabla P|^2\|_{L^{\infty}(\Omega_t)}\lesssim_{M_s}1,\hspace{5mm}\|\Delta^{-1}\nabla_B\Delta |\nabla P|^2\|_{L^{\infty}(\Omega_t)}\lesssim_{M_s}1.
\end{equation*}
We conclude that
\begin{equation*}
\|F_2\|_{L^2(\Omega_t)}+\|F_{B,2}\|_{L^2(\Gamma_t)}\lesssim_{M_s}(E^k)^{\frac{1}{2}}.
\end{equation*}
Next, we estimate $\nabla\mathcal{H}\mathcal{N}^{k-2}\mathcal{M}$ and $\mathcal{N}^{k-2}\nabla_B \mathcal{M}$. By \Cref{higherpowers}, \Cref{baltrace}, \Cref{Lambdakest}, \Cref{boundaryest} and the energy coercivity, we have
\begin{equation*}
\|\nabla\mathcal{H}\mathcal{N}^{k-2}\mathcal{M}\|_{L^2(\Omega_t)}\lesssim_{M_s}\|\mathcal{N}^{k-2}\mathcal{M}\|_{H^{\frac{1}{2}}(\Gamma_t)}\lesssim_{M_s} \|\mathcal{M}\|_{H^{k-\frac{3}{2}}(\Gamma_t)}+\|\Gamma_t\|_{H^k}\|\mathcal{M}\|_{L^{\infty}(\Gamma_t)}\lesssim_{M_s}(E^k)^{\frac{1}{2}}.
\end{equation*}
To estimate $\mathcal{N}^{k-2}\nabla_B \mathcal{M}$, we observe that $\nabla_B\mathcal{M}=\mathcal{M}'\mathcal{T}+\mathcal{M}'$ where $\mathcal{M}'$ is a multilinear expression in the variables
$n_{\Gamma_t},\nabla P,\nabla B,\nabla W^\pm, \nabla\mathcal{B}^\pm, \nabla D_t^\pm P, \nabla_B n_{\Gamma_t}$ and $\nabla_B\nabla P$ and $\mathcal{T}$ is either $\nabla\nabla_B\mathcal{B}^\pm$, $\nabla\nabla_B D_t^\pm P$, $\nabla\nabla_BB$ or $\nabla\nabla_BW^\pm$. Using the partition $\mathcal{T}=\Phi_{\leq j}\mathcal{T}+\Phi_{\geq j}\mathcal{T}$ and the fact that $\|\mathcal{M}'\|_{C^{\epsilon}(\Gamma_t)}\lesssim_{M_s}1$, we can argue similarly to the above to obtain
\begin{equation*}
\begin{split}
\|\mathcal{N}^{k-2}\nabla_B \mathcal{M}\|_{L^2(\Gamma_t)}\lesssim_{M_s}&\|\mathcal{T}\|_{H^{k-\frac{3}{2}}(\Omega_t)}+(\|\mathcal{M}'\|_{H^{k-\frac{3}{2}}(\Gamma_t)}+\|\Gamma_t\|_{H^k})\sup_{j>0}2^{-\frac{j}{2}}\|\Phi_{\leq j}\mathcal{T}\|_{L^{\infty}(\Omega_t)}
\\
+&\sup_{j>0}2^{j(k-2-\epsilon)}\|\Phi_{\geq j}\mathcal{T}\|_{H^{\frac{1}{2}+\epsilon}(\Omega_t)}
\\
\lesssim_{M_s} &\|\mathcal{T}\|_{H^{k-\frac{3}{2}}(\Omega_t)}+(E^k)^{\frac{1}{2}}\|\mathcal{T}\|_{H^{s-\frac{3}{2}}(\Omega_t)}.
\end{split}
\end{equation*}
Using \Cref{Lambdakest}, this finally gives
\begin{equation*}
\|\mathcal{N}^{k-2}\nabla_B \mathcal{M}\|_{L^2(\Gamma_t)}\lesssim_{M_s}(E^k)^{\frac{1}{2}}.
\end{equation*}
Next, we control $\nabla_nD_t^\mp\mathcal{B}^\pm$ and $\nabla_B\nabla_n D_t^\mp\mathcal{B}^\pm$. We begin by recalling the decomposition
\begin{equation*}
\begin{split}
\Delta \mathcal{B}^\pm &=\Delta W^\pm\cdot\nabla P+2\nabla^2P\cdot\nabla W^\pm.
\end{split}
\end{equation*}
Using the equations \eqref{pm equations} and the Laplace equation for the pressure, we can write
\begin{equation*}\label{decompositionforBpm}
\begin{split}
\Delta D_t^\mp\mathcal{B}^\pm &=\mathcal{M}_2(\nabla^2 W^\mp,\nabla\mathcal{B}^\pm)+\mathcal{M}_2(\nabla W^\mp,\nabla^2\mathcal{B}^\pm)+ \mathcal{M}_2(\nabla^2 W^\pm,\nabla D_t^\mp P)+\mathcal{M}_2(\nabla W^\pm,\nabla^2 D_t^\mp P)
\\
&+\mathcal{M}_2(\nabla^2P, \nabla^2P)+\mathcal{M}_3(\nabla^2 W^\pm, \nabla W^\mp, \nabla P)+\mathcal{M}_3(\nabla^2 P, \nabla W,\nabla W).
\end{split}
\end{equation*}
Since $D_t^\mp\mathcal{B}^\pm=0$ on $\Gamma_t$, it is straightforward to verify using \Cref{direst}, \Cref{Multilinearest}, \Cref{Linfest} and the energy coercivity bound that 
\begin{equation*}
\|D_t^\mp\mathcal{B}^\pm\|_{C^1(\Omega_t)}\lesssim_{M_s}\|D_t^\mp\mathcal{B}^\pm\|_{H^s(\Omega_t)}\lesssim_{M_s}1,\hspace{5mm}\|D_t^\mp\mathcal{B}^\pm\|_{H^k(\Omega_t)}\lesssim_{M_s}(E^k)^{\frac{1}{2}}.
\end{equation*}
Consequently, from \Cref{EEcorollary} we have
\begin{equation*}
\|\nabla\mathcal{H}\mathcal{N}^{k-2}\nabla_nD_t^\mp\mathcal{B}^\pm\|_{L^2(\Omega_t)}\lesssim_{M_s} (E^k)^{\frac{1}{2}}.
\end{equation*}
From \Cref{baselineDN2}, \Cref{higherpowers}, \Cref{boundaryest} and the above estimates, we see that
\begin{equation*}
\|\mathcal{N}^{k-2}[\nabla_n,\nabla_B]D_t^\mp\mathcal{B}^\pm\|_{L^2(\Gamma_t)}\lesssim_{M_s}(E^k)^{\frac{1}{2}}.
\end{equation*}
Hence, it remains to estimate $\mathcal{N}^{k-2}\nabla_n\nabla_BD_t^\mp\mathcal{B}^\pm$. Using \Cref{Movingsurfid}, \Cref{commutatorremark}, \Cref{direst} and the above estimates for $D_t^\mp\mathcal{B}^\pm$, it is easy to verify that
\begin{equation*}
\|[\nabla_B,\Delta^{-1}]\Delta D_t^\mp\mathcal{B}^\pm\|_{H^{k-\frac{1}{2}}(\Omega_t)}\lesssim_{M_s}(E^k)^{\frac{1}{2}},\hspace{5mm}\|[\nabla_B,\Delta^{-1}]\Delta D_t^\mp\mathcal{B}^\pm\|_{L^{\infty}(\Omega_t)}\lesssim_{M_s}1
\end{equation*}
and, moreover,
\begin{equation*}
\|\nabla_BD_t^\mp\mathcal{B}^\pm\|_{L^{\infty}(\Omega_t)}\lesssim_{M_s}\|D_t^\mp\mathcal{B}^\pm\|_{C^1(\Omega_t)}\lesssim_{M_s}1.
\end{equation*}
By expanding $\nabla_B\Delta D_t^\mp\mathcal{B}^\pm$, we also easily obtain the bound 
\begin{equation*}
\begin{split}
\|\nabla_B\Delta D_t^\mp\mathcal{B}^\pm\|_{H^{k-\frac{5}{2}}(\Omega_t)}\lesssim_{M_s} (E^k)^{\frac{1}{2}}
\end{split}
\end{equation*}
as a consequence of straightforward algebraic manipulation, \Cref{Multilinearest} and \Cref{Lambdakest}. Hence, by \Cref{baselineDN2}, \Cref{higherpowers}, \Cref{direst} and the above estimates, we obtain
\begin{equation*}
\|\mathcal{N}^{k-2}\nabla_n \nabla_B D_t^\mp \mathcal{B}^\pm\|_{L^2(\Gamma_t)}\lesssim_{M_s}(E^k)^{\frac{1}{2}}.
\end{equation*}
We remark that above, we implicitly used the fact that $k\geq 3$ which allows us to avoid negative regularity Sobolev spaces in the above estimates. This concludes the estimates for $f$ and $\nabla_B f$.
\subsubsection{Control of \texorpdfstring{$\langle \nabla\mathcal{H}\mathcal{N}^{k-2}\nabla_B a, \mathcal{R}_{B,2}\rangle_{L^2(\Omega_t)}$}{}}\label{RB2} We recall that
\begin{equation*}
\mathcal{R}_{B,2}=\nabla\mathcal{H}\mathcal{N}^{k-2}\nabla_B(D_t^\pm a-\mathcal{G}^\pm)=-\nabla\mathcal{H}\mathcal{N}^{k-2}\nabla_B\nabla_n\Delta^{-1}D_t^\pm\Delta P=-\nabla\mathcal{H}\mathcal{N}^{k-2}\nabla_B\nabla_n\mathcal{C}^\pm.
\end{equation*}
To estimate $\langle \nabla\mathcal{H}\mathcal{N}^{k-2}\nabla_Ba,\mathcal{R}_{B,2}\rangle_{L^2(\Omega_t)}$, we will rely on the following estimate which captures the enhanced regularity of the free surface in the direction of the magnetic field. For our purposes, this will be conveniently quantified by measuring the regularity of the vector field $\nabla_B^2$ applied to both $a$ and $P$ below. We will abuse language slightly and  refer to this property as a ``regularizing effect". This is a slight misnomer since this enhanced regularity is really a priori hidden in the fixed-time boundary condition $B\cdot n_{\Gamma_t}=0$ and is not really a property of the flow itself, but rather of the data.
\begin{lemma}[Regularizing effect]\label{partitionofBa}
The following estimates hold:
\begin{equation*}
\|\nabla_B^2a\|_{H^{k-2}(\Gamma_t)}+\|\nabla_B^2P\|_{H^{k-\frac{1}{2}}(\Omega_t)}\lesssim_{M_s} (E^k)^{\frac{1}{2}}.
\end{equation*}
\end{lemma}
\begin{remark}
Using the identity $n_{\Gamma_t}=-a^{-1}\nabla P$ (or more directly, the identity $\nabla_B^2 n_{\Gamma_t}=-\nabla_B (\nabla^{\top} B\cdot n_{\Gamma_t})$) and adapting the  proof below, one can also establish the bound
\begin{equation*}
\|\nabla_B^2 n_{\Gamma_t}\|_{H^{k-2}(\Gamma_t)}\lesssim_{M_s} (E^k)^{\frac{1}{2}}\lesssim_{M_s} M_k
\end{equation*}
which can be thought of as a more direct manifestation of the regularizing effect on the free surface. However, we do not need such a bound in our analysis below, so we omit the proof.
\end{remark}
\begin{proof}
We begin with the harder estimate, which is  $\nabla_B^2a$. We first observe the identity $\nabla_Ba=-n_{\Gamma_t}\cdot\nabla_B\nabla P$ and use the definition of $P$ (which can be thought of as implicitly measuring the regularity of the free surface) to write
\begin{equation*}
\begin{split}
-n_{\Gamma_t}\cdot\nabla_B\nabla P&=n_{\Gamma_t}\cdot\nabla B\cdot \nabla P-\nabla_n\Delta^{-1}(\Delta B\cdot\nabla P+2\nabla B\cdot \nabla^2 P)+\nabla_n\Delta^{-1}\nabla_B(\partial_iW_j^+\partial_jW_i^-)
\\
&=:\mathcal{T}_1+\mathcal{T}_2+\mathcal{T}_3.
\end{split}
\end{equation*}
We begin by estimating $\nabla_B\mathcal{T}_1$. Using \Cref{boundaryest}, \Cref{baltrace} and similar analysis to earlier in the section, we see that
\begin{equation*}
\begin{split}
\|\nabla_B n_{\Gamma_t}\cdot\nabla B\cdot\nabla P\|_{H^{k-2}(\Gamma_t)}\lesssim_{M_s}& \|\nabla_B n_{\Gamma_t}\|_{L^{\infty}(\Gamma_t)}(E^k)^{\frac{1}{2}}+\|\nabla_B n_{\Gamma_t}\|_{H^{k-2}(\Gamma_t)}\lesssim_{M_s}(E^k)^{\frac{1}{2}}.
\end{split}
\end{equation*}
Moreover, we have
\begin{equation*}
\begin{split}
\|n_{\Gamma_t}\cdot\nabla_B (\nabla B\cdot\nabla P)\|_{H^{k-2}(\Gamma_t)}\lesssim_{M_s}& \|\nabla_B (\nabla B\cdot\nabla P)\|_{H^{k-\frac{3}{2}}(\Omega_t)}
\\
+&\|\Gamma_t\|_{H^{k-\frac{1}{2}}}\sup_{j>0}2^{-\frac{j}{2}}\|\Phi_{\leq j}\nabla_B (\nabla B\cdot\nabla P)\|_{L^{\infty}(\Omega_t)}
\\
+&\sup_{j>0}2^{j(k-2-\epsilon)}\|\Phi_{\geq j}(\nabla_B (\nabla B\cdot\nabla P))\|_{H^{\frac{1}{2}+\epsilon}(\Omega_t)},
\end{split}
\end{equation*}
which yields
\begin{equation*}
\begin{split}
\|n_{\Gamma_t}\cdot\nabla_B (\nabla B\cdot\nabla P)\|_{H^{k-2}(\Gamma_t)}&\lesssim_{M_s}\|\nabla_B (\nabla B\cdot\nabla P)\|_{H^{k-\frac{3}{2}}(\Omega_t)}+\|\Gamma_t\|_{H^{k-\frac{1}{2}}}\lesssim_{M_s} (E^k)^{\frac{1}{2}}.
\end{split}
\end{equation*}
Thus, we have
\begin{equation*}
\|\nabla_B \mathcal{T}_1\|_{H^{k-2}(\Gamma_t)}\lesssim_{M_s} (E^k)^{\frac{1}{2}}.
\end{equation*}
To estimate the second term, let us first define $F:=\Delta B\cdot\nabla P+2\nabla B\cdot \nabla^2 P$. Then, as in previous analysis, using that $B$ is divergence-free, we note that we can write
\begin{equation*}
F=\nabla\cdot\mathcal{M}_2(\nabla B,\nabla P),\hspace{5mm}[\Delta^{-1},\nabla_B]F=\Delta^{-1}[\nabla\cdot\mathcal{M}_2(\nabla B,\nabla \Delta^{-1}F)].
\end{equation*}
By Sobolev embeddings, \Cref{direst} and \Cref{Multilinearest}, we therefore have
\begin{equation}\label{Flinfest}
\|\Delta^{-1}F\|_{C^{\frac{1}{2}}(\Omega_t)}+\|[\Delta^{-1},\nabla_B]F\|_{L^{\infty}(\Omega_t)}\lesssim_{M_s} 1.
\end{equation}
Moreover, we have 
\begin{equation*}
-\nabla_B\mathcal{T}_2=[\nabla_B,\nabla_n] \Delta^{-1}F-\nabla_n[\Delta^{-1},\nabla_B]F+\nabla_n\Delta^{-1}\nabla_B F.
\end{equation*}
Using \Cref{boundaryest}, \Cref{baltrace}, \Cref{L1bound}, \Cref{Multilinearest}, \eqref{Flinfest} and the energy coercivity, we may control the first two terms on the right by $(E^k)^{\frac{1}{2}}$. To control the last term, since $B$ is divergence-free, we may write $\nabla_B F=\nabla\cdot \nabla_B\mathcal{M}_2(\nabla B,\nabla P)+\nabla\cdot\mathcal{M}_3(\nabla B,\nabla B,\nabla P)$. Therefore, we have by \Cref{Multilinearest},
\begin{equation*}
\|\nabla_B F\|_{H^{s-\frac{5}{2}}(\Omega_t)}\lesssim_{M_s}\|\nabla_B\mathcal{M}_2(\nabla B,\nabla P)\|_{H^{s-\frac{3}{2}}(\Omega_t)}+\|\mathcal{M}_3(\nabla B,\nabla B,\nabla P)\|_{H^{s-\frac{3}{2}}(\Omega_t)}\lesssim_{M_s}1.
\end{equation*}
Similarly, using \Cref{Multilinearest} and  the energy coercivity, we have
\begin{equation*}
\|\nabla_BF\|_{H^{k-\frac{5}{2}}(\Omega_t)}\lesssim_{M_s} (E^k)^{\frac{1}{2}}.
\end{equation*}
It follows from \Cref{L1bound}, Sobolev embeddings and \Cref{direst} that
\begin{equation*}
\|\nabla_B\mathcal{T}_2\|_{H^{k-2}(\Gamma_t)}\lesssim_{M_s} \|\nabla_B F\|_{H^{k-\frac{5}{2}}(\Omega_t)}+\|\Gamma_t\|_{H^k}\|\Delta^{-1}\nabla_B F\|_{L^{\infty}(\Omega_t)}\lesssim_{M_s} (E^k)^{\frac{1}{2}}.
\end{equation*}
Finally, to estimate $\nabla_B\mathcal{T}_3$, we write
\begin{equation*}
\nabla_B\mathcal{T}_3=-n_{\Gamma_t}\cdot\nabla B\cdot\nabla\Delta^{-1}\nabla_B(\partial_iW_j^+\partial_jW_i^-)+\nabla_n\nabla_B\Delta^{-1}\nabla_B(\partial_iW_j^+\partial_jW_i^-)=:\mathcal{T}_3^1+\mathcal{T}_3^2.
\end{equation*}
Using \Cref{boundaryest}, Sobolev embeddings, \Cref{baltrace} and \Cref{direst}, we find that
\begin{equation*}
\|\mathcal{T}_3^1\|_{H^{k-2}(\Gamma_t)}\lesssim_{M_s} (E^k)^{\frac{1}{2}}.
\end{equation*}
On the other hand, by \Cref{L1bound}, we have
\begin{equation*}
\|\nabla_n\nabla_B\Delta^{-1}\nabla_B(\partial_iW_j^+\partial_jW_i^-)\|_{H^{k-2}(\Gamma_t)}\lesssim_{M_s} \|\nabla_B\Delta^{-1}\nabla_B(\partial_iW_j^+\partial_jW_i^-)\|_{H^{k-\frac{1}{2}}(\Omega_t)}+\|\Gamma_t\|_{H^k}.
\end{equation*}
Using the algebra property \eqref{algebrapropproduct} and then \Cref{direst}, we can  crudely estimate (i.e.~treating $\nabla_B$ as a derivative of order one)
\begin{equation*}
\|\nabla_B\Delta^{-1}\nabla_B(\partial_iW_j^+\partial_jW_i^-)\|_{H^{k-\frac{1}{2}}(\Omega_t)}\lesssim_{M_s} \|B\|_{H^{k-\frac{1}{2}}(\Omega_t)}+\|\Gamma_t\|_{H^k}+\|\nabla_B(\partial_iW_j^+\partial_jW_i^-)\|_{H^{k-\frac{3}{2}}(\Omega_t)}.
\end{equation*}
Using \Cref{Multilinearest} to estimate the last term on the right, as well as the energy coercivity, we have
\begin{equation*}
\|\mathcal{T}_3^2\|_{H^{k-2}(\Gamma_t)}\lesssim_M (E^k)^{\frac{1}{2}}.
\end{equation*}
Combining everything above, we obtain the desired bound for $\nabla_B^2a$. The estimate for $\nabla_B^2P$ is simpler. One can proceed by expanding
\begin{equation*}
\nabla_BP=\Delta^{-1}(\Delta B\cdot\nabla P+2\nabla B\cdot\nabla^2 P)-\Delta^{-1}\nabla_B(\partial_iW_j^+\partial_jW_i^-)
\end{equation*}
and then estimating each term on the right in a similar (in fact, more straightforward) manner to the estimates for $\mathcal{T}_2$ and $\mathcal{T}_3$ above. This completes the proof of the lemma. 
\end{proof}
A straightforward corollary of the above analysis is the following bound.
\begin{corollary}\label{partitionofBa2}
There holds
\begin{equation*}
\|\mathcal{N}^{k-2}\nabla_B^2a\|_{L^2(\Gamma_t)}\lesssim_{M_s} (E^k)^{\frac{1}{2}}.
\end{equation*}
\end{corollary}
\begin{proof}
First, we observe that by combining the estimate for $\nabla_B^2 P$ in \Cref{partitionofBa} with \Cref{Multilinearest}, we obtain
\begin{equation}\label{BsquaredofP}
\|\nabla_B^2\nabla P\|_{H^{k-\frac{3}{2}}(\Omega_t)}\lesssim_{M_s} (E^k)^{\frac{1}{2}}.
\end{equation}
By partitioning $-\nabla_B^2 a=(\nabla_B n_{\Gamma_t}\cdot\nabla_B\nabla P+n_{\Gamma_t}\cdot\Phi_{< j}\nabla_B^2\nabla P)+n_{\Gamma_t}\cdot\Phi_{\geq j}\nabla_B^2\nabla P$ and applying \Cref{baselineDN2}, \Cref{higherpowers} and using the regularization bounds for $\Phi_{\leq j}$, we have
\begin{equation*}
\begin{split}
\|\mathcal{N}^{k-2}\nabla_B^2a\|_{L^2(\Gamma_t)} \lesssim_{M_s}&\|\nabla_B^2a\|_{H^{k-2}(\Gamma_t)}+\|\Gamma_t\|_{H^k}(\|\nabla_Bn_{\Gamma_t}\|_{L^{\infty}(\Gamma_t)}\|\nabla_B\nabla P\|_{L^{\infty}(\Omega_t)}+\|\nabla_B^2\nabla P\|_{H^{s-2}(\Omega_t)})
\\
+&\|\nabla_B^2\nabla P\|_{H^{k-\frac{3}{2}}(\Omega_t)}.
\end{split}
\end{equation*}
From \Cref{partitionofBa} and \eqref{BsquaredofP}, we conclude the desired bound.
\end{proof}
Now, let us return to the estimate for  $\langle\nabla\mathcal{H}\mathcal{N}^{k-2}\nabla_Ba,\mathcal{R}_{B,2}\rangle_{L^2(\Omega_t)}$. To simplify notation slightly, let us write $\mathcal{D}=\nabla_n\mathcal{C}^\pm$. We note the  preliminary bounds
\begin{equation}\label{Destimates}
\|\mathcal{N}^{k-1}\mathcal{D}\|_{L^2(\Gamma_t)}+\|\mathcal{D}\|_{H^{k-1}(\Gamma_t)}\lesssim_{M_s}(E^k)^{\frac{1}{2}},\hspace{5mm}\|\mathcal{D}\|_{L^{\infty}(\Gamma_t)}\lesssim_{M_s}1,
\end{equation}
which follow from \eqref{correctionsestimates}, \Cref{EEcorollary} and \Cref{L1bound}. We write
\begin{equation*}
\mathcal{R}_{B,2}=\nabla\mathcal{H}[\mathcal{N}^{k-2},\nabla_B]\mathcal{D}+\nabla\mathcal{H}\nabla_B\mathcal{N}^{k-2}\mathcal{D}.
\end{equation*}
We  estimate the former term using the $H^{\frac{1}{2}}\to H^1$ bound for $\mathcal{H}$, \Cref{commutatorremark} and \eqref{Destimates},
\begin{equation*}
\|\nabla\mathcal{H}[\mathcal{N}^{k-2},\nabla_B]\mathcal{D}\|_{L^2(\Omega_t)}\lesssim_{M_s} \|\mathcal{D}\|_{H^{k-\frac{3}{2}}(\Gamma_t)}+(\|\Gamma_t\|_{H^k}+\|B\|_{H^k(\Omega_t)})\|\mathcal{D}\|_{L^{\infty}(\Gamma_t)}\lesssim_{M_s} (E^k)^{\frac{1}{2}}.
\end{equation*}
It remains to estimate $\langle\nabla\mathcal{H}\mathcal{N}^{k-2}\nabla_Ba,\nabla\mathcal{H}\nabla_B\mathcal{N}^{k-2}\mathcal{D}\rangle_{L^2(\Omega_t)}$. 
By applying \Cref{commutatorremark}, we first observe  that
\begin{equation*}
\|[\nabla_B,\nabla\mathcal{H}]\mathcal{N}^{k-2}\mathcal{D}\|_{L^2(\Omega_t)}+\|[\nabla_B,\nabla\mathcal{H}\mathcal{N}^{k-2}]\nabla_B a\|_{L^2(\Omega_t)}\lesssim_{M_s} (E^k)^{\frac{1}{2}}.
\end{equation*}
Using the above and skew-adjointness of $\nabla_B$ on $L^2(\Omega_t)$, we can integrate by parts $\nabla_B$ onto $\nabla_Ba$ and then integrate by parts again to move one factor of $\mathcal{N}$ to the other term to obtain 
\begin{equation*}
\begin{split}
\langle\nabla\mathcal{H}\mathcal{N}^{k-2}\nabla_B a,\nabla\mathcal{H}\nabla_B\mathcal{N}^{k-2}\mathcal{D}\rangle_{L^2(\Omega_t)}&\leq -\langle\nabla\mathcal{H}\mathcal{N}^{k-2}\nabla_B^2 a,\nabla\mathcal{H}\mathcal{N}^{k-2}\mathcal{D}\rangle_{L^2(\Omega_t)}+C(M_s)E^k
\\
&=-\langle\mathcal{N}^{k-2}\nabla_B^2 a,\mathcal{N}^{k-1}\mathcal{D}\rangle_{L^2(\Gamma_t)}+C(M_s)E^k.
\end{split}
\end{equation*}
Making use of Cauchy-Schwarz, \Cref{partitionofBa2}, \eqref{Destimates} and the energy coercivity, we then obtain
\begin{equation*}
\langle\nabla\mathcal{H}\mathcal{N}^{k-2}\nabla_B a,\nabla\mathcal{H}\nabla_B\mathcal{N}^{k-2}\mathcal{D}\rangle_{L^2(\Omega_t)}\lesssim_{M_s} E^k+\|\mathcal{N}^{k-1}\mathcal{D}\|_{L^2(\Gamma_t)}\|\mathcal{N}^{k-2}\nabla_B^2a\|_{L^2(\Gamma_t)}\lesssim_{M_s}E^k.
\end{equation*}
Combining everything  gives the desired bound
\begin{equation*}
\langle\nabla\mathcal{H}\mathcal{N}^{k-2}\nabla_Ba,\mathcal{R}_{B,2}\rangle_{L^2(\Omega_t)}\lesssim_{M_s}E^k.
\end{equation*}
\subsubsection{Estimates for \texorpdfstring{$\nabla\mathcal{H}[\mathcal{N}^{k-2},a]\mathcal{N}a$}{}  and  \texorpdfstring{$[\mathcal{N}^{k-2},a]\mathcal{N}\nabla_Ba$}{}} To conclude the energy propagation bound in \Cref{Energy est. thm}, it remains to establish the  commutator estimates 
\begin{equation}\label{Leibnizcommutator}
\|[\mathcal{N}^{k-2},a]\mathcal{N}a\|_{H^{\frac{1}{2}}(\Gamma_t)}\lesssim_{M_s}(E^k)^{\frac{1}{2}},\hspace{5mm}\|[\mathcal{N}^{k-2},a]\mathcal{N}\nabla_B a\|_{L^2(\Gamma_t)}\lesssim_{M_s}(E^k)^{\frac{1}{2}}.
\end{equation} 
We will show the details for the first estimate in \eqref{Leibnizcommutator} and then remark on the minor differences required to prove the second estimate. 
We mention that a much more refined version of the first estimate was established in \cite{Euler}. This was needed to prove a sharp, pointwise continuation criterion for the free boundary Euler equations. Here, we will not need anything this precise, so we opt to provide a simpler proof which is sufficient for our purposes.
\medskip

As in \cite[Lemma 7.12]{Euler}, we have the commutator identity
\begin{equation}\label{higherorderleibniz}
[\mathcal{N}^{k-2},a]\mathcal{N}a=\sum_{n+m=k-3}\mathcal{N}^n(\mathcal{N}a\mathcal{N}^{m+1}a)-2\mathcal{N}^n\nabla_n\Delta^{-1}(\nabla\mathcal{H}a\cdot\nabla\mathcal{H}\mathcal{N}^{m+1}a)
\end{equation}
where $m$ and $n$ are integers ranging from $0$ to $k-3$. We begin by estimating the latter summand. To simplify notation, we set $F:=\nabla\mathcal{H}a\cdot\nabla G$ where $G:=\mathcal{H}\mathcal{N}^{m+1}a$. We also define $\mathcal{N}_{<j}:=n_\Gamma\cdot\nabla\Phi_{\leq j}\mathcal{H}$ and $\mathcal{N}_{\geq j}:=n_\Gamma\cdot\nabla\Phi_{\geq j}\mathcal{H}$. By \Cref{EEcorollary}, Sobolev embeddings and \Cref{direst}, we have
\begin{equation*}
\begin{split}
\|\mathcal{N}^n\nabla_n\Delta^{-1}F\|_{H^{\frac{1}{2}}(\Gamma_t)}&\lesssim \|F\|_{H^n(\Omega_t)}+\|\Gamma_t\|_{H^k}\sup_{j>0}2^{-j(m+1)}\|\nabla\mathcal{H}a\cdot\nabla G_j^1\|_{H^{s-2}(\Omega_t)}
\\
&+\sup_{j>0}2^{j(n+1-\epsilon)}\|\nabla\cdot (\nabla\mathcal{H}a G_j^2)\|_{H^{-1}(\Omega_t)},
\end{split}
\end{equation*}
where we define the sequence of partitions of $G$ by
\begin{equation*}
G_j^1=\mathcal{H}\mathcal{N}^{m+1}_{<j}a,\hspace{5mm}G_j^2=\sum_{0\leq i\leq m}\mathcal{H}\mathcal{N}_{<j}^i\mathcal{N}_{\geq j}\mathcal{N}^{m-i}a.
\end{equation*}
To estimate $\nabla\mathcal{H}a\cdot\nabla G_j^1$ we use \Cref{Multilinearest}, Sobolev embeddings and \Cref{Hbounds} to obtain
\begin{equation*}
\begin{split}
\|\nabla\mathcal{H}a\cdot\nabla G_j^1\|_{H^{s-2}(\Omega_t)}&\lesssim_{M_s} \|\mathcal{H}a\|_{C^{\frac{1}{2}}(\Omega_t)}\|G_j^1\|_{H^{s-\frac{1}{2}}(\Omega_t)}+\|\mathcal{H}a\|_{H^{s-\frac{1}{2}}(\Omega_t)}\|G_j^1\|_{C^{\frac{1}{2}}(\Omega_t)}
\\
&\lesssim_{M_s}\|a\|_{H^{s-1}(\Gamma_t)}\|G_j^1\|_{H^{s-\frac{1}{2}}(\Omega_t)}.
\end{split}
\end{equation*}
We also clearly have $\|a\|_{H^{s-1}(\Gamma_t)}\lesssim_{M_s}1$. On the other hand, using \Cref{boundaryest}, \Cref{baltrace}, Sobolev embeddings, the regularization bounds for $\Phi_{\leq j}$ and \Cref{Hbounds} we have
\begin{equation*}
\|G_j^1\|_{H^{s-\frac{1}{2}}(\Omega_t)}\lesssim_{M_s}\|\Phi_{\leq j}\mathcal{H}\mathcal{N}_{<j}^ma\|_{H^{s+\frac{1}{2}}(\Omega_t)}\lesssim_{M_s} 2^j\|\mathcal{N}^m_{<j}a\|_{H^{s-1}(\Gamma_t)}. 
\end{equation*}
Iterating, we find that
\begin{equation*}
\|G_j^1\|_{H^{s-\frac{1}{2}}(\Omega_t)}\lesssim_{M_s} 2^{j(m+1)}\|a\|_{H^{s-1}(\Gamma_t)}\lesssim_{M_s}2^{j(m+1)}.
\end{equation*}
To estimate $\nabla\cdot (\nabla\mathcal{H}a G_j^2)$, we use H\"older's inequality and Sobolev embeddings to obtain
\begin{equation*}
\|\nabla\cdot (\nabla\mathcal{H}a G_j^2)\|_{H^{-1}(\Omega_t)}\lesssim_{M_s}\|\nabla\mathcal{H}aG_j^2\|_{L^2(\Omega_t)}\lesssim_{M_s}\|G_j^2\|_{H^{\frac{1}{2}}(\Omega_t)}.
\end{equation*}
From the $H^{\epsilon}\to H^{\frac{1}{2}+\epsilon}$ bound for $\mathcal{H}$, the regularization bounds for $\Phi_{\leq j}$, the trace inequality and the bound $\|n_{\Gamma_t}\|_{C^{\epsilon}(\Gamma_t)}\lesssim_{M_s}1$, we see that 
\begin{equation*}
\|G_j^2\|_{H^{\frac{1}{2}}(\Omega_t)}\lesssim_{M_s} 2^{-j(n+1-\epsilon)}\sum_{0\leq i\leq m}\|\mathcal{H}\mathcal{N}^{m-i}a\|_{H^{n+i+\frac{5}{2}}(\Omega_t)}.
\end{equation*}
Using \Cref{higherpowers}, \Cref{Linfest} and the energy coercivity, we have
\begin{equation*}
2^{j(n+1-\epsilon)}\|G_j^2\|_{H^{\frac{1}{2}}(\Omega_t)}\lesssim_{M_s} \|a\|_{H^{k-1}(\Gamma_t)}+\|\Gamma_t\|_{H^k}\|a\|_{C^{\epsilon}(\Gamma_t)}\lesssim_{M_s}(E^k)^{\frac{1}{2}}.
\end{equation*}
To estimate $F$, using the convention in \Cref{SSRO}, we perform a bilinear frequency decomposition, 
\begin{equation*}
F=\nabla(\mathcal{H}a)^l\cdot \nabla G^{\leq l}+\nabla(\mathcal{H}a)^{\leq l}\cdot \nabla G^l.
\end{equation*}
Estimating the latter term is straightforward. Expanding out $n$ derivatives applied to this term,  using the regularization bounds for $\Phi_{\leq l}$, and by summing in $l$, we have
\begin{equation*}
\|\nabla (\mathcal{H}a)^{\leq l}\cdot \nabla G^{l}\|_{H^n(\Omega_t)}\lesssim_{M_s} \|\mathcal{H}a\|_{C^{\frac{1}{2}+\epsilon}(\Omega_t)}\|G\|_{H^{n+\frac{3}{2}}(\Omega_t)}.
\end{equation*}
Then, using \Cref{higherpowers}, we see that
\begin{equation*}
\|\mathcal{H}a\|_{C^{\frac{1}{2}+\epsilon}(\Omega_t)}\|G\|_{H^{n+\frac{3}{2}}(\Omega_t)}\lesssim_{M_s} (E^k)^{\frac{1}{2}}.
\end{equation*}
To estimate $\nabla (\mathcal{H}a)^l\cdot \nabla G^{\leq l}$, we decompose
\begin{equation*}
\nabla (\mathcal{H}a)^l\cdot \nabla G^{\leq l}=\nabla (\mathcal{H}a)^l\cdot \nabla (G_l^1)^{\leq l}+\nabla (\mathcal{H}a)^l\cdot \nabla (G_l^2)^{\leq l},
\end{equation*}
where $G_l^1$ and $G_l^2$ are as above. Using the previously established bounds for $G_l^1$ and $G_l^2$, we have
\begin{equation*}
\|\nabla (\mathcal{H}a)^l\cdot \nabla (G_l^1)^{\leq l}\|_{H^n(\Omega_t)}\lesssim_{M_s} 2^{-l(m+\frac{3}{2})}\|G_l^1\|_{C^1(\Omega_t)}\|\mathcal{H}a\|_{H^{k-\frac{1}{2}}(\Omega_t)}\lesssim_{M_s} (E^k)^{\frac{1}{2}}.
\end{equation*}
On the other hand, arguing similarly to the estimate for $G_j^2$, it is easy to see that 
\begin{equation*}
\|\nabla (\mathcal{H}a)^l\cdot \nabla (G_l^2)^{\leq l}\|_{H^n(\Omega_t)}\lesssim_{M_s} 2^{l(n+\frac{1}{2}-\epsilon)}\|G_l^2\|_{H^1(\Omega_t)}\|\mathcal{H}a\|_{C^{\frac{1}{2}+\epsilon}(\Omega_t)}\lesssim_{M_s}(E^k)^{\frac{1}{2}}.
\end{equation*}
This gives us control over the second summand in \eqref{higherorderleibniz}. To deal with the other summand, we  use the partition $\mathcal{N}a\mathcal{N}^{m+1}a=H_j^1+H_j^2$ defined by taking $H_j^1:=\Phi_{\leq j}(\mathcal{H}\mathcal{N}a\mathcal{H}\mathcal{N}_{<j}^{m+1}a)$.
Indeed, it is straightforward to verify that
\begin{equation*}
\|H_j^1\|_{L^{\infty}(\Omega_t)}\lesssim_{M_s} 2^{j(m+1)},\hspace{5mm}2^{j(n+\frac{1}{2}-\epsilon)}\|H_j^2\|_{H^{\frac{1}{2}+\epsilon}(\Omega_t)}\lesssim_{M_s}(E^k)^{\frac{1}{2}}.
\end{equation*}
Therefore, by \Cref{higherpowers}, \Cref{baltrace}, \Cref{Multilinearest} and similar arguments as above, we obtain 
\begin{equation*}
\|\mathcal{N}^n(\mathcal{N}a\mathcal{N}^{m+1}a)\|_{H^{\frac{1}{2}}(\Gamma_t)}\lesssim_{M_s}(E^k)^{\frac{1}{2}},
\end{equation*}
which ultimately gives 
\begin{equation*}
\|\nabla\mathcal{H}[\mathcal{N}^{k-2},a]\mathcal{N}a\|_{L^2(\Omega_t)}\lesssim_{M_s}\|[\mathcal{N}^{k-2},a]\mathcal{N}a\|_{H^{\frac{1}{2}}(\Gamma_t)}\lesssim_{M_s}(E^k)^{\frac{1}{2}}.
\end{equation*}
By completely analogous reasoning (i.e.~essentially by replacing $\mathcal{N}^{m+1}a$ with $\mathcal{N}^{m+1}\nabla_Ba$ in the expansion \eqref{higherorderleibniz}), we have 
\begin{equation*}
\|[\mathcal{N}^{k-2},a]\mathcal{N}\nabla_Ba\|_{L^2(\Gamma_t)}\lesssim_{M_s}(E^k)^{\frac{1}{2}}.
\end{equation*}
This finally concludes the proof of \Cref{Energy est. thm}.
\section{Construction of regular solutions}\label{Existence section} 
In this section, we provide a novel method for constructing smooth solutions to the free boundary MHD equations. In contrast to previous works which tend to assume that the initial domain is $\mathbb{T}^2\times (0,1)$, our construction requires no supplemental constraints on the initial data. By carefully combining the results in this section with the difference bounds in \Cref{DF} and the energy estimates in \Cref{HEB}, we will see in \Cref{RS} that solutions in the low regularity regime may be obtained as unique limits of the regular solutions that we construct here. 
\medskip

Most contemporary approaches to constructing solutions to free boundary fluid equations involve a combination of one or more of the following approaches: (i) Fixing the domain using Lagrangian coordinates, (ii) Nash Moser iteration and (iii) Constructing solutions to the corresponding problem with surface tension and then taking the zero surface tension limit. Another approach that was implemented in the case of the free boundary Euler equations in the setting of a laterally infinite ocean with graph geometry can be found in \cite{MR4263411}. The article \cite{MR4263411} relies on a paradifferential reduction akin to the methods employed by Alazard-Burq-Zuily \cite{MR3260858} in the setting of water waves, as well as a subtle iteration scheme. Although conceivably possible, none of these approaches are straightforward at all to implement in the general setting that we consider in this paper. They either suffer from being indirect (in the case of considering the zero surface tension limit), require a sophisticated functional setup,  or are much better adapted to situations where the free surface can be parameterized by a single coordinate patch (for instance, when making use of Lagrangian coordinates or the paradifferential method).  For these reasons, we propose a direct, geometric approach, implemented fully within the Eulerian coordinates. As we will see, our scheme is robust, can be utilized in domains with very general geometries and  has the potential to be  adapted to a wide class of free boundary problems. In particular, we do not need to assume that our domains are simply connected, which stands in stark contrast to  all previous works on the free boundary MHD equations.
\medskip

Our overarching strategy will be to construct approximate solutions to the free boundary MHD equations by discretizing the problem in time on a small enough time-scale $\epsilon>0$. The most crude implementation of this scheme would be to try to view the MHD equations as an ODE and use an Euler type iteration. This approach, unsurprisingly, completely fails. However, it is instructive to outline the main obstructions to using this strategy. A na\"ive implementation of this method would be to define an approximate solution $(v(\epsilon), B(\epsilon),\Gamma(\epsilon))$ at time $\epsilon$  in terms of the variables $W^\pm(\epsilon)$  by taking
\begin{equation}\label{naivetransport}
\begin{cases}
&W^\pm(\epsilon)=W_0^\pm+\epsilon(-v_0\cdot\nabla W^\pm_0-\nabla P_0\pm\nabla_{B_0}W_0^\pm),
\\
&\Gamma(\epsilon)=(I+\epsilon v_0)\Gamma_0.
\end{cases}
\end{equation}
Of course (barring the technical issue of $W^\pm(\epsilon)$ and $W_0^\pm$ not sharing the same domain), the term $\epsilon (v_0\cdot\nabla W_0^\pm)$ leads to a full derivative loss when directly estimating the requisite Sobolev norm of the new iterate $W^\pm(\epsilon)$. It turns out that one can partially ameliorate this derivative loss by building the transport structure of the equations into each iteration. However, one is still left with a $\frac{1}{2}$ derivative loss coming from the pressure gradient (where the free surface regularity enters at leading order) and from the term $\nabla_{B_0}W^\pm_0$. Although not a priori obvious, the motivation for this latter $\frac{1}{2}$ derivative loss comes from the dynamic problem,  where one heuristically interprets $\nabla_{B_0}$ as a half derivative. 
\medskip


Our  goal will be to retain the simplicity of the Euler method (together with the implicit transport step mentioned above),  but ameliorate the remaining derivative loss by performing a suitable regularization of the data prior to carrying out each iteration. Roughly speaking, we will split the time step into two main pieces:
\begin{enumerate}
    \item Regularization.
    \item Euler plus transport.
\end{enumerate}
To ensure that the regularity of the data is preserved over a unit time-scale (or $\approx \epsilon^{-1}$ many iterations), the regularization step needs to be performed very carefully. In particular, one has to design a regularization that preserves the natural boundary conditions of the problem -- which is especially difficult in the case of the magnetic boundary condition $B\cdot n_\Gamma=0$ -- but also preserves at leading order the size of the energy functional constructed in \Cref{HEB}, which  serves as our quantitative measure of the regularity of each iterate. A detailed synopsis of the regularization procedure will be presented in \Cref{Outline existence}. However, to give a brief summary, our strategy will be to take a modular approach and decouple the regularization process into two main steps. In the first step, we will simultaneously regularize the free surface (at the same scale as our discretization) and the irrotational components of the variables $W^\pm_{\epsilon}$ (which are measured by the good variables $\mathcal{G}^\pm$ from \Cref{HEB}), taking care to ensure that the dynamic and magnetic boundary conditions are preserved along the way. In the second main step  we perform an additional  regularization of $W^\pm_0$ in the direction of the vector field $\nabla_{B_\epsilon}$. This serves to ameliorate the derivative loss coming from the term $\epsilon \nabla_{B_0}W^\pm_0$ described in the above na\"ive iteration. Note that, a priori, the term $\nabla_{B_0}W^\pm_0$ seems to lead to a full derivative loss in the requisite estimates, as, quantitatively, it should be of size $\delta^{-1}$ at top order if $W^\pm_0$ is regularized at some scale $\delta>0$. The purpose of our secondary regularization is to reduce this to a $\frac{1}{2}$ derivative loss, i.e.,~we will ensure that $\nabla_{B_0}$ scales like a factor of $\delta^{-\frac{1}{2}}$ when applied to $W^\pm_0$. This is consistent with how the vector field $\nabla_{B}$ scales in the dynamic problem. 
\medskip

We believe that the first step in the above approach is quite robust and should be applicable to constructing solutions to other free boundary problems. In particular, this first step is entirely sufficient for constructing solutions to the free boundary Euler equations (i.e., when  the  magnetic field is zero) and gives a new and relatively simple proof in that setting (as most of the difficulties we will encounter are due to the magnetic field). The second step is, of course, more specialized to the problem at hand. In these two steps, it should be noted that we essentially avoid regularizing the variables $\omega_0^\pm$   (except in the direction of $\nabla_B$) and instead directly propagate their regularity through each iteration. This is motivated by the fact that the variables $\omega_0^\pm$ are essentially transported by the velocity, up to the term $\nabla_{B_0}\omega_0^\pm$, which is formally orthogonal to $\omega_0^\pm$ in $L^2(\Omega_0)$.
\medskip

The iteration strategy that we employ in this section takes some very rough inspiration from the time discretization approach carried out in the case of a compressible gas in \cite{ifrim2020compressible}. We stress, however, that aside from the broad structure of the argument (i.e.~regularization plus Euler iteration), the main difficulties here are completely different. One obvious reason for this is that the surface of a liquid carries a non-trivial energy, and so the geometry of the free boundary hypersurface $\Gamma_0$ plays a leading role when propagating the energy bounds through each iteration. Moreover, given the matched regularity of the magnetic field and the free surface, it turns out to be an extremely delicate task to understand how the tangency condition $B\cdot n_\Gamma=0$ is affected when regularizing the hypersurface. As we will see below,  properly understanding these issues  will require a host of subtle technical innovations.  
\subsection{Basic setup and simplifications}\label{basicsimp}
We start by fixing a smooth reference hypersurface $\Gamma_*$ and a collar neighborhood $\Lambda_*:=\Lambda(\Gamma_*,\epsilon_0,\delta_0).$ Here, $\epsilon_0$ and $\delta_0$ are small but fixed positive constants.  Given $k\gg\frac{d}{2}+1$  and an initial state $(v_0, B_0, \Gamma_0)\in\mathbf{H}^k$, our objective is to construct a solution $(v(t),B(t),\Gamma_t)\in\mathbf{H}^k$ to the free boundary MHD equations on a time interval $[0,T]$ whose length depends solely on the size of $\|(v_0, B_0, \Gamma_0)\|_{\mathbf{H}^k}$, the lower bound in the Taylor sign condition and the collar neighborhood $\Lambda_*$. 
\medskip

Fix $M>0$. Given a small time step $\epsilon>0$ and a suitable triple of initial data $(v_0,B_0,\Gamma_0)\in\mathbf{H}^k$ with $\|(v_0,B_0,\Gamma_0)\|_{\mathbf{H}^k}\leq M$, our goal will be to construct a sequence $(v_{\epsilon}(j\epsilon),B_{\epsilon}(j\epsilon),\Gamma_{\epsilon}(j\epsilon))\in\mathbf{H}^k$ with the following properties:
\begin{enumerate}
\item (Norm bound). There is a uniform constant $c_0>0$ depending only on $\Lambda_*$, $M$ and the lower bound in the Taylor sign condition such that if $j$ is an integer with $0\leq j\leq c_0\epsilon^{-1}$, then 
\begin{equation*}
\|(v_{\epsilon}(j\epsilon),B_{\epsilon}(j\epsilon),\Gamma_{\epsilon}(j\epsilon))\|_{\mathbf{H}^k}\leq C(M),
\end{equation*}
where $C(M)>0$ is some constant depending on $M$. 
\item (Approximate solution).  We have
\begin{equation*}
\begin{cases}
&v_{\epsilon}((j+1)\epsilon)=v_{\epsilon}(j\epsilon)-\epsilon[v_{\epsilon}(j\epsilon)\cdot\nabla v_{\epsilon}(j\epsilon)-B_{\epsilon}(j\epsilon)\cdot\nabla B_{\epsilon}(j\epsilon)+\nabla P_{\epsilon}(j\epsilon)]+\mathcal{O}_{C^{3}}(\epsilon^2),
\\
&B_{\epsilon}((j+1)\epsilon)=B_{\epsilon}(j\epsilon)-\epsilon[v_{\epsilon}(j\epsilon)\cdot\nabla B_{\epsilon}(j\epsilon)-B_{\epsilon}(j\epsilon)\cdot\nabla v_{\epsilon}(j\epsilon)]+\mathcal{O}_{C^3}(\epsilon^2),
\\
&\nabla\cdot B_{\epsilon}((j+1)\epsilon)=\nabla\cdot v_{\epsilon}((j+1)\epsilon)=0,
\\
 &B_{\epsilon}((j+1)\epsilon)\cdot n_{\epsilon}((j+1)\epsilon)=0,
\\
&\Omega_{\epsilon}((j+1)\epsilon)=(I+\epsilon v_{\epsilon}(j\epsilon))(\Omega_{\epsilon}(j\epsilon))+\mathcal{O}_{C^3}(\epsilon^2).
\end{cases}
\end{equation*}
\end{enumerate}
By virtue of the size of $k$ and the above properties, the Taylor sign condition will be propagated over $\approx\epsilon^{-1}$ many iterations (with an implicit constant depending on the initial lower bound for $a_0$ and on $M$). We will therefore suppress the lower bound in the Taylor sign condition from our notation in the sequel. An important feature of the above iteration scheme is that it suffices to only carry out a single step. More specifically, we have the following theorem.
\begin{theorem}\label{onestepiteration}
 Let $k$ be a sufficiently large integer and let $M > 0$. Consider an initial data $(v_0, B_0,\Gamma_0)\in\mathbf{H}^k$
 so that $\|(v_0, B_0, \Gamma_0)\|_{\mathbf{H}^k}\leq M$ and $\omega_0^{\pm}$ and $\Gamma_0$ satisfy the bootstrap hypothesis
\begin{equation}\label{inductiveregbound}
\begin{split}
\|\omega_0^\pm\|_{H^k(\Omega_0)}\leq K(M)\epsilon^{-\frac{3}{2}},\hspace{5mm}\|\Gamma_0\|_{H^{k+\frac{1}{2}+\gamma}}\leq K(M)\epsilon^{-\frac{1}{2}-\gamma}
\end{split}
\end{equation}
for some fixed $0<\gamma\ll 1$ and sufficiently large positive constant $K(M)>0$. Under these hypotheses, there exists a one step iterate $(v_0,B_0,\Gamma_0)\mapsto (v_1,B_1,\Gamma_1)$ with the following properties:
\begin{enumerate}
\item (Energy monotonicity).
\begin{equation}\label{EMBITT}
E^{k}(v_1, B_1,\Gamma_1)\leq E^{k}(v_0, B_0,\Gamma_0)+C(M)\epsilon.
\end{equation}
\item (Good pointwise approximation). 
\begin{equation}
\label{approx-sln}
\begin{cases}
&v_1=v_0-\epsilon (v_0\cdot\nabla v_0-B_0\cdot\nabla B_0+\nabla P_0)+\mathcal{O}_{C^3}(\epsilon^2),\hspace{5mm}\text{on}\hspace{2mm}\Omega_1\cap\Omega_0,
\\
&B_1=B_0-\epsilon(v_0\cdot \nabla B_0 -B_0\cdot\nabla v_0)+\mathcal{O}_{C^3}(\epsilon^2),\hspace{5mm}\text{on}\hspace{2mm}\Omega_1\cap\Omega_0,
\\
&\nabla\cdot v_1=\nabla\cdot B_1=0,\hspace{5mm}\text{on}\hspace{2mm}\Omega_1,
\\
&B_1\cdot n_1=0,\hspace{5mm}\text{on}\hspace{2mm}\Gamma_1,
\\
&\Omega_{1}=(I+\epsilon v_{0})(\Omega_{0})+\mathcal{O}_{C^3}(\epsilon^2).
\end{cases}
\end{equation}
\item (Updated bootstrap). There holds
\begin{equation}\label{regboundprop}
\begin{split}
\|\omega_1^\pm\|_{H^k(\Omega_1)}\leq (1+C(M)\epsilon)K(M)\epsilon^{-\frac{3}{2}},\hspace{5mm}\|\Gamma_1\|_{H^{k+\frac{1}{2}+\gamma}}\leq K(M)\epsilon^{-\frac{1}{2}-\gamma}
\end{split}
\end{equation}
with the same constant $K(M)$ as above.
\end{enumerate}
\end{theorem}
The rotational bound in \eqref{regboundprop} ensures that the constant in the high regularity bootstrap \eqref{inductiveregbound} for $\omega_1^\pm$ only grows by an amount comparable to $\epsilon$ times the initial bound, which implies that on a unit time-scale (or $\approx_M \epsilon^{-1}$ many iterations), $\omega^\pm$  retains its quantitative regularity. The reason we must track this bound through each iteration is because we  do not directly regularize (except in a very mild way) the rotational parts of $W^\pm$ in our analysis. Heuristically, we can get away with this because, in the dynamic problem,  $\omega^\pm$ is approximately transported by the vector field $W^\mp$. Hence, there should be no severe loss of regularity when carrying out the Euler plus transport iteration. We remark, however, that the situation is more complicated in practice, as we will need to use the velocity field $v$ to transport the free surface rather than the variables $W^\mp$ (which would define two different transported domains). 
\begin{remark}
The reader may wonder at this point if it would be simpler to regularize the full variable $W^\pm_0$ in each iteration rather than just its irrotational part (to avoid needing to precisely track the higher regularity $H^k$ bounds for $\omega_0^\pm$).  Attempting this introduces a whole host of technical issues which seem very difficult (if at all possible) to overcome. Such a regularization has to preserve the boundary conditions and the size of the energy functional. This is even further complicated by the fact that the rotational part of the energy functional involves fractional order Sobolev norms on $\Omega$ which are somewhat awkward to work with in a bounded domain (in contrast to the boundary $\Gamma$ where we have ellipticity of $\mathcal{N}$ and where the dynamics of the irrotational part of $W^\pm$ essentially live).
\end{remark}
The bootstrap assumption on the free surface in \eqref{inductiveregbound} is  for technical convenience to avoid certain logarithmic divergences in the forthcoming estimates. This part of the bootstrap will be  easy to close, as we will directly regularize the free surface in each iteration. The energy monotonicity property \eqref{EMBITT}, along with the energy coercivity bound from \Cref{Energy est. thm}, will ensure that the resulting sequence $(v_{\epsilon}(j\epsilon),B_{\epsilon}(j\epsilon),\Gamma_{\epsilon}(j\epsilon))$ of approximate solutions that we construct remains uniformly bounded in $\mathbf{H}^k$ for $j\ll_M\epsilon^{-1}$. The second property  in \Cref{onestepiteration} will ensure that $(v_{\epsilon}(j\epsilon),B_{\epsilon}(j\epsilon),\Gamma_{\epsilon}(j\epsilon))$ converges in a weaker topology to a solution of the equation. 
\medskip

We remark that the assumption on the initial iterate \eqref{inductiveregbound}  is harmless in practice. To achieve this, we can use the cruder regularization in \Cref{envbounds} to replace the first iterate in the resulting sequence $(v_{\epsilon}(j\epsilon),B_{\epsilon}(j\epsilon),\Gamma_{\epsilon}(j\epsilon))$ with a suitable $\epsilon^{-1}$ scale regularization, so that the base case is satisfied. Importantly, such a regularization is only performed a single time -- on the initial iterate -- as it is merely bounded on $\mathbf H^k$.  In contrast, we require the much stricter energy monotonicity bound \eqref{EMBITT} for all other iterations, which will be crucial for propagating the $\mathbf{H}^k$ bounds over unit time-scales.  
\begin{remark}
In \eqref{EMBITT}, it is imperative that the constant be of the form $1+ C(M)\epsilon$, as we will run our iteration over $\approx_M \epsilon^{-1}$ steps. Note that the definition of the $H^k$ norm involves both integer and half-integer Sobolev spaces. Therefore, we must be careful to properly define fractional Sobolev spaces so that interpolation arguments keep the form of the good constant in \eqref{EMBITT}. Hence, for this section, we define $H^s(\Omega)$ for fractional $s$ via direct interpolation of integer order spaces as in \eqref{Interpolation on domains}. As mentioned below \eqref{Interpolation on domains}, this definition is uniformly equivalent to our previous definition, for domains with boundary in the collar.  
\end{remark}
\subsection{Outline of the argument}\label{Outline existence} We now briefly overview the section. The first step is selecting a suitable regularization scale. To help motivate this, we begin by recalling that for an exact solution to the dynamic problem, the estimates in \Cref{HEB} give us the a priori bound
\begin{equation*}
E^k(v(t),B(t),\Gamma(t))=E^k(v(0),B(0),\Gamma(0))+\mathcal{O}_M(t)
\end{equation*}
on some time interval depending on the data size. Our basic strategy in this section will be to carry out a discrete version of this energy bound for our approximate solutions. Phrased in terms of a single iterate (with time step $\epsilon$), we want to show that
\begin{equation*}
E^k(v_1,B_1,\Gamma_1)=E^k(v_0,B_0,\Gamma_0)+\mathcal{O}_M(\epsilon).
\end{equation*}
By carefully expanding the difference of the two energies $E^k(v_1,B_1,\Gamma_1)$ and $E^k(v_0,B_0,\Gamma_0)$, we will be able to exhibit the same leading cancellation observed in \Cref{HEB}. However, unlike for exact solutions, in this setting we will also have to estimate additional quadratic error terms, which very schematically corresponds to estimating error terms akin to  $\epsilon^2\|D_tW^\pm\|_{H^k(\Omega_t)}^2$. For such error terms to be of size $\mathcal{O}_M(\epsilon)$, we need our regularization scale to be such that $D_t$ scales like $\epsilon^{-\frac{1}{2}}$. Since in the exact dynamic problem $D_t$ scales like a spatial derivative of order $\frac{1}{2}$, this suggests that the scale we perform our regularization at should be $\epsilon^{-1}$.
\medskip

Our regularization will consist of three steps, for which we now provide a schematic overview. The first step will be to regularize the free surface, along with the irrotational variables $\mathcal{N}W^\pm\cdot n_\Gamma$, which correspond to the good variables $\mathcal{G}^\pm$. Regularizing these variables in tandem is natural due to the coupling between the variables $\mathcal{G}^\pm$ and $a$ which we observed in the dynamic problem in \Cref{HEB}. To construct the regularized free surface $\Gamma_{\epsilon}$ from $\Gamma_0$, we will define a smooth, one-parameter family of hypersurfaces $(\Gamma_{\delta})_{0\leq\delta\leq\epsilon}$ by solving a suitable heat-type equation (in the parameter $\delta$) on the reference hypersurface $\Gamma_*$. To regularize $W^\pm$, we will construct a corresponding one-parameter family $(W^\pm_{\delta})_{0\leq\delta\leq\epsilon}$ by partially regularizing the trace of $W^\pm_0$. There are two critical properties that our regularization operators will satisfy. First, they will not create significant energy growth. Secondly, they will preserve (at leading order) the magnetic boundary condition. To achieve this, we construct $W^\pm_{\delta}$ by defining an appropriate  non-local gradient flow in the parameter $\delta$. This flow will be carefully crafted to ensure (among other things) that the boundary term $B_{\delta}\cdot n_{\Gamma_\delta}$ obeys, at leading order, a suitable parabolic equation. This will allow us to propagate the boundary condition $B_0\cdot n_{\Gamma_0}=0$ from $\delta=0$ to $\delta=\epsilon$. 
\medskip

The second step in our scheme is very simple. Its purpose is to crudely regularize the full variables $W^\pm_{0}$ at the scale $\epsilon^{-3}$. This is performed purely for technical reasons in order to close the bootstrap for $\Gamma_1$ in \eqref{regboundprop} at the end of each iteration. This step will also provide us with some crude quantitative high regularity bounds for $W^\pm_{\epsilon}$, which will help us to carry out the final, third step of our regularization. Such a construction will be carried out using an essentially standard mollification.

\medskip
The third and final step in our regularization scheme is to regularize the variables $W_0^\pm$ in the direction of the vector field $\nabla_B$. To be able to estimate the resulting error terms in the Euler step of the iteration, we will need $\nabla_B$ to scale like $\epsilon^{-\frac{1}{2}}$. We will achieve this by solving a suitable nonlinear elliptic equation on the fixed regularized domain $\Omega_{\epsilon}$. We will postpone providing a precise description of the requisite construction here. However, the na\"ive model case to consider is the following second-order elliptic equation
\begin{equation*}
W^\pm_{\epsilon}-\epsilon\nabla_{B_{\epsilon}}^2W^\pm_{\epsilon}=W^\pm_0.
\end{equation*}
We will not work with this particular equation since (among other things) the solution will cause the property \eqref{approx-sln} to fail. However, to understand the spirit of our argument, this model provides a good first approximation. As in the first step of the regularization scheme, it is a delicate matter to understand how such an elliptic equation affects the energy and the boundary conditions. This is where most of the work will be concentrated.
\medskip 

The last step in our construction is to carry out the Euler plus transport iteration to define the new variable $(v_1, B_1, \Gamma_1)$. The main difficulty in this section is to ensure that the energy only grows by a factor of $1+C(M)\epsilon$. We will  guarantee this by performing a discrete version of the energy estimate in \Cref{HEB}. Since we are working with approximate solutions, this will not be entirely trivial, but we now have the massive benefit of being able to allow for implicit constants to be of size $C(M)$. To exhibit the energy cancellations observed in \Cref{HEB}, we carefully relate the good variables $a$, $\mathcal{G}^\pm$, and $\omega^\pm$ for the new iterate to the corresponding good variables for the regularized data. 
\medskip

After carrying out the above steps, we will then straightforwardly construct $\mathbf{H}^k$ solutions to the free boundary MHD equations by applying our iteration  $\mathcal{O}_M(\epsilon^{-1})$ many times and sending the parameter $\epsilon\to 0$. At the end of this section, in \Cref{existencecommutators},  we also collect and prove some of the various commutator estimates we will use in our analysis. 
\subsection{Step 1: Free surface and irrotational regularization}\label{DRS}
We begin with the first regularization step. For this, we have the following proposition. 
\begin{proposition}\label{domainregularizationprop} 
 Let $(v_0, B_0,\Gamma_0)\in\mathbf{H}^k$ be as in \Cref{onestepiteration}. There exists a domain $\Omega_{\epsilon}$ contained in $\Omega_0$ with boundary $\Gamma_{\epsilon}\in\Lambda_*$ and divergence-free functions $v_{\epsilon}$ and $B_{\epsilon}$ on $\Omega_\epsilon$ such that $B_{\epsilon}$ is tangent to $\Gamma_{\epsilon}$ and 
 \begin{enumerate}
\item (Energy monotonicity).
\begin{equation}\label{Emonsurf}
E^{k}(v_{\epsilon},B_{\epsilon},\Gamma_{\epsilon})\leq (1+C(M)\epsilon)E^{k}(v_0, B_0,\Gamma_0).
\end{equation}
\item (Good pointwise approximation).
\begin{equation}\label{PP approx}
\eta_{\epsilon}=\eta_{0}+\mathcal{O}_{C^3}(\epsilon^2)\hspace{2mm}on\hspace{2mm}\Gamma_*,\hspace{5mm} (v_{\epsilon},B_{\epsilon})=(v_0,B_0)+\mathcal{O}_{C^3}(\epsilon^2)\hspace{2mm}on\hspace{2mm}\Omega_\epsilon.
\end{equation}
\item (Domain regularization bound). For every $\alpha\geq 0$, there holds,
\begin{equation}\label{surfbound-par}
\|\Gamma_{\epsilon}\|_{H^{k+\alpha}}\lesssim_{M,\alpha} \epsilon^{-\alpha}.
\end{equation}
\item (Regularization bounds). We have the following irrotational regularization bounds (in two equivalent flavors) 
\begin{equation*}
\|\mathcal{N}W^\pm_{\epsilon}\cdot n_{\Gamma_\epsilon}\|_{H^{k-\frac{1}{2}}(\Gamma_{\epsilon})}+\|\nabla^{\top}W^\pm_{\epsilon}\cdot n_{\Gamma_\epsilon}\|_{H^{k-\frac{1}{2}}(\Gamma_{\epsilon})}\lesssim_M \epsilon^{-1}.
\end{equation*}
\item (Bootstrap propagation). If $K(M)$  in \eqref{inductiveregbound}  is large enough then 
\begin{equation*}
\|\omega_{\epsilon}^\pm\|_{H^k(\Omega_\epsilon)}\leq (1+C(M)\epsilon)K(M)\epsilon^{-\frac{3}{2}},\hspace{5mm}\|\Gamma_{\epsilon}\|_{H^{k+\frac{1}{2}+\gamma}}\leq \frac{1}{2}K(M)\epsilon^{-\frac{1}{2}-\gamma}.
\end{equation*}
That is, the constant for $\Gamma_{\epsilon}$ has improved by a factor of $2$ and the constant for $\omega_{\epsilon}^\pm$ has only grown by a factor of $(1+C(M)\epsilon)$.
 \end{enumerate}
\end{proposition}
\begin{remark}
In a loose sense, property (iv) states that the irrotational component of $W^\pm_{\epsilon}$ is regularized at scale $\epsilon^{-1}$ when measured in $H^{k+1}(\Omega_\epsilon)$. This is one ingredient that will be needed to prove the necessary energy monotonicity bound in the transport step later on. 
\end{remark}
To regularize $\Gamma_0$, we begin by defining a family of preliminary parabolic regularizations of its parameterization $\eta_0$ in collar coordinates,
\begin{equation*}
\tilde{\eta}_{\delta}=e^{\delta^2\Delta_{\Gamma_*}}\eta_0,
\end{equation*}
for each $\delta\in (0,\epsilon]$. Here, $\Delta_{\Gamma_*}$ is the Laplace-Beltrami operator for $\Gamma_*$. The reason for using the operator $e^{\delta^2\Delta_{\Gamma_*}}$ instead of $e^{-\delta|D|}$ is to ensure that when $k$ is large enough, we have $\|\partial_{\delta}\tilde{\eta}_{\delta}\|_{H^{k-2}(\Gamma_*)}\lesssim_M \delta$. In particular, this ensures that we have $\tilde{\eta}_{\delta}=\eta_0+\mathcal{O}_{C^3}(\delta^2)$. Moreover, we have the regularization bound $\|\tilde{\eta}_{\delta}\|_{H^{k+\alpha}(\Gamma_*)}\lesssim_{M,\alpha}\delta^{-\alpha}.$ The hypersurface defined by $\tilde{\eta}_{\epsilon}$ will not quite be what we use for our regularized domain. Instead, we slightly correct $\tilde{\eta}_{\delta}$ by defining through the collar parameterization
\begin{equation*}
\eta_{\delta}=\tilde{\eta}_{\delta}-C\delta^2,
\end{equation*}
where $C$ is some positive constant depending on $M$ only, imposed to ensure that the domain  $\Omega_{\delta}$ associated to $\Gamma_{\delta}$ is contained in $\Omega_0$. We then define our regularized surface $\Gamma_{\epsilon}$ (and correspondingly the domain $\Omega_{\epsilon}$) using the function $\eta_{\epsilon}$. Clearly, $\Gamma_{\epsilon}$ satisfies properties (ii)-(iii) in \Cref{domainregularizationprop} and also (v) if $K(M)$ is large enough. Before proceeding to the proof of energy monotonicity, we note that the above construction gives rise to a flow velocity $V_{\delta}$ in the parameter $\delta$ for each $\delta\in (0,\epsilon]$ for the family of hypersurfaces $\Gamma_{\delta}$ by composing $\partial_{\delta}\eta_{\delta}\nu$ with the inverse of the collar coordinate parameterization $x\mapsto x+\eta_{\delta}(x)\nu(x)$. We may assume that $V_{\delta}$ is defined on $\Omega_{\delta}$ by defining
\begin{equation*}
V_{\delta}:=\mathcal{H}_*(\nu\partial_{\delta}\eta_{\delta})\circ\Phi_{\delta}^{-1}
\end{equation*}
where $\mathcal{H}_*$ is the harmonic extension operator for the reference domain $\Omega_*$ and $\Phi_{\delta}=id_{\Omega_*}+\mathcal{H}_*(\nu\eta_{\delta})$, which is a diffeomorphism from $\Omega_*$ to $\Omega_{\delta}$ with $\|\Phi_{\delta}\|_{H^{k+\frac{1}{2}}(\Omega_*)}\lesssim_M 1$. For simplicity, we write $\Psi_{\delta}:=\Phi_{\delta}^{-1}$, which satisifes a similar bound. As a consequence, we observe that for $1\leq s\leq k+\frac{1}{2}$, we have
\begin{equation*}
\|V_{\delta}\|_{H^s(\Omega_{\delta})}\lesssim_M \|V_{\delta}\|_{H^{s-\frac{1}{2}}(\Gamma_{\delta})}.
\end{equation*}
We use $D_{\delta}:=\partial_{\delta}+V_{\delta}\cdot\nabla$ to denote the associated material derivative, which will be tangent to the family of hypersurfaces $\Gamma_{\delta}$. For a function $f$ defined on $\Gamma_{\delta}$ or $\Omega_{\delta}$, we will write $f_*$ as shorthand for its pullback $f\circ\Phi_{\delta}$ to the reference hypersurface or domain, respectively. We will also write $f^*$ as shorthand for $f\circ\Phi_{\delta}^{-1}$ for a function $f$ defined on $\Gamma_*$ or $\Omega_*$.
\medskip

To define $v_{\epsilon}$ and $B_{\epsilon}$, we need the following proposition, which we state for a general vector field $X$.
\begin{proposition}\label{Bconstruction} 
Let $k>\frac{d}{2}+1$ be a large integer. Assume that $\Gamma_0$ is in $H^{k+\frac12}$. 
Let $X_0$ be a divergence-free $H^k$ vector field on $\Omega_0$. Then for $\delta\in [0,\epsilon]$, there exists a unique divergence-free solution $X:=X_{\delta}\in C([0,\epsilon] ; H^k(\Omega_\delta))$ 
for the evolution
\begin{equation}\label{parabolicregularizationconstruction}
\begin{cases}
&D_{\delta}X_{\delta}=\nabla\phi_{\delta}+\nabla\Delta^{-1}(\partial_iV_{\delta,j}\partial_jX_{\delta,i})
\hspace{5mm}\text{on}\hspace{2mm}\Omega_{\delta},
\\
&\Delta\phi_{\delta} = 0\hspace{5mm}\text{on}\hspace{2mm}\Omega_{\delta},
\\
&(\nabla_n\phi_{\delta})_*=2\delta \Delta_{\Gamma_*}X_{\delta,*}\cdot n_{\delta,*} - \frac{2\delta}{|\Gamma_{\delta}|}\int_{\Gamma_{\delta}}  (\Delta_{\Gamma_*}X_{\delta,*}\cdot n_{\delta,*})^* dS  \hspace{5mm}\text{on}\hspace{2mm}\Gamma_{*}.
\end{cases}
\end{equation}
Moreover, we have the  energy estimate 
\begin{equation}\label{X-delta}
\sup_{0\leq\delta\leq\epsilon}\|X_{\delta}\|_{H^k(\Omega_{\delta})}^2 \lesssim C(\|\Gamma_0\|_{H^{k+\frac{1}{2}}})\|X_0\|_{H^k(\Omega_0)}^2.
\end{equation}
\end{proposition}

We remark that the assumption that $\Gamma_0$ is in $H^{k+\frac12}$ is somewhat excessive, making this a soft result. However, it  has the merit that it simplifies the proof and it suffices for
our applications. A stronger result, which will exploit the dissipative effect of the Neumann data and which only assumes that $\Gamma_0$  is in $H^{k}$ is also valid, and is a consequence of the sharper energy estimates
proved later in this section. The softer bound with dependence on $\Gamma_0$ in $H^{k+\frac{1}{2}}$ (which thanks to \eqref{inductiveregbound} depends only on $\epsilon$ and not $\delta$) will be used to set up the necessary bootstrap when proving the improved estimate later. One important, but simple, ingredient for the proof that we will also use extensively in the sequel is the following lemma, which asserts that the normal $n_{\delta}$ and mean curvature $\kappa_{\delta}$ satisfy heat-type equations (up to lower order terms).
\begin{lemma}\label{epsderivativeofnormal} There exist smooth functions $R_{\delta}^1$ and $R_{\delta}^2$ defined on $\Gamma_{*}$ such that
\begin{equation*}
\frac{d}{d\delta} n_{\delta,*}=2\delta \Delta_{\Gamma_*}n_{\delta,*}+R_{\delta}^1,\hspace{5mm}\frac{d}{d\delta}\kappa_{\delta,*}=2\delta\Delta_{\Gamma_*}\kappa_{\delta,*}+R_{\delta}^2,
\end{equation*}
with
\begin{equation*}
\|R_{\delta}^2\|_{H^s(\Gamma_*)}\lesssim_M \delta(1+\|\Gamma_{\delta}\|_{H^{s+2}}),\hspace{5mm}\|R_{\delta}^2\|_{H^s(\Gamma_*)}\lesssim_M \delta(1+\|\Gamma_{\delta}\|_{H^{s+3}}),\hspace{5mm}s\geq 0.
\end{equation*}
\end{lemma}
The proof is a straightforward computation in collar coordinates. We omit the details.
\medskip

Now, we establish \Cref{Bconstruction}.
\begin{proof}
The difficulty here lies in the fact that the regularities of $\Gamma_0$ and of $X_0$ 
are closely matched. However, the above equation is a linear evolution for $X_\delta$, with $\Gamma_\delta$ already given. We take advantage of this fact in order to decouple the two regularities in the existence and uniqueness parts, relegating the coupled regularities to the stage of proving energy estimates. So, to start with, let us generously assume that 
\begin{equation}\label{better-Gamma}
\Gamma_0 \in  H^{k+\frac32}. 
\end{equation}
Motivated by the natural div-curl structure of the system, one can prove energy estimates for the evolution \eqref{parabolicregularizationconstruction} by establishing energy estimates for the pair
\[
(\omega_\delta := \nabla \times X_\delta, f_\delta := X_\delta \cdot n_\delta) \in H^{k-1}(\Omega_\delta) \times H^{k-\frac12}(\Gamma_\delta).
\]
Thanks to \eqref{epsderivativeofnormal} and the product rule, we have the equations
\begin{equation}\label{transp-par}
\begin{cases}
&D_{\delta}\omega_{\delta, ij}= \partial_iV_{\delta,j}\partial_jX_{\delta,i}-
\partial_jV_{\delta,i}\partial_iX_{\delta,j}
,
\\
&\frac{d}{d\delta}f_{\delta,*} - 2\delta \Delta_{\Gamma_*} f_{\delta,*}= -  4\delta \nabla^\top X_{\delta,*} \cdot \nabla^\top n_{\delta,*} - \frac{2\delta}{|\Gamma_{\delta}|}\int_{\Gamma_{\delta}}  (\Delta_{\Gamma_*}X_{\delta,*}\cdot n_{\delta,*})^* dS
\\
& \qquad \qquad \qquad \qquad \ \ \ +(\nabla_n\Delta^{-1}(\partial_iV_{\delta,j}\partial_jX_{\delta,i}))_*+R_{\delta},
\end{cases}
\end{equation}
where $X_\delta$  can be estimated (using \Cref{Balanced div-curl}) in an elliptic fashion  by $\omega_\delta$ and $f_\delta$ via a div-curl system, with bounds
\begin{equation}\label{div-curl-ell}
\| X_\delta\|_{H^k(\Omega_\delta)} \leq C(\|\Gamma_0\|_{H^{k+\frac{1}{2}}})(\|\omega_\delta \|_{H^{k-1}(\Omega_\delta)}+ \|f_\delta \|_{H^{k-\frac12}(\Gamma_\delta)}
+ \| X_\delta\|_{L^2(\Omega_\delta)})
\end{equation}
and
where $R_{\delta}$ satisfies
\begin{equation*}\label{remainderforgradientflow}
\int_{0}^{\delta}\|R_{\tau}\|_{H^{k-\frac{1}{2}}(\Gamma_{\tau})}d\tau\lesssim_M \delta^{\frac{1}{2}}\sup_{0\leq\tau\leq\delta}\|X_{\tau}\|_{H^{k-\frac{1}{2}}(\Gamma_{\tau})}.
\end{equation*}
Here we carefully note that the implicit constant in \eqref{div-curl-ell}  depends only on the $H^{k+\frac12}$ norm of $\Gamma_0$.  We also remark  that the full unique solvability of the div-curl system requires a finite dimensional set of topological constraints, see \cite{Mitrea} and references therein. We emphasize that we are not assuming that our domain is simply-connected, so we cannot directly appeal to such results to obtain solutions. 
\medskip

We observe that \eqref{transp-par} is a coupled transport/parabolic system. 
The transport part in the moving domain $\Omega_\delta$ advected by $V_\delta$ is classically well-posed in $H^{k-1}$, while the parabolic part is well-posed 
in $H^{k-\frac12}$. However, in order to view the source terms on the right-hand side as perturbative, we would have to dynamically propagate the constraints mentioned above.
To bypass this difficulty, we will use the system \eqref{transp-par} only for the purpose 
of proving energy estimates, but not for existence. Instead, for existence we use a mild 
frequency truncation of the Neumann trace of $\phi_\delta$, and consider the regularized system
\begin{equation}\label{parabolicregularizationconstructionj}
\begin{cases}
&D_{\delta}X^l_{\delta}=\nabla\phi_{\delta}^l+\nabla\Delta^{-1}(\partial_iV_{\delta,j}\partial_j X_{\delta,i}^l)
\hspace{5mm}\text{on}\hspace{2mm}\Omega_{\delta},
\\
&\Delta\phi_{\delta}^l = 0\hspace{5mm}\text{on}\hspace{2mm}\Omega_{\delta},
\\
&(\nabla_n\phi_{\delta}^l)_*=2\delta  P_{<l}(\Delta_{\Gamma_*} X_{\delta,*}^l\cdot n_{\delta,*}) - \frac{2\delta}{|\Gamma_{\delta}|}\int_{\Gamma_{\delta}}  (P_{<l}(\Delta_{\Gamma_*}X^l_{\delta,*}\cdot n_{\delta,*}))^* dS  \hspace{5mm}\text{on}\hspace{2mm}\Gamma_{*},
\end{cases}
\end{equation}
where $P_{<l}$ stands for a smooth  spectral projector  localized at frequencies $\lesssim 2^l$ associated to $\Delta_{\Gamma_*}$, with symbol
\begin{equation}\label{Pl}
p_{<l}(\lambda) = (1+ 2^{-2l} \lambda^2)^{-1}.
\end{equation}
With the spectral truncation added, the map $X_\delta^l \to \nabla \phi_\delta^l$ is bounded in $H^k(\Omega_\delta)$. Therefore,  the evolution 
\eqref{parabolicregularizationconstructionj}
can be seen as a transport equation with a perturbative source term, which is easily solvable in $H^k$. 
\medskip

Now, we return to the original system \eqref{parabolicregularizationconstruction}, for which we would like to obtain solutions 
$X_\delta$ as limits of $X_\delta^l$  as $l \to \infty$ on a subsequence. The difficulty is that, a priori, the implicit  constants in the 
$H^k$ solvability for \eqref{parabolicregularizationconstructionj}
depend on $l$. To avoid this, we need better bounds for the regularized system, which are 
independent of $l$; this is where the improved boundary regularity in \eqref{better-Gamma} 
plays a key role. We will rely on the div-curl bound \eqref{div-curl-ell} and
consider the counterpart of the system \eqref{transp-par}, namely 
\begin{equation}\label{transp-par-l}
\begin{cases}
&D_{\delta}\omega_{\delta, ij}^l= \partial_iV_{\delta,j}\partial_jX^l_{\delta,i}-
\partial_jV_{\delta,i}\partial_iX^l_{\delta,j}
,
\\
& \frac{d}{d\delta} f_{\delta,*}^l - 2\delta \Delta_{\Gamma_*} P_{<l} f_{\delta,*}^l
= - 4\delta  P_{<l}(\nabla^\top X_{\delta,*}^l\cdot \nabla^\top n_{\delta,*})+ 2\delta P_{\geq l}(  X_{\delta,*}^l \cdot \Delta_{\Gamma_*} n_{\delta,*})
\\ & \qquad \qquad \qquad \qquad \qquad
- \frac{2\delta}{|\Gamma_{\delta}|}\int_{\Gamma_{\delta}}  (P_{<l}(\Delta_{\Gamma_*}X^l_{\delta,*}\cdot n_{\delta,*}))^* dS
+ (\nabla_n\Delta^{-1}(\partial_iV_{\delta,j}\partial_jX^l_{\delta,i}))_*.
\end{cases}
\end{equation}
Here we can use \eqref{Pl} and \eqref{better-Gamma} in order to estimate the source terms in the second equation perturbatively and 
still obtain a short time transport/parabolic bound of the form 
\begin{equation*}
\begin{split}
\|\omega^l_\delta \|_{L^\infty H^{k-1}(\Omega_\delta)}
+ \| f^l_{\delta,*} \|_{L^\infty H^{k-\frac12}(\Gamma_*)}
&+ \| \delta^\frac12 P_{<l}^\frac12 f^l_{\delta,*} \|_{L^2 H^{k+\frac12}(\Gamma_*)} 
\\
&\lesssim C(\|\Gamma_0\|_{H^{k+\frac{3}{2}}})(\| X_0\|_{H^{k}(\Omega_0)} + \delta^\frac12 \| X_\delta^l \|_{L^\infty H^k(\Omega_\delta)}),
\end{split}
\end{equation*}
which is coupled with a direct transport bound for $X^l_\delta$ in a weaker topology,
\[
\| X^l_\delta\|_{L^\infty H^{k-2}(\Omega_\delta)} \lesssim 
\| X_0\|_{H^{k-2}(\Omega_0)} + \delta \| X^l_\delta\|_{L^\infty H^{k-2}(\Omega_\delta)},
\]
both with implicit constants independent of $l$.
Combining these two estimates with \eqref{div-curl-ell} for small enough $\delta$ (depending on $\|\Gamma_0\|_{H^{k+\frac{3}{2}}}$) we arrive at 
\begin{equation}\label{X_delta-l}
\| X_\delta^{l}\|_{L^\infty H^k(\Omega_\delta)} + \| X^l_\delta \cdot n_\delta \|_{L^\infty H^{k-\frac12}(\Gamma_\delta)}+
\| \delta^\frac12 P_{<l}^\frac12 (X^l_\delta \cdot n_\delta)_* \|_{L^2 H^{k+\frac12}(\Gamma_*)}
 \lesssim 
C(\|\Gamma_0\|_{H^{k+\frac{3}{2}}})\| X_0\|_{H^{k}(\Omega_0)},
\end{equation}
with implicit constants independent of $l$. Iterating this procedure, we can extend this bound to hold for $\delta\in [0,\epsilon]$ (with possibly enlarged implicit constants, but still not depending on $l$).  This uniform bound allows us to use Arzela-Ascoli to pass to the limit on a subsequence as $l \to \infty$ in order to obtain a solution $X_\delta$ for the original system \eqref{parabolicregularizationconstruction}
which satisfies 
\begin{equation}\label{X_delta-a}
\| X_{\delta}\|_{L^\infty H^k(\Omega_\delta)}  + \|  X_{\delta} \cdot n_\delta \|_{L^\infty H^{k-\frac12}(\Gamma_\delta)}+
\| \delta^\frac12  ( X_{\delta} \cdot n_\delta) \|_{L^2 H^{k+\frac12}(\Gamma_\delta)}
 \lesssim 
C(\|\Gamma_0\|_{H^{k+\frac{3}{2}}})\| X_0\|_{H^{k}(\Omega_0)}.
\end{equation}
The proof of this estimate, by passing to a weak limit in  \eqref{X_delta-l}, uses the $H^{k+\frac32}$ regularity of $\Gamma_0$ and the bound $\|n_{\delta}\|_{H^{k+\frac{3}{2}}(\Gamma_\delta)}\lesssim_M \delta^{-1}\|\Gamma_0\|_{H^{k+\frac{3}{2}}}$ in order to bound the second term on the right in the second equation in \eqref{transp-par-l}. However, at this point we can reprove this bound directly
via the system \eqref{transp-par}, corresponding to $l = \infty$, where this term no longer appears. This shows that \eqref{X_delta-a} 
holds for all $H^k$ solutions to \eqref{parabolicregularizationconstruction},
with an implicit constant which only depends on 
the $H^{k+\frac12}$ regularity of $\Gamma_0$. That is, we can show that
\begin{equation*}
\|  X_{\delta}\|_{L^\infty H^k(\Omega_\delta)}  + \|  X_{\delta} \cdot n_\delta \|_{L^\infty H^{k-\frac12}(\Gamma_\delta)}+
\| \delta^\frac12  ( X_{\delta} \cdot n_\delta) \|_{L^2 H^{k+\frac12}(\Gamma_\delta)}
 \lesssim 
C(\|\Gamma_0\|_{H^{k+\frac{1}{2}}})\| X_0\|_{H^{k}(\Omega_0)}.
\end{equation*}
In particular, this implies the uniqueness of this solution. 
\medskip

We remark that \eqref{X_delta-a} directly implies 
\eqref{X-delta}. Thus, we have proved that, if $\Gamma_0$ is more regular as in \eqref{better-Gamma}, then our problem has a unique solution 
$X_\delta$ in $H^k(\Omega_\delta)$ with $X_\delta \cdot n_\delta$ in $H^{k-\frac12}(\Gamma_\delta)$. It now remains to extend our result to the case when $\Gamma_\delta$ has only $H^{k+\frac12}$ regularity.
 The analysis above already shows that we have unique solvability for $X_0$ in, say, $H^{k-1}$. We only need to improve its regularity 
and show that \eqref{X-delta} holds.
For this, we approximate $\Gamma_0$ with a sequence of regular data $X_{0j}$. For these data we have solutions $X_{\delta j}$ in $H^k$.  We then use the fact that the bound \eqref{X-delta}  holds uniformly for $X_{0j}$; 
 the solution $X_{\delta}$ can thus be obtained as a weak limit of $X_{\delta j}$ on a subsequence, still with the bound \eqref{X-delta} satisfied.
\end{proof}
\begin{remark}
The above construction  achieves two important goals. First, it is designed so that we  have (owing to the definition of $\eta_{\delta}$ above  and of $(v_\delta,\tilde{B}_\delta)$ below) the approximate relation $$\frac{d}{d\delta}(\tilde{B}_{\delta,*}\cdot n_{\delta,*})\approx 2\delta \Delta_{\Gamma_*}(\tilde{B}_{\delta,*}\cdot n_{\delta,*}).$$ This will allow us to propagate (with small error) the tangency condition for the magnetic field $B_{\delta}$ from $\delta=0$ to $\delta=\epsilon$. Secondly, the above regularization will also effectively regularize the irrotational components of $v$ and $B$, which can be measured by the variables $\mathcal{N}W^\pm_{\delta}\cdot n_{\Gamma_\delta}$.
\end{remark}
\textbf{Defining $B_{\epsilon}$ and $v_{\epsilon}$}. For $\delta\in (0,\epsilon]$, we define the one-parameter family $v_{\delta}$ by \Cref{Bconstruction}. Defining $B_{\epsilon}$ is a bit more delicate as we need to enforce the tangency condition $B_{\epsilon}\cdot n_{\Gamma_\epsilon}=0$. In line with the above constructions for $\Gamma_\epsilon$ and $v_\epsilon$, we will define a one-parameter family $B_{\delta}$ of magnetic fields associated to each domain $\Omega_{\delta}$, where $\delta\in [0,\epsilon]$.
\medskip

We begin by defining $\tilde{B}_{\delta}$ by using \Cref{Bconstruction}. We then let 
\begin{equation*}
B_{\delta}:=\tilde{B}_{\delta}^{rot}:=\tilde{B}_{\delta}-\nabla\mathcal{H}\mathcal{N}^{-1}(\tilde{B}_{\delta}\cdot n_{\Gamma_\delta}).
\end{equation*}
It is easy to see that -- with the above definitions for $v_\delta$ and $B_\delta$ -- the second condition in \eqref{PP approx} follows from the fundamental theorem of calculus, \Cref{Bconstruction} and the fact that $B_0\cdot n_{\Gamma_0}=0$. 
\medskip

Given the family $(v_\delta, B_\delta, \Gamma_\delta)$ as above, we define the associated quantities $P_\delta$, $D_t^\pm P_\delta$, $a_\delta$ and $D_t^\pm a_\delta$ on $\Omega_{\delta}$ and $\Gamma_{\delta}$ by using the relevant Poisson equations, as in \Cref{CTEF}. We will use the notation $\mathcal{N}_{\delta}$ to refer to the Dirichlet-to-Neumann operator for $\Gamma_{\delta}$. For notational convenience,  we will drop the $\delta$ subscript  whenever unambiguous. That is, we will write $\mathcal{N}$ as shorthand for $\mathcal{N}_{\delta}$, $\Gamma$ as shorthand for $\Gamma_{\delta}$, and so-forth.  
By slight abuse, we also introduce the following notation:
\begin{equation*}
\|f\|_{\mathbf{H}^s(\Omega_{\delta})}:=\|(f,\nabla_{B_\delta}f)\|_{H^{s}(\Omega_{\delta})\times H^{s-\frac{1}{2}}(\Omega_{\delta})},\hspace{5mm}\|f\|_{\mathbf{H}^s(\Gamma_\delta)}:=\|(f,\nabla_{B_\delta}f)\|_{H^{s}(\Gamma_{\delta})\times H^{s-\frac{1}{2}}(\Gamma_{\delta})},\hspace{5mm}s\geq\frac{1}{2},
\end{equation*}
and also for each $p\geq 1$,
\begin{equation*}
\|f\|_{L^p_{\delta}\mathbf{H}^s(\Omega_\delta)}:=\left(\int_{0}^{\delta}\|f(\tau)\|_{\mathbf{H}^s(\Omega_\tau)}^pd\tau\right)^{\frac{1}{p}},\hspace{5mm}\|f\|_{L^p_{\delta}\mathbf{H}^s(\Gamma_\delta)}:=\left(\int_{0}^{\delta}\|f(\tau)\|_{\mathbf{H}^s(\Gamma_\tau)}^pd\tau\right)^{\frac{1}{p}},
\end{equation*}
with the obvious modifications for $p=\infty$.
\medskip

\textbf{Energy monotonicity}.
Now, we aim to prove the energy monotonicity bound \eqref{Emonsurf}. To efficiently organize the required estimates, we split each of the main components of the energy $E^k_\pm$ into an irrotational component and a rotational component. That is, we define $E^k_\pm:=1+\|W^\pm\|_{L^2(\Omega)}^2+ E_{r}^k+E^k_i$ where
\begin{equation*}
\begin{split}
E_r^k&:=\|\omega^\pm\|_{H^{k-1}(\Omega)}^2+\|\nabla_B\omega^\pm\|_{H^{k-\frac{3}{2}}(\Omega)}^2,
\\
E_i^k&:=2\|a^{\frac{1}{2}}\mathcal{N}^{k-1}a\|_{L^2(\Gamma)}^2+\|\nabla\mathcal{H}\mathcal{N}^{k-2}\nabla_Ba\|_{L^2(\Omega)}^2+\|\nabla\mathcal{H}\mathcal{N}^{k-2}\mathcal{G}^\pm\|_{L^2(\Omega)}^2+\|a^{-\frac{1}{2}}\mathcal{N}^{k-2}\nabla_B\mathcal{G}^\pm\|_{L^2(\Gamma)}^2.
\end{split}
\end{equation*}
In order to measure the parabolic smoothing effect that the above regularization has on both the surface and on $v$ and $B$, we define the quantity
\begin{equation*}
\mathcal{J}(\delta):=\left(\int_{0}^{\delta}\tau\|\kappa_{\tau}\|_{\mathbf{H}^{k-1}(\Gamma_{\tau})}^2+\tau\|\mathcal{N}_{\tau}W^\pm_{\tau}\cdot n_\tau\|_{\mathbf{H}^{k-\frac{1}{2}}(\Gamma_{\tau})}^2d\tau\right)^{\frac{1}{2}}.
\end{equation*}
The required energy monotonicity bound \eqref{Emonsurf} is an immediate consequence of the following proposition, which is where the bulk of our efforts will be situated.
\begin{proposition}\label{E1estimate}
There exist positive constants $c_1$ and $c_2$ with $c_1\gg c_2$ such that for every $0<\delta\leq \epsilon$ there holds:
\begin{enumerate}
\item (Rotational energy bound).
\begin{equation}\label{rotregbounds}
E^k_r(\delta)\leq E^k_r(0)+c_2\mathcal{J}^2(\delta)+C(M)\delta.
\end{equation}
\item (Propagation of vorticity bootstrap). We have
\begin{equation}\label{bootstrapprop}
\|\omega^\pm_{\delta}\|_{H^{k}(\Omega_\delta)}\leq (1+C(M)\delta)\delta^{-\frac{3}{2}}K(M),
\end{equation}
where $K(M)$ is the same constant as in \eqref{inductiveregbound}. 
\item (Irrotational energy bound).
\begin{equation}\label{parabolicemonbounds}
E^k_i(\delta)+c_1\mathcal{J}^2(\delta)\leq E^k_i(0)+C(M)\delta.
\end{equation}
\end{enumerate}
\end{proposition}
\begin{remark}\label{equivalenceofboundarycond}
 By pigeonholing (and possibly replacing $\epsilon$ with some $\epsilon'\approx \epsilon$), it is easy to see that \Cref{E1estimate}  implies property (iv) in \Cref{domainregularizationprop} for the boundary term $\mathcal{N}_{\epsilon}W^\pm_{\epsilon}\cdot n_{\epsilon}$. To estimate the other boundary term $\nabla^{\top} W^\pm_{\epsilon}\cdot n_{\epsilon}$, we may simply use the pointwise (in $\delta$) bound 
 \begin{equation}\label{pointwisedeltabound}
 \begin{split}
\|\nabla^{\top}W_\delta^\pm\cdot n_\delta\|_{H^{k-\frac{1}{2}}(\Gamma_\delta)}&\lesssim_M \delta^{-\frac{1}{2}}+\|\mathcal{N}\nabla^{\top}W_\delta^\pm\cdot n_\delta\|_{H^{k-\frac{3}{2}}(\Gamma_\delta)}
\\
&\lesssim_M \delta^{-\frac{1}{2}}+\|[\nabla^{\top},\mathcal{N}]W_\delta^\pm\|_{H^{k-\frac{3}{2}}(\Gamma_\delta)}+\|\mathcal{N}W_\delta^\pm\cdot n_\delta\|_{H^{k-\frac{1}{2}}(\Gamma_\delta)}
\\
&\lesssim_M \delta^{-\frac{1}{2}}+\|\mathcal{N}W_\delta^\pm\cdot n_\delta\|_{H^{k-\frac{1}{2}}(\Gamma_{\delta})}
\end{split}
 \end{equation}
where throughout, we used the bound $\|\Gamma_{\delta}\|_{H^{k+\frac{1}{2}}}\lesssim_M \delta^{-\frac{1}{2}}$. Note that, in the first line, we used the ellipticity of $\mathcal{N}$ and the Leibniz rule for $\mathcal{N}$. In the second line, we used \Cref{tangradientbound}, and in the third line we used \Cref{materialcom} and \Cref{commutatorremark}. We will use \eqref{pointwisedeltabound} (or simple variations of it when using different Sobolev indices) in the analysis below.
 \end{remark}
 \begin{proof}
It will be convenient to make the following bootstrap hypotheses
 \begin{equation}\label{bootstrap h-J}
 \|W^\pm_{\delta}\|_{\mathbf{H}^{k}(\Omega_\delta)}+\mathcal{J}(\delta)\leq C'(M),\hspace{5mm} \|\omega_{\delta}^\pm\|_{H^{k}(\Omega_{\delta})}\leq 2K(M)\delta^{-\frac{3}{2}}
 \end{equation}
 for some sufficiently large constant $C'(M)$ and where $K(M)$ is as in \eqref{inductiveregbound}. Both of these constants will be automatically improved by the bounds  \eqref{rotregbounds}-\eqref{parabolicemonbounds} and the energy coercivity bound in \Cref{Energy est. thm}.
 \medskip
 
  We begin by establishing the following technical lemma which relates the velocity function $V_{\delta}$ to the mean curvature.
\begin{lemma}\label{gradBV}
The following bounds hold for $0<\delta\leq \epsilon$.
\begin{enumerate}
\item (Curvature bound I).
\begin{equation}\label{curvbound1}
\|V_{\delta}\|_{L^2_{\delta}H^{k}(\Omega_{\delta})}+\|V_{\delta}\|_{L^2_{\delta}H^{k-\frac{1}{2}}(\Gamma_{\delta})}\lesssim_M \delta+\mathcal{J}(\delta).
\end{equation}
\item (Curvature bound II).
\begin{equation}\label{curvbound2}
\|\nabla_{B_{\delta}}V_{\delta}\|_{L^2_{\delta}H^{k-\frac{1}{2}}(\Omega_{\delta})}+\|\nabla_{B_{\delta}}V_{\delta}\|_{L^2_{\delta}H^{k-1}(\Gamma_{\delta})}\lesssim_M \delta+\mathcal{J}(\delta).
\end{equation}
\end{enumerate}
\end{lemma}
\begin{proof}
We begin with the much easier bound \eqref{curvbound1}. By definition of $V_{\delta}$ and the bound $\|\Gamma_{\delta}\|_{H^k}\lesssim_M 1$, we have $\|V_{\delta}\|_{H^k(\Omega_{\delta})}\lesssim_M \|V_{\delta}\|_{H^{k-\frac{1}{2}}(\Gamma_{\delta})}$. Therefore, it suffices to estimate the latter term on the left-hand side of \eqref{curvbound1}. By definition of $V_{\delta}$ and  \Cref{curvaturebound} as well as the parabolic estimate $\|\Gamma_{\delta}\|_{H^{k+\frac{3}{2}}}\lesssim_M \delta^{-\frac{1}{2}}\|\Gamma_{\frac{\delta}{2}}\|_{H^{k+1}}$, we have
\begin{equation*}
\|V_{\delta}\|_{H^{k-\frac{1}{2}}(\Gamma_{\delta})}\lesssim_M \delta^{\frac{1}{2}}\left(1+\|\kappa_{\frac{\delta}{2}}\|_{H^{k-1}(\Gamma_{\frac{\delta}{2}})}\right).
\end{equation*}
This easily gives \eqref{curvbound1} by integrating in $\delta$ and rescaling. Next, we prove the second estimate \eqref{curvbound2}. The main ingredient will be the parabolic regularity bound
\begin{equation}\label{curvbound3}
\|\nabla_{B_{\delta}}\kappa_{\delta}\|_{H^{k-1}(\Gamma_{\delta})}\lesssim_M \delta^{-\frac{1}{2}}(1+\|\kappa_{\frac{\delta}{2}}\|_{\mathbf{H}^{k-1}(\Gamma_{\frac{\delta}{2}})}),\hspace{5mm}0<\delta\leq\epsilon.
\end{equation}
To prove this, we will need the following energy bound for $\alpha=0,1$ and every $\frac{\delta}{2}\leq\delta'\leq\delta$:
\begin{equation}\label{preliminaryheatbound}
\begin{split}
I_{\alpha,\delta'}\lesssim_M \delta^{-\alpha}+\delta^{1-2\alpha}\|\kappa_{\frac{\delta}{2}}\|_{H^{k-1}(\Gamma_{\frac{\delta}{2}})}^2+\|\nabla_{B_{\delta'}}\kappa_{\delta'}\|_{H^{k-\frac{3}{2}+\alpha}(\Gamma_{\delta'})}^2,
\end{split}
\end{equation}
where 
\begin{equation*}
I_{\alpha,\delta'}:=\|\nabla_{B_{\delta}}\kappa_{\delta}\|_{H^{k-\frac{3}{2}+\alpha}(\Gamma_\delta)}^2+\int_{\delta'}^{\delta}\tau\|\nabla_{B_\tau}\kappa_\tau\|_{H^{k-\frac{1}{2}+\alpha}(\Gamma_\tau)}^2d\tau.
\end{equation*}
Our starting point is to recall that from \Cref{epsderivativeofnormal}, we have the  heat equation for $\kappa_{\delta}$,
\begin{equation}\label{heatforkappa}
\frac{d}{d\delta}\kappa_{\delta,*}-2\delta\Delta_{\Gamma_*}\kappa_{\delta,*}=R_{\delta,*},
\end{equation}
where we can bound using \Cref{epsderivativeofnormal}, the parabolic bound for $\Gamma_{\delta}$ and \Cref{curvaturebound},
\begin{equation}\label{Rtaubound}
\|R_{\tau,*}\|_{H^{k-\frac{3}{2}+\alpha}(\Gamma_*)}\lesssim_M \tau+ \tau^{\frac{1}{2}-\alpha}\|\kappa_{\frac{\tau}{2}}\|_{H^{k-1}(\Gamma_{\frac{\tau}{2}})},\hspace{5mm}0\leq\tau\leq\delta.
\end{equation}
Our next aim will be to commute \eqref{heatforkappa} with $\nabla_{B_{\delta}}$ and to obtain energy estimates for the corresponding heat equation for $\nabla_{B_{\delta}}\kappa_{\delta}$. To estimate the requisite error terms, we will need the bounds
\begin{equation}\label{crudeDBbounds}
\int_{\delta'}^{\delta}\tau^{-1}\left(\|D_{\tau}B_{\tau}\|_{H^{k-2+\alpha}(\Omega_\tau)}^2+\tau^2\|B_{\tau}\|_{H^{k+\alpha}(\Omega_\tau)}^2\right)d\tau\lesssim_M \delta^{-\alpha}.
\end{equation}
We will focus on the case $\alpha=1$ as the case $\alpha=0$ is simpler and follows similar reasoning. We begin by estimating $B_{\tau}$. For this, we use the div-curl estimate in \Cref{Balanced div-curl} to obtain
\begin{equation*}
\|B_{\tau}\|_{H^{k+1}(\Omega_{\tau})}\lesssim_M 1+\|\Gamma_{\tau}\|_{H^{k+\frac{1}{2}}}+\|\nabla\times B_{\tau}\|_{H^{k}(\Omega_{\tau})}+\|\nabla^{\top}B_{\tau}\cdot n_\tau\|_{H^{k-\frac{1}{2}}(\Gamma_{\tau})}.
\end{equation*}
By   \eqref{pointwisedeltabound}, \eqref{bootstrap h-J} and the bound $\|\Gamma_{\tau}\|_{H^{k+\frac{1}{2}}}\lesssim_M \tau^{-\frac{1}{2}}$, we obtain the required bound for the component involving the integral of $\tau\|B_{\tau}\|_{H^{k+1}(\Omega_\tau)}^2$. On the other hand, by \Cref{Balanced div-curl} and the obvious bound $\|D_{\tau}B_{\tau}\|_{L^2(\Omega_{\tau})}\lesssim_M \tau$, we have
\begin{equation*}
\|D_{\tau}B_{\tau}\|_{H^{k-1}(\Omega_{\tau})}\lesssim_M \tau +\|\nabla\times D_{\tau}B_{\tau}\|_{H^{k-2}(\Omega_{\tau})}+\|\nabla\cdot D_{\tau}B_{\tau}\|_{H^{k-2}(\Omega_{\tau})}+\|D_{\tau}B_{\tau}\cdot n_\tau\|_{H^{k-\frac{3}{2}}(\Gamma_{\tau})}.
\end{equation*}
Using the definition of $B_{\tau}$ and the parabolic bounds for $\Gamma_{\tau}$, it is easy to estimate the second and third term on the right by $\tau^{\frac{1}{2}}$. For the last term, we have using the boundary condition $B_{\tau}\cdot n_\tau=0$ and \Cref{Movingsurfid}, the identity $D_{\tau}B_{\tau}\cdot n_\tau=B_{\tau}\cdot\nabla^{\top}V_{\tau}\cdot n_\tau$ which can be estimated in $H^{k-\frac{3}{2}}(\Gamma_{\tau})$ (using the parabolic estimate for $\Gamma_{\tau}$ and \Cref{curvaturebound}) by $\tau^{\frac{1}{2}}\|\kappa_{\frac{\tau}{2}}\|_{H^{k-1}(\Gamma_{\frac{\tau}{2}})}$. By rescaling and the bootstrap hypothesis \eqref{bootstrap h-J}, this gives \eqref{crudeDBbounds}. We can now commute the heat equation \eqref{heatforkappa} for the curvature with $\nabla_{B_{\delta}}$ to obtain
\begin{equation*}
\frac{d}{d\delta}(\nabla_{B_{\delta}}\kappa_{\delta})_*-2\delta\Delta_{\Gamma_*}(\nabla_{B_\delta}\kappa_{\delta})_*=R'_{\delta,*}
\end{equation*}
where by \eqref{Rtaubound} and \eqref{crudeDBbounds}, there holds
\begin{equation*}
\int_{\delta'}^{\delta}\tau^{-1}\|R'_{\tau,*}\|_{H^{k-\frac{5}{2}+\alpha}(\Gamma_*)}^2d\tau\lesssim_M \delta^{-\alpha}+\delta^{1-2\alpha}\|\kappa_{\frac{\delta}{2}}\|_{H^{k-1}(\Gamma_{\frac{\delta}{2}})}^2.
\end{equation*}
From this and a simple energy estimate, we deduce \eqref{preliminaryheatbound}. To obtain \eqref{curvbound3}, we first observe that by averaging in the $\alpha=0$ estimate and directly using the $\alpha=1$ estimate, we have the parabolic regularity bound
\begin{equation*}
\|\nabla_{B_{\delta}}\kappa_{\delta}\|_{H^{k-\frac{1}{2}}(\Gamma_\delta)}\lesssim_M \delta^{-1}(1+\|\kappa_{\frac{\delta}{2}}\|_{\mathbf{H}^{k-1}(\Gamma_{\frac{\delta}{2}})}).
\end{equation*}
Therefore, interpolating and combining the above estimate with the $\alpha=0$ bound we have
\begin{equation*}
\begin{split}
\|\nabla_{B_{\delta}}\kappa_{\delta}\|_{H^{k-1}(\Gamma_\delta)}^2&\lesssim_M \|\nabla_{B_{\delta}}\kappa_{\delta}\|_{H^{k-\frac{1}{2}}(\Gamma_\delta)}\|\nabla_{B_{\delta}}\kappa_{\delta}\|_{H^{k-\frac{3}{2}}(\Gamma_\delta)}
\\
&\lesssim_M \delta^{-1}(1+\|\kappa_{\frac{\delta}{2}}\|_{\mathbf{H}^{k-1}(\Gamma_{\frac{\delta}{2}})})(1+\delta^{\frac{1}{2}}\|\kappa_{\frac{\delta}{2}}\|_{H^{k-1}(\Gamma_{\frac{\delta}{2}})}+\|\nabla_{B_{\frac{\delta}{2}}}\kappa_{\frac{\delta}{2}}\|_{H^{k-\frac{3}{2}}(\Gamma_{\frac{\delta}{2}})})
\\
&\lesssim_M \delta^{-1}(1+\|\kappa_{\frac{\delta}{2}}\|_{\mathbf{H}^{k-1}(\Gamma_{\frac{\delta}{2}})})^2
\end{split}
\end{equation*}
as desired. Now, we turn to \eqref{curvbound2}. Our first step is to reduce matters to controlling the latter term on the left-hand side of \eqref{curvbound2}. For this, we compute using the chain rule,
\begin{equation}\label{chainruleboundV}
\nabla_{B_{\delta}}V_{\delta}=\nabla_{B_{\delta}}\Psi_{\delta}\cdot (\nabla\mathcal{H}_*(\nu\partial_{\delta}\eta_{\delta}))^*,\hspace{5mm}x\in \overline{\Omega}_{\delta}. 
\end{equation}
Therefore, by elliptic regularity and a change of variables, there holds
\begin{equation*}
\begin{split}
\|\nabla_{B_{\delta}}V_{\delta}\|_{L_{\delta}^2H^{k-\frac{1}{2}}(\Omega_{\delta})}&\lesssim \|(\nabla_{B_{\delta}}\Psi_{\delta})_*\cdot\nabla\mathcal{H}_*(\nu\partial_{\delta}\eta_{\delta})\|_{L^2_{\delta}H^{k-\frac{1}{2}}(\Omega_*)}
\\
&\lesssim \|\Delta_{\Gamma_*}((\nabla_{B_{\delta}}\Psi_{\delta})_*\cdot\nabla\mathcal{H}_*(\nu\partial_{\delta}\eta_{\delta}))\|_{L^2_{\delta}H^{k-\frac{5}{2}}(\Omega_*)}+\|(\nabla_{B_{\delta}}\Psi_{\delta})_*\cdot\nabla\mathcal{H}_*(\nu\partial_{\delta}\eta_{\delta})\|_{L^2_{\delta}H^{k-1}(\Gamma_*)},
\end{split}
\end{equation*}
where the implicit constants depend on $M$. It is straightforward from the definition of $\partial_{\delta}\eta_{\delta}$ and \Cref{curvaturebound} to estimate the first term in the second line above by the right-hand side of \eqref{curvbound2}. For the second term, we can use the fact that $\Psi_{\delta}$ is a diffeomorphism from $\Gamma_*$ to $\Gamma_{\delta}$ and \eqref{chainruleboundV} to obtain
\begin{equation*}
\|(\nabla_{B_{\delta}}\Psi_{\delta})_*\cdot\nabla\mathcal{H}_*(\nu\partial_{\delta}\eta_{\delta})\|_{L^2_{\delta}H^{k-1}(\Gamma_*)}\lesssim_M \|\nabla_{B_{\delta}}V_{\delta}\|_{L_{\delta}^2H^{k-1}(\Gamma_{\delta})}.
\end{equation*}
Thus, it remains to estimate $\|\nabla_{B_{\delta}}V_{\delta}\|_{L_{\delta}^2H^{k-1}(\Gamma_{\delta})}$. Thanks to \eqref{curvbound3}, it suffices to prove that
\begin{equation}\label{curv2auxil}
\|\nabla_{B_{\delta}}V_{\delta}\|_{H^{k-1}(\Gamma_{\delta})}\lesssim_M \delta(1+\|\kappa_{\delta}\|_{H^{k-1}(\Gamma_{\delta})}+\|\nabla_{B_{\delta}}\kappa_{\delta}\|_{H^{k-1}(\Gamma_{\delta})}).
\end{equation}
To elucidate the key points in the computation, we first consider the case where $\Gamma_*$ is the hyperplane $\{x_d=0\}$ and the collar parameterization of $\eta_{\delta}$ is a (compactly supported) graph satisfying the heat equation $\partial_{\delta}\eta_{\delta}=2\delta\Delta_{\Gamma_*}\eta_{\delta}$. In this setting, the mean curvature $\kappa_{\delta}$ is related to $\eta_{\delta}$ by the  elliptic equation,
\begin{equation*}
\kappa_{\delta,*}=-\partial_j\left(\frac{\partial_j\eta_{\delta}}{\sqrt{1+|\nabla \eta_{\delta}|^2}}\right)=-\frac{\Delta\eta_{\delta}}{(1+|\nabla\eta_{\delta}|^2)^{\frac{1}{2}}}+\frac{\partial_i\eta_{\delta}\partial_j\eta_{\delta}\partial_i\partial_j\eta_{\delta}}{(1+|\nabla\eta_{\delta}|^2)^{\frac{3}{2}}}=:L_{\delta}\eta_{\delta}.
\end{equation*}
By commuting the elliptic operator $L_{\delta}$ above with $\partial_{\delta}$ and using the heat equation for $\eta_{\delta}$ as well as the identity $\partial_{\delta}\eta_{\delta}=V_{\delta,*}\cdot\nu$, we obtain
\begin{equation*}
\begin{split}
L_{\delta}(V_{\delta,*}\cdot\nu)=L_{\delta}(\partial_{\delta}\eta_{\delta})=\frac{d}{d\delta}\kappa_{\delta,*}+R_{\delta}
\end{split}
\end{equation*}
where 
\begin{equation*}
\|R_{\delta}\|_{H^{k-2}(\Gamma_*)}\lesssim_M \delta\|\Gamma_{\delta}\|_{H^{k+1}}\lesssim_M\delta(1+\|\kappa_{\delta}\|_{H^{k-1}(\Gamma_{\delta})}).
\end{equation*}
Commuting with $\nabla_{B_\delta}$ and carrying out straightforward estimates, we find that
\begin{equation*}
\frac{d}{d\delta}(\nabla_{B_{\delta}}\kappa_{\delta})_{*}=L_{\delta}((\nabla_{B_{\delta}}V_{\delta})_*\cdot\nu)+R'_{\delta},
\end{equation*}
where
\begin{equation*}
\|R_{\delta}'\|_{H^{k-3}(\Gamma_*)}\lesssim_M \delta(1+\|\kappa_{\delta}\|_{H^{k-1}(\Gamma_{\delta})}).
\end{equation*}
In a tight enough collar neighborhood, one may use standard elliptic estimates, the heat equation for $\nabla_{B_\delta}\kappa_{\delta}$ and the definition of $V_{\delta}$ to see that
\begin{equation*}
\|(\nabla_{B_{\delta}} V_{\delta})_*\cdot\nu\|_{H^{k-1}(\Gamma_{*})}\lesssim_M \delta (1+\|\kappa_{\delta}\|_{H^{k-1}(\Gamma_{\delta})}+\|\nabla_{B_{\delta}}\kappa_{\delta}\|_{H^{k-1}(\Gamma_{\delta})}),
\end{equation*}
which by definition of $V_{\delta}$ and $\nu$ suffices for \eqref{curv2auxil}. The same bound holds in a  (tight enough) collar neighborhood over a general reference hypersurface $\Gamma_*$ by performing a similar computation to the above, using the local coordinates on $\Gamma_*$.
\end{proof}
Now, we return to the proof of \Cref{E1estimate}. 
\medskip

We begin with the proof of \eqref{rotregbounds}, which is the easiest part. Let us define $\tilde{\omega}_0^\pm:=\omega_0^\pm\circ (\Phi_{0}\circ\Phi_{\delta}^{-1})=\omega_{0,*}^\pm\circ\Phi_{\delta}^{-1}$.
We  note the identity $D_{\delta}\omega^\pm_{\delta}=-[\nabla\times, V_{\delta}\cdot\nabla]W^\pm_{\delta}$. Therefore, from the bound $\|W^\pm_{\delta}\|_{H^{k}(\Omega_\delta)}\lesssim_M 1$, Cauchy-Schwarz and \Cref{gradBV}, we have
\begin{equation*}\label{differentialvort}
\|\omega^\pm_{\delta}-\tilde{\omega}_0^\pm\|_{H^{k-1}(\Omega_{\delta})}\lesssim_M \int_{0}^{\delta}\|D_{\tau}\omega^\pm_{\tau}\|_{H^{k-1}(\Omega_\tau)}d\tau\lesssim_M \int_{0}^{\delta}\|V_{\tau}\|_{H^{k-\frac{1}{2}}(\Gamma_\tau)}d\tau\lesssim_M \delta^{\frac{1}{2}}(\delta+ \mathcal{J}(\delta)).
\end{equation*}
By Cauchy-Schwarz and a change of variables, we obtain 
\begin{equation*}\label{firstE1bound}
\|\omega^\pm_{\delta}\|_{H^{k-1}(\Omega_\delta)}^2\leq \|\omega^\pm_0\|_{H^{k-1}(\Omega_0)}^2+c_2\mathcal{J}^2(\delta)+C(M)\delta.
\end{equation*}
Before we turn to estimating the component of $E^k_r$ which involves $\nabla_{B_{\delta}}\omega^\pm_{\delta}$, we will  establish the bound \eqref{bootstrapprop}. Similarly to the above, we have by Sobolev product estimates and the weak bound $\|V_\tau\|_{H^{k-2}(\Omega_\tau)}\lesssim_M\tau$ (which follows from the definition of $V_\tau$),
\begin{equation*}
\|\omega_{\delta}^\pm-\tilde{\omega}_0^\pm\|_{H^{k}(\Omega_\delta)}\lesssim_M \int_{0}^{\delta}\|V_{\tau}\|_{H^{k+1}(\Omega_{\tau})}+\tau\|W_{\tau}^\pm\|_{H^{k+1}(\Omega_\tau)}d\tau.
\end{equation*}
Using the parabolic gain for the irrotational part of $W^\pm_{\tau}$ from \eqref{bootstrap h-J} we obtain by a div-curl analysis (and \eqref{pointwisedeltabound}), 
\begin{equation*}
\|\omega_{\delta}^\pm-\tilde{\omega}_0^\pm\|_{H^{k}(\Omega_\delta)}\lesssim_M \int_{0}^{\delta}\|V_{\tau}\|_{H^{k+1}(\Omega_{\tau})}+\tau\|\omega_{\tau}^\pm\|_{H^{k}(\Omega_\tau)}d\tau+1.
\end{equation*}
Using the bootstrap hypothesis on $\Gamma_0$ from \eqref{inductiveregbound} and the parabolic regularization bound \eqref{surfbound-par} for $\Gamma_{\tau}$ we have
\begin{equation*}
\|V_{\tau}\|_{H^{k+1}(\Omega_{\tau})}\lesssim_M \tau\|\Gamma_{\tau}\|_{H^{k+\frac{5}{2}}}\lesssim_M\tau^{\gamma-1}\|\Gamma_0\|_{H^{k+\frac{1}{2}+\gamma}}\lesssim_M \tau^{\gamma-1}\epsilon^{-\frac{1}{2}-\gamma}.
\end{equation*}
We therefore conclude that
\begin{equation*}
\|\omega_{\delta}^\pm-\tilde{\omega}_0^\pm\|_{H^{k}(\Omega_\delta)}\lesssim_M \int_{0}^{\delta}\epsilon^{-\frac{1}{2}-\gamma}\tau^{\gamma-1}+\tau\|\omega_{\tau}^\pm\|_{H^{k}(\Omega_\tau)}d\tau\lesssim_{M,\gamma} \epsilon^{-\frac{1}{2}}+\delta\sup_{0\leq \tau\leq \delta}\|\omega_{\tau}^\pm\|_{H^{k}(\Omega_\tau)},
\end{equation*}
which gives 
\begin{equation*}
\|\omega_{\delta}^\pm\|_{H^{k}(\Omega_\delta)}\leq \|\omega_0^\pm\|_{H^{k}(\Omega_0)}+C(M)\epsilon^{-\frac{1}{2}}
\end{equation*}
and establishes \eqref{bootstrapprop}. We observe that as a consequence of  \eqref{bootstrapprop}, the bootstrap \eqref{bootstrap h-J}, \Cref{Balanced div-curl} and \eqref{pointwisedeltabound}, we  also  have the crude bound
\begin{equation}\label{fullbootstrap}
\int_{0}^{\delta}\tau\|W_{\tau}^\pm\|_{H^{k+\frac{1}{2}}(\Omega_\tau)}d\tau\lesssim_M\delta,
\end{equation}
which we will use below. Now, to estimate the remaining component of $E^k_r$, we average as above to obtain
\begin{equation*}
\|\nabla_{B_{\delta}}\omega^\pm_{\delta}\|_{H^{k-\frac{3}{2}}(\Omega_\delta)}\leq \|\nabla_{B_{0}}\omega_0^\pm\|_{H^{k-\frac{3}{2}}(\Omega_0)}+C(M)\int_{0}^{\delta}\|D_{\tau}(\nabla_{B_{\tau}}\omega^\pm_{\tau})\|_{H^{k-\frac{3}{2}}(\Omega_{\tau})}d\tau.
\end{equation*}
A simple computation gives
\begin{equation*}
D_{\tau}(\nabla_{B_{\tau}}\omega^\pm_{\tau})=-\nabla_{B_{\tau}}([\nabla\times, V_{\tau}\cdot\nabla]W_{\tau}^\pm)+[D_{\tau},\nabla_{B_{\tau}}]\omega^\pm_{\tau}.
\end{equation*}
Straightforward algebraic manipulations and Sobolev product estimates then yield
\begin{equation}\label{fiveterms}
\begin{split}
\|D_{\tau}(\nabla_{B_{\tau}}\omega^\pm_{\tau})\|_{H^{k-\frac{3}{2}}(\Omega_\tau)}\lesssim_M &\tau\|\nabla_{B_{\tau}}W^\pm_{\tau}\|_{H^{k-\frac{1}{2}}(\Omega_\tau)}+\|V_{\tau}\|_{H^{k-\frac{1}{2}}(\Omega_{\tau})}+\|\nabla_{B_{\tau}}V_{\tau}\|_{H^{k-\frac{1}{2}}(\Omega_{\tau})}
\\
&+\|D_{\tau}B_{\tau}\|_{H^{k-\frac{3}{2}}(\Omega_\tau)}+\tau\|\omega_{\tau}^\pm\|_{H^{k-\frac{1}{2}}(\Omega_\tau)}.
\end{split}
\end{equation}
The first two terms on the right-hand side of \eqref{fiveterms} are of size $\mathcal{O}_M(\tau)$ and $\mathcal{O}_M(1)$, respectively. By \Cref{gradBV} and Cauchy-Schwarz, we can estimate the $\delta$-integral of the third term by 
\begin{equation*}
\int_{0}^{\delta}\|\nabla_{B_{\tau}}V_{\tau}\|_{H^{k-\frac{1}{2}}(\Omega_{\tau})}d\tau\lesssim_M \delta^{\frac{1}{2}}(\delta+\mathcal{J}(\delta))\leq C(M)\delta+c_2\mathcal{J}^2(\delta)
\end{equation*}
where $c_2$ is an appropriately small constant. To estimate the fourth term we observe that, by definition, we have
\begin{equation*}
\|D_{\tau}B_{\tau}\|_{H^{k-\frac{3}{2}}(\Omega_\tau)}\leq \|D_{\tau}\tilde{B}_{\tau}\|_{H^{k-\frac{3}{2}}(\Omega_\tau)}+\|D_{\tau}\tilde{B}^{ir}_{\tau}\|_{H^{k-\frac{3}{2}}(\Omega_\tau)}.
\end{equation*}
Using the definition of $D_{\tau}\tilde{B}_\tau$ and elliptic regularity we may bound
\begin{equation*}
\|D_{\tau}\tilde{B}_\tau\|_{H^{k-\frac{3}{2}}(\Omega_\tau)}\lesssim_M 1+\tau\|B_{\tau}\|_{H^{k+\frac{1}{2}}(\Omega_\tau)}+\|V_{\tau}\|_{H^{k-\frac{3}{2}}(\Omega_\tau)}.
\end{equation*}
We next recall that $\|V_{\tau}\|_{H^{k-\frac{3}{2}}(\Omega_\tau)}\lesssim_M 1$. In light of \eqref{fullbootstrap}, this implies that 
\begin{equation*}
\int_{0}^{\delta}\|D_{\tau}\tilde{B}_{\tau}\|_{H^{k-\frac{3}{2}}(\Omega_\tau)}d\tau\lesssim_M\delta.
\end{equation*}
 On the other hand, using \Cref{Balanced div-curl}, $\tilde{B}_{\tau}\cdot n_{\tau}=\mathcal{O}_{H^{k-3}}(\tau^2)$ and the parabolic bounds \eqref{surfbound-par} for $\Gamma_{\tau}$, we have
\begin{equation*}
\|D_{\tau}\tilde{B}^{ir}_\tau\|_{H^{k-\frac{3}{2}}(\Omega_\tau)}\lesssim_M 1+\|D_{\tau}\tilde{B}^{ir}_\tau\cdot n_{\tau}\|_{H^{k-2}(\Gamma_{\tau})}\lesssim_M 1+\|D_{\tau}(\tilde{B}_{\tau}\cdot n_{\tau})\|_{H^{k-2}(\Gamma_{\tau})}\lesssim_M 1+\|D_\tau\tilde{B}_{\tau}\|_{H^{k-\frac{3}{2}}(\Omega_{\tau})}.
\end{equation*}
From this  we conclude that 
\begin{equation*}
\int_{0}^{\delta}\|D_{\tau}B_{\tau}\|_{H^{k-\frac{3}{2}}(\Omega_\tau)}d\tau\lesssim_M \delta.
\end{equation*}
Finally, the last term in \eqref{fiveterms} is of size $\mathcal{O}_M(1)$ thanks to \eqref{bootstrapprop}. Combining everything and applying Cauchy-Schwarz, we  obtain
\begin{equation*}
\|\nabla_{B_{\delta}}\omega^\pm_{\delta}\|_{H^{k-\frac{3}{2}}(\Omega_\delta)}\leq \|\nabla_{B_{0}}\omega_0^\pm\|_{H^{k-\frac{3}{2}}(\Omega_0)}+C(M)\delta^{\frac{1}{2}}\mathcal{J}(\delta)+C(M)\delta.
\end{equation*}
Therefore, we have
\begin{equation*}
\|\nabla_{B_{\delta}}\omega^\pm_{\delta}\|_{H^{k-\frac{3}{2}}(\Omega_\delta)}^2\leq \|\nabla_{B_{0}}\omega_0^\pm\|_{H^{k-\frac{3}{2}}(\Omega_0)}^2+c_2\mathcal{J}^2(\delta)+C(M)\delta,
\end{equation*}
which completes the proof of \eqref{rotregbounds}. \end{proof}
Next, we turn to the proof of \eqref{parabolicemonbounds}. We will need the following estimates for $D_{\delta}W^\pm_{\delta}$.
\begin{lemma}\label{Ddeltabounds} For $c\ll 1$ we have
\begin{equation}\label{DdeltaWbound}
\|D_{\delta}W^\pm_\delta\|_{L_{\delta}^1\mathbf{H}^{k-1}(\Omega_\delta)}\lesssim_M \delta+\delta^{\frac{1}{2}}\mathcal{J}(\delta),\hspace{5mm}\|\delta^{-\frac{1}{2}}D_{\delta}W^\pm_\delta\|_{L_{\delta}^2\mathbf{H}^{k-\frac{3}{2}}(\Omega_\delta)}\lesssim_M \delta^{\frac{1}{2}}+c\mathcal{J}(\delta).
\end{equation}
\end{lemma}
\begin{proof}
It suffices to prove the same bounds separately for $D_{\delta}B_\delta$  and $D_{\delta}v_\delta$. We will show the details for $D_{\delta}B_\delta$ as the latter is similar (in fact, a bit simpler). We begin by establishing the $L_{\delta}^1H^{k-1}$ and $L_{\delta}^2H^{k-\frac{3}{2}}$ bounds. From the regularization bounds \eqref{surfbound-par} for $\Gamma_{\delta}$ and the definition of $D_{\delta}B_\delta$, it is straightforward to estimate for $\alpha\in \{0,\frac{1}{2}\}$,
\begin{equation*}
\|\nabla\times D_{\delta}B_\delta\|_{H^{k-\frac{5}{2}+\alpha}(\Omega_\delta)}+\|\nabla\cdot D_{\delta}B_\delta\|_{H^{k-\frac{5}{2}+\alpha}(\Omega_\delta)}\lesssim_M \delta^{1-\alpha}.
\end{equation*}
Therefore, by the div-curl estimate in \Cref{Balanced div-curl}, we have  
\begin{equation*}
\|D_{\delta}B_\delta\|_{H^{k-\frac{3}{2}+\alpha}(\Omega_\delta)}\lesssim_M \delta^{1-\alpha}+\|D_{\delta}B_\delta\cdot n_\delta\|_{H^{k-2+\alpha}(\Gamma_\delta)}.
\end{equation*}
By definition of $D_{\delta}B_\delta$, the relation $B_\delta:=\tilde{B}_\delta-\tilde{B}^{ir}_\delta$ and the regularization bounds \eqref{surfbound-par} for $\Gamma_{\delta}$, it is easy to see that
\begin{equation*}\label{normal cancellation}
\|D_{\delta}B_\delta\cdot n_\delta\|_{H^{k-2+\alpha}(\Gamma_\delta)}\lesssim_M \delta^{1-\alpha}+\|2\delta\Delta_{\Gamma_*}\tilde{B}_{\delta,*}\cdot n_{\delta,*}-(D_{\delta}\nabla\mathcal{H}\mathcal{N}^{-1}(\tilde{B}_\delta\cdot n_\delta)\cdot n_\delta)_*\|_{H^{k-2+\alpha}(\Gamma_*)}.
\end{equation*}
Next, we analyze the term $D_{\delta}\nabla\mathcal{H}\mathcal{N}^{-1}(\tilde{B}_\delta\cdot n_\delta)\cdot n_\delta$ in the above estimate. For $\alpha\in \{0,\frac{1}{2}\}$, we have by \Cref{epsderivativeofnormal}, the bound $\|\tilde{B}_{\delta}\|_{H^{k}(\Omega_{\delta})}\lesssim_M 1$  and the regularization estimates for $\Gamma_{\delta}$,
\begin{equation*}
\begin{cases}
&\partial_{\delta}(\tilde{B}_{\delta}\cdot n_{\delta})_*= 2\delta\Delta_{\Gamma_*}(\tilde{B}_{\delta}\cdot n_{\delta})_*+R_{\delta},\hspace{5mm}\text{on}\hspace{2mm}\Gamma_{*},
\\
&(\tilde{B}_0\cdot n_0)_*=0,\hspace{5mm}\text{on}\hspace{2mm}\Gamma_{*},
\end{cases}
\end{equation*}
with $\|R_{\delta}\|_{H^{k-2+\alpha}(\Gamma_*)}\lesssim_M \delta^{1-\alpha}$. 
Consequently,
\begin{equation*}\label{heatbound}
\|(\partial_{\delta}-2\delta\Delta_{\Gamma_*})(\tilde{B}_{\delta}\cdot n_{\delta})_*\|_{H^{k-2+\alpha}(\Gamma_*)}\lesssim_M \delta^{1-\alpha}.
\end{equation*}
From the fundamental theorem of calculus and the fact that $\tilde{B}_0\cdot n_0=0$ we also have stronger bounds for $\tilde{B}_{\delta}\cdot n_{\delta}$ in weaker topologies. Specifically,  we have
\begin{equation*}\label{weakbound}
\|\tilde{B}_{\delta}\cdot n_{\delta}\|_{H^{k-3}(\Gamma_\delta)}\lesssim_M \delta^2.
\end{equation*}
Using the above estimates our goal will be to show that
\begin{equation*}\label{goal}
\|D_{\delta}B_\delta\cdot n_\delta\|_{H^{k-2+\alpha}(\Gamma_\delta)}\lesssim_M \delta^{\frac{1}{2}}+\delta\|\mathcal{N}B_\delta\cdot n_\delta\|_{H^{k-1+\alpha}(\Gamma_\delta)}.
\end{equation*}
Indeed, from the above estimates, the condition $\alpha\leq\frac{1}{2}$, the bounds for $V_{\delta}$ and the identity $D_{\delta}n_{\delta}=-\nabla^{\top}V_{\delta}\cdot n_\delta$, we have
\begin{equation*}
\begin{split}
(D_{\delta}\nabla\mathcal{H}\mathcal{N}^{-1}(\tilde{B}_\delta\cdot n_\delta)\cdot n_\delta)_*&=\frac{d}{d\delta}(\tilde{B}_{\delta,*}\cdot n_{\delta,*})+\mathcal{O}_{H^{k-2+\alpha}}(\delta^{\frac{1}{2}})=2\delta\Delta_{\Gamma_*}(\tilde{B}_\delta\cdot n_\delta)_*+\mathcal{O}_{H^{k-2+\alpha}}(\delta^{\frac{1}{2}})
\\
&=2\delta\Delta_{\Gamma_*}(\tilde{B}_\delta^{ir}\cdot n_\delta)_*+\mathcal{O}_{H^{k-2+\alpha}}(\delta^{\frac{1}{2}})
\\
&=2\delta\Delta_{\Gamma_*}\tilde{B}^{ir}_{\delta,*}\cdot n_{\delta,*}+\mathcal{O}_{H^{k-2+\alpha}}(\delta^{\frac{1}{2}}).
\end{split}
\end{equation*}
It follows that
\begin{equation*}
\begin{split}
\|D_{\delta}B_\delta\cdot n_\delta\|_{H^{k-2+\alpha}(\Gamma_\delta)}\lesssim_M &\, \delta^{\frac{1}{2}}+\delta\|\Delta_{\Gamma_*}\tilde{B}_{\delta,*}\cdot n_{\delta,*}-\Delta_{\Gamma_*}\tilde{B}^{ir}_{\delta,*}\cdot n_{\delta,*}\|_{H^{k-2+\alpha}(\Gamma_*)}
\\
=&\, \delta^{\frac{1}{2}}+\delta\|\Delta_{\Gamma_*}B_{\delta,*}\cdot n_{\delta,*}\|_{H^{k-2+\alpha}(\Gamma_*)}.
\end{split}
\end{equation*}
It then suffices to show that
\begin{equation*}\label{collartoDN}
\delta\|\Delta_{\Gamma_*}B_{\delta,*}\cdot n_{\delta,*}\|_{H^{k-2+\alpha}(\Gamma_*)}\lesssim_M \delta^{\frac{1}{2}}+\delta\|\mathcal{N}B_\delta\cdot n_\delta\|_{H^{k-1+\alpha}(\Gamma_\delta)}.
\end{equation*}
By \Cref{ellipticity} and the Leibniz rule for $\mathcal{N}$, we have
\begin{equation*}
\delta\|\Delta_{\Gamma_*}B_{\delta,*}\cdot n_{\delta,*}\|_{H^{k-2+\alpha}(\Gamma_*)}\lesssim_M\delta^{\frac{1}{2}}+\delta\|\mathcal{N}^2(\Delta_{\Gamma_*} B_{\delta,*})^*\cdot n_{\delta}\|_{H^{k-4+\alpha}(\Gamma_{\delta})}.
\end{equation*}
From \Cref{coordinatescommutator}, the identities in \Cref{Movingsurfid} and \Cref{higherpowers} we  conclude that
\begin{equation}\label{doubleDNcom}
\begin{split}
\delta\|\mathcal{N}^2(\Delta_{\Gamma_*} B_{\delta,*})^*\cdot n_{\delta}\|_{H^{k-4+\alpha}(\Gamma_{\delta})}&\lesssim_M \delta^{\frac{1}{2}}+\delta\|\Delta_{\Gamma_*}(\mathcal{N}^2 B_{\delta})_*\cdot n_{\delta,*}\|_{H^{k-4+\alpha}(\Gamma_*)}
\\
&\lesssim_M \delta^{\frac{1}{2}}+\delta\|\mathcal{N}B_\delta\cdot n_\delta\|_{H^{k-1+\alpha}(\Gamma_\delta)},
\end{split}
\end{equation}
as desired. Next, we establish the analogous estimates for $\nabla_{B_\delta} D_{\delta}W_\delta^\pm$ in $L_{\delta}^1H^{k-\frac{3}{2}}$ and $L_{\delta}^2H^{k-2}$. To conveniently track error terms in our analysis we will write $R_B:=R_{B}(\delta)$ to denote a generic remainder term on $\Omega_{\delta}$ or $\Gamma_\delta$ such that
\begin{equation*}
\|R_B\|_{L_{\delta}^1H^{k-\frac{3}{2}}(\Omega_\delta)}\lesssim_M \delta+\delta^{\frac{1}{2}}\mathcal{J}(\delta),\hspace{5mm}\|\delta^{-\frac{1}{2}}R_B\|_{L_{\delta}^2H^{k-2}(\Omega_\delta)}\lesssim_M \delta^{\frac{1}{2}}+c\mathcal{J}(\delta)
\end{equation*}
if $R_B$ is defined on $\Omega_\delta$, or
\begin{equation*}
\|R_B\|_{L_{\delta}^1H^{k-2}(\Gamma_\delta)}\lesssim_M \delta+\delta^{\frac{1}{2}}\mathcal{J}(\delta),\hspace{5mm}\|\delta^{-\frac{1}{2}}R_B\|_{L_{\delta}^2H^{k-\frac{5}{2}}(\Gamma_\delta)}\lesssim_M \delta^{\frac{1}{2}}+c\mathcal{J}(\delta)
\end{equation*}
 if $R_B$ is defined on $\Gamma_{\delta}$. Similarly to before, we observe the estimates
\begin{equation*}
\|\nabla\cdot\nabla_{B_\delta}D_{\delta}B_{\delta}\|_{H^{k-3+\alpha}(\Omega_\delta)}+\|\nabla\times\nabla_{B_\delta}D_{\delta}B_{\delta}\|_{H^{k-3+\alpha}(\Omega_\delta)}\lesssim_M \delta+\|D_{\delta}B_{\delta}\|_{H^{k-2+\alpha}(\Omega_\delta)}+\|\nabla_{B_\delta}V_{\delta}\|_{H^{k-2+\alpha}(\Omega_\delta)}.
\end{equation*}
Using the above bounds for $D_{\delta}B_{\delta}$  and the crude bound
\begin{equation*}
\|\nabla_{B_\delta}V_{\delta}\|_{H^{k-2+\alpha}(\Omega_\delta)}\lesssim_M \|V_{\delta}\|_{H^{k-\frac{3}{2}+\alpha}(\Gamma_{\delta})}\lesssim_M \delta^{\frac{1}{2}-\alpha}
\end{equation*}
it is easy to see that
\begin{equation*}
\nabla\cdot\nabla_{B_\delta}D_{\delta}B_{\delta}=R_B,\hspace{5mm}\nabla\times\nabla_{B_\delta} D_{\delta}B_{\delta}=R_B.
\end{equation*}
Hence, by the div-curl estimate in \Cref{Balanced div-curl}, we have reduced matters to obtaining a suitable approximation for $\nabla_{B_\delta}D_{\delta}B_{\delta}\cdot n_\delta$ in $H^{k-\frac{5}{2}+\alpha}(\Gamma_\delta)$. Our  objective will be to show that
\begin{equation}\label{DBDdB}
\nabla_{B_\delta}D_{\delta}B_{\delta}\cdot n_\delta=2\delta\nabla_{B_\delta}(\Delta_{\Gamma_*}B_{\delta,*}\cdot n_{\delta,*})^*+R_B.
\end{equation}
We begin as above by observing that
\begin{equation*}
\nabla_{B_\delta}D_{\delta}B_{\delta}\cdot n_\delta=2\delta\nabla_{B_\delta}(\Delta_{\Gamma_*}\tilde{B}_{\delta,*}\cdot n_{\delta,*})^*-\nabla_{B_\delta}D_{\delta}\tilde{B}_\delta^{ir}\cdot n_\delta+R_B.
\end{equation*}
Arguing as in the estimate for $D_{\delta}B_{\delta}\cdot n_\delta$ we may use the smallness of $\tilde{B}_{\delta}\cdot n_{\delta}$ in $H^{k-3}(\Gamma_\delta)$ to obtain
\begin{equation*}
\nabla_{B_\delta}D_{\delta}\tilde{B}_\delta^{ir}\cdot n_\delta=\nabla_{B_\delta}D_{\delta}(\tilde{B}_\delta\cdot n_\delta)+R_B.
\end{equation*}
Then, using the heat equation for $\tilde{B}_\delta\cdot n_\delta$ we may compute that
\begin{equation*}
\nabla_{B_\delta}D_{\delta}(\tilde{B}_\delta\cdot n_\delta)=2\delta\nabla_{B_\delta}(\Delta_{\Gamma_*}(\tilde{B}_{\delta,*}\cdot n_{\delta,*}))^*+R_B=2\delta\nabla_{B_\delta}(\Delta_{\Gamma_*}\tilde{B}^{ir}_{\delta,*}\cdot n_{\delta,*})^*+R_B,
\end{equation*}
which yields \eqref{DBDdB}. Utilizing the ellipticity of $\mathcal{N}$, \Cref{Movingsurfid} and performing some commutator estimates similar to \eqref{doubleDNcom}, we get 
\begin{equation*}
\|\delta\nabla_{B_\delta}(\Delta_{\Gamma_*}B_{\delta,*}\cdot n_{\delta,*})^*\|_{H^{k-\frac{5}{2}+\alpha}(\Gamma_\delta)}\lesssim_M \delta^{\frac{1}{2}-\alpha}+\delta\|\mathcal{N}B_\delta\cdot n_\delta\|_{\mathbf{H}^{k-1+\alpha}(\Gamma_\delta)},
\end{equation*}
which is sufficient to complete the proof.
\end{proof}
Now, we turn to the estimate for $E^k_{i}$. We begin by focusing on the first two terms; namely,
\begin{equation*}
\|a^{\frac{1}{2}}\mathcal{N}^{k-1}a\|_{L^2(\Gamma)}^2+\|\nabla\mathcal{H}\mathcal{N}^{k-2}\nabla_Ba\|_{L^2(\Omega)}^2=:E^k_{a}+E^k_{\nabla_B a}.
\end{equation*}
\medskip
To exploit the parabolic regularizing effect on the surface $\Gamma_{\delta}$, we will need the following lemma which extracts the leading parts of the good variables $\mathcal{N}a$ and $D_{\delta}\mathcal{N}a$. We remark that, in the text below, we will very often drop the $\delta$ subscript, in order to declutter notation.
\begin{lemma}\label{othersurfaceenergy}
The following identities hold: 
\begin{equation}\label{without_grad_B}
\frac{d}{d\delta}(\mathcal{N}a)_*=2\delta a_*\Delta_{\Gamma_*}\kappa_*+R_*,\hspace{5mm}\mathcal{N}a=a\kappa+Q,
\end{equation}
where for some $0<c\ll c_2$ sufficiently small, there holds
\begin{equation*}\label{adecompbounds}
\|\delta^{-\frac{1}{2}}R\|_{L^2_{\delta}\mathbf{H}^{k-3}(\Gamma_\delta)}\lesssim_M \delta^{\frac{1}{2}}+c\mathcal{J}(\delta),\hspace{5mm}\|R\|_{L^1_{\delta}\mathbf{H}^{k-\frac{5}{2}}(\Gamma_\delta)}\lesssim_M \delta+\delta^{\frac{1}{2}}\mathcal{J}(\delta),\hspace{5mm}\|Q\|_{L^{\infty}_{\delta}\mathbf{H}^{k-\frac{3}{2}}(\Gamma_\delta)}\lesssim_M 1.
\end{equation*}
\end{lemma}
\begin{proof}
Thanks to Lemmas~\ref{epsderivativeofnormal} and \ref{Ddeltabounds}, the relation \eqref{without_grad_B} will be proved if we can establish the identity $\mathcal{N}a=a\kappa+\tilde{R}$, where $\tilde{R}$ satisfies
\begin{equation*}\label{Rbounds}
\|\tilde{R}\|_{L^{\infty}_{\delta}\mathbf{H}^{k-\frac{3}{2}}(\Gamma_\delta)}\lesssim_M 1,\hspace{5mm}\|\delta^{-\frac{1}{2}}D_{\delta}\tilde{R}\|_{L^2_{\delta}\mathbf{H}^{k-3}(\Gamma_\delta)}\lesssim_M \delta^{\frac{1}{2}}+c\mathcal{J}(\delta),\hspace{5mm}\|D_{\delta}\tilde{R}\|_{L^1_{\delta}\mathbf{H}^{k-\frac{5}{2}}(\Gamma_\delta)}\lesssim_M \delta+\delta^{\frac{1}{2}}\mathcal{J}(\delta).
\end{equation*}
The basic strategy of the proof follows a similar line of reasoning as \cite[Equation (8.15)]{Euler}, although here the estimates are a bit more complicated. We begin by relating $\mathcal{N} a$ to the mean curvature. First, we recall that $\Delta_{\Gamma}P=0$ as well as the general formula 
\begin{equation*}
\Delta P|_{\Gamma}=\Delta_{\Gamma}P-\kappa n_\Gamma\cdot\nabla P+D^2P(n_\Gamma,n_\Gamma).
\end{equation*}
Using the above and the relation $a=-n_\Gamma\cdot\nabla P$ we observe that 
\begin{equation*}
\begin{split}
a\kappa &=-n_in_j\partial_i\partial_jP+\Delta P
\\
&=-n_in_j\partial_i\partial_jP-\partial_iW^+_{j}\partial_jW^-_{i}
\\
&=-n_in_j\partial_i\partial_jP+\tilde{R},
\end{split}
\end{equation*}
where in the last line, we used \eqref{DdeltaWbound} to check the remainder property for $D_{\delta}\tilde{R}$. We now further expand using the Laplace equation for $P$,
\begin{equation*}
\begin{split}
-n_in_j\partial_i\partial_jP &=n_j\mathcal{N}(n_ja)+n_j\nabla_n\Delta^{-1}\partial_j(\partial_i W^+_{k}\partial_kW^-_{i})
\\
&=n_j\mathcal{N}(n_ja)+\tilde{R}.
\end{split}
\end{equation*}
Next, we write
\begin{equation*}
\begin{split}
n_j\mathcal{N}(n_ja)&=\mathcal{N}a+an_j\mathcal{N}n_j-2 n_j\nabla_n \Delta^{-1}(\nabla\mathcal{H} n_j\cdot\nabla\mathcal{H}a)
\\
&=\mathcal{N}a+a n_j\mathcal{N}n_j+\tilde{R}
\\
&=\mathcal{N}a+\tilde{R},
\end{split}
\end{equation*}
where in the first line we used the Leibniz rule \eqref{DNLeibniz} for $\mathcal{N}$ and in the third line we used the Leibniz rule again and the fact that $\mathcal{N}(n_jn_j)=0$.
\end{proof}
A simple corollary of \Cref{othersurfaceenergy} is the following.
\begin{corollary}\label{adecompwithB}
We have
\begin{equation}\label{with_grad_b}
\frac{d}{d\delta}(\mathcal{N}\nabla_Ba)_*=2\delta a_*\Delta_{\Gamma_*}(\nabla_B\kappa)_*+R_*,\hspace{5mm}\mathcal{N}\nabla_Ba=a\nabla_B\kappa+Q,
\end{equation}
where for some $0<c\ll c_2$ sufficiently small, 
\begin{equation*}
\|\delta^{-\frac{1}{2}}R\|_{L_{\delta}^2H^{k-\frac{7}{2}}(\Gamma_\delta)}\lesssim_M \delta^{\frac{1}{2}}+c\mathcal{J}(\delta),\hspace{5mm}\|R\|_{L_{\delta}^1H^{k-3}(\Gamma_\delta)}\lesssim_M \delta+\delta^{\frac{1}{2}}\mathcal{J}(\delta),\hspace{5mm}\|Q\|_{L_{\delta}^{\infty}H^{k-2}(\Gamma_\delta)}\lesssim_M 1.
\end{equation*}
\end{corollary}
\begin{proof}
In light of \Cref{othersurfaceenergy}, to prove \eqref{with_grad_b} it  suffices to apply $\nabla_B$ to \eqref{without_grad_B} and show that the commutators produce acceptable errors.  We begin by proving that
\begin{equation*}
[\nabla_B,\mathcal{N}]a=Q,\hspace{5mm}[\nabla_B,D_{\delta}]\mathcal{N}a+D_{\delta}[\nabla_B,\mathcal{N}]a=R,
\end{equation*}
where $R$ and $Q$ denote error terms as specified in \Cref{adecompwithB}. Thanks to \Cref{Movingsurfid}, we have the bounds $\|[\nabla_B,\mathcal{N}]a\|_{H^{k-2}(\Gamma_\delta)}\lesssim_M \|a\|_{H^{k-1}(\Gamma_\delta)}\lesssim_M 1$. Using that $\|D_{\delta}B\|_{H^{k-3}(\Gamma_{\delta})}\lesssim_M \delta$, $\|V_{\delta}\|_{H^{k-2}(\Gamma_\delta)}\lesssim_M \delta$ and $\|\mathcal{N}a\|_{H^{k-2}(\Gamma_\delta)}\lesssim_M 1$, it is straightforward to verify the required bound for $[\nabla_B,D_{\delta}]\mathcal{N}a$. For the remaining term, we use the identities in \Cref{Movingsurfid} to expand
\begin{equation*}
[\nabla_B,\mathcal{N}]a=\nabla_Bn\cdot \nabla\mathcal{H}a-n\cdot ((\nabla B)^*(\nabla\mathcal{H}a))+\nabla_n\Delta^{-1}(2\nabla B\cdot\nabla^2\mathcal{H}a+\Delta B\cdot\nabla\mathcal{H}a).
\end{equation*}
Then,  using  the identities in \Cref{Movingsurfid} and \Cref{Ddeltabounds} once again, it may be verified through straightforward (but slightly tedious) computations that $D_{\delta}[\nabla_B,\mathcal{N}]a$ satisfies the required estimates. 
\medskip

To finish the proof, we must commute $\nabla_B$ through the first term on the right-hand side of \eqref{with_grad_b}, and, in particular, show that the commutator with the Laplace-Beltrami operator can be absorbed into $R$. This is an exercise in local coordinates, and is therefore left to the reader.
\end{proof}
Now, we have enough information to establish the requisite parabolic-type energy estimate for $E^k_a$ and $E^k_{\nabla_Ba}$. We begin with an estimate for the former term. First, using \Cref{Movingsurfid} and the relevant bounds in \Cref{BEE}, we have
\begin{equation*}
\|D_{\delta}a\|_{L^{\infty}(\Gamma_\delta)}\lesssim_M 1,\hspace{5mm} \|[\mathcal{N}^{k-2},D_{\delta}]\mathcal{N}a\|_{L^2(\Gamma_\delta)}\lesssim_M 1.
\end{equation*}
Therefore, we conclude from \Cref{othersurfaceenergy} and \eqref{Leibniz on moving hypersurfaces general} that
\begin{equation}\label{aEE}
\begin{split}
\frac{d}{d\delta}E^k_a&\lesssim_M 1+\delta \langle a\mathcal{N}^{k-2}(a_*\Delta_{\Gamma_*}\kappa_*)^*,\mathcal{N}^{k-2}(a\kappa)\rangle+\delta\langle a\mathcal{N}^{k-2}(a\Delta_{\Gamma_*}\kappa_*)^*,\mathcal{N}^{k-2}Q\rangle
\\
&+\langle a\mathcal{N}^{k-2}R,\mathcal{N}^{k-2}(a\kappa)\rangle+\langle a\mathcal{N}^{k-2}R,\mathcal{N}^{k-2}Q\rangle.
\end{split}
\end{equation}
Using the Taylor sign condition, commuting $a$ with $\mathcal{N}^{k-2}$ (using the Leibniz rule for $\mathcal{N}$) and then  commuting $\Delta_{\Gamma_*}$ with $\mathcal{N}^{k-2}$ using \Cref{coordinatescommutator} and integration by parts, we have 
\begin{equation*}
\delta\langle a\mathcal{N}^{k-2}(a_*\Delta_{\Gamma_*}\kappa_*)^*,\mathcal{N}^{k-2}(a\kappa)\rangle\lesssim_M 1-\delta\|(-\Delta_{\Gamma_*})^{\frac{1}{2}}(\mathcal{N}^{k-2}\kappa_{\delta})_*\|_{L^2(\Gamma_*)}^2\lesssim_M 1-\delta\|\kappa_{\delta}\|_{H^{k-1}(\Gamma_{\delta})}^2.
\end{equation*}
By self-adjointness and ellipticity of $\mathcal{N}$ as well as the parabolic regularization bounds for $\Gamma_{\delta}$ in \Cref{epsderivativeofnormal} and \Cref{curvaturebound}, we may estimate
\begin{equation*}
\delta\langle a\mathcal{N}^{k-2}(a_*\Delta_{\Gamma_*}\kappa_*)^*,\mathcal{N}^{k-2}Q\rangle\lesssim_M\delta\|Q\|_{H^{k-\frac{3}{2}}(\Gamma_\delta)}\|\kappa_{\delta}\|_{H^{k-\frac{1}{2}}(\Gamma_\delta)}\lesssim_M 1+\delta^{\frac{1}{2}}\|\kappa_{\frac{1}{2}\delta}\|_{H^{k-1}(\Gamma_{\frac{1}{2}\delta})}.
\end{equation*}
By Cauchy-Schwarz, self-adjointness of $\mathcal{N}$ and \Cref{higherpowers} we moreover have
\begin{equation*}
\begin{split}
\langle a\mathcal{N}^{k-2}R,\mathcal{N}^{k-2}(a\kappa)\rangle+\langle a\mathcal{N}^{k-2}R,\mathcal{N}^{k-2}Q\rangle &\lesssim_M \|R\|_{H^{k-3}(\Gamma_\delta)}\|\kappa\|_{H^{k-1}(\Gamma_\delta)}+\|R\|_{H^{k-\frac{5}{2}}(\Gamma_\delta)}\|Q\|_{H^{k-\frac{3}{2}}(\Gamma_\delta)}.
\end{split}
\end{equation*}
Hence, by integrating from $0$ to $\delta$, rescaling, and applying Cauchy-Schwarz, we obtain the parabolic estimate,
\begin{equation}\label{aEE22}
E^k_a(\delta)+c_1\int_{0}^{\delta}\tau\|\kappa_{\tau}\|_{H^{k-1}(\Gamma_\tau)}^2d\tau\leq E^k_a(0)+c_2\mathcal{J}^2(\delta)+C(M)\delta,
\end{equation}
where $c_1$ and $c_2$ are constants with $c_2\ll c_1$ as in the statement of \Cref{E1estimate}.
\medskip

Next, we prove the analogous estimate for $E^k_{\nabla_Ba}$. Let $R$ and $Q$ denote error terms as in \Cref{adecompwithB}. Similarly to the above, we have the commutator estimate $\|[\nabla\mathcal{H}\mathcal{N}^{k-2},D_{\delta}]\mathcal{N}\nabla_B a\|_{L^2(\Gamma_\delta)}\lesssim_M 1$. Moreover, using self-adjointness of $\mathcal{N}$, we have the identity
\begin{equation*}
\langle\nabla\mathcal{H}f,\nabla\mathcal{H}g\rangle_{L^2(\Omega)}=\langle\mathcal{N}^{\frac{1}{2}}f,\mathcal{N}^{\frac{1}{2}}g\rangle_{L^2(\Gamma)}.
\end{equation*}
Hence, similarly to the computation for $E_{\nabla_Ba}^k$, we have
\begin{equation*}
\frac{d}{d\delta}E^k_{\nabla_Ba}=\frac{d}{d\delta}\|\nabla\mathcal{H}\mathcal{N}^{k-2}\nabla_Ba\|_{L^2(\Omega_\delta)}^2\lesssim_M 1+ \langle \mathcal{N}^{k-\frac{3}{2}}\nabla_Ba,\mathcal{N}^{k-\frac{5}{2}} D_{\delta}\mathcal{N}\nabla_Ba\rangle.
\end{equation*}
Using  \Cref{gradBV}, \Cref{adecompwithB}, \Cref{coordinatescommutator}, \Cref{Leibnizcom} and following a similar line of reasoning to the above estimate for $E^k_a$, we obtain
\begin{equation}\label{BaEE}
\begin{split}
E^k_{\nabla_Ba}(\delta)+c_1\int_{0}^{\delta}\tau\|\nabla_{B_{\tau}}\kappa_{\tau}\|_{H^{k-\frac{3}{2}}(\Gamma_\tau)}^2d\tau\leq E^k_{\nabla_Ba}(0)+c_2\mathcal{J}^2(\delta)+C(M)\delta.
\end{split}
\end{equation}
Next, we turn to the estimates for the remaining components of $E^k_i$ given by 
\begin{equation*}
\|\nabla\mathcal{H}\mathcal{N}^{k-2}\mathcal{G}^\pm\|_{L^2(\Omega)}^2+\|a^{-\frac{1}{2}}\mathcal{N}^{k-2}\nabla_B\mathcal{G}^\pm\|_{L^2(\Gamma)}^2=:E^k_{\mathcal{G}^\pm}+E^k_{\nabla_B\mathcal{G}^\pm}.
\end{equation*}
We begin by establishing a suitable analogue of \Cref{othersurfaceenergy} but for the good variables $\mathcal{G}^\pm$. 
\begin{lemma}\label{Bbounds1}
There exist functions $R$, $Q$ and $Q_B$ such that
\begin{equation}\label{Gdecomp}
\frac{d}{d\delta}(\mathcal{N}\mathcal{G}^\pm)_*=-2\delta a_*\Delta_{\Gamma_*}(\mathcal{N}(\mathcal{N}W^\pm\cdot n))_*+R_*,\hspace{5mm} \mathcal{G}^\pm=-a\mathcal{N}W^\pm\cdot n+Q
\end{equation}
and
\begin{equation}\label{Bbounds2}
\nabla_B\mathcal{G}^\pm=-a\nabla_B(\mathcal{N}W^\pm\cdot n)+Q_B
\end{equation}
where for some $0<c\ll c_2$ there holds
\begin{equation*}
\|\delta^{-\frac{1}{2}}R\|_{L_{\delta}^2\mathbf{H}^{k-\frac{7}{2}}(\Gamma_\delta)}\lesssim_M \delta^{\frac{1}{2}}+c\mathcal{J}(\delta),\hspace{5mm}\|\delta^{\frac{1}{2}}(Q,Q_B)\|_{L_{\delta}^{2}H^{k-\frac{1}{2}}(\Gamma_\delta)\times L_{\delta}^{2}H^{k-1}(\Gamma_\delta)}\lesssim_M \delta^{\frac{1}{2}}+c\mathcal{J}(\delta).
\end{equation*}
\end{lemma}
\begin{proof}
We start with the differential identity. We first recall that
\begin{equation*}
\mathcal{G}^\pm=\nabla_nW^\pm\cdot\nabla P-\nabla_n\Delta^{-1}(\Delta W^\pm\cdot\nabla P+2\nabla W^\pm\cdot \nabla^2 P).
\end{equation*}
We note that by definition we have $P=-\Delta^{-1}(\partial_iW_j^+\partial_jW_i^-)$. Therefore, using the elliptic estimates in \Cref{BEE}, the commutator identities in \Cref{Movingsurfid} and the regularization bounds for $\Gamma_\delta$ we may collect the estimates 
\begin{equation*}
\|V\|_{\mathbf{H}^{k-\frac{3}{2}}(\Gamma_{\delta})}+\|D_{\delta} n\|_{\mathbf{H}^{k-\frac{5}{2}}(\Gamma_{\delta})}+\|D_{\delta}P\|_{\mathbf{H}^{k-1}(\Omega_{\delta})}\lesssim_M\delta^{\frac{1}{2}}+\|D_{\delta}W^\pm\|_{H^{k-2}(\Omega_{\delta})}.
\end{equation*}
From these estimates, \Cref{Ddeltabounds}, the identities in  \Cref{Movingsurfid} and the identity $\mathcal{H}f=f-\Delta^{-1}\Delta f$ we obtain
\begin{equation*}
\begin{split}
D_{\delta}\mathcal{N}\mathcal{G}^\pm &=\mathcal{N}(\nabla_nD_{\delta}W^\pm\cdot\nabla P-\nabla_n\Delta^{-1}(\Delta D_{\delta}W^\pm\cdot\nabla P))+R
\\
&=\mathcal{N}\nabla_n(D_{\delta}W^\pm\cdot\nabla P-\Delta^{-1}\Delta( D_{\delta}W^\pm\cdot\nabla P))+R
\\
&=-a\mathcal{N}^2(D_{\delta}W^\pm\cdot n)+R.
\end{split}
\end{equation*}
Arguing as in the proof of \Cref{Ddeltabounds} and using the commutator identity for $[\nabla_B,\mathcal{N}]$ from \Cref{Movingsurfid}, we see that
\begin{equation*}
\|\mathcal{N}^2(D_{\delta}W^\pm\cdot n-2\delta(\Delta_{\Gamma_*}B_*\cdot n_*)^*)\|_{\mathbf{H}^{k-\frac{7}{2}}(\Gamma_\delta)}\lesssim_M \|D_{\delta}W^\pm\cdot n-2\delta(\Delta_{\Gamma_*}B_*\cdot n_*)^*\|_{\mathbf{H}^{k-\frac{3}{2}}(\Gamma_\delta)}\lesssim_M \delta^{\frac{1}{2}}.
\end{equation*}
Therefore, from \Cref{fractionalpowerapprox} and arguing as in \Cref{Ddeltabounds}, we have
\begin{equation*}
\begin{split}
\frac{d}{d\delta}(\mathcal{N}\mathcal{G}^\pm)_*&=-2a_*\delta(\mathcal{N}^2(\Delta_{\Gamma_*}W^\pm_*\cdot n_*)^*)_*+R_*
\\
&=-2a_*\delta\Delta_{\Gamma_*}(\mathcal{N}^2W^\pm\cdot n)_*+R_*
\\
&=-2a_*\delta\Delta_{\Gamma_*}(\mathcal{N}(\mathcal{N}W^\pm\cdot n))_*+R_*,
\end{split}
\end{equation*}
which establishes the differential identity in \eqref{Gdecomp}. Now, we turn to the second decomposition. Using the crude paraproduct expansion $fg=f^lg^{\leq l}+f^{\leq l}g^l$ outlined in \eqref{bilpara}  and writing $W$ as shorthand for $W^\pm$, we observe the bounds
\begin{equation*}
\|\nabla W^{\leq l}\cdot \nabla P^{l}\|_{H^{k}(\Omega_{\delta})}\lesssim_M 1+\|P\|_{H^{k+1}(\Omega_{\delta})}\lesssim_M \|\Gamma_{\delta}\|_{H^{k+\frac{1}{2}}}+1\lesssim_M \delta^{-\frac{1}{2}},
\end{equation*}
where in the second inequality we used \Cref{direst} and in the third we used the parabolic regularization bound \eqref{surfbound-par} for $\Gamma_{\delta}$.  Similarly, one may check that 
\begin{equation*}
\|\nabla_n\Delta^{-1}(\Delta W^{\leq l}\cdot\nabla P^l+\nabla W\cdot\nabla ^2 P)\|_{H^{k-\frac{1}{2}}(\Gamma_{\delta})}\lesssim_M \delta^{-\frac{1}{2}},\hspace{5mm}\|\nabla W^{\leq l}\cdot\nabla^2 P^l\|_{H^{k-1}(\Omega_{\delta})}\lesssim_M \delta^{-\frac{1}{2}}.
\end{equation*}
Consequently, we have
\begin{equation*}
\mathcal{G}^\pm=\mathcal{N}(\nabla_nW^l\cdot\nabla P^{\leq l}-\nabla_n\Delta^{-1}(\Delta W^l\cdot\nabla P^{\leq l}))+Q=-\mathcal{N}(W^l\cdot\nabla P^{\leq l})+Q.
\end{equation*}
From the Leibniz rule for $\mathcal{N}$ and a similar analysis to the above we may write
\begin{equation*}
\begin{split}
\mathcal{N}(W^l\cdot\nabla P^{\leq l})&= \mathcal{N}W^l\cdot\nabla P^{\leq l}+W^l\cdot\mathcal{N}\nabla P^{\leq l}-2\nabla_n\Delta^{-1}(\nabla\mathcal{H}W^l\cdot\nabla \mathcal{H}\nabla P^{\leq l})
\\
&=\mathcal{N}W^l\cdot\nabla P^{\leq l}+Q
\\
&=\mathcal{N}W\cdot\nabla P-\mathcal{N}W^{\leq l}\cdot\nabla P^{l}+Q
\\
&=\mathcal{N}W\cdot\nabla P+Q,
\end{split}
\end{equation*}
which from the identity $\nabla P=-an$ yields the desired decomposition. Now, we turn to the final decomposition \eqref{Bbounds2}. A vital ingredient for the proof of \eqref{Bbounds2} is the following technical proposition which will allow us to control certain commutators involving $\nabla_B$ in terms of $\mathcal{J}(\delta)$.
\begin{proposition}\label{technicalprop}
Let $F\in L_{\delta}^{\infty}H^{k-\frac{3}{2}}(\Omega_\delta)$ with $\|F\|_{L_{\delta}^{\infty}H^{k-\frac{3}{2}}(\Omega_\delta)}\lesssim_M 1$. For ${Q}_B$ as above, there holds
\begin{equation*}
\nabla_Bn=Q_B,\hspace{5mm}[\nabla_B,\nabla_n\Delta^{-1}]F=Q_B.
\end{equation*}
\end{proposition}
\begin{proof}
We first prove the bound for $\nabla_B n$, which is simpler. We know from the identities in \Cref{Movingsurfid} and \Cref{ellipticity} that
\begin{equation*}
\begin{split}
\|\nabla_Bn\|_{H^{k-1}(\Gamma_\delta)}=\|\nabla^{\top}B\cdot n\|_{H^{k-1}(\Gamma_\delta)}&\lesssim_M 1+\|\mathcal{N}\nabla^{\top}B\cdot n\|_{H^{k-2}(\Gamma_\delta)}
\\
&\lesssim_M 1+\|[\nabla^{\top},\mathcal{N}]B\|_{H^{k-2}(\Gamma_\delta)}+\|\mathcal{N}B\cdot n\|_{H^{k-1}(\Gamma_\delta)}.
\end{split}
\end{equation*}
By \Cref{commutatorremark}, we may bound $\|[\nabla^{\top},\mathcal{N}]B\|_{H^{k-2}(\Gamma_\delta)}\lesssim_M \|B\|_{H^{k-1}(\Gamma_\delta)}\lesssim_M 1$. Moreover, by interpolation, we have $\|\mathcal{N}B\cdot n\|_{H^{k-1}(\Gamma_\delta)}\lesssim_M \|\mathcal{N}B\cdot n\|_{H^{k-\frac{1}{2}}(\Gamma_\delta)}^{\frac{1}{2}}$. Therefore, by Cauchy-Schwarz, we obtain $\nabla_Bn=Q_B$, as needed. For the latter bound, we observe by commuting $\nabla_B$ with $\Delta^{-1}$ and $\nabla_n$ that
\begin{equation*}
[\nabla_B,\nabla_n\Delta^{-1}]F=-\nabla_nB\cdot\nabla\Delta^{-1}F+\nabla_n\Delta^{-1}(\Delta B\cdot \nabla \Delta^{-1}F+2\nabla B\cdot\nabla^2 \Delta^{-1}F).
\end{equation*}
Using the hypothesis for $F$ and arguing as in the proof of \eqref{Bbounds1}, we have
\begin{equation*}
[\nabla_B,\nabla_n\Delta^{-1}]F=\mathcal{N}B\cdot\nabla\Delta^{-1}F+Q_B.
\end{equation*}
Since $\nabla\Delta^{-1}F$ is normal to $\Gamma_\delta$, we deduce
\begin{equation*}
\|[\nabla_B,\nabla_n\Delta^{-1}]F\|_{H^{k-1}(\Gamma_\delta)}\lesssim_M \|Q_B\|_{H^{k-1}(\Gamma_\delta)}+\|\mathcal{N}B\cdot n\|_{H^{k-1}(\Gamma_\delta)}\lesssim_M \|Q_B\|_{H^{k-1}(\Gamma_\delta)}+\|\mathcal{N}B\cdot n\|_{H^{k-\frac{1}{2}}(\Gamma_\delta)}^{\frac{1}{2}}.
\end{equation*}
The desired estimate then follows similarly to the bound for $\nabla_B n$.
\end{proof}
Now, we return to complete the proof of \eqref{Bbounds2}.  We first note that an important immediate corollary of  \Cref{technicalprop} is that
\begin{equation*}
(\nabla_B\nabla P)_{|\Gamma_{\delta}}=Q_B.
\end{equation*}
Arguing as in the proof of \eqref{Gdecomp}, we  obtain 
\begin{equation*}
\begin{split}
\nabla_B\mathcal{G}^\pm &=\nabla_B(\nabla_nW^l\cdot\nabla P^{\leq l}-\nabla_n\Delta^{-1}(\Delta W^l\cdot\nabla P^{\leq l}))+Q_B
\\
&=-\nabla_B\mathcal{N}(W^l\cdot\nabla P^{\leq l})+Q_B
\\
&=-\nabla_B(\mathcal{N}W^l\cdot\nabla P^{\leq l})+Q_B
\\
&=-\nabla_B(\mathcal{N}W\cdot\nabla P)+Q_B,
\end{split}
\end{equation*}
which concludes the proof of \Cref{Bbounds1}.
\end{proof}
Now, we turn to the estimates for $E^k_{\mathcal{G}^\pm}$ and $E^k_{\nabla_B\mathcal{G}^\pm}$. To simplify notation, we write $\mathcal{F}^\pm:=\mathcal{N}(\mathcal{N}W^\pm\cdot n)$. For the first energy component we may use \Cref{Movingsurfid} and the bound $\|V_{\delta}\|_{H^{k-\frac{3}{2}}(\Omega_\delta)}\lesssim_M \delta^{\frac{1}{2}}\leq 1$ to obtain $\|[D_{\delta},\nabla\mathcal{H}\mathcal{N}^{k-3}]\mathcal{N}\mathcal{G}^\pm\|_{L^2(\Omega_\delta)}\lesssim_M 1$. Therefore, we have
\begin{equation*}
\begin{split}
\frac{d}{d\delta}E^k_{\mathcal{G}^\pm}&=\frac{d}{d\delta}\|\nabla\mathcal{H}\mathcal{N}^{k-2}\mathcal{G}^\pm\|_{L^2(\Omega_\delta)}^2\lesssim_M 1+\langle \mathcal{N}^{k-\frac{3}{2}}\mathcal{G}^\pm,\mathcal{N}^{k-\frac{5}{2}}D_{\delta}\mathcal{N}\mathcal{G}^\pm\rangle_{L^2(\Gamma_{\delta})},
\end{split}
\end{equation*}
which thanks to \Cref{Bbounds1}, \Cref{Leibnizcom} and self-adjointness of $\mathcal{N}$ yields
\begin{equation*}
\begin{split}
\frac{d}{d\delta}E^k_{\mathcal{G}^\pm}\lesssim_M 1+\delta^{\frac{1}{2}}\|\mathcal{F}^\pm\|_{H^{k-\frac{3}{2}}(\Gamma_\delta)}+\delta &\langle a^2\mathcal{N}^{k-\frac{5}{2}}(\Delta_{\Gamma_*}\mathcal{F}^\pm_*)^*,\mathcal{N}^{k-\frac{5}{2}}\mathcal{F}^\pm\rangle-\delta\langle a\mathcal{N}^{k-\frac{5}{2}}(\Delta_{\Gamma_*}\mathcal{F}^\pm_*)^*, \mathcal{N}^{k-\frac{3}{2}}Q\rangle
\\
-&\langle a\mathcal{N}^{k-\frac{5}{2}}\mathcal{F}^\pm,\mathcal{N}^{k-\frac{5}{2}}R\rangle+\langle \mathcal{N}^{k-\frac{5}{2}}R,\mathcal{N}^{k-\frac{3}{2}}Q\rangle.
\end{split}
\end{equation*}
By \Cref{coordinatescommutator}, integration by parts, \Cref{higherpowers},  self-adjointness of $\mathcal{N}$ and Cauchy-Schwarz, we  conclude that
\begin{equation*}
\frac{d}{d\delta}E^k_{\mathcal{G}^\pm}\lesssim_M 1-\delta\|\mathcal{F}^\pm\|_{H^{k-\frac{3}{2}}(\Gamma_\delta)}^2+\delta\|Q\|_{H^{k-\frac{1}{2}}(\Gamma_\delta)}^2+\delta^{-1}\|R\|_{H^{k-\frac{7}{2}}(\Gamma_\delta)}^2.
\end{equation*}
By \Cref{ellipticity}, we also have $\|\mathcal{N}W^\pm\cdot n\|_{H^{k-\frac{1}{2}}(\Gamma_\delta)}\lesssim_M 1+\|\mathcal{F}^\pm\|_{H^{k-\frac{3}{2}}(\Gamma_\delta)}$. Therefore, from \eqref{Gdecomp} and integrating the above from $0$ to $\delta$, we obtain
\begin{equation}\label{Geneset}
E^k_{\mathcal{G}^\pm}(\delta)+c_1\int_{0}^{\delta}\tau\|\mathcal{N}W_\tau^\pm\cdot n_\tau\|_{H^{k-\frac{1}{2}}(\Gamma_\tau)}^2d\tau\leq E^k_{\mathcal{G}^\pm}(0)+c_2\mathcal{J}^2(\delta)+C(M)\delta.
\end{equation}
Finally, we turn to estimating $E^k_{\nabla_B\mathcal{G}^\pm}$. Using the identities in \Cref{Movingsurfid}, the bounds for $V_{\delta}$ and the definition of $D_{\delta}B$, we have
\begin{equation*}
\|D_{\delta}\mathcal{N}\nabla_B\mathcal{G}^\pm-\nabla_BD_{\delta}\mathcal{N}\mathcal{G}^\pm\|_{H^{k-4}(\Gamma_\delta)}\lesssim_M\delta.
\end{equation*}
Moreover, using the notation for $\mathcal{F}^\pm$ from earlier we see that
\begin{equation*}
\|\nabla_B(\Delta_{\Gamma_*}\mathcal{F}^\pm_*)^*-(\Delta_{\Gamma_*}(\nabla_B\mathcal{F}^\pm)_*)^*\|_{H^{k-4}(\Gamma_\delta)}\lesssim_M \|\mathcal{N}W^\pm\cdot n\|_{H^{k-1}(\Gamma_\delta)}\lesssim_M \|\mathcal{N}W^\pm\cdot n\|_{H^{k-\frac{1}{2}}(\Gamma_\delta)}^{\frac{1}{2}}.
\end{equation*}
Therefore, thanks to \Cref{Bbounds1}, there exist functions $R_B$ and $Q_B$ such that
\begin{equation*}
\frac{d}{d\delta}(\mathcal{N}\nabla_B\mathcal{G}^\pm)_*=-2\delta a_*\Delta_{\Gamma_*}(\nabla_B\mathcal{F}^\pm)_*+(R_B)_*,\hspace{5mm}\nabla_B\mathcal{G}^\pm=-a\nabla_B(\mathcal{N}W^\pm\cdot n)+Q_B
\end{equation*}
where
\begin{equation*}
\|\delta^{-\frac{1}{2}}R_B\|_{L^2_{\delta}H^{k-4}(\Gamma_\delta)}+\|\delta^{\frac{1}{2}}Q_B\|_{L_{\delta}^2H^{k-1}(\Gamma_\delta)}\lesssim_M \delta^{\frac{1}{2}}+c_2\mathcal{J}(\delta).
\end{equation*}
Consequently, arguing as in the estimate for $E^k_{\mathcal{G}^\pm}$, we have
\begin{equation*}
\begin{split}
\frac{d}{d\delta}E^k_{\nabla_B\mathcal{G}^\pm}&=\frac{d}{d\delta}\|a^{-\frac{1}{2}}\mathcal{N}^{k-2}\nabla_B\mathcal{G}^\pm\|_{L^2(\Gamma_\delta)}^2\lesssim_M 1-\delta\|\nabla_B\mathcal{F}^\pm\|_{H^{k-2}(\Gamma_\delta)}^2+\delta\|Q_B\|_{H^{k-1}(\Gamma_\delta)}^2+\delta^{-1}\|R_B\|_{H^{k-4}(\Gamma_\delta)}^2.
\end{split}
\end{equation*}
As $\|[\nabla_B,\mathcal{N}]\|_{H^{k-1}\to H^{k-2}}\lesssim_M 1$, there holds
\begin{equation*}
\begin{split}
\|\nabla_B(\mathcal{N}W^\pm\cdot n)\|_{H^{k-1}(\Gamma_\delta)}&\lesssim_M 1+ \|\mathcal{N}\nabla_B(\mathcal{N}W^\pm\cdot n)\|_{H^{k-2}(\Gamma_\delta)}
\\
&\lesssim_M 1+\|\nabla_B\mathcal{F}^\pm\|_{H^{k-2}(\Gamma_\delta)}+\|\mathcal{N}W^\pm\cdot n\|_{H^{k-1}(\Gamma_\delta)}
\\
&\lesssim_M 1+\|\nabla_B\mathcal{F}^\pm\|_{H^{k-2}(\Gamma_\delta)}+\|\mathcal{N}W^\pm\cdot n\|_{H^{k-\frac{1}{2}}(\Gamma_\delta)}^{\frac{1}{2}}.
\end{split}
\end{equation*}
Hence, integrating the differential inequality for $E^k_{\nabla_B\mathcal{G}^\pm}$ and using Cauchy-Schwarz, we obtain
\begin{equation}\label{DBGest}
E^k_{\nabla_B\mathcal{G}^\pm}(\delta)+c_1\int_{0}^{\delta}\tau\|\nabla_{B_\tau}(\mathcal{N}W_{\tau}^\pm\cdot n_\tau)\|_{H^{k-1}(\Gamma_\tau)}^2d\tau\leq E^k_{\nabla_B\mathcal{G}^\pm}(0)+c_2\mathcal{J}^2(\delta)+C(M)\delta.
\end{equation}
Combining \eqref{aEE22}, \eqref{BaEE}, \eqref{Geneset} and \eqref{DBGest}   we finally conclude the proof of \eqref{parabolicemonbounds} and thus the proof of \Cref{E1estimate}.
\subsection{Step 2: A mild regularization} To close the bootstrap \eqref{inductiveregbound} at the end of our iteration, we will need a very mild regularization for the full velocity and magnetic fields which will allow us to estimate  norms slightly beyond $H^{k+1}$ in the transport step of the argument below. 
\begin{proposition}\label{mildreg}
Let $(v,B,\Gamma)\in \mathbf{H}^k$ be a state with $\|(v,B,\Gamma)\|_{\mathbf{H}^k}\leq C(M)$ and $\|\Gamma\|_{H^{k+\alpha}}\lesssim_{M,\alpha}\epsilon^{-\alpha}$ for $\alpha>0$. Assume moreover that the pair $(v,B)$ satisfies the regularization bounds
\begin{equation}\label{importantbounds}
\|W^\pm\|_{H^{k+1}(\Omega)}\lesssim_M\epsilon^{-\frac{3}{2}},\hspace{5mm}\|\mathcal{N}W^\pm\cdot n_\Gamma\|_{{H}^{k-\frac{1}{2}}(\Gamma)}\lesssim_M\epsilon^{-1},\hspace{5mm}\|\omega^\pm\|_{H^k(\Omega)}\leq K(M)\epsilon^{-\frac{3}{2}}.
\end{equation}
Then there exists a regularized state $(v_{\epsilon},B_{\epsilon},\Gamma)$ (with the same domain) such that $(v_{\epsilon}, B_{\epsilon})$ satisfies the first two bounds in \eqref{importantbounds} as well as the following properties:
\begin{enumerate}
\item (Good pointwise approximation).
\begin{equation*}
W^\pm_{\epsilon}=W^\pm+\mathcal{O}_{C^3(\Omega)}(\epsilon^2).
\end{equation*}
\item (Higher order regularization bounds).
\begin{equation*}
\|W^\pm_{\epsilon}\|_{H^{k+1+\alpha}(\Omega)}\lesssim_{M,\alpha} \epsilon^{-\frac{3}{2}-3\alpha},\hspace{5mm} \alpha\geq 0.
\end{equation*}
\item (Propagation of vorticity bound).
\begin{equation*}
\|\omega_{\epsilon}^\pm\|_{H^k(\Omega)}\leq K(M)\epsilon^{-\frac{3}{2}}(1+C(M)\epsilon).
\end{equation*}
\item (Energy monotonicity).
\begin{equation*}
E^k(v_{\epsilon},B_{\epsilon},\Gamma)\leq E^k(v,B,\Gamma)+C(M)\epsilon.
\end{equation*}
\end{enumerate}
\end{proposition}
\begin{proof} We regularize $v$ and $B$ in a na\"ive way as follows: First, we define for some $C>0$ sufficiently large, the map
\begin{equation*}
y(x):=x-C\epsilon^3\nu(x),\hspace{5mm}x\in\Omega.
\end{equation*}
 Here, $\nu$ is a smooth, unit vector field on $\mathbb{R}^d$ which is uniformly transverse to hypersurfaces in $\Lambda_*$ (see \Cref{Collarcoords}). The above map is a diffeomorphism from $\Omega$ to a domain contained in $\Omega$ whose boundary is at distance $\approx C\epsilon^3$ away from $\Gamma$. For a function $f$ defined on $\Omega$, we let $\tilde{f}(x):=f(y)$ and define the regularizations
\begin{equation*}
B_{\epsilon}:=(\chi_{\epsilon^{-3}}*\tilde{B})^{rot},\hspace{5mm}v_{\epsilon}:=v^{ir}+(\chi_{\epsilon^{-3}}*\tilde{v})^{rot},
\end{equation*}
where $\chi_{\epsilon^{-3}}(x)=\epsilon^{-3d}\chi(\epsilon^{-3}x)$ is a standard mollifier (see \cite[Appendix C]{MR2597943}) at scale $\epsilon^{-3}$. The definitions of $\tilde{v}$ and $\tilde{B}$ ensure that  (if $C$ is large enough),  $v_{\epsilon}$ and $B_{\epsilon}$ are well-defined on $\Omega$. Using the standard mollifier properties and interpolating the obvious bounds for integer $s$, the map $f\mapsto \chi_{\epsilon^{-3}}*\tilde{f}$ is easily seen to have $H^s\to H^s$ norm at most $(1+C(M)\epsilon)$ for every $s\geq 0$. This, along with the elliptic estimates in \Cref{BEE}, can be used to establish properties (i)-(iii).
\medskip  

To establish property (iv), we note that given the bounds $\|v\|_{H^{k+1}(\Omega)}, \|B\|_{H^{k+1}(\Omega)}\lesssim_M \epsilon^{-\frac{3}{2}}$, we have $(B_{\epsilon},v_{\epsilon})=(B,v)+\mathcal{O}_{H^k(\Omega)}(\epsilon)$. The energy monotonicity bound follows in a straightforward fashion. 
\medskip

We are left to establish the bound $\|\mathcal{N}W^\pm_{\epsilon}\cdot n_\Gamma\|_{H^{k-\frac{1}{2}}(\Gamma)}\lesssim_M\epsilon^{-1}$. It suffices to show that
\begin{equation*}
\|\mathcal{N}(W^\pm_{\epsilon}-W^\pm)\cdot n_\Gamma\|_{H^{k-\frac{1}{2}}(\Gamma)}\lesssim_M 1.
\end{equation*}
For this, we use the fact that $W_{\epsilon}^\pm-W^\pm$ is rotational (i.e.~tangent to $\Gamma$) and the Leibniz rule for $\mathcal{N}$ to estimate
\begin{equation*}
\|\mathcal{N}(W^\pm_{\epsilon}-W^\pm)\cdot n_\Gamma\|_{H^{k-\frac{1}{2}}(\Gamma)}\lesssim_M \|W_{\epsilon}^\pm-W^\pm\|_{H^{k}(\Omega)}+\|W_{\epsilon}^\pm-W^\pm\|_{H^{k-3}(\Omega)}\|\Gamma\|_{H^{k+\frac{3}{2}}}\lesssim_M 1,
\end{equation*}
as required.
\end{proof}
\subsection{Step 3: Regularization in the direction of the magnetic field}
In the transport step of our argument we will need to estimate error terms involving our good variables differentiated in the direction of $B$. For such error terms to be treated perturbatively, we will need to regularize, in a suitable sense, the velocity and magnetic fields from the previous step in the direction of $B$. We will seek  a bound morally (but not exactly) of the form
\begin{equation*}
\|\nabla_{B_{\epsilon}}W^\pm_{\epsilon}\|_{\mathbf{H}^k(\Omega)}\lesssim_M \epsilon^{-\frac{1}{2}}\|W^\pm\|_{\mathbf{H}^k(\Omega)}.
\end{equation*}
Note that the factor of $\epsilon^{-\frac{1}{2}}$ is consistent with the fact that $\nabla_{B_{\epsilon}}$ scales like a $\frac{1}{2}$-derivative when applied to solutions to the free boundary MHD equations. In an ideal world, one might try to regularize a function $u$ in this way by solving the equation
\begin{equation*}
u_{\epsilon}-\epsilon\nabla_{B_\epsilon}^2u_{\epsilon}=u.
\end{equation*}
However, two technical issues arise when trying to execute this approach. The first is that the error between $u_{\epsilon}$ and $u$ in weaker topologies (such as $C^3$) will only be of size $\epsilon$. The second is that, in practice, the magnetic field will have limited regularity, which makes propagating high regularity bounds for $u_{\epsilon}$ quite challenging. To address the first issue, we will replace the $-\epsilon\nabla_{B_{\epsilon}}^2$ term in the regularizing equation by $\epsilon^2\nabla_{B_{\epsilon}}^4$. This will still give us the desired regularization bounds, but will also ensure that the regularized data has an error of size at most $\mathcal{O}(\epsilon^2)$ in $C^3$ if $k$ is large enough. To address the second issue, we will opt to regularize only a high-frequency band of the function $u$. More precisely, we will decompose $u:=u^l+u^h$ into low and high-frequency parts given by
\begin{equation*}
u^l:=\Phi_{\leq\epsilon^{-\frac{1}{8}}}u,\hspace{5mm} u^h:=u-u_l.
\end{equation*}
We will then define the regularized data $u_{\epsilon}$ via the  equation $u_{\epsilon}:=u^l+u_{\epsilon}^h$ where
\begin{equation*}
u_{\epsilon}^h+\epsilon^2\nabla_{B_\epsilon}^4u_{\epsilon}^h=u^h.
\end{equation*}
If $B$ and $u$ are regular enough, (say, in $H^k$ with norm of size $M$) this definition ensures that in sufficiently weak topologies (relative to $H^k$) the function $u_{\epsilon}^h$ will be rapidly decaying in the parameter $\epsilon$. This will allow us to crudely interpret the variable coefficient term $\epsilon^2\nabla_{B_\epsilon}^4u_{\epsilon}^h$ in a paradifferential fashion as a low-high paraproduct. For instance, when  estimating the part of this expression (in $H^k$, say) where the magnetic field is differentiated a large number of times relative to $u_{\epsilon}^h$,  we will be able to use  (if $k$ is large enough) that $u_{\epsilon}^h$ gains a certain number of factors of $\epsilon$ in a weaker topology to compensate for the imbalance. The following proposition outlines the general construction that we will need.
\begin{proposition}\label{magfieldreg}
Let $\epsilon>0$ be sufficiently small and let $\Gamma$ be parabolically regularized at scale $\epsilon^{-1}$ as in \eqref{surfbound-par}. Let $u$ be a smooth function defined on $\Omega$ and let $X$ be a smooth rotational vector field on $\Omega$ with $\|X\|_{H^{k+\alpha}(\Omega)}\lesssim_M \epsilon^{-3\alpha}$ for $0\leq\alpha\leq 2$ (for instance, $B_{\epsilon}$ as defined in \Cref{mildreg}). For each such $X$, we define the elliptic-type operator $\mathcal{L}_X:=Id+\epsilon^2\nabla_X^4$.
\begin{enumerate}
\item There exists a unique smooth solution $u_{\epsilon}:=\mathcal{L}^{-1}_Xu$ to the equation $\mathcal{L}_{X}u_{\epsilon}=u$ which satisfies the energy estimate
\begin{equation}\label{genrotregbound}
\|u_{\epsilon}\|_{H^{s}(\Omega)}^2+\sum_{2\leq j\leq 4}\epsilon^j\|\nabla_{X}^ju_{\epsilon}\|_{H^s(\Omega)}^2\leq (1+C(M)\epsilon)\|u\|_{H^s(\Omega)}^2,\hspace{5mm}0\leq s\leq k.
\end{equation}
Moreover, if $k$ is large enough, then the high-frequency regularization $u_{\epsilon}^h:=\mathcal{L}_X^{-1}u^h$ satisfies the higher regularity bound
\begin{equation}\label{cruderrotregbound}
\sum_{0\leq j\leq 4}\epsilon^{j}\|\nabla_{X}^ju_{\epsilon}^h\|_{H^{s}(\Omega)}^2\lesssim_M \|u\|_{H^{s}(\Omega)}^2,\hspace{5mm}0\leq s\leq k+2.
\end{equation}
\item Let $u$ and $Z$ be smooth vector fields satisfying $\|u\|_{H^k(\Omega)}+\|Z\|_{H^k(\Omega)}\lesssim_M 1$. There exists a unique, smooth vector field $\tilde{X}_{\epsilon}$ which satisfies the nonlinear equation
\begin{equation*}
\mathcal{L}_{Y_{\epsilon}}\tilde{X}_{\epsilon}=u^h
\end{equation*}
where $Y_{\epsilon}:=(\tilde{X}_{\epsilon}+Z)^{rot}$.
\end{enumerate}
\end{proposition} 
The first and second parts of \Cref{magfieldreg} will be used to construct the regularized velocity and magnetic fields, respectively.
\begin{proof}
We begin with (i). To establish \eqref{genrotregbound}, it suffices to prove the claim for integers $0\leq s=m\leq k$. The general bound follows from interpolation, using the map from $H^m(\Omega)\to (H^m(\Omega))^4$ given by 
\begin{equation*}
u\mapsto (u_{\epsilon},\epsilon\nabla_X^2u_{\epsilon},\epsilon^{\frac{3}{2}}\nabla_X^3u_{\epsilon},\epsilon^{2}\nabla_X^4u_{\epsilon}),
\end{equation*}
with $l^2$ norm on the product space.  Moreover, given the estimate \eqref{genrotregbound}, existence and uniqueness follows from a standard duality argument since the adjoint equation and the original equation are the same (as the assumptions on $X$ ensure that $\nabla_X$ is skew-adjoint). Now, we turn to the energy estimate. To simplify notation, we use $\partial^m$  to denote a differential operator of the form $\partial^{\alpha}$ where $\alpha$ is a multi-index of order $m$. Our starting point is the identity
\begin{equation}\label{L2id}
\begin{split}
\|\partial^mu_{\epsilon}\|_{L^2(\Omega)}^2&=\|\partial^mu\|_{L^2(\Omega)}^2+2\langle \partial^mu_{\epsilon},\partial^m(u_{\epsilon}-u)\rangle-\|\partial^m(u_{\epsilon}-u)\|_{L^2(\Omega)}^2
\\
&=\|\partial^mu\|_{L^2(\Omega)}^2-2\epsilon^2\langle \partial^m u_{\epsilon},\partial^m\nabla_{X}^4u_{\epsilon}\rangle-\epsilon^4\|\partial^m\nabla_{X}^4u_{\epsilon}\|_{L^2(\Omega)}^2.
\end{split}
\end{equation}
Next, we observe that since $0\leq m\leq k$, we have the bound
\begin{equation*}
\|[\nabla_{X},\partial^m]g\|_{L^2(\Omega)}\leq C(M)\|g\|_{H^m(\Omega)}
\end{equation*}
for every $g\in H^m(\Omega)$. Therefore, since $\nabla_{X}$ is skew-adjoint, we have by a simple application of Cauchy-Schwarz,
\begin{equation*}\label{rotrege1}
-2\epsilon^2\langle \partial^m u_{\epsilon},\partial^m\nabla_{X}^4u_{\epsilon}\rangle\leq 2\epsilon^2\langle \partial^m \nabla_{X}u_{\epsilon},\partial^m\nabla_{X}^3u_{\epsilon}\rangle+c\epsilon^3\|\nabla_{X}^3u_{\epsilon}\|_{H^m(\Omega)}^2+C(M)\epsilon\|u_{\epsilon}\|_{H^m(\Omega)}^2,
\end{equation*}
for some $0<c\ll 1$ sufficiently small. Integrating by parts and applying Cauchy Schwarz again we see that
\begin{equation*}
2\epsilon^2\langle \partial^m \nabla_{X}u_{\epsilon},\partial^m\nabla_{X}^3u_{\epsilon}\rangle\leq -\frac{3}{2}\epsilon^2\|\partial^m\nabla_{X}^2u_{\epsilon}\|_{L^2(\Omega)}^2+C(M)\epsilon^2(\|\nabla_{X}u_{\epsilon}\|_{H^m(\Omega)}^2+\|u_{\epsilon}\|^2_{H^{m}(\Omega)}).
\end{equation*}
Moreover, we have
\begin{equation*}\label{bint1}
\epsilon^2\|\partial^m\nabla_{X} u_{\epsilon}\|_{L^2(\Omega)}^2\lesssim_M \epsilon^3\|\nabla_{X}^2u_{\epsilon}\|_{H^m(\Omega)}^2+\epsilon\|u_{\epsilon}\|_{H^m(\Omega)}^2
\end{equation*}
and
\begin{equation*}
\epsilon^3\|\partial^m\nabla_{X}^3u_{\epsilon}\|_{L^2(\Omega)}^2\leq \frac{1}{2}\epsilon^4\|\nabla_{X}^4u_{\epsilon}\|_{H^m(\Omega)}^2+\frac{1}{2}\epsilon^2\|\nabla_{X}^2u_{\epsilon}\|_{H^m(\Omega)}^2+C(M)\epsilon^3\|\nabla_{X}^2u_{\epsilon}\|^2_{H^m(\Omega)}+C(M)\epsilon\|u_{\epsilon}\|_{H^{m}(\Omega)}^2.
\end{equation*}
Combining the above estimates completes the proof of \eqref{genrotregbound}. To obtain \eqref{cruderrotregbound}, we proceed in the same fashion as above, except now we rely on the fact that $u_{\epsilon}^h$ gains factors of $\epsilon$ in weaker topologies to estimate the commutators that appear in the above argument. As an example, when carrying out the estimate in $H^{k+2}(\Omega)$, we  have to estimate the commutator $[\partial^{k+2},\nabla_X]u_{\epsilon}^h$ in $L^2(\Omega)$. By Sobolev product estimates, we have
\begin{equation*}
\|[\nabla_X,\partial^{k+2}]u_{\epsilon}^h\|_{L^2(\Omega)}\lesssim_M \|X\|_{H^{k+2}(\Omega)}\|u_{\epsilon}^h\|_{H^{k-100}(\Omega)}+\|u_{\epsilon}^h\|_{H^{k+2}(\Omega)}.
\end{equation*}
Using the bounds $\|X\|_{H^{k+2}(\Omega)}\lesssim_M\epsilon^{-6}$ and $\|u_{\epsilon}^h\|_{H^{k-100}(\Omega)}\lesssim_M \epsilon^6\|u\|_{H^{k+2}(\Omega)}$ (the latter of which follows from \eqref{genrotregbound}) we  have
\begin{equation*}
\|[\nabla_X,\partial^{k+2}]u_{\epsilon}^h\|_{L^2(\Omega)}\lesssim_M \|u\|_{H^{k+2}(\Omega)}+\|u_{\epsilon}^h\|_{H^{k+2}(\Omega)}.
\end{equation*}
Using this strategy and following a similar line of reasoning as in \eqref{genrotregbound} we obtain the bound
\begin{equation*}
\|u_{\epsilon}^h\|_{H^{k+2}(\Omega)}^2+\sum_{2\leq j\leq 4}\epsilon^j\|\nabla_{X}^ju_{\epsilon}^h\|_{H^{k+2}(\Omega)}^2\lesssim_M \|u\|_{H^{k+2}(\Omega)}^2.
\end{equation*}
Using that $\nabla_{X}$ is skew-adjoint and that $u_{\epsilon}^h$ is at high-frequency, we can interpolate between the bounds for $\nabla_{X}^2u_{\epsilon}^h$ and $u_{\epsilon}^h$ to obtain
\begin{equation*}
\sum_{0\leq j\leq 4}\epsilon^j\|\nabla_{X}^ju_{\epsilon}^h\|_{H^{k+2}(\Omega)}^2\lesssim_M \|u\|_{H^{k+2}(\Omega)}^2.
\end{equation*}
Our desired bound follows from carrying out this procedure for each integer $0\leq m\leq k+2$ and then interpolating. We omit the straightforward modifications.
\medskip

Now, we establish (ii). We define the initialization $\tilde{X}^0:=u^h$, $Y^0:=(u^h+Z)^{rot}$. For each non-negative integer $n\geq 0$, we define inductively $X^{n+1}:=(\tilde{X}^{n+1})^{rot}$ and $Y^{n+1}:=X^{n+1}+Z^{rot}$ where $\tilde{X}^{n+1}$ is the unique $H^k$ solution to the equation
\begin{equation*}
\mathcal{L}_{Y^{n}}\tilde{X}^{n+1}=u^h.
\end{equation*}
Such a function is well-defined under the inductive hypotheses (where we use the convention $X^{-1}=0$ and $Y^{-1}=0$) and we have
\begin{equation}\label{iterationschemeinduction}
\|X^n\|_{H^k(\Omega)}\leq C_b(M),\hspace{5mm} \|\tilde{X}^n\|_{H^k(\Omega)}^2+\sum_{2\leq j\leq 4}\epsilon^j\|\nabla_{Y^{n-1}}^{j}\tilde{X}^n\|_{H^k(\Omega)}^2\leq C_b'(M)\|u\|_{H^k(\Omega)}^2,
\end{equation}
where $C_b(M)\gg C_b'(M)$ is some sufficiently large (but fixed) bootstrap parameter to be chosen. We aim to show that $\tilde{X}^n$ converges in $H^s(\Omega)$ for every $0\leq s<k$ with limit $\tilde{X}_{\epsilon}$ in $H^k(\Omega)$. By interpolation, it suffices to establish \eqref{iterationschemeinduction} for $X^{n+1}$ and $\tilde{X}^{n+1}$, respectively, as well as the weak Lipschitz type bound
\begin{equation}\label{iterationdifferencebound}
\sum_{0\leq j\leq 4}\epsilon^j\|\nabla_{Y^n}^j(\tilde{X}^{n+1}-\tilde{X}^n)\|_{H^2(\Omega)}^2\leq \frac{1}{2}\sum_{0\leq j\leq 4}\epsilon^j\|\nabla_{Y^{n-1}}^j(\tilde{X}^{n}-\tilde{X}^{n-1})\|_{H^2(\Omega)}^2,\hspace{5mm}n\geq 1.
\end{equation}
\begin{remark}
We perform the above estimate  in $H^2(\Omega)$ rather than $L^2(\Omega)$ or $H^1(\Omega)$ as it will be simpler from a technical standpoint to perform some of the elliptic estimates that will manifest in our analysis below (in particular, we will avoid dealing with any negative regularity Sobolev spaces).
\end{remark}
We begin with \eqref{iterationschemeinduction}. We observe that the bound for $\tilde{X}^{n+1}$ follows immediately from (i) if $\epsilon$ is sufficiently small and $C_b'(M)$ is large enough relative to the $H^k\to H^k$ norm of $\Phi_{\leq\epsilon^{-\frac{1}{8}}}$. We now focus on  proving the $H^k(\Omega)$ bound for $X^{n+1}$. We remark that in this part of the argument, implicit constants may depend on the $H^k(\Omega)$ bound for $u$ and the constants coming from elliptic regularity estimates, but will not depend on the constant $C_b(M)$ above. 
\medskip

We observe that  the above definitions imply that
\begin{equation*}
\tilde{X}^{n+1}\cdot n_\Gamma=-\epsilon^2\nabla_{Y^n}^4\tilde{X}^{n+1}\cdot n_\Gamma+u^h\cdot n_\Gamma.
\end{equation*}
Since $u^h$ is at high-frequency, we have $\|u^h\cdot n_\Gamma\|_{H^{k-\frac{1}{2}}(\Gamma)}\lesssim_M 1$. Moreover, using  $\|\Gamma\|_{H^{k+\frac{1}{2}}}\lesssim_M\epsilon^{-\frac{1}{2}}$ and  the $H^k(\Omega)$ bound for $\nabla_{Y^n}^4\tilde{X}^{n+1}$, we conclude that
\begin{equation*}
\|\tilde{X}^{n+1}\cdot n_\Gamma\|_{H^{k-\frac{1}{2}}(\Gamma)}\lesssim_M 1+\epsilon^2\|\nabla_{Y^n}^4\tilde{X}^{n+1}\|_{H^k(\Omega)}\lesssim_M 1.
\end{equation*}
A similar computation using that $u^h=\mathcal{O}_{H^{k-10}(\Omega)}(\epsilon)$ yields the weak bound
\begin{equation*}
\|\nabla\cdot \tilde{X}^{n+1}\|_{H^{k-20}(\Omega)}\lesssim_M \epsilon.
\end{equation*}
Consequently, the balanced elliptic estimates from \Cref{BEE} imply that $(\tilde{X}^{n+1})^{ir}:=\nabla\mathcal{H}\mathcal{N}^{-1}((\tilde{X}^{n+1}-\nabla\Delta^{-1}\nabla\cdot\tilde{X}^{n+1})\cdot n_\Gamma)$ satisfies the bound
\begin{equation*}
\|(\tilde{X}^{n+1})^{ir}\|_{H^k(\Omega)}\lesssim_M 1.
\end{equation*}
If $C_b(M)$ in the bootstrap hypothesis is initially chosen to be large enough (to control all of the implicit constants in the elliptic and trace type estimates above) we obtain, say,
\begin{equation*}
\|(\tilde{X}^{n+1})^{ir}\|_{H^k(\Omega)}^2\leq\frac{1}{2}C_b(M).
\end{equation*}
Combining this with the bound for $\tilde{X}^{n+1}$ and taking $C_b(M)$ large enough closes the bootstrap for $X^{n+1}$. It therefore remains to establish \eqref{iterationdifferencebound}. 
Expanding $\mathcal{L}_{Y^n}(\tilde{X}^{n+1}-\tilde{X}^n)$ gives the identity
\begin{equation*}
\begin{split}
\mathcal{L}_{Y^{n}}(\tilde{X}^{n+1}-\tilde{X}^n)=\epsilon^2(\nabla_{Y^n}^4-\nabla_{Y^{n-1}}^4)\tilde{X}^{n}=\sum_{j=0}^3\epsilon^2\nabla_{Y^{n-1}}^j\nabla_{X^n-X^{n-1}}\nabla_{Y^{n}}^{3-j}\tilde{X}^n.
\end{split}
\end{equation*}
Applying the estimate from (i) and interpolating to control $\epsilon\|\nabla_{X^n}(\tilde{X}^{n+1}-\tilde{X}^n)\|_{H^2(\Omega)}^2$ we obtain the bound
\begin{equation*}\label{diffbounditeration}
\sum_{0\leq j\leq 4}\epsilon^j\|\nabla_{Y^n}(\tilde{X}^{n+1}-\tilde{X}^n)\|_{H^2(\Omega)}^2\lesssim_M \epsilon^4\sup_{0\leq j\leq 3}\|\nabla_{Y^{n-1}}^j(X^{n}-X^{n-1})\|_{H^2(\Omega)}^2.
\end{equation*}
If $\epsilon>0$ is small enough (relative to $M$) we will be able to conclude the desired estimate if we can show that for each $0\leq j\leq 3$,
\begin{equation}\label{diffbounditeration2}
\|\nabla_{Y^{n-1}}^j(X^{n}-X^{n-1})\|_{H^2(\Omega)} \lesssim_M\sup_{0\leq j\leq 3}\|\nabla_{Y^{n-1}}^j(\tilde{X}^{n}-\tilde{X}^{n-1})\|_{H^2(\Omega)}=:\mathcal{A}_n.
\end{equation}
If $j=0$, this follows in a straightforward manner from the estimates in \Cref{BEE}. If $j\geq 1$, we will rely on a div-curl estimate. We observe that since $X^{n+1}-X^n$ is divergence-free and tangent to $\Gamma$, we have the reduction estimate
\begin{equation*}
\|\nabla\cdot \nabla_{Y^{n-1}}^j(X^n-X^{n-1})\|_{H^1(\Omega)}
\\
+\|\nabla_{Y^{n-1}}^j(X^{n}-X^{n-1})\cdot n_\Gamma\|_{H^{\frac{3}{2}}(\Gamma)}\lesssim_M \sup_{0\leq l\leq j-1}\|\nabla_{Y^{n-1}}^l(X^n-X^{n-1})\|_{H^2(\Omega)}.
\end{equation*}
Moreover, since $\nabla\times (X^{n}-X^{n-1})=\nabla\times (\tilde{X}^{n}-\tilde{X}^{n-1})$ we have 
\begin{equation*}
\|\nabla\times (X^{n}-X^{n-1})\|_{H^1(\Omega)}\lesssim_M \mathcal{A}_n+\sup_{0\leq l\leq j-1}\|\nabla_{Y^{n-1}}^l(X^n-X^{n-1})\|_{H^2(\Omega)}.
\end{equation*}
Therefore, by \Cref{Balanced div-curl} and the estimate for $j=0$, we obtain \eqref{diffbounditeration2}.
\end{proof}
\medskip

\textbf{Energy monotonicity and the regularized variables}. Let $(v,B,\Gamma)$ be a state satisfying the bounds in the conclusion of \Cref{mildreg}. Our objective here will be to regularize $v$ as well as the curl of $B$ in the direction of the magnetic field on the fixed, regularized domain $\Omega$, while retaining the regularization bounds for $(v, B,\Gamma)$. To do this, we define the regularizations $v_{\epsilon}$ and $B_{\epsilon}$ using \Cref{magfieldreg}:
\begin{equation*}\label{rotregdef}
B_{\epsilon}:=(B^l+\mathcal{L}^{-1}_{B_{\epsilon}}B^h)^{rot},\hspace{5mm}v_{\epsilon}:=(v^l+\mathcal{L}^{-1}_{B_{\epsilon}}v^h)^{div}=:w_{\epsilon}^{div},
\end{equation*}
where the superscript $div$ denotes projection onto divergence-free functions and, once again, we  generically write $u^l:=\Phi_{\leq \epsilon^{-\frac{1}{8}}}u$ and $u^h:=u-u^l$. We  set the convention that by $u_{\epsilon}^h$ we mean $\mathcal{L}_{B_{\epsilon}}^{-1}u^h$; we define the corresponding rotational good variables  $\omega^\pm_{\epsilon}:=\nabla\times W_{\epsilon}^\pm$ as well as the variables $a_{\epsilon}$, 
 $D_t^\pm a_\epsilon$, $\mathcal{G}^\pm_{\epsilon}$ and so-forth in the usual manner. We also use $\zeta_{\epsilon}$ as a shorthand for $\nabla\times W_{\epsilon}^{h,\pm}$. With the above notation and conventions, we have $\omega_{\epsilon}^\pm=\nabla\times W^l+\zeta_{\epsilon}$. We may think of $\zeta_{\epsilon}$ as the high-frequency part of $\omega_{\epsilon}$. We  define the uncorrected magnetic field and the corresponding uncorrected variables $\tilde{W}^\pm_{\epsilon}$ by
 \begin{equation*}
\tilde{B}_{\epsilon}:=B^l+\mathcal{L}^{-1}_{B_{\epsilon}}B^h,\hspace{5mm}\tilde{W}^\pm_{\epsilon}:=v_{\epsilon}\pm\tilde{B}_{\epsilon}.
 \end{equation*}
 \begin{remark}
 Although  $\tilde{B}_{\epsilon}$ will have good regularization bounds in the direction of the magnetic field $B_{\epsilon}$,  $B_{\epsilon}$ itself will not, due to the dependence of the rotational projection on the regularity of the free surface.  This is why  we carefully distinguish between the above variables. Importantly,  since $\nabla\times B_{\epsilon}=\nabla\times\tilde{B}_{\epsilon}$, the curl of $B_{\epsilon}$ will inherit the improved regularity of $\tilde{B}_{\epsilon}$. 
\end{remark}

 For the remainder of this section, we take $\mathcal{L}^{-1}$ to mean $\mathcal{L}^{-1}_{B_{\epsilon}}$ and also drop the $\pm$ superscripts from expressions involving the variables $W^\pm$. The following proposition will help us to understand the affect of the above regularizations on the energy as well as to  quantify the errors incurred in weaker topologies.
\begin{proposition}\label{rotregEmon} 
 Given $(v,B,\Gamma)\in\mathbf{H}^k$ from above, we have the following estimates:
 \begin{enumerate}
 \item (Good pointwise approximation).
\begin{equation*}\label{approxdomain}
(v_{\epsilon},B_{\epsilon})=(v_{\epsilon},\tilde{B}_{\epsilon})+\mathcal{O}_{C^3(\Omega)}(\epsilon^2)=(v,B)+\mathcal{O}_{C^3(\Omega)}(\epsilon^2).
\end{equation*}
\item (Regularization bounds). Let $0\leq s\leq k$. For the uncorrected variables $\tilde{W}^\pm_{\epsilon}$, there holds
\begin{equation}\label{gradbregbounds}
\|\nabla_{B_{\epsilon}}\tilde{W}_{\epsilon}^\pm\|_{\mathbf{H}^s(\Omega)}\lesssim_M \epsilon^{-\frac{1}{2}}\|W^\pm\|_{\mathbf{H}^s(\Omega)}.
\end{equation}
Moreover, we have the regularization bounds (the first and third being retained from the previous step)
\begin{equation}\label{retainedbounds}
\|W_{\epsilon}^\pm\|_{H^{k+1+\alpha}(\Omega)}\lesssim_{M,\alpha}\epsilon^{-\frac{3}{2}-3\alpha},\hspace{5mm}\|\nabla_{B_{\epsilon}}\tilde{W}^\pm_{\epsilon}\|_{H^{k+1}(\Omega)}\lesssim_M\epsilon^{-2},\hspace{5mm}\|\mathcal{N}W^\pm_{\epsilon}\cdot n_\Gamma\|_{H^{k-\frac{1}{2}}(\Gamma)}\lesssim_M \epsilon^{-1},
\end{equation}
where $0\leq\alpha\leq 1$.
\item (Rotational energy bounds I). Let $\epsilon>0$ be sufficiently small, $\delta\in [0,\epsilon]$ and  $1\leq s\leq k$. There is a universal constant $C>0$ such that the following estimate holds:
\begin{equation}\label{genrotregbound1}
\|\omega_{\epsilon}\pm\delta\nabla_{B_{\epsilon}}\omega_{\epsilon}\|_{H^{s}(\Omega)}^2+C\sum_{2\leq j\leq 4}\|\nabla_{B_{\epsilon}}^j\zeta_{\epsilon}\|_{H^s(\Omega)}^2\leq (1+C(M)\epsilon)\|\omega\|_{H^{s}(\Omega)}^2+C(M)\epsilon\|W\|_{H^{s+1}(\Omega)}^2.
\end{equation}
\item (Rotational energy bounds II). Let $\epsilon>0$ be sufficiently small, $\delta\in [0,\epsilon]$ and  $1\leq s\leq k-1$. The following estimate holds:
\begin{equation}\label{genrotregbound2}
\begin{split}
\|\nabla_{B_{\epsilon}}\omega_{\epsilon}\pm\delta\nabla_{B_{\epsilon}}^2\omega_{\epsilon}\|_{H^{s}(\Omega)}^2+C\sum_{2\leq j\leq 4}\|\nabla_{B_{\epsilon}}^{j+1}\zeta_{\epsilon}\|_{H^s(\Omega)}^2&\leq (1+C(M)\epsilon)\|\nabla_{B}\omega\|_{H^{s}(\Omega)}^2+C(M)\epsilon\|W\|_{\mathbf{H}^{s+\frac{3}{2}}(\Omega)}^2.
\end{split}
\end{equation}
\item (Energy monotonicity).
\begin{equation*}\label{Emonsurf2}
E^{k}(v_{\epsilon},B_{\epsilon},\Gamma)\leq E^{k}(v, B ,\Gamma)+C(M)\epsilon.
\end{equation*}
 \end{enumerate}
\end{proposition}
\begin{proof}
If $k$ is large enough, property (i) is clear from Sobolev embeddings and the definition of $B_{\epsilon}$ and $v_{\epsilon}$. Next, we move to property (ii). Thanks to the first part of \Cref{magfieldreg} and the fact that $\mathcal{L}^{-1}$ commutes with $\nabla_{B_{\epsilon}}$, we have 
\begin{equation*}
\|\nabla_{B_{\epsilon}}\tilde{B}_{\epsilon}\|_{\mathbf{H}^s(\Omega)}\lesssim_M\epsilon^{-\frac{1}{2}}\|(B,\nabla_{B_{\epsilon}}B)\|_{H^s(\Omega)\times H^{s-\frac{1}{2}}(\Omega)}\lesssim_M \epsilon^{-\frac{1}{2}}\|B\|_{\mathbf{H}^s(\Omega)}
\end{equation*}
where we used product estimates, the high regularity bound $\|B\|_{H^{k+\frac{1}{2}}(\Omega)}\lesssim_M \epsilon^{-\frac{3}{4}}$ and the error bound $\|B_{\epsilon}-B\|_{H^{k-4}(\Omega)}\lesssim_M \epsilon^2$ to replace $\nabla_{B_{\epsilon}}$ with $\nabla_B$ in the last inequality.  We  also have, by definition of $v_{\epsilon}$,
\begin{equation*}
\|\nabla_{B_{\epsilon}}v_{\epsilon}\|_{\mathbf{H}^s(\Omega)}\lesssim_M \epsilon^{-\frac{1}{2}}\|v\|_{\mathbf{H}^s(\Omega)}+\|\nabla_{B_{\epsilon}}\nabla\Delta^{-1}\nabla\cdot (v^l+\mathcal{L}^{-1}v^h)\|_{\mathbf{H}^s(\Omega)}.
\end{equation*}
Since $v$ is divergence-free, it is easy to see that $\nabla\cdot (v^l+\mathcal{L}^{-1}v^h)=\mathcal{O}_{H^{k-100}(\Omega)}(\epsilon^{10})$, where the choice of topology here is somewhat arbitrary. Combining this bound with  the commutator identities in \Cref{Movingsurfid}, the bound $\|\Gamma\|_{H^{k+\frac{1}{2}}}\lesssim_M\epsilon^{-\frac{1}{2}}$, the balanced elliptic estimates in \Cref{BEE} and the surface regularization bounds, we have
\begin{equation*}
\|\nabla_{B_{\epsilon}}v_{\epsilon}\|_{\mathbf{H}^s(\Omega)}\lesssim_M \epsilon^{-\frac{1}{2}}\|v\|_{\mathbf{H}^s(\Omega)},
\end{equation*}
which yields \eqref{gradbregbounds}. Now, we prove the regularization bounds in \eqref{retainedbounds}. First, from the elliptic estimates in \Cref{BEE} and the regularization bounds for $W^l$ as well as \eqref{cruderrotregbound}, we have
\begin{equation*}
\|W_{\epsilon}\|_{H^{k+1+\alpha}(\Omega)}\lesssim_M \|\Gamma\|_{H^{k+\frac{3}{2}+\alpha}}+\|W^l+W_{\epsilon}^h\|_{H^{k+1+\alpha}(\Omega)}\lesssim_M \epsilon^{-\frac{3}{2}-\alpha}+\|W_{\epsilon}^h\|_{H^{k+1+\alpha}(\Omega)}\lesssim_M\epsilon^{-\frac{3}{2}-3\alpha}
\end{equation*}
for $0\leq\alpha\leq 1$. To prove the second bound, one proceeds in an almost identical fashion to \eqref{gradbregbounds} except we use the second part of \Cref{magfieldreg} and in carrying out the various commutator and elliptic estimates, we make use of the bounds $\|\Gamma\|_{H^{k+\frac{3}{2}}}\lesssim_M\epsilon^{-\frac{3}{2}}$, $\|B_{\epsilon}\|_{H^{k+2}(\Omega)}\lesssim_M\epsilon^{-\frac{9}{2}}$ and $\|B_{\epsilon}\|_{H^{k+1}(\Omega)}\lesssim_M\epsilon^{-\frac{3}{2}}$ instead of the weaker bounds used in \eqref{gradbregbounds}.
\medskip

Now, we move to the other, more difficult, bound in \eqref{retainedbounds}. We first note that since $B_{\epsilon}$ is tangent to $\Gamma$, we obtain easily from the Leibniz rule for $\mathcal{N}$, the regularization bounds for $\Gamma$ and the weak bound $\|B_{\epsilon}-B\|_{H^{k-5}(\Omega)}\lesssim_M\epsilon^2$,
\begin{equation*}
\|\mathcal{N}(B_{\epsilon}-B)\cdot n_\Gamma\|_{H^{k-\frac{1}{2}}(\Gamma)}\lesssim_M 1+\|B_{\epsilon}-B\|_{H^k(\Omega)}\lesssim_M 1.
\end{equation*}
Therefore, it suffices to show that
\begin{equation*}\label{virrotbound}
\|\mathcal{N}v_{\epsilon}\cdot n_\Gamma\|_{H^{k-\frac{1}{2}}(\Gamma)}\lesssim_M \epsilon^{-1}.
\end{equation*}
Using the surface regularization bounds, the above will follow if we can establish that
\begin{equation}\label{twowbounds}
\|w_{\epsilon}-v_{\epsilon}\|_{H^{k+1}(\Omega)}\lesssim_M\epsilon^{-1},\hspace{5mm}\|\mathcal{N}w_{\epsilon}\cdot n_\Gamma\|_{H^{k-\frac{1}{2}}(\Gamma)}\lesssim_M \epsilon^{-1}.
\end{equation}
For the first bound, we observe from the balanced elliptic estimates and the  weak $\mathcal{O}_{H^{k-5}(\Omega)}(\epsilon^2)$ error bound for $v_{\epsilon}-w_{\epsilon}$ that
\begin{equation*}
\|w_{\epsilon}-v_{\epsilon}\|_{H^{k+1}(\Omega)}\lesssim_M 1+\|\nabla\cdot w_{\epsilon}\|_{H^{k}(\Omega)}\lesssim_M\epsilon^{-1}+\|\nabla\cdot v_{\epsilon}^h\|_{H^k(\Omega)}.
\end{equation*}
To estimate $\nabla\cdot v_{\epsilon}^h$, we observe that since $v$ is divergence-free, we have the equation
\begin{equation*}
\nabla\cdot v_{\epsilon}^h+\epsilon^2\nabla_{B_{\epsilon}}^4(\nabla\cdot v_{\epsilon}^h)=-[\nabla\cdot,\Phi_{\leq\epsilon^{-\frac{1}{8}}}]v-\epsilon^2[\nabla\cdot, \nabla_{B_{\epsilon}}^4]v_{\epsilon}^h.
\end{equation*}
The first commutator on the right-hand side is essentially localized at frequency $\approx \epsilon^{-\frac{1}{8}}$. Moreover, as $v_{\epsilon}^h$ is at high frequency, we can estimate, by expanding the second commutator,
\begin{equation*}
\epsilon^2\|[\nabla\cdot,\nabla_{B_{\epsilon}}^4]v_{\epsilon}^h\|_{H^{k}(\Omega)}\lesssim_M 1+\sup_{0\leq j\leq 3}\epsilon^2\|\nabla_{B_{\epsilon}}^jv_{\epsilon}^h\|_{H^k(\Omega)}\lesssim_M 1+\epsilon^{\frac{1}{2}}\|v\|_{H^{k+1}(\Omega)}\lesssim_M\epsilon^{-1}.
\end{equation*}
Consequently, we have
\begin{equation*}
\|\nabla\cdot v_{\epsilon}^h\|_{H^{k}(\Omega)}\lesssim_M\epsilon^{-1}.
\end{equation*}
This establishes the first bound in \eqref{twowbounds}. To establish the second bound, we observe that it suffices to prove the bound with $w_{\epsilon}$ replaced by $v_{\epsilon}^h$ in light of the regularization bounds for $\Gamma$ and $v^l$. Arguing similarly to the above and using the commutator identities in \Cref{Movingsurfid} we can write
\begin{equation*}
\mathcal{N}v_{\epsilon}^h\cdot n_\Gamma+\epsilon^2\nabla_{B_{\epsilon}}^4(\mathcal{N}v_{\epsilon}^h\cdot n_\Gamma)=\mathcal{N}v^h\cdot n_\Gamma+R_{\epsilon}
\end{equation*}
where $R_{\epsilon}$ is an error term satisfying
\begin{equation*}
\|R_{\epsilon}\|_{H^{k-\frac{1}{2}}(\Gamma)}\lesssim_M \epsilon^{-1}.
\end{equation*}
An energy estimate akin to the one used to prove \eqref{genrotregbound} yields
\begin{equation*}
\|\mathcal{N}v_{\epsilon}^h\cdot n_\Gamma\|_{H^{k-\frac{1}{2}}(\Gamma)}\lesssim_M\epsilon^{-1}
\end{equation*}
as desired. 
\medskip 

We now focus on establishing properties (iii) and (iv). We will show the full details for (iii) and then outline the main differences for establishing (iv). Our first aim is to reformulate the estimate (iii) in terms of a bound where we can interpolate between integer-based Sobolev spaces. This is slightly tricky since the right-hand side of the above estimate involves both $\omega$ and $W$ with significantly different weights in the parameter $\epsilon$. Before proceeding, we introduce the following notation for various commutator expressions that will appear in the analysis below:
\begin{equation*}
\mathcal{C}^1:=\epsilon^2\mathcal{L}^{-1}[\nabla_{B_{\epsilon}}^4,\nabla\times]\mathcal{L}^{-1}\Phi_{\geq\epsilon^{-\frac{1}{8}}},\hspace{5mm}\mathcal{C}^2:=\epsilon^2\mathcal{L}^{-1}\nabla_{B_{\epsilon}}^4[\nabla\times,\Phi_{\leq\epsilon^{-\frac{1}{8}}}].
\end{equation*}
We  observe the identities
\begin{equation}\label{vortcommute}
\omega_{\epsilon}=\omega^l+\omega_{\epsilon}^h+\mathcal{C}^1W+\mathcal{C}^2W,\hspace{5mm}\zeta_{\epsilon}=\omega_{\epsilon}^h-[\nabla\times, \Phi_{\leq\epsilon^{-\frac{1}{8}}}]W+\mathcal{C}^1W+\mathcal{C}^2W.
\end{equation}
We also collect the simple bounds 
\begin{equation*}
\epsilon^{\frac{j}{2}}\|\nabla_{B_{\epsilon}}^j\mathcal{C}^1\|_{H^{s+1}(\Omega)\to H^s(\Omega)}\lesssim_M \epsilon^{\frac{1}{2}},\hspace{5mm}\epsilon^{\frac{j}{2}}\|\nabla_{B_{\epsilon}}^j\mathcal{C}^2\|_{H^{s+1}(\Omega)\to H^s(\Omega)}\lesssim_M\epsilon^{\frac{3}{2}},\hspace{5mm}1\leq s\leq k,\hspace{5mm} 0\leq j\leq 4,
\end{equation*}
which use the fact that $W_{\epsilon}^h$ and the commutator in the latter term are ``localized" to ``frequency" essentially $\lesssim_M\epsilon^{-\frac{1}{8}}$ and $\approx \epsilon^{-\frac{1}{8}}$ respectively. For $2\leq j\leq 4$, it is easy to see  from these estimates that we have
\begin{equation*}\label{differentvorticities}
\sum_{2\leq j\leq 4}\epsilon^j\|\nabla_{B_{\epsilon}}^j\zeta_{\epsilon}\|_{H^s(\Omega)}^2\leq 2\sum_{2\leq j\leq 4}\epsilon^j\|\nabla_{B_{\epsilon}}^j\omega_{\epsilon}^h\|_{H^s(\Omega)}^2+C(M)\epsilon\|W\|_{H^{s+1}(\Omega)}^2. 
\end{equation*}
Therefore, it suffices to establish \eqref{genrotregbound2} with $\omega_{\epsilon}^h$ in place of $\zeta_{\epsilon}$ in the second term on the left-hand side.
Motivated by the above, let us define the linear maps given by 
\begin{equation*}
\mathcal{T}:H^{s+1}(\Omega)\to H^{s}(\Omega),\hspace{5mm}f\mapsto c\epsilon^{-\frac{1}{2}}(\mathcal{C}^1+\mathcal{C}^2)f,\hspace{5mm}\mathcal{S}:H^s(\Omega)\to H^s(\Omega),\hspace{5mm}u\mapsto u_{\epsilon}:=u^l+u_{\epsilon}^h.
\end{equation*}
Here, $c>0$ is some small $M$ and $k$ dependent parameter chosen so that for every $1\leq s\leq k$, there holds
\begin{equation}\label{Tbound}
\|\mathcal{T}\|_{H^{s+1}(\Omega)\to H^{s}(\Omega)}\leq\frac{1}{100}.
\end{equation}
By interpolation, it suffices therefore to show the general bound
\begin{equation}\label{interpolationsetup2}
\|(u_{\epsilon}+\mathcal{T}f)\pm\delta\nabla_{B_{\epsilon}}(u_{\epsilon}+\mathcal{T}f)\|_{H^{s}(\Omega)}^2 
+ \frac{1}{2}\sum_{2\leq j\leq 4}\epsilon^j \| \nabla_{B_{\epsilon}}^ju_\epsilon^h\|_{H^{s}(\Omega)}^2 
\leq (1+C\epsilon) (\|u\|_{H^{s}(\Omega)}^2 
+  \|f\|_{H^{s+1}(\Omega)}^2)
\end{equation}
when $s=m$ is a non-negative integer. In light of the above discussion, we remark that the bound \eqref{genrotregbound1} will follow by taking $f=c^{-1}\epsilon^{\frac{1}{2}}W$ and $u=\omega$. We remark further that it suffices to prove \eqref{interpolationsetup2} in the case $\delta=0$ since we have for $w_{\epsilon}:=(u_{\epsilon}+\mathcal{T}f)$,
\begin{equation*}
\|w_{\epsilon}+\delta\nabla_{B_{\epsilon}}w_{\epsilon}\|_{H^m(\Omega)}^2\leq \|w_{\epsilon}\|_{H^m(\Omega)}^2+2\delta \langle w_{\epsilon},\nabla_{B_{\epsilon}}w_{\epsilon}\rangle_{H^m(\Omega)}+\delta^2C(M)(\|\nabla_{B_{\epsilon}}u_{\epsilon}\|_{H^m(\Omega)}^2+\|\nabla_{B_{\epsilon}}\mathcal{T}f\|_{H^m(\Omega)}^2),
\end{equation*}
and the second and third terms on the right can (by skew-adjointness of $\nabla_{B_{\epsilon}}$, the $\delta=0$ case and the above bounds) be estimated by $C(M)\epsilon(\|u\|_{H^m(\Omega)}^2+\|f\|_{H^{m+1}(\Omega)}^2)$.
Now, we turn to \eqref{interpolationsetup2} in the case $\delta=0$. We begin with the expansion
\begin{equation*}\label{L2id2}
\begin{split}
\|w_{\epsilon}\|_{H^m(\Omega)}^2&=\|u\|_{H^m(\Omega)}^2+2\langle w_{\epsilon},w_{\epsilon}-u\rangle_{H^m(\Omega)}-\|w_{\epsilon}-u\|_{H^m(\Omega)}^2.
\end{split}
\end{equation*}
Using the equation for $u_{\epsilon}$, the last term can be expanded as
\begin{equation*}
\begin{split}
-\|w_{\epsilon}-u\|_{H^m(\Omega)}^2=-\epsilon^4\|\nabla_{B_{\epsilon}}^4u^h_{\epsilon}\|_{H^m(\Omega)}^2-\|\mathcal{T}f\|_{H^m(\Omega)}^2+2\epsilon^2\langle\nabla_{B_{\epsilon}}^4u_{\epsilon}^h,\mathcal{T}f\rangle_{H^m(\Omega)}.
\end{split}
\end{equation*}
By Cauchy-Schwarz and \eqref{Tbound}, we obtain
\begin{equation*}\label{interpolationbound1}
-\|w_{\epsilon}-u\|_{H^m(\Omega)}^2\leq -\epsilon^4\frac{1}{2}\|\nabla_{B_{\epsilon}}^4u^h_{\epsilon}\|_{H^m(\Omega)}^2+\frac{1}{10}\|f\|_{H^{m+1}(\Omega)}^2.
\end{equation*}
Next, we expand the second term in \eqref{L2id}. We have
\begin{equation*}\label{twotermsinterpolation}
\begin{split}
2\langle w_{\epsilon},w_{\epsilon}-u\rangle_{H^m(\Omega)}&=-2\epsilon^2\langle u^l,\nabla_{B_{\epsilon}}^4u^h_{\epsilon}\rangle_{H^m(\Omega)}-2\epsilon^2\langle u^h_{\epsilon},\nabla_{B_{\epsilon}}^4u_{\epsilon}^h\rangle_{H^m(\Omega)}+2\langle u^l,\mathcal{T}f\rangle_{H^m(\Omega)}+2\langle u_{\epsilon}^h,\mathcal{T}f\rangle_{H^m(\Omega)}
\\
&-2\epsilon^2\langle \nabla_{B_{\epsilon}}^4u_{\epsilon}^h,\mathcal{T}f\rangle_{H^m(\Omega)}+2\|\mathcal{T}f\|_{H^m(\Omega)}^2=:I_1+I_2+I_3+I_4+I_5+I_6.
\end{split}
\end{equation*}
We now estimate the terms $I_1,\dots,I_6$.
\medskip

\noindent
\textbf{$I_1$ Estimate}. For this, we simply commute $\nabla_{B_{\epsilon}}^4$  and integrate $\nabla_{B_{\epsilon}}^4$ by parts onto the low frequency term $u^l$ to estimate
\begin{equation*}
I_1\lesssim_M \epsilon^2\|u^l\|_{H^{m+4}(\Omega)}\|u_{\epsilon}^h\|_{H^m(\Omega)}+\epsilon\|u\|_{H^m(\Omega)}^2\lesssim_M \epsilon\|u\|_{H^m(\Omega)}^2,
\end{equation*}
where we used that $u_{\epsilon}^h$ is at high frequency to compensate for commutator errors.
\medskip

\noindent
\textbf{$I_2$ Estimate}. We commute two factors of $\nabla_{B_{\epsilon}}^2$ and integrate by parts to obtain
\begin{equation*}
I_2\leq -2\epsilon^2\|\nabla_{B_{\epsilon}}^2u^h_{\epsilon}\|_{H^m(\Omega)}^2+C(M)\epsilon\|u\|_{H^m(\Omega)}^2.
\end{equation*}
To estimate $I_3$ and $I_4$, we need to rely more heavily on the explicit structure of the term $\mathcal{T}f$. By expanding the first commutator in the definition of $\mathcal{T}$ and using the bounds for $\mathcal{L}^{-1}f^h$, we observe that there exist functions $g_{\epsilon}^1$ and $g_{\epsilon}^2$ such that (if $c$ is small enough) 
\begin{equation*}
\mathcal{T}f=\epsilon^{\frac{3}{2}}\nabla_{B_{\epsilon}}^3g_{\epsilon}^1+g_{\epsilon}^2,\hspace{5mm} \|g_{\epsilon}^1\|_{H^m(\Omega)}\leq \frac{1}{4}\|f\|_{H^{m+1}(\Omega)},\hspace{5mm}\|g_{\epsilon}^2\|_{H^m(\Omega)}\lesssim_M\epsilon^{\frac{1}{2}} \|f\|_{H^{m+1}(\Omega)}.
\end{equation*}
\textbf{$I_3$ Estimate}. Using the above decomposition, we can integrate three factors of $\nabla_{B_{\epsilon}}^3$ similarly to the $I_1$ estimate to obtain
\begin{equation*}
I_3\lesssim_M \epsilon^{\frac{3}{2}}\|u^l\|_{H^{m+3}(\Omega)}\|f\|_{H^{m+1}(\Omega)}+\epsilon^{\frac{1}{2}}\|u^l\|_{H^m(\Omega)}\|f\|_{H^{s+1}(\Omega)}\leq C(M)\epsilon\|u\|_{H^m(\Omega)}^2+\frac{1}{10}\|f\|_{H^{m+1}(\Omega)}^2.
\end{equation*}
\textbf{$I_4$ Estimate}. Using the above decomposition again and integrating by parts, we have
\begin{equation*}
\begin{split}
I_4&\leq \frac{\epsilon^{\frac{3}{2}}}{2}\|\nabla_{B_{\epsilon}}^3u^h_{\epsilon}\|_{H^m(\Omega)}\|f\|_{H^{m+1}(\Omega)}+\frac{1}{4}\|f\|_{H^{m+1}(\Omega)}^2+C(M)\epsilon\|u\|_{H^m(\Omega)}^2
\\
&\leq \frac{\epsilon^3}{4}\|\nabla_{B_{\epsilon}}^3u^h_{\epsilon}\|_{H^m(\Omega)}^2+\frac{1}{2}\|f\|_{H^{m+1}(\Omega)}^2+C(M)\epsilon\|u\|_{H^m(\Omega)}^2.
\end{split}
\end{equation*}
Integrating by parts and using Cauchy-Schwarz, we may then estimate 
\begin{equation*}
\frac{\epsilon^3}{4}\|\nabla_{B_{\epsilon}}^3u^h_{\epsilon}\|_{H^m(\Omega)}^2\leq \frac{\epsilon^4}{8}\|\nabla_{B_{\epsilon}}^4u^h_{\epsilon}\|_{H^m(\Omega)}^2+ \frac{\epsilon^2}{8}\|\nabla_{B_{\epsilon}}^2u^h_{\epsilon}\|_{H^m(\Omega)}^2+C(M)\epsilon\|u\|_{H^m(\Omega)}^2.
\end{equation*}
\textbf{$I_5$ and $I_6$ estimates}. Here, we simply use Cauchy-Schwarz and the operator bound for $\mathcal{T}$ to estimate
\begin{equation*}
I_5+I_6\leq \epsilon^4\frac{1}{4}\|\nabla_{B_{\epsilon}}^4u_{\epsilon}^h\|_{H^m(\Omega)}^2+\frac{1}{4}\|f\|_{H^{m+1}(\Omega)}^2.
\end{equation*}
Combining all of the above estimates yields the bound \eqref{interpolationsetup2} and thus \eqref{genrotregbound1} (with, say, $C=\frac{1}{100}$). Next, we move on to \eqref{genrotregbound2}. Using \eqref{vortcommute} and arguing similarly to the above, it suffices to establish the required bound with $\omega_{\epsilon}^h$ on the left-hand side in place of $\zeta_{\epsilon}$. We now describe how to set up the requisite interpolation argument. To begin, we introduce the notation
\begin{equation*}
\mathcal{C}^1_B:=[\mathcal{C}^1,\nabla_{B_{\epsilon}}]-\epsilon^2\nabla_{B_{\epsilon}}^4\mathcal{L}^{-1}[\nabla_{B_{\epsilon}},\Phi_{\leq\epsilon^{-\frac{1}{8}}}]\nabla\times,\hspace{5mm}\mathcal{C}^2_B:=\nabla_{B_{\epsilon}}\mathcal{C}^2.
\end{equation*}
Applying $\nabla_{B_{\epsilon}}$ to \eqref{vortcommute} and using the above notation, we  obtain
\begin{equation*}
\nabla_{B_{\epsilon}}\omega_{\epsilon}=(\nabla_{B_{\epsilon}}\omega)^l+\mathcal{L}^{-1}(\nabla_{B_{\epsilon}}\omega)^h+\mathcal{C}_1\nabla_{B_{\epsilon}}W+(\mathcal{C}^1_B+\mathcal{C}^2_B)W.
\end{equation*}
Analogously to before, we collect the estimates which can be verified through straightforward (but slightly tedious) computation: 
\begin{equation*}
\|\mathcal{C}_B^1\|_{H^{s+\frac{3}{2}}(\Omega)\to H^s(\Omega)}\lesssim_M \epsilon^{\frac{1}{2}},\hspace{5mm}\|\mathcal{C}_B^2\|_{H^{s+\frac{3}{2}}(\Omega)\to H^s(\Omega)}\lesssim_M \epsilon^{\frac{3}{2}},\hspace{5mm} 1\leq s\leq k-1.
\end{equation*}
We then define the map $\mathcal{T}_B: H^{s+1}(\Omega)\times H^{s+\frac{3}{2}}(\Omega)\to H^{s}(\Omega)$ by
\begin{equation*}
\mathbf{f}:=(f_1,f_2)\mapsto c\epsilon^{-\frac{1}{2}}(\mathcal{C}^1f_1+(\mathcal{C}^1_B+\mathcal{C}^2_B)f_2),
\end{equation*}
where $0<c\ll 1$ is chosen as before so that $\|\mathcal{T}_{B}\|\leq\frac{1}{100}$. As in the proof of \eqref{genrotregbound1}, we can reduce matters to establishing the estimate
\begin{equation*}
\|u_{\epsilon}+\mathcal{T}_B\mathbf{f}\|_{H^s(\Omega)}^2+\frac{1}{2}\sum_{2\leq j\leq 4}\epsilon^j\|\nabla_{B_{\epsilon}}^ju_{\epsilon}^h\|_{H^s(\Omega)}^2\leq (1+C(M)\epsilon)(\|u\|_{H^s(\Omega)}^2+\|\mathbf{f}\|_{H^{s+1}(\Omega)\times H^{s+\frac{3}{2}}(\Omega)}^2)
\end{equation*}
for integer $s=m$. The proof of this bound proceeds similarly  to the analogous bound  \eqref{genrotregbound1} by observing a decomposition for $\mathcal{T}_B$ similar to that of $\mathcal{T}$ above and then performing the analogous energy-type estimate. We omit the details of these straightforward (albeit somewhat tedious) modifications.
\medskip

Now, we establish the energy monotonicity bound (v). Let us write
\begin{equation*}
\mathcal{J}_{\epsilon}:=\sum_{2\leq j\leq 4}\epsilon^j(\|\nabla_{B_\epsilon}^jv_{\epsilon}^h\|_{\mathbf{H}^k(\Omega)}^2+\|\nabla_{B_\epsilon}^j\zeta_{\epsilon}^\pm\|_{\mathbf{H}^{k-1}(\Omega)}^2).
\end{equation*}
We will prove the stronger bound
\begin{equation*}
E^k(v_{\epsilon},B_{\epsilon},\Gamma)+C\mathcal{J}_{\epsilon}\leq E^k(v,B,\Gamma)+C(M)\epsilon,
\end{equation*}
for some constant $C>0$. Below, we write
\begin{equation*}
R_{\epsilon}:=c\mathcal{J}_{\epsilon}+C(M)\epsilon,
\end{equation*}
where $0<c\ll C$ is a sufficiently small positive constant. From properties (iii)-(iv) in \Cref{rotregEmon} we have
\begin{equation*}\label{vorticityemon}
\begin{split}
\|\omega_{\epsilon}^\pm\|_{\mathbf{H}^{k-1}(\Omega)}^2+C\sum_{2\leq j\leq 4}\epsilon^j\|\nabla_{B_{\epsilon}}^j\zeta_{\epsilon}^\pm\|_{\mathbf{H}^{k-1}(\Omega)}^2&\leq (1+C(M)\epsilon)\|\omega^\pm\|_{\mathbf{H}^{k-1}(\Omega)}^2+R_{\epsilon}.
\end{split}
\end{equation*}
To estimate perturbative errors in the surface components of the energy, we will need the  bound
\begin{equation}\label{Wdiffbound}
\|W^\pm_{\epsilon}-W^\pm\|_{H^{k-\frac{1}{2}}(\Omega)}\leq R_{\epsilon}.
\end{equation}
 To prove this bound, we use the div-curl estimate in \Cref{Balanced div-curl}, the bounds for $W_{\epsilon}^h$ and the identity $W^\pm_{\epsilon}-W^\pm=-\epsilon^2(\nabla_{B_{\epsilon}}^4W^h_{\epsilon})^{rot}-\epsilon^2(\nabla_{B_{\epsilon}}^4v_{\epsilon}^h)^{ir}$ to estimate 
\begin{equation*}
\|W^\pm_{\epsilon}-W^\pm\|_{H^{k-\frac{1}{2}}(\Omega)}\lesssim_M \epsilon^2\|\nabla\times\nabla_{B_{\epsilon}}^4W_{\epsilon}^h\|_{H^{k-\frac{3}{2}}(\Omega)}+\epsilon^2\|\nabla_{B_{\epsilon}}^4v_{\epsilon}^h\|_{H^{k-\frac{1}{2}}(\Omega)}+\epsilon.
\end{equation*}
The desired bound then follows by commuting $\nabla\times$ with $\nabla_{B_{\epsilon}}^4$ in the first term on the right-hand side above (using that $W_{\epsilon}^h$ is at high frequency to compensate for errors coming from these commutations) and from Cauchy-Schwarz. An immediate consequence of the above bound and elliptic regularity is the estimate
\begin{equation*}\label{PdiffboundR}
\begin{split}
\|a_{\epsilon}-a\|_{H^{k-1}(\Gamma)}+\|P_{\epsilon}-P\|_{H^{k+\frac{1}{2}}(\Omega)}\leq R_{\epsilon}.
\end{split}
\end{equation*}
From this we deduce that
\begin{equation*}\label{firstaestimate}
\int_{\Gamma}a_{\epsilon}|\mathcal{N}^{k-1}a_{\epsilon}|^2dS\leq \int_{\Gamma}a|\mathcal{N}^{k-1}a|^2dS+R_{\epsilon}.
\end{equation*}
Using the above,  self-adjointness of $\mathcal{N}$ and \Cref{Movingsurfid}, we also obtain
\begin{equation*}
\begin{split}
\int_{\Gamma}\mathcal{N}^{k-\frac{3}{2}}\nabla_{B_{\epsilon}}a_{\epsilon}\mathcal{N}^{k-\frac{3}{2}}\nabla_{B_{\epsilon}}(a_{\epsilon}-a)dS&\leq \int_{\Gamma}\mathcal{N}^{k-2}\nabla_{B_{\epsilon}}a_{\epsilon}\nabla_{B_{\epsilon}}\mathcal{N}^{k-1}(a_{\epsilon}-a)dS+R_{\epsilon}
\\
&\leq -\int_{\Gamma}\mathcal{N}^{k-2}\nabla_{B_{\epsilon}}^2a_{\epsilon}\mathcal{N}^{k-1}(a_{\epsilon}-a)dS+R_{\epsilon}.
\end{split}
\end{equation*}
We then further estimate
\begin{equation*}
-\int_{\Gamma}\mathcal{N}^{k-2}\nabla_{B_{\epsilon}}^2a_{\epsilon}\mathcal{N}^{k-1}(a_{\epsilon}-a)dS\lesssim_M \|\nabla_{B_{\epsilon}}^2a_{\epsilon}\|_{H^{k-2}(\Gamma)}\|a_{\epsilon}-a\|_{H^{k-1}(\Gamma)}.
\end{equation*}
By \Cref{partitionofBa}, we have $\|\nabla_{B_{\epsilon}}^2a_{\epsilon}\|_{H^{k-2}(\Gamma)}\lesssim_M 1$. Consequently, from the above analysis and \eqref{Wdiffbound}, we obtain
\begin{equation*}\label{secondaestimate}
\int_{\Omega}|\nabla\mathcal{H}\mathcal{N}^{k-2}\nabla_{B_{\epsilon}}a_{\epsilon}|^2dx\leq \int_{\Omega}|\nabla\mathcal{H}\mathcal{N}^{k-2}\nabla_{B}a|^2dx+R_{\epsilon}.
\end{equation*}
Next, we turn to the energy bounds for $\mathcal{G}_\epsilon^\pm$ and $\nabla_{B_{\epsilon}}\mathcal{G}_\epsilon^\pm$.
  We will need the following lemma to deal with some low-frequency error terms in our analysis.
\begin{lemma}\label{Gdecomp2}
Let $j\in\{0,1,2\}$. There holds
\begin{equation}\label{firstboundGdecomp2}
\|\nabla_{B_{\epsilon}}^j(\mathcal{G}_{\epsilon}^\pm-\mathcal{N}W^\pm_{\epsilon}\cdot \nabla P_{\epsilon})\|_{H^{k-1-\frac{j}{2}}(\Gamma)}\lesssim_M 1.
\end{equation}
Moreover, we have
\begin{equation}\label{errboundG}
\|\mathcal{G}_{\epsilon}^\pm-\mathcal{G}^\pm-\mathcal{N}(v_{\epsilon}-v)\cdot \nabla P_{\epsilon}\|_{H^{k-1}(\Gamma)}\leq R_{\epsilon}.
\end{equation}
\end{lemma}
\begin{proof}
We begin with \eqref{firstboundGdecomp2}. The cases $j=0$ and $j=1$ follows a similar line of reasoning to the proof of the decomposition for $\mathcal{G}_{\epsilon}^\pm$ and $\nabla_{B_{\epsilon}}\mathcal{G}_{\epsilon}^\pm$ in \Cref{Bbounds1} (just with different numerology because the norms are different), so we omit the details. To handle the case $j=2$, we begin by writing
\begin{equation*}
\mathcal{G}^\pm=\nabla_nW^\pm\cdot \nabla P-\nabla_n\Delta^{-1}(\Delta W^\pm\cdot \nabla P+2\nabla W^\pm\cdot \nabla^2 P).
\end{equation*}
Next, we recall from \Cref{partitionofBa} that we have the enhanced regularity bound 
\begin{equation*}
\|\nabla_{B_{\epsilon}}^2P_\epsilon\|_{H^{k-\frac{1}{2}}(\Omega)}\lesssim_M 1.
\end{equation*}
Moreover, from the identity $n_\Gamma=-a_\epsilon^{-1}\nabla P_\epsilon$, it is also a straightforward application of \Cref{partitionofBa} and Sobolev product estimates to obtain
\begin{equation*}
\|\nabla_{B_{\epsilon}}^2n_\Gamma\|_{H^{k-2}(\Gamma)}\lesssim_M 1.
\end{equation*}
Combining this with the bound $\|\nabla_{B_{\epsilon}}^2W_{\epsilon}^\pm\|_{H^{k-\frac{3}{2}}(\Omega)}+\|\nabla_{B_{\epsilon}}B_{\epsilon}\|_{H^{k-\frac{1}{2}}(\Omega)}\lesssim_M 1$, we have
\begin{equation*}
\begin{split}
\nabla_{B_{\epsilon}}^2\mathcal{G}_{\epsilon}^\pm&=\nabla_n\nabla_{B_{\epsilon}}^2W^\pm_{\epsilon}\cdot\nabla P_{\epsilon}-\nabla_n\Delta^{-1}(\Delta \nabla_{B_{\epsilon}}^2W_{\epsilon}^\pm\cdot\nabla P_{\epsilon})+\mathcal{O}_{H^{k-2}(\Gamma)}(1)=\mathcal{N}\nabla_{B_{\epsilon}}^2W_{\epsilon}^\pm\cdot \nabla P_{\epsilon}+\mathcal{O}_{H^{k-2}(\Gamma)}(1)
\\
&=\nabla_{B_{\epsilon}}^2(\mathcal{N}W_{\epsilon}^\pm\cdot \nabla P_{\epsilon})+\mathcal{O}_{H^{k-2}(\Gamma)}(1),
\end{split}
\end{equation*}
as deisred. For the bound \eqref{errboundG}, we observe that in light of \eqref{Wdiffbound} and the definition of $\mathcal{G}^\pm$  we have the inequality
\begin{equation*}
\|\mathcal{G}^\pm_{\epsilon}-\mathcal{G}^\pm-(\nabla_n(W_{\epsilon}-W)\cdot\nabla P_\epsilon-\nabla_n\Delta^{-1}(\Delta (W_{\epsilon}-W)\cdot\nabla P_{\epsilon})\|_{H^{k-1}(\Gamma)}\leq R_{\epsilon}.
\end{equation*}
Using the bounds $\|W_{\epsilon}-W\|_{H^{k-5}(\Omega)}\lesssim_M\epsilon^2$,  $\|P_{\epsilon}\|_{H^{k+1}(\Omega)}\lesssim_M \epsilon^{-\frac{1}{2}}$  (which follows from the regularization bounds for the surface), $\|W^\pm_{\epsilon}\|_{H^{k+\frac{1}{2}}(\Omega)}\lesssim_M\epsilon^{-1}$ and the identity $\nabla_n=\mathcal{N}+\nabla_n\Delta^{-1}\Delta$, we have
\begin{equation*}
\|\mathcal{G}^\pm_{\epsilon}-\mathcal{G}^\pm-\mathcal{N}(W_{\epsilon}-W)\cdot\nabla P_{\epsilon}\|_{H^{k-1}(\Gamma)}\leq R_{\epsilon}.
\end{equation*}
Finally, we use that $W_{\epsilon}^\pm-W^\pm=\mp\epsilon^2(\nabla_{B_{\epsilon}}^4\tilde{B}_{\epsilon}^h)^{rot}+v_{\epsilon}-v$ and the Leibniz rule for $\mathcal{N}$ (leveraging the irrotationality $(\nabla_{B_{\epsilon}}^4\tilde{B}_{\epsilon}^h)^{rot}$) to estimate
\begin{equation*}
\epsilon^2\|\mathcal{N}(\nabla_{B_{\epsilon}}^4\tilde{B}_{\epsilon}^h)^{rot}\cdot\nabla P_\epsilon\|_{H^{k-1}(\Gamma)}\lesssim_M \epsilon+\epsilon^2\|(\nabla_{B_{\epsilon}}^4\tilde{B}_{\epsilon}^h)^{rot}\|_{H^{k-\frac{1}{2}}(\Omega)}\leq R_{\epsilon},
\end{equation*}
where the last inequality follows analogously to the proof of \eqref{Wdiffbound}.
\end{proof}
Next, we recall that $v_{\epsilon}=w_{\epsilon}^{div}$, where $w_{\epsilon}=v^l+v_{\epsilon}^h$. We will show that -- up to an error of size $R_{\epsilon}$ -- all instances of $v_{\epsilon}$ above can be replaced by $w_{\epsilon}$. To accomplish this, we begin by writing the equation for $\nabla\cdot w_{\epsilon}$. We compute that
\begin{equation}\label{divergenceqnforweps}
\nabla\cdot w_{\epsilon}=(1-\mathcal{L}^{-1})[\nabla\cdot,\Phi_{\leq\epsilon^{-\frac{1}{8}}}]v+\epsilon^2\mathcal{L}^{-1}[\nabla_{B_{\epsilon}}^4,\nabla\cdot]v_{\epsilon}^h,
\end{equation}
which (by writing $1-\mathcal{L}^{-1}=\epsilon^2\mathcal{L}^{-1}\nabla_{B_{\epsilon}}^4$) gives the estimates $\|\nabla\cdot w_{\epsilon}\|_{\mathbf{H}^{k-1}(\Omega)}\leq R_{\epsilon}$, $\|\nabla\cdot w_{\epsilon}\|_{\mathbf{H}^{k-5}(\Omega)}\lesssim_M\epsilon^2$, and thus, by the balanced elliptic estimates in \Cref{BEE},
\begin{equation}\label{vtowapprox}
\|v_{\epsilon}-w_{\epsilon}\|_{\mathbf{H}^{k}(\Omega)}=\|\nabla\Delta^{-1}\nabla\cdot w_{\epsilon}\|_{\mathbf{H}^{k}(\Omega)}\leq R_{\epsilon}.
\end{equation}
\begin{remark}\label{divergenceremarkB}
By inspecting the equation \eqref{divergenceqnforweps} and writing out the analogous equation for $\nabla\cdot \tilde{B}_{\epsilon}$ (which simply amounts to replacing $v$ with $B$ above), we also have  
\begin{equation*}
\|\nabla\Delta^{-1}\nabla\cdot w_{\epsilon}\|_{\mathbf{H}^{k}(\Omega)}+\|\nabla\Delta^{-1}\nabla\cdot \tilde{B}_{\epsilon}\|_{\mathbf{H}^{k}(\Omega)}\lesssim_M \epsilon^{\frac{1}{2}}
\end{equation*}
which we will need to use in the next section, but not for the remainder of the proof here.
\end{remark}
Next, for efficient bookkeeping, we define for $0\leq j\leq 5$,
\begin{equation*}
Q^l_j:=\mathcal{N}\nabla_{B_{\epsilon}^l}^jv^l\cdot\nabla P_{\epsilon},\hspace{5mm}Q^h_j:=\mathcal{N}\nabla_{B_{\epsilon}}^jv_{\epsilon}^h\cdot\nabla P_{\epsilon},
\end{equation*}
 where we write $B_{\epsilon}^l:=\Phi_{\leq \epsilon^{-\frac{1}{8}}}B_{\epsilon}$. Below, we will use  the following bounds repeatedly:
 \begin{equation}\label{frequentuse}
\epsilon^{\frac{j+1}{2}}(\|Q_j^h\|_{H^{k-\frac{3}{2}}(\Gamma)}+\|Q_{j+1}^h\|_{H^{k-2}(\Gamma)})\leq C(M)\epsilon^{\frac{j+1}{2}}\|\nabla_{B_{\epsilon}}^jv_{\epsilon}^h\|_{\mathbf{H}^{k-\frac{3}{2}}(\Omega)}\leq  C(M)\epsilon+R_{\epsilon},\hspace{5mm}2\leq j\leq 4,
 \end{equation}
where the last inequality follows from Cauchy-Schwarz. Now, we are ready to estimate the energy components $E^k_{\mathcal{G}_{\epsilon}^\pm}$  and $E^k_{\nabla_{B_{\epsilon}}\mathcal{G}_{\epsilon}^\pm}$. For the first component, our starting point is the elementary expansion
\begin{equation*}
E^k_{\mathcal{G}^\pm_{\epsilon}}=E^k_{\mathcal{G}^\pm}+\sum_{\alpha\in\{+,-\}}2\langle\nabla\mathcal{H}\mathcal{N}^{k-2}\mathcal{G}_{\epsilon}^\alpha,\nabla\mathcal{H}\mathcal{N}^{k-2}(\mathcal{G}^\alpha_{\epsilon}-\mathcal{G}^\alpha)\rangle_{L^2(\Omega)}-\|\nabla\mathcal{H}\mathcal{N}^{k-2}(\mathcal{G}^\alpha_{\epsilon}-\mathcal{G}^\alpha)\|_{L^2(\Omega)}^2.
\end{equation*}
By Poincare's inequality (noting that $\mathcal{N}^{k-2}(\mathcal{G}_{\epsilon}^\pm-\mathcal{G}^\pm)$ has mean zero on $\Gamma$) and ellipticity of $\mathcal{N}$, we may estimate
\begin{equation*}
\|\mathcal{G}^\pm_{\epsilon}-\mathcal{G}^\pm\|_{H^{k-\frac{3}{2}}(\Gamma)}^2\leq C(M)\|\nabla\mathcal{H}\mathcal{N}^{k-2}(\mathcal{G}_{\epsilon}^\pm-\mathcal{G}^\pm)\|_{L^2(\Omega)}^2+R_{\epsilon}.
\end{equation*}
Hence, by \Cref{Gdecomp2}, \eqref{vtowapprox} and the equation $w_{\epsilon}$, we have
\begin{equation*}
\|Q_4^h\|_{H^{k-\frac{3}{2}}(\Gamma)}^2\leq C(M)\|\nabla\mathcal{H}\mathcal{N}^{k-2}(\mathcal{G}_{\epsilon}^\pm-\mathcal{G}^\pm)\|_{L^2(\Omega)}^2+R_{\epsilon}.
\end{equation*}
By again using \Cref{Gdecomp2}, \eqref{vtowapprox} and the definition of $w_{\epsilon}$, it follows  by integration by parts that
\begin{equation}\label{annoyingcrossterms}
\begin{split}
\sum_{\alpha\in\{+,-\}}2\langle\nabla\mathcal{H}&\mathcal{N}^{k-2}\mathcal{G}_{\epsilon}^\alpha,\nabla\mathcal{H}\mathcal{N}^{k-2}(\mathcal{G}^\alpha_{\epsilon}-\mathcal{G}^\alpha)\rangle_{L^2(\Omega)}\leq-4\epsilon^2\langle\nabla\mathcal{H}\mathcal{N}^{k-2}(Q_0^l+Q_0^h),\nabla\mathcal{H}\mathcal{N}^{k-2}Q_4^h\rangle_{L^2(\Omega)}
\\
&+C(M)\sum_{\alpha\in\{+,-\}}\|\mathcal{N}^{k-1}(\mathcal{G}^{\alpha}-\mathcal{N}W_{\epsilon}^{\alpha})\|_{L^2(\Gamma)}\|\mathcal{N}^{k-2}(\mathcal{N}(v_{\epsilon}-v)\cdot\nabla P_{\epsilon})\|_{L^2(\Gamma)}+ R_{\epsilon}.
\end{split}
\end{equation}
Thanks to \eqref{Wdiffbound}, we may estimate
\begin{equation*}
\|\mathcal{N}^{k-2}(\mathcal{N}(v_{\epsilon}-v)\cdot\nabla P_{\epsilon})\|_{L^2(\Gamma)}\lesssim_{M}\|v_{\epsilon}-v\|_{H^{k-\frac{1}{2}}(\Omega)}\leq R_{\epsilon}.
\end{equation*}
Hence, by combining the above estimates, we have
\begin{equation}\label{Ek energy existence}
\begin{split}
E^k_{\mathcal{G}_{\epsilon}^\pm}&\leq E^k_{\mathcal{G}^\pm}-4\epsilon^2\langle\nabla\mathcal{H}\mathcal{N}^{k-2}Q_0^l,\nabla\mathcal{H}\mathcal{N}^{k-2}Q_4^h\rangle_{L^2(\Omega)}-4\epsilon^2\langle\nabla\mathcal{H}\mathcal{N}^{k-2}Q_0^h,\nabla\mathcal{H}\mathcal{N}^{k-2}Q_4^h\rangle_{L^2(\Omega)}
\\
&-C(M)\epsilon^4\|Q_4^h\|_{H^{k-\frac{3}{2}}(\Gamma)}^2+R_{\epsilon}.
\end{split}
\end{equation}
Notice next that by commuting two factors of $\nabla_{B_{\epsilon}}$ with $\mathcal{N}$ (observing that such commutations contribute errors of type $R_{\epsilon}$) and integrating by parts, we may bound
\begin{equation*}
\begin{split}
-4\epsilon^2\langle\nabla\mathcal{H}\mathcal{N}^{k-2}Q_0^h,\nabla\mathcal{H}\mathcal{N}^{k-2}Q_4^h\rangle_{L^2(\Omega)}\leq -C(M)\epsilon^2\|Q_2^h\|_{H^{k-\frac{3}{2}}(\Gamma)}^2+R_{\epsilon}.
\end{split}
\end{equation*}
Our aim will be to show that the remaining term $-4\epsilon^2\langle\nabla\mathcal{H}\mathcal{N}^{k-2}Q_0^l,\nabla\mathcal{H}\mathcal{N}^{k-2}Q_4^h\rangle_{L^2(\Omega)}$ in \eqref{Ek energy existence} can be controlled by $R_{\epsilon}$. The key idea here is that $Q_0^l$ and $Q_4^h$ should be ``almost" orthogonal, since the leading part of $Q_0^l$ involves the low-frequency factor $v^l$. Nevertheless, exploiting this is somewhat delicate. The strategy will be to shift a suitable number of factors of $\nabla_{B_{\epsilon}}$ onto the low-frequency term $v^l$. By commuting a factor of $\nabla_{B_{\epsilon}}$ onto $Q_0^l$, we obtain
\begin{equation*}
-4\epsilon^2\langle\nabla\mathcal{H}\mathcal{N}^{k-2}Q_0^l,\nabla\mathcal{H}\mathcal{N}^{k-2}Q_4^h\rangle_{L^2(\Omega)}\leq 4\epsilon^2\langle\nabla\mathcal{H}\mathcal{N}^{k-2}Q_1^l,\nabla\mathcal{H}\mathcal{N}^{k-2}Q_3^h\rangle_{L^2(\Omega)}+R_{\epsilon}
\end{equation*}
where we used the estimate
\begin{equation*}
\nabla_{B_{\epsilon}}Q_0^l=Q_1^l+\mathcal{O}_{H^{k-\frac{3}{2}}(\Gamma)}(1),
\end{equation*}
which follows from the commutator identities in \Cref{Movingsurfid} and  the straightforward estimate
\begin{equation*}
\|\nabla_{B_{\epsilon}^l-B_{\epsilon}} v^l\|_{H^k(\Omega)}\lesssim_M \|B_{\epsilon}-B\|_{H^k(\Omega)}+\|B_{\epsilon}-B\|_{H^{k-10}(\Omega)}\|v^l\|_{H^{k+1}(\Omega)}\lesssim_M 1.
\end{equation*}
Iterating this process and carrying out similar estimates, we find that
\begin{equation*}
\begin{split}
-4\epsilon^2\langle\nabla\mathcal{H}\mathcal{N}^{k-2}Q_0^l,\nabla\mathcal{H}\mathcal{N}^{k-2}Q_4^h\rangle_{L^2(\Omega)}&\leq -4\epsilon^2\langle\nabla\mathcal{H}\mathcal{N}^{k-2}Q_2^l,\nabla\mathcal{H}\mathcal{N}^{k-2}Q_2^h\rangle_{L^2(\Omega)}+R_{\epsilon}
\\
&\lesssim_M \epsilon^2\|\nabla_{B_{\epsilon}^l}^2v^l\|_{H^k(\Omega)}\|\nabla_{B_{\epsilon}}^2v_{\epsilon}^h\|_{H^k(\Omega)}+R_\epsilon
\\
&\leq R_{\epsilon},
\end{split}
\end{equation*}
where we used the bound $\|\nabla_{B_{\epsilon}^l}^2v^l\|_{H^k(\Omega)}\lesssim_M \epsilon^{-\frac{1}{2}}$. Consequently, we finally have
\begin{equation}\label{mainEkgest}
E^k_{\mathcal{G}_{\epsilon}^\pm}\leq E^k_{\mathcal{G}^\pm}-C(M)\epsilon^2\|Q_2^h\|_{H^{k-\frac{3}{2}}(\Gamma)}^2-C(M)\epsilon^4\|Q_4^h\|_{H^{k-\frac{3}{2}}(\Gamma)}^2+R_{\epsilon}.
\end{equation}
This will be the essential part of our estimate for $E^k_{\mathcal{G}^\pm}$. The analogous estimate for $E^k_{\nabla_{B_{\epsilon}}\mathcal{G}^\pm}$ is mostly similar. Our aim is to obtain the estimate
\begin{equation}\label{DBenergyestimate}
\begin{split}
E_{\nabla_{B_{\epsilon}}\mathcal{G}_{\epsilon}^\pm}^k&\leq E_{\nabla_{B}\mathcal{G}^\pm}^k-C(M)\epsilon^2\|Q_3^h\|_{H^{k-2}(\Gamma)}^2-C(M)\epsilon^4\|Q_5^h\|_{H^{k-2}(\Gamma)}^2+R_{\epsilon}.
\end{split}
\end{equation}
This follows along a very similar line of reasoning to the estimate for $E^k_{\mathcal{G}_{\epsilon}^\pm}$ except for one significant difference. The analogue of \eqref{annoyingcrossterms} is the following:
\begin{equation*}
\begin{split}
\sum_{\alpha\in\{+,-\}}2\langle a^{-1}_{\epsilon}\mathcal{N}^{k-2}&\nabla_{B_{\epsilon}}\mathcal{G}_{\epsilon}^\alpha,\mathcal{N}^{k-2}\nabla_{B_{\epsilon}}(\mathcal{G}^\alpha_{\epsilon}-\mathcal{G}^\alpha)\rangle_{L^2(\Gamma)}\leq-4\epsilon^2\langle a_{\epsilon}^{-1}\mathcal{N}^{k-2}\nabla_{B_{\epsilon}}(Q_0^l+Q_0^h),\mathcal{N}^{k-2}\nabla_{B_{\epsilon}}Q_4^h\rangle_{L^2(\Gamma)}
\\
&+C(M)\sum_{\alpha\in\{+,-\}}\|\mathcal{N}^{k-2}\nabla_{B_{\epsilon}}^2(\mathcal{G}^{\alpha}-\mathcal{N}W_{\epsilon}^{\alpha})\|_{L^2(\Gamma)}\|\mathcal{N}^{k-2}(\mathcal{N}(v_{\epsilon}-v)\cdot\nabla P_{\epsilon})\|_{L^2(\Gamma)}+ R_{\epsilon},
\end{split}
\end{equation*}
where the term in the second line arises by using \eqref{surfintbyparts} to integrate a factor of $\nabla_{B_{\epsilon}}$ by parts on $\Gamma$ (in contrast to the corresponding term in \eqref{annoyingcrossterms} where we instead integrated $\nabla$ by parts). Similarly to the analysis for $E^k_{\mathcal{G}_{\epsilon}^\pm}$, we  have
\begin{equation*}
\begin{split}
-4\epsilon^2\langle a_{\epsilon}^{-1}\mathcal{N}^{k-2}\nabla_{B_{\epsilon}}Q_0^l,\mathcal{N}^{k-2}\nabla_{B_{\epsilon}}Q_4^h\rangle_{L^2(\Gamma)}&\leq 4\epsilon^2\langle a_{\epsilon}^{-1}\mathcal{N}^{k-2}Q_2^l,\mathcal{N}^{k-2}Q_4^h\rangle_{L^2(\Gamma)}+R_{\epsilon}
\\
&\leq -4\epsilon^2\langle a_{\epsilon}^{-1}\mathcal{N}^{k-2}Q_3^l,\mathcal{N}^{k-2}Q_3^h\rangle_{L^2(\Gamma)}+R_{\epsilon}
\\
&\lesssim_M \epsilon^2\|\nabla_{B_{\epsilon}^l}^3v^l\|_{H^{k-\frac{1}{2}}(\Omega)}\|\nabla_{B_{\epsilon}}^3v_{\epsilon}^h\|_{H^{k-\frac{1}{2}}(\Omega)}+R_\epsilon
\leq R_{\epsilon}.
\end{split}
\end{equation*}
By integrating by parts two factors of $\nabla_{B_{\epsilon}}$, we also have
\begin{equation*}
-4\epsilon^2\langle a_{\epsilon}^{-1}\mathcal{N}^{k-2}\nabla_{B_{\epsilon}}Q_0^h,\mathcal{N}^{k-2}\nabla_{B_{\epsilon}}Q_4^h\rangle_{L^2(\Gamma)}\leq -4\epsilon^2\langle \mathcal{N}^{k-2}Q_3^h,\mathcal{N}^{k-2}Q_3^h\rangle_{L^2(\Gamma)}+R_{\epsilon}.
\end{equation*}
Then, arguing as in the $E^k_{\mathcal{G}_{\epsilon}^\pm}$ estimate to handle the remaining terms, we obtain \eqref{DBenergyestimate}. To finally conclude the energy monotonicity bound, we observe that by \Cref{Balanced div-curl}, the Taylor sign condition and the Leibniz rule for $\mathcal{N}$, we can estimate
\begin{equation*}
\mathcal{J}_{\epsilon}\leq C(M)\sum_{2\leq j\leq 4}\epsilon^j(\|\nabla_{B_{\epsilon}}^j\zeta_{\epsilon}\|_{\mathbf{H}^{k-1}(\Omega)}^2+\|Q_j^h\|_{H^{k-\frac{3}{2}}(\Gamma)}^2+\|Q_{j+1}^h\|_{H^{k-2}(\Gamma)}^2)+R_{\epsilon}.
\end{equation*}
Combining this with \eqref{mainEkgest} and \eqref{DBenergyestimate}, we obtain
\begin{equation*}
E^k(v_{\epsilon},B_{\epsilon},\Gamma_{\epsilon})+C\mathcal{J}_{\epsilon}\leq E^k(v,B,\Gamma)+R_{\epsilon},
\end{equation*}
which gives the energy monotonicity bound (v).
\end{proof}
\subsection{Step 4: Euler plus transport iteration}\label{EPTI} Given a small time step $0<\epsilon\ll 1$ and an initial data $(v_0,B_0,\Gamma_0)\in \mathbf{H}^k$ satisfying the hypotheses of \Cref{onestepiteration}, we construct a regularized data $(v_{\epsilon}, B_{\epsilon},\Gamma_{\epsilon})$ as follows: First, we input the initial state $(v_0,B_0,\Gamma_0)$ into \Cref{domainregularizationprop} to produce a partially regularized state $(\tilde{v}_{\epsilon}, \tilde{B}_{\epsilon},\tilde{\Gamma}_{\epsilon})$. Then, we input the  state $(\tilde{v}_{\epsilon}, \tilde{B}_{\epsilon},\tilde{\Gamma}_{\epsilon})$ into \Cref{mildreg}, which  produces a new state with the same regularization properties as $(\tilde{v}_{\epsilon}, \tilde{B}_{\epsilon},\tilde{\Gamma}_{\epsilon})$ but also  certain  high regularity bounds. Finally, we insert the outcome of  \Cref{mildreg}   into \Cref{rotregEmon} to obtain the fully regularized state $(v_{\epsilon}, B_{\epsilon},\Gamma_{\epsilon})$. Clearly, $(v_{\epsilon}, B_{\epsilon},\Gamma_{\epsilon})$ obeys the energy monotonicity bound and stays within distance $\mathcal{O}_{C^3}(\epsilon^2)$ of $(v_0,B_0,\Gamma_0)$. The regularization properties of $(v_{\epsilon}, B_{\epsilon},\Gamma_{\epsilon})$ are summarized in the following proposition.
\begin{proposition}\label{Reg bound list}
The following properties hold for the regularized data $(v_\epsilon,B_\epsilon,\Gamma_\epsilon)$:
\begin{enumerate}
\item (Surface regularization). For each $\alpha\geq 0$, we have 
\begin{equation*}
\|\Gamma_{\epsilon}\|_{H^{k+\alpha}}\lesssim_{M,\alpha} \epsilon^{-\alpha}.
\end{equation*}
\item (Rotational and irrotational bounds). The velocity and magnetic fields satisfy the bounds:
\begin{equation*}
\|\omega_{\epsilon}^\pm\|_{H^{k}(\Omega_\epsilon)}\leq (1+C(M)\epsilon)K(M)\epsilon^{-\frac{3}{2}},\hspace{5mm}\|\mathcal{N}_{\epsilon}W^\pm_{\epsilon}\cdot n_{\epsilon}\|_{H^{k-\frac{1}{2}}(\Gamma_\epsilon)}\lesssim_M\epsilon^{-1}.
\end{equation*}
\item (Higher regularity bounds). For each $0\leq \alpha\leq 1$, we have the higher regularity bounds
\begin{equation*}
\|W_\epsilon^\pm\|_{H^{k+1+\alpha}(\Omega_\epsilon)}\lesssim_M\epsilon^{-\frac{3}{2}-3\alpha}.
\end{equation*}
\item (Regularization in the direction of the magnetic field). We have the bounds
\begin{equation*}
\|\nabla_{B_{\epsilon}}\tilde{W}_{\epsilon}^\pm\|_{\mathbf{H}^k(\Omega_\epsilon)}\lesssim_M\epsilon^{-\frac{1}{2}},\hspace{5mm}\|\nabla_{B_{\epsilon}}\tilde{W}^\pm_{\epsilon}\|_{H^{k+1}(\Omega_\epsilon)}\lesssim_M\epsilon^{-2}.
\end{equation*}
\end{enumerate}
\end{proposition}
\begin{remark}
We remark importantly that in statement (iv) of \Cref{Reg bound list}, the improved regularity is witnessed by   $\tilde{W}_{\epsilon}^\pm$ rather than $W^\pm_\epsilon$.
\end{remark}
The objective of this subsection is to construct the iterate $(v_1, B_1,\Gamma_1)$ from the regularized data $(v_{\epsilon}, B_{\epsilon},\Gamma_{\epsilon})$. In essence, we will produce the iterate $(v_1, B_1,\Gamma_1)$  by flowing the regularized data $(v_{\epsilon}, B_{\epsilon},\Gamma_{\epsilon})$ along a discrete version of the free boundary MHD evolution.  A na\"ive  Euler type iteration (phrased in terms of the $W^\pm$ variables) would suggest an iteration akin to the following:
\begin{equation*}
\begin{split}
W^\pm_{1}&:=W^\pm_{\epsilon}-\epsilon (\nabla_{v_{\epsilon}} W^\pm_{\epsilon}\mp\nabla_{B_{\epsilon}}W^\pm_{\epsilon}+\nabla P_{\epsilon}) ,
\end{split} 
\end{equation*}
which is to be supplemented with the domain transport
\begin{equation*}
x_1(x):=x+\epsilon v_{\epsilon}(x).    
\end{equation*}
Unfortunately, this na\"ive scheme loses a full derivative in each iteration. Therefore, it is better to perform the above two steps in tandem. This will halve the derivative loss and allow us to uncover a discrete version of the energy cancellation seen in \Cref{HEB}. The  regularization bounds in \Cref{Reg bound list} will then be used to control any remaining errors. Such errors now only ``lose" half a derivative, which makes estimating the various quadratic error terms appearing in our analysis below much easier. To carry out this procedure, we have the following proposition. 
\begin{proposition}\label{Final transport +Euler}
Given $(v_{\epsilon}, B_{\epsilon},\Gamma_{\epsilon})$ as in  the previous step, there exists an iteration $(v_{\epsilon}, B_{\epsilon}, \Gamma_{\epsilon})\mapsto (v_1, B_1, \Gamma_1)$ 
such that the following properties hold:
\begin{enumerate}
    \item (Approximate solution). 
\begin{equation*}
\begin{cases}
&W^\pm_1=W_{\epsilon}-\epsilon (\nabla_{v_{\epsilon}} W^\pm_{\epsilon}\mp \nabla_{B_{\epsilon}}W^\pm_{\epsilon}+\nabla P_{\epsilon})+\mathcal{O}_{C^3}(\epsilon^2),\hspace{5mm}\text{on}\hspace{2mm}\Omega_1\cap\Omega_0,
\\
&\nabla\cdot W^\pm_1=0,\hspace{5mm}\text{on}\hspace{2mm}\Omega_1,
\\
&B_1\cdot n_1=0,\hspace{5mm}\text{on}\hspace{2mm}\Gamma_1,
\\
&\Omega_{1}=(I+\epsilon v_{\epsilon})(\Omega_{\epsilon}).
\end{cases}
\end{equation*}
\item (Energy monotonicity bound).
\begin{equation*}
E^{k}(v_1,B_1,\Gamma_1)\leq (1+C(M)\epsilon)E^{k}(v_{\epsilon}, B_{\epsilon},\Gamma_{\epsilon}).
\end{equation*}
\item (Propagated regularization bounds for $(v_1,B_1,\Gamma_1)$). For $0<\gamma\ll 1$ as in \eqref{inductiveregbound} there holds
\begin{equation*}
\|\omega_1^\pm\|_{H^k(\Omega_1)}\leq (1+C(M)\epsilon)K(M)\epsilon^{-\frac{3}{2}},\hspace{5mm}\|\Gamma_1\|_{H^{k+\frac{1}{2}+\gamma}}\leq \frac{1}{2}K(M)\epsilon^{-\frac{1}{2}-\gamma}.
\end{equation*}
\end{enumerate}
\end{proposition}
We define the change of coordinates $x_1(x) := x + \epsilon v_{\epsilon}(x)$
and the iterated domain $\Omega_1$ by
\begin{equation*}
\Omega_1:=(I+\epsilon v_{\epsilon})\Omega_{\epsilon}.    
\end{equation*}
We remark that the latter bound in (iii) is required to close the bootstrap \eqref{inductiveregbound} from the parabolic regularization step.
\medskip

\textbf{Uncorrected variables}. We begin by defining the uncorrected iteration variables
\begin{equation*}\label{Iterated v_1}
\tilde{v}_1(x_1):=v_{\epsilon}-\epsilon(\nabla P_{\epsilon}-\nabla_{B_{\epsilon}} \tilde{B}_{\epsilon})
\end{equation*}
and
\begin{equation*}
\tilde{B}_1(x_1):=B_{\epsilon}+\epsilon \nabla_{B_{\epsilon}}v_{\epsilon}.    
\end{equation*}
We note that $\tilde{v}_1$ and $\tilde{B}_1$ are not divergence-free, so we will need to perform suitable corrections. 
\begin{remark}
On the right-hand side of the definition of $\tilde{v}_1$ we use $\nabla_{B_{\epsilon}}\tilde{B}_{\epsilon}$ as opposed to $\nabla_{B_{\epsilon}}B_{\epsilon}$. Such terms agree up to an error of size $\epsilon^2$ in the $C^3$ topology, but in $H^k$, only the term $\nabla_{B_{\epsilon}}\tilde{B}_{\epsilon}$ enjoys the $\epsilon^{-\frac{1}{2}}$ regularization bound. This will nonetheless turn out to be sufficient for proving a suitable energy monotonicity bound in this stage of the argument.
\end{remark}

Before we define the appropriate corrections, we will take a detour to first show that $\tilde{B}_1\cdot n_1=0$ and also close the bootstrap for $\Gamma_1$. For this, we need the following lemma.
\begin{lemma}\label{normalapprox} 
The following identity holds:
\begin{equation*}
n_1(x_1)-n_\epsilon(x)=-\epsilon\nabla^{\top}v_{\epsilon}\cdot n_\epsilon-\epsilon\nabla^{\top}v_{\epsilon}\cdot (n_1(x_1)-n_\epsilon)-\frac{1}{2}n_\epsilon|n_1(x_1)-n_\epsilon
|^2.
\end{equation*}
\end{lemma}
We remark that $-\nabla^{\top}v_{\epsilon}\cdot n_\epsilon$ is nothing more than the expression for $D_tn_\epsilon$ (see \Cref{Movingsurfid}).
\begin{proof}
It is convenient to write the normal in terms of $\nabla P_{\epsilon}$ and the Taylor term (although the gradient of any non-degenerate defining function will also work). By doing this, we have 
\begin{equation*}
\begin{split}
n_1(x_1)-n_\epsilon(x)&=a_\epsilon^{-1}\nabla P_\epsilon-a_1^{-1}(x_1)(\nabla P_1)(x_1)
\\
&=-a_1^{-1}(x_1)((\nabla P_1)(x_1)-\nabla P_\epsilon(x))+(a_\epsilon^{-1}-a^{-1}_1(x_1))\nabla P_\epsilon
\\
&=-\epsilon \nabla v_{\epsilon}\cdot n_1(x_1)-a_1^{-1}(x_1)\nabla (P_1(x_1)-P_\epsilon(x))-a_1^{-1}(x_1)(a_1(x_1)-a_\epsilon)n_\epsilon.
\end{split}
\end{equation*}
We then compute that
\begin{equation*}
\begin{split}
-a_1^{-1}(x_1)(a_1(x_1)-a_\epsilon)=&a_1^{-1}(x_1)(n_1(x_1)\cdot(\nabla P_1)(x_1)-n_\epsilon\cdot\nabla P_\epsilon)
\\
=&a_1^{-1}(x_1)(n_1(x_1)-n_\epsilon)\cdot(\nabla P_1)(x_1)+\epsilon n_\epsilon\cdot \nabla v_{\epsilon}\cdot n_1(x_1)
\\
+&a_1^{-1}(x_1)n_\epsilon\cdot \nabla(P_1(x_1)-P_\epsilon).
\end{split}
\end{equation*}
Hence, by the dynamic boundary condition,
\begin{equation*}
\begin{split}
n_1(x_1)-n_\epsilon(x)&=-\epsilon\nabla^{\top}v_{\epsilon}\cdot n_1(x_1)-n_\epsilon(n_1(x_1)-n_\epsilon)\cdot n_1(x_1)
\\
&=-\epsilon\nabla^{\top}v_{\epsilon}\cdot n_\epsilon-\epsilon\nabla^{\top}v_{\epsilon}\cdot (n_1(x_1)-n_\epsilon)-\frac{1}{2}n_\epsilon|n_1(x_1)-n_\epsilon|^2.
\end{split}
\end{equation*}
This completes the proof.
\end{proof}
From the above identity, we immediately obtain the weak bound
\begin{equation*}
\|n_1(x_1)-n_\epsilon\|_{H^{k-2}(\Gamma_\epsilon)}\lesssim_M\epsilon.
\end{equation*}
Thanks to this, \Cref{normalapprox} and by interpolating the regularization bound $\|v_{\epsilon}\|_{H^{k+2}(\Omega_\epsilon)}\lesssim_M \epsilon^{-\frac{3}{2}-3}$ against the improved irrotational bound $\|\nabla^{\top}v_{\epsilon}\cdot n_\epsilon\|_{H^{k-\frac{1}{2}}(\Gamma_\epsilon)}\lesssim_M\epsilon^{-1}$ we find that for $0\leq\gamma\leq 1$, we have
\begin{equation*}
\|n_1(x_1)-n_\epsilon\|_{H^{k-\frac{1}{2}+\gamma}(\Gamma_\epsilon)}\lesssim_M \epsilon^{-C\gamma},
\end{equation*}
for some constant $C>0$. If $\gamma$ is small enough, this implies the bound $\|\Gamma_1\|_{H^{k+\frac{1}{2}+\gamma}}\lesssim_M \epsilon^{-\frac{1}{2}-\gamma}$, which closes the bootstrap for $\Gamma_1$.
\medskip

Next, we use \Cref{normalapprox} to show that $\tilde{B}_1\cdot n_1=0$. Using the equation for $\tilde{B}_1$ and the fact that $B_{\epsilon}$ is tangent to $\Gamma_{\epsilon}$, we have
\begin{equation*}
\begin{split}
\tilde{B}_1(x_1)\cdot n_1(x_1)&=(\tilde{B}_1(x_1)-B_{\epsilon})\cdot n_1(x_1)+B_{\epsilon}\cdot (n_1(x_1)-n_\epsilon)
\\
&=\epsilon\nabla_{B_{\epsilon}}v_{\epsilon}\cdot n_1(x_1)+B_{\epsilon}\cdot (n_1(x_1)-n_\epsilon).
\end{split}
\end{equation*}
Using \Cref{normalapprox} to expand the latter term on the right-hand side, this implies   (in light of $B_{\epsilon}\cdot n_\epsilon=0$)  that
\begin{equation*}
\tilde{B}_1(x_1)\cdot n_1(x_1)=0,
\end{equation*}
as desired.
\medskip

\textbf{Divergence-free corrections}. Next, we correct $\tilde{B}_1$ and $\tilde{v}_1$ to be divergence-free (while still preserving the tangency of the magnetic field). We define the full iterates $v_1$ and $B_1$ by
\begin{equation*}
v_1:=\tilde{v}_1-\nabla\Delta^{-1}_{\Omega_1}(\nabla\cdot \tilde{v}_1),\hspace{5mm}B_1:=(\tilde{B}_1-\nabla\Delta_{\Omega_1}^{-1}(\nabla\cdot \tilde{B}_1))^{rot},
\end{equation*}
so that $v_1$ and $B_1$ are divergence-free and $B_1$ is tangent to $\Gamma_1$. Next, we show that the errors induced by these corrections are small in $\mathbf{H}^k$. 
We begin with the error estimates for $B_1$, which are more complicated. We first observe the identity
\begin{equation*}\label{Berr}
\tilde{B}_1-B_1=\nabla\Delta^{-1}_{\Omega_1}(\nabla\cdot\tilde{B}_1)+\nabla\mathcal{H}_1\mathcal{N}_1^{-1}(\nabla_n\Delta^{-1}(\nabla\cdot\tilde{B}_1)).
\end{equation*}
Next, we compute from the equation for $\tilde{B}_1$,
\begin{equation}\label{Bdiveqn}
\begin{split}
(\nabla\cdot \tilde{B}_1)(x_1)&=\nabla\cdot (\tilde{B}_1(x_1))-\epsilon\nabla v_{\epsilon}\cdot(\nabla \tilde{B}_1)(x_1)
\\
&=\epsilon\nabla v_{\epsilon}\cdot\nabla (B_1(x_1)-B_{\epsilon})+\epsilon^2\nabla v_{\epsilon}\cdot\nabla v_{\epsilon}\cdot (\nabla \tilde{B}_1)(x_1)
\\
&=\epsilon^2(\nabla v_{\epsilon}\cdot\nabla\nabla_{B_{\epsilon}}v_{\epsilon}+\nabla v_{\epsilon}\cdot\nabla v_{\epsilon}\cdot (\nabla \tilde{B}_1)(x_1)).
\end{split}
\end{equation}
From the equation for $\tilde{B}_1$ and the regularization bounds for $v_{\epsilon}$ and  $\nabla_{B_{\epsilon}}v_{\epsilon}$ from \Cref{Reg bound list}, we see that 
\begin{equation}\label{B1bounds}
\|\tilde{B}_1\|_{H^k(\Omega_1)}\lesssim_M 1,\hspace{5mm}\|\tilde{B}_1\|_{H^{k+\frac{1}{2}}(\Omega_1)}\lesssim_M\epsilon^{-\frac{3}{4}}.
\end{equation}
From the first bound in \eqref{B1bounds} and the above identity for $\nabla\cdot\tilde{B}_1$, we have
\begin{equation*}\label{Bdivbounds}
\|\nabla\cdot \tilde{B}_1\|_{H^{k-1}(\Omega_1)}\lesssim_M\epsilon,\hspace{5mm} \|\nabla\cdot \tilde{B}_1\|_{H^{k-2}(\Omega_1)}\lesssim_M\epsilon^{2}.
\end{equation*}
From this, the bound $\|\Gamma_1\|_{H^{k+\frac{1}{2}}}\lesssim_M\epsilon^{-\frac{1}{2}}$ and the estimates in \Cref{BEE}, we obtain
\begin{equation*}
\|B_1-\tilde{B}_1\|_{H^k(\Omega_1)}\lesssim_M\epsilon,\hspace{5mm}\|B_1-\tilde{B}_1\|_{H^{k-1}(\Omega_1)}\lesssim_M\epsilon^{2},\hspace{5mm}\|B_1\|_{H^k(\Omega_1)}\lesssim_M 1.
\end{equation*}
In order to establish the analogous bounds in the stronger space $\mathbf{H}^k$, we need the following lemma.
\begin{lemma}\label{gradbchain}
Let $0\leq s\leq k-1$. For every $f\in H^{s+1}(\Gamma_1)$ and $g\in H^{s+\frac{3}{2}}(\Omega_{1})$ there holds 
\begin{equation*}
\|(\nabla_{B_1}f)(x_1)-\nabla_{B_{\epsilon}}(f(x_1))\|_{H^s(\Gamma_\epsilon)}\lesssim_M \epsilon^2\|f\|_{H^{s+1}(\Gamma_1)}+\epsilon\|f\|_{H^s(\Gamma_1)}
\end{equation*}
and
\begin{equation*}
\|(\nabla_{B_1}g)(x_1)-\nabla_{B_{\epsilon}}(g(x_1))\|_{H^{s+\frac{1}{2}}(\Omega_\epsilon)}\lesssim_M \epsilon^2\|g\|_{H^{s+\frac{3}{2}}(\Omega_1)}+\epsilon\|g\|_{H^{s+\frac{1}{2}}(\Omega_1)}.
\end{equation*}
\end{lemma}
\begin{proof}
We will prove the first (and slightly harder) case. The latter case will follow similarly. Writing $f=\mathcal{H}_1f$ and using the chain rule and the definition of $\tilde{B}_1$, we have
\begin{equation*}
\begin{split}
\nabla_{B_\epsilon}(f(x_1))&=B_{\epsilon}\cdot(\nabla\mathcal{H}_1f)(x_1)+\epsilon\nabla_{B_{\epsilon}}v_{\epsilon}\cdot (\nabla\mathcal{H}_1f)(x_1)
\\
&=(\nabla_{B_1}f)(x_1)+((\tilde{B}_1-B_1)\cdot\nabla\mathcal{H}_1f)(x_1).
\end{split}
\end{equation*}
Since $\tilde{B}_1=B_1+\mathcal{O}_{H^k(\Omega_1)}(\epsilon)$ and $\tilde{B}_1=B_1+\mathcal{O}_{H^{k-1}(\Omega_1)}(\epsilon^2)$ we may therefore use  Sobolev product estimates and elliptic regularity to obtain the  bound
\begin{equation*}
\|((\tilde{B}_1-B_1)\cdot\nabla\mathcal{H}_1f)(x_1)\|_{H^s(\Gamma_\epsilon)}\lesssim_M \epsilon^2\|f\|_{H^{s+1}(\Gamma_1)}+\epsilon\|f\|_{H^s(\Gamma_1)},
\end{equation*}
as desired.
\end{proof}
Using \Cref{gradbchain}, the second bound in \eqref{B1bounds} and the regularization bound for $\nabla_{B_{\epsilon}}^2v_{\epsilon}$ from \Cref{Reg bound list}, we obtain $\|\tilde{B}_1\|_{\mathbf{H}^k(\Omega_1)}\lesssim_M 1$. Moreover, thanks to  \eqref{Bdiveqn} (and again the bound for $\nabla_{B_{\epsilon}}^2v_{\epsilon}$), we have
\begin{equation*}
\|\nabla\cdot B_1\|_{\mathbf{H}^{k-1}(\Omega_1)}\lesssim_M \epsilon,\hspace{5mm}\|\nabla\cdot B_1\|_{\mathbf{H}^{k-2}(\Omega_1)}\lesssim_M \epsilon^2.
\end{equation*}
Combining the above with the commutator identities in \Cref{Movingsurfid}, the bound $\|B_{\epsilon}\|_{H^{k+\frac{1}{2}}(\Omega_\epsilon)}\lesssim_M\epsilon^{-\frac{3}{4}}$ and the estimates from \Cref{BEE}, we see that
\begin{equation*}
\|\nabla\Delta^{-1}(\nabla\cdot B_1)\|_{\mathbf{H}^k(\Omega_1)}\lesssim_M\epsilon,\hspace{5mm}\|\nabla\Delta^{-1}(\nabla\cdot B_1)\|_{\mathbf{H}^{k-1}(\Omega_1)}\lesssim_M\epsilon^2.
\end{equation*}
Using the div-curl estimate in \Cref{Balanced div-curl}, the regularization bounds for $\Gamma_{1}$ and product rule we also have 
\begin{equation*}
\begin{split}
\|\nabla_{B_{1}}\nabla\mathcal{H}_1\mathcal{N}_1^{-1}(\nabla_n\Delta^{-1}(\nabla\cdot\tilde{B}_1))\|_{H^{k-\frac{1}{2}}(\Omega_1)}&\lesssim_M\epsilon+\|\nabla_{B_{1}}\nabla\mathcal{H}_1\mathcal{N}_1^{-1}(\nabla_n\Delta^{-1}(\nabla\cdot\tilde{B}_1))\cdot n_1\|_{H^{k-1}(\Gamma_1)}
\\
&\lesssim_M \epsilon+\|\nabla_{B_{1}}(\nabla_n\Delta^{-1}(\nabla\cdot\tilde{B}_1))\|_{H^{k-1}(\Gamma_1)}
\\
&\lesssim_M\epsilon.
\end{split}
\end{equation*}
Combining everything above finally yields the bounds
\begin{equation*}
\|B_1-\tilde{B}_1\|_{\mathbf{H}^k(\Omega_1)}\lesssim_M\epsilon,\hspace{5mm}\|B_1\|_{\mathbf{H}^k(\Omega_1)}\lesssim_M 1.
\end{equation*}
Next, we estimate the error between $v_1$ and $\tilde{v}_1$. We compute from the equation for $\tilde{v}_1$,
\begin{equation*}
\begin{split}
(\nabla\cdot \tilde{v}_1)(x_1)&=\nabla\cdot (\tilde{v}_1(x_1))-\epsilon\nabla v_{\epsilon}\cdot(\nabla v_1)(x_1)
\\
&=\epsilon^2\nabla v_{\epsilon}\cdot\nabla v_{\epsilon}\cdot (\nabla v_1)(x_1)-\epsilon\partial_i(v_{1,j}(x_1)-v_{\epsilon,j})\partial_jv_{\epsilon,i}-\epsilon\partial_iB_{\epsilon,j}\partial_jB_{\epsilon,i}+\epsilon\partial_i\tilde{B}_{\epsilon,j}\partial_j\tilde{B}_{\epsilon,i}.
\end{split}
\end{equation*}
From the estimates in \Cref{BEE} and the regularization bounds for $\tilde{B}_{\epsilon}$, $v_{\epsilon}$ and $\Gamma_{\epsilon}$ (this last bound is required to estimate the pressure), we have $\|\tilde{v}_1\|_{H^{k}(\Omega_1)}\lesssim_M 1$ and $\|\tilde{v}_1\|_{H^{k+\frac{1}{2}}(\Omega_1)}\lesssim_M\epsilon^{-\frac{3}{4}}$. Arguing as in the estimate for $\tilde{B}_1$ and using that $\tilde{B}_{\epsilon}=B_{\epsilon}+\mathcal{O}_{H^{k-5}(\Omega_\epsilon)}(\epsilon^2)$, we have
\begin{equation*}
\|\nabla\cdot \tilde{v}_1\|_{\mathbf{H}^{k-1}(\Omega_1)}\lesssim_M\epsilon,\hspace{5mm}\|\nabla\cdot \tilde{v}_1\|_{\mathbf{H}^{k-5}(\Omega_1)}\lesssim_M\epsilon^2
\end{equation*}
and
\begin{equation*}
\|v_1-\tilde{v}_1\|_{\mathbf{H}^k(\Omega_1)}\lesssim_M\epsilon,\hspace{5mm}\|v_1-\tilde{v}_1\|_{\mathbf{H}^{k-4}(\Omega_1)}\lesssim_M\epsilon^2.
\end{equation*}
Combining everything above, we have established property (i) in \Cref{Final transport +Euler} and also the approximate equation 
\begin{equation}\label{approxtransporteqn}
W_1^\pm(x_1)=W_{\epsilon}-\epsilon(\mp\nabla_{B_{\epsilon}}\tilde{W}^\pm_{\epsilon}+\nabla P_{\epsilon})+\mathcal{O}_{\mathbf{H}^k(\Omega_\epsilon)}(\epsilon)
\end{equation}
 in $\mathbf{H}^k$ by adding and subtracting the equations for $\tilde{v}_1$ and $\tilde{B}_1$. 
 \begin{remark}
 We once again emphasize that the latter term on the right-hand side of the approximate equation \eqref{approxtransporteqn} involves the ``uncorrected" variable $\tilde{W}_{\epsilon}^\pm$ whereas the first term involves the corrected variable $W_{\epsilon}^\pm$. As mentioned before, this choice is made to ensure that we can estimate any quadratic errors appearing in the analysis below. It will turn out that $W_{\epsilon}$ and $\nabla_{B_{\epsilon}}\tilde{W}_{\epsilon}$ exhibit a strong enough almost orthogonal structure, which will let us also handle ``first-order" error terms in the requisite energy monotonicity bound.
\end{remark} 
 Now, we verify property (iii)  in \Cref{Final transport +Euler}, which is the bootstrap bound 
 \begin{equation*}
 \|\omega_1^\pm\|_{H^{k}(\Omega_1)}\leq \epsilon^{-\frac{3}{2}}K(M)(1+C(M)\epsilon).
 \end{equation*}
 We can work with the equation for $\tilde{W}_1^\pm$ because $\nabla\times \tilde{W}_1^\pm=\nabla\times W_1$. Using the definition of $\tilde{v}_1$ and $\tilde{B}_1$, we obtain the relation
\begin{equation*}
\tilde{W}^\pm_1(x_1)=W^\pm_{\epsilon}\pm\epsilon\nabla_{B_{\epsilon}}\tilde{W}^\pm_{\epsilon}-\epsilon\nabla P_{\epsilon}.
\end{equation*}
Before estimating $\tilde{\omega}^\pm_1$, by using the above equation and \Cref{Reg bound list}, we record the estimate
\begin{equation*}
\|\tilde{W}_1^\pm\|_{H^{k+1}(\Omega_1)}\lesssim_M \epsilon^{-\frac{3}{2}}.
\end{equation*}
Consequently, 
\begin{equation*}
\omega^\pm_1(x_1)=\omega_{\epsilon}^\pm\pm\epsilon\nabla_{B_{\epsilon}}\omega_{\epsilon}^\pm+\mathcal{O}_{H^{k}(\Omega_\epsilon)}(\epsilon^{-\frac{1}{2}}),
\end{equation*}
which thanks to property (iii) in \Cref{rotregEmon} and a change of variables yields
\begin{equation*}
\|\omega_1^\pm\|_{H^k(\Omega_1)}\leq \epsilon^{-\frac{3}{2}}K(M)(1+C(M)\epsilon),
\end{equation*}
as desired. This establishes property (iii).
\medskip

Next, we work towards establishing the energy monotonicity bound (ii). The main step is to relate the good variables associated to the iterate $W_1$ and $\Gamma_1$ to the good variables associated to $W_{\epsilon}$ and $\Gamma_{\epsilon}$ at the regularity level of the energy. 
Here, we define the ``uncorrected" good variables $\tilde{\mathcal{G}}^\pm_{\epsilon}$ by
\begin{equation*}\label{uncorrectedG}
\tilde{\mathcal{G}_{\epsilon}}^\pm:=\nabla_n\tilde{W}_{\epsilon}^\pm\cdot\nabla P_{\epsilon}-\nabla_n\Delta^{-1}(\Delta\tilde{W}_{\epsilon}^\pm\cdot\nabla P_{\epsilon}+2\nabla\tilde{W}_{\epsilon}^\pm\cdot\nabla^2 P_{\epsilon}).
\end{equation*}
We have the following lemma.
\begin{lemma}[Relations between the good variables]\label{goodvarrel} The following relations hold:
\begin{enumerate}
\item (Relation for $\omega_1$).
\begin{equation*}
\begin{split}
\omega_1^\pm(x_1)-\omega_{\epsilon}^\pm\pm\epsilon\nabla_{B_{\epsilon}}\omega^\pm_{\epsilon}=\mathcal{O}_{H^{k-1}(\Omega_\epsilon)}(\epsilon).
\end{split}
\end{equation*}
\item (Relation for $\nabla_{B_{1}}\omega_1$).
\begin{equation*}
(\nabla_{B_1}\omega_1^\pm)(x_1)-\nabla_{B_{\epsilon}}\omega^\pm_\epsilon \pm\nabla_{B_{\epsilon}}^2\omega_{\epsilon}^\pm=\mathcal{O}_{H^{k-\frac{3}{2}}(\Omega_\epsilon)}(\epsilon) .   
\end{equation*}
\item (Relation for $a_1$).
\begin{equation*}\label{arelation}
a_1(x_1)-a_{\epsilon}-\epsilon D_t a_{\epsilon}=\mathcal{O}_{H^{k-1}(\Gamma_\epsilon)}(\epsilon).
\end{equation*}
\item (Relation for $\nabla_{B_1}a_1$). There exists $F$ with $\|F\|_{H^{k-1}(\Gamma_\epsilon)}\lesssim_M \epsilon$ and $\|\nabla_{B_{\epsilon}}(F-\epsilon D_ta_{\epsilon})\|_{H^{k-\frac{3}{2}}(\Gamma_\epsilon)}\lesssim_M\epsilon^{\frac{1}{2}}$ such that
\begin{equation*}
(\nabla_{B_1}a_1)(x_1)-\nabla_{B_{\epsilon}}a_{\epsilon}-\epsilon \nabla_{B_{\epsilon}}D_ta_{\epsilon}+\nabla_{B_{\epsilon}}F=\mathcal{O}_{H^{k-\frac{3}{2}}(\Gamma_\epsilon)}(\epsilon)    .
\end{equation*}
\item (Relation for $\mathcal{G}^\pm_1$). 
\begin{equation*}
\mathcal{G}_1^\pm(x_1)-\mathcal{G}_{\epsilon}^\pm+\epsilon a_{\epsilon}\mathcal{N}_{\epsilon}a_{\epsilon}=\pm\epsilon\nabla_{B_{\epsilon}}\tilde{\mathcal{G}}^\pm_{\epsilon}+\mathcal{O}_{H^{k-\frac{3}{2}}(\Gamma_\epsilon)}(\epsilon).
\end{equation*}
\item (Relation for $\nabla_{B_{\epsilon}}\mathcal{G}^\pm_1$). 
\begin{equation*}
(\nabla_{B_1}\mathcal{G}_1^\pm)(x_1)-\nabla_{B_{\epsilon}}\mathcal{G}_{\epsilon}^\pm+\epsilon a_{\epsilon}\mathcal{N}_{\epsilon}\nabla_{B_{\epsilon}}a_{\epsilon}=\pm\epsilon\nabla_{B_{\epsilon}}^2\tilde{\mathcal{G}}^\pm_{\epsilon}+\mathcal{O}_{H^{k-2}(\Gamma_\epsilon)}(\epsilon).
\end{equation*}
\end{enumerate}
\end{lemma}
The error term $F$ (which is a technical artifact of the definition of our iteration) should be thought of as playing a similar role to the error term $D^\pm a-\mathcal{G}^\pm$ in \Cref{RB2}.  
\begin{proof} 
The relation for $\omega_1^\pm$ follows immediately from \eqref{approxtransporteqn} by taking the curl of the equation and using that $\nabla\times \tilde{W}^\pm_{\epsilon}=\nabla\times W^\pm_{\epsilon}$. The relation for $\nabla_{B_1}\omega_1^\pm$ follows by taking curl of \eqref{approxtransporteqn}, applying $\nabla_{B_{\epsilon}}$, and using \eqref{gradbchain} together with the bound $\|W_1\|_{H^{k+\frac{1}{2}}(\Omega_1)}\lesssim_M\epsilon^{-1}$.
\medskip

To obtain the estimates (iii) and (iv) we will need  to find a suitable relation between the pressures $P_1$ and $P_{\epsilon}$. For this, we have the following lemma.
\begin{lemma}\label{pressureapproxtrans}
There exists a function $R_{\epsilon}$ such that
\begin{equation*}
P_1(x_1)-P_{\epsilon}-\epsilon D_tP_{\epsilon}+\Delta^{-1}R_{\epsilon}=\mathcal{O}_{\mathbf{H}^{k+\frac{1}{2}}(\Omega_\epsilon)}(\epsilon)
\end{equation*}
where
\begin{equation*}
\|R_{\epsilon}\|_{H^{k-\frac{3}{2}}(\Omega_\epsilon)}\lesssim_M \epsilon,\hspace{5mm} \|\nabla_{B_{\epsilon}}(\Delta^{-1}R_{\epsilon}-\epsilon D_tP_{\epsilon})\|_{H^{k}(\Omega_\epsilon)}\lesssim_M\epsilon^{\frac{1}{2}}.
\end{equation*}
\end{lemma}
The above lemma will ultimately tell us that $R_{\epsilon}$  plays a perturbative role in establishing identity (iii) in \Cref{goodvarrel} but will contribute to the definition of the non-perturbative term $F$ when establishing identity (iv).
\begin{proof}
From the Laplace equation for $P_1$, the dynamic boundary condition, tangency of $B_1$ and the regularity bounds for $\Gamma_1$ and $W^\pm_1$, we have the preliminary bounds $\|P_1\|_{H^{k+\frac{1}{2}}(\Omega_1)}\lesssim_M 1$, $\|\nabla_{B_1}P_1\|_{H^{k}(\Omega_1)}\lesssim_M 1$ and $\|P_1\|_{H^{k+1}(\Omega_1)}\lesssim_{M}\epsilon^{-\frac{3}{4}}$. Moreover, using that $P_1(x_1)-P_{\epsilon}=0$ on $\Gamma_{\epsilon}$, it is easy to see by expanding $\Delta (P_1(x_1)-P_{\epsilon})$ that we have (for instance) the preliminary  bound $\|P_1(x_1)-P_{\epsilon}\|_{H^{k-2}(\Omega_\epsilon)}\lesssim_M\epsilon$. Now, we expand
\begin{equation*}
\Delta (P_1(x_1)-P_{\epsilon})=(\Delta P_1)(x_1)-\Delta P_{\epsilon}+\epsilon\Delta v_{\epsilon}\cdot (\nabla P_{1})(x_1)+2\epsilon\nabla v_{\epsilon}\cdot(\nabla^2 P_1)(x_1)+\epsilon^2\nabla v_{\epsilon}\cdot \nabla v_{\epsilon}\cdot (\nabla^2 P_1)(x_1).
\end{equation*}
We obtain from this and the bounds $\|v_{\epsilon}\|_{H^{k+\frac{1}{2}}(\Omega_\epsilon)}\lesssim_M\epsilon^{-\frac{3}{4}}$ and $\|\nabla_{B_{\epsilon}}v_{\epsilon}\|_{H^{k}(\Omega_\epsilon)}\lesssim_M\epsilon^{-\frac{1}{2}}$ (which follow from \Cref{Reg bound list}) that 
\begin{equation*}\label{startingprelation}
\Delta (P_1(x_1)-P_{\epsilon})=(\Delta P_1)(x_1)-\Delta P_{\epsilon}+\epsilon\Delta v_{\epsilon}\cdot \nabla P_{\epsilon}+2\epsilon\nabla v_{\epsilon}\cdot\nabla^2 P_{\epsilon}+\mathcal{O}_{\mathbf{H}^{k-\frac{3}{2}}(\Omega_\epsilon)}(\epsilon).
\end{equation*}
We expand 
\begin{equation*}
\begin{split}
(\Delta P_1)(x_1)-\Delta P_{\epsilon}&=\partial_iW_{\epsilon,j}^+\partial_jW_{\epsilon,i}^--(\partial_i W_{1,j}^+\partial_jW_{1,i}^-)(x_1)
\\
&=\partial_i(W_{\epsilon,j}^+-W_{1,j}^+(x_1))\partial_jW_{\epsilon,i}^-+\partial_iW_{\epsilon,j}^+\partial_j(W_{\epsilon,i}^--W_{1,i}^-(x_1))+\mathcal{O}_{\mathbf{H}^{k-\frac{3}{2}}(\Omega_\epsilon)}(\epsilon)
\\
&=2\epsilon\nabla^2 P_{\epsilon}\cdot\nabla v_{\epsilon}-\epsilon \nabla\nabla_{B_{\epsilon}}\tilde{W}_{\epsilon}^+\cdot\nabla W^-_{\epsilon}+\epsilon \nabla W_{\epsilon}^+\cdot\nabla \nabla_{B_{\epsilon}}\tilde{W}_{\epsilon}^-+\mathcal{O}_{\mathbf{H}^{k-\frac{3}{2}}(\Omega_\epsilon)}(\epsilon).
\end{split}
\end{equation*}
From the Laplace equation for $D_tP_{\epsilon}$, we see that 
\begin{equation*}
P_1(x_1)-P_{\epsilon}-\epsilon D_tP_{\epsilon}+\Delta^{-1}R_{\epsilon}=\mathcal{O}_{\mathbf{H}^{k+\frac{1}{2}}(\Omega_\epsilon)}(\epsilon)
\end{equation*}
where 
\begin{equation*}
R_{\epsilon}:=-\epsilon \nabla\nabla_{B_{\epsilon}}\tilde{B}_{\epsilon}^{ir}\cdot\nabla W_\epsilon^-+\epsilon \nabla W_{\epsilon}^+\cdot\nabla \nabla_{B_{\epsilon}}\tilde{B}_{\epsilon}^{ir}.
\end{equation*}
As a consequence of the identity $\tilde{B}_{\epsilon}^{ir}=\tilde{B}_{\epsilon}-B_{\epsilon}$, we find that $\|\nabla_{B_{\epsilon}}\tilde{B}_{\epsilon}^{ir}\|_{H^{k-\frac{1}{2}}(\Omega_{\epsilon})}\lesssim_M 1$. Therefore, we have 
\begin{equation*}\label{Restimatestransport}
\|R_{\epsilon}\|_{H^{k-\frac{3}{2}}(\Omega_\epsilon)}\lesssim_M\epsilon.
\end{equation*}
Moreover, by definition of $R_{\epsilon}$ and from the expansions above, it is easy to verify that
\begin{equation*}
\| \nabla_{B_{\epsilon}}(\Delta^{-1}R_{\epsilon}-\epsilon D_tP_{\epsilon})\|_{H^{k}(\Omega_\epsilon)}\lesssim_M\epsilon^{\frac{1}{2}},
\end{equation*}
which completes the proof of the lemma.
\end{proof}
Returning now to the proof of properties (iii) and (iv) in \Cref{goodvarrel}, we use the above lemma to obtain
\begin{equation*}
\begin{split}
(\nabla P_{1})(x_1)&=\nabla P_{\epsilon}+\epsilon\nabla D_t P_{\epsilon}-\epsilon\nabla v_{\epsilon}\cdot(\nabla P_{1})(x_1)-\nabla\Delta^{-1}R_{\epsilon}+\mathcal{O}_{\mathbf{H}^{k-\frac{1}{2}}(\Omega_\epsilon)}(\epsilon)
\\
&=\nabla P_{\epsilon}+\epsilon D_t\nabla P_{\epsilon}-\nabla\Delta^{-1}R_{\epsilon}+\mathcal{O}_{\mathbf{H}^{k-\frac{1}{2}}(\Omega_\epsilon)}(\epsilon).
\end{split}
\end{equation*}
 From this we see that
\begin{equation*}
\begin{split}
a_1(x_1)&=a_{\epsilon}+\epsilon D_ta_{\epsilon}+\nabla_n\Delta^{-1}R_{\epsilon}-(n_{\Gamma_1}(x_1)-n_{\Gamma_\epsilon})\cdot(\nabla P_{1})(x_1)+\mathcal{O}_{\mathbf{H}^{k-1}(\Gamma_{\epsilon})}(\epsilon)
\\
&=a_{\epsilon}+\epsilon D_ta_{\epsilon}+\nabla_n\Delta^{-1}R_{\epsilon}+\mathcal{O}_{\mathbf{H}^{k-1}(\Gamma_{\epsilon})}(\epsilon),
\end{split}
\end{equation*}
where in the last line we used 
$$(n_{\Gamma_1}(x_1)-n_{\Gamma_\epsilon})\cdot (\nabla P_1)(x_1)=-a_1(x_1)(n_{\Gamma_1}(x_1)-n_{\Gamma_\epsilon})\cdot n_{\Gamma_1}(x_1)=-a_1(x_1)\frac{1}{2}|n_{\Gamma_1}(x_1)-n_{\Gamma_\epsilon}|^2=\mathcal{O}_{\mathbf{H}^{k-1}(\Gamma_\epsilon)}(\epsilon).$$ From the bounds for $R_{\epsilon}$, relation (iii) immediately follows. To conclude the proof of relation (iv), we define $F:=-\nabla_n\Delta^{-1}R_{\epsilon}$ and apply \Cref{gradbchain}.
\medskip

Next, we prove the relation for $\mathcal{G}_1^\pm$. We  recall that
\begin{equation*}
\begin{split}
\mathcal{G}^\pm&= \nabla_nW^\pm\cdot\nabla P-\nabla_n\Delta^{-1}(\Delta W^\pm\cdot\nabla P+2\nabla W^\pm\cdot\nabla^2 P)=-n_\Gamma\cdot\mathcal{A}^\pm=-n_\Gamma\cdot (-\nabla W^\pm\cdot\nabla P+\nabla\mathcal{B}^\pm).
\end{split}
\end{equation*}
We see that
\begin{equation}\label{Dtaeqndiff}
\begin{split}
\mathcal{A}_1^\pm(x_1)-\mathcal{A}_{\epsilon}^\pm &=-((\nabla W^\pm_1\cdot\nabla P_{1})(x_1)-\nabla W^\pm_{\epsilon}\cdot\nabla P_{\epsilon})+((\nabla\mathcal{B}_1^\pm)(x_1)-\nabla\mathcal{B}^\pm_{\epsilon})
\\
&=-((\nabla W^\pm_1)(x_1)-\nabla W^\pm_{\epsilon})\cdot\nabla P_{\epsilon}+((\nabla\mathcal{B}_1^\pm)(x_1)-\nabla\mathcal{B}^\pm_{\epsilon})+\mathcal{O}_{\mathbf{H}^{k-1}(\Omega_\epsilon)}(\epsilon).
\end{split}
\end{equation}
To control the second term on the right-hand side above, we write out the Laplace equation for $\mathcal{B}^\pm_1(x_1)$ and use \Cref{gradbchain} together with the relevant elliptic estimates in \Cref{BEE} to obtain
\begin{equation*}
\begin{split}
\Delta (\mathcal{B}^\pm_1(x_1))&=(\Delta \mathcal{B}^\pm_1)(x_1)+\epsilon\Delta W_{\epsilon}^\pm\cdot(\nabla \mathcal{B}^\pm_1)(x_1)+2\epsilon\nabla W_{\epsilon}^\pm\cdot (\nabla^2 \mathcal{B}^\pm_1)(x_1)+\mathcal{O}_{\mathbf{H}^{k-2}(\Omega_\epsilon)}(\epsilon)
\\
&=(\Delta \mathcal{B}^\pm_1)(x_1)+\mathcal{O}_{\mathbf{H}^{k-2}(\Omega_\epsilon)}(\epsilon).
\end{split}
\end{equation*}
From the equation for $W_1^\pm$, the Laplace equation for $P_{\epsilon}$ and the chain rule, we then observe the relation
\begin{equation*}
\begin{split}
(\Delta W_1^\pm)(x_1)&=\Delta (W_1^\pm(x_1))+\mathcal{O}_{\mathbf{H}^{k-2}(\Omega_\epsilon)}(\epsilon)
\\
&=\Delta W_{\epsilon}^\pm-\epsilon\nabla\Delta P_{\epsilon}\pm\epsilon\Delta(\nabla_{B_{\epsilon}}\tilde{W}_{\epsilon}^\pm)+\mathcal{O}_{\mathbf{H}^{k-2}(\Omega_\epsilon)}(\epsilon)
\\
&=\Delta W_{\epsilon}^\pm\pm\epsilon\nabla_{B_{\epsilon}}\Delta \tilde{W}_{\epsilon}^\pm+\mathcal{O}_{\mathbf{H}^{k-2}(\Omega_\epsilon)}(\epsilon).
\end{split}
\end{equation*}
Defining 
\begin{equation*}
\tilde{\mathcal{B}}_\epsilon^\pm:=\Delta^{-1}(\Delta\tilde{W}_{\epsilon}^\pm\cdot\nabla P_{\epsilon}+2\nabla\tilde{W}_{\epsilon}^\pm\cdot\nabla^2 P_{\epsilon})
\end{equation*}
and using \Cref{partitionofBa}, we  obtain the identity 
\begin{equation*}
\begin{split}
(\Delta \mathcal{B}^\pm_1)(x_1)&=\Delta \mathcal{B}^\pm_{\epsilon}+(\Delta W_1^\pm\cdot\nabla P_1)(x_1)-\Delta W^\pm_{\epsilon}\cdot\nabla P_{\epsilon}+2(\nabla W^\pm_1\cdot\nabla ^2 P_1)(x_1)-2(\nabla W^\pm_{\epsilon}\cdot\nabla ^2 P_{\epsilon})
\\
&=\Delta (\mathcal{B}^\pm_{\epsilon}\mp\epsilon\nabla_{B_{\epsilon}}\tilde{\mathcal{B}}^\pm_{\epsilon})+2\left(\nabla W^\pm_{\epsilon}\cdot ((\nabla^2 P_1)(x_1)-\nabla^2P_{\epsilon})\right)+\mathcal{O}_{\mathbf{H}^{k-2}(\Omega_\epsilon)}(\epsilon)
\\
&=\Delta (\mathcal{B}^\pm_{\epsilon}\mp\epsilon\nabla_{B_{\epsilon}}\tilde{\mathcal{B}}^\pm_{\epsilon})+\mathcal{O}_{\mathbf{H}^{k-2}(\Omega_\epsilon)}(\epsilon),
\end{split}    
\end{equation*}
where in the last line, we used  the relation $\epsilon D_tP_{\epsilon}-\Delta^{-1}R_{\epsilon}=\mathcal{O}_{\mathbf{H}^{k}(\Omega_\epsilon)}(\epsilon)$. Combining the above with \eqref{Dtaeqndiff} we obtain, by elliptic regularity,
\begin{equation*}
\mathcal{A}^\pm_1(x_1)-\mathcal{A}^\pm_{\epsilon}\mp\epsilon\nabla_{B_{\epsilon}}\nabla \tilde{\mathcal{B}}^\pm_{\epsilon}=((\nabla W^\pm_1)(x_1)-\nabla W^\pm_{\epsilon})\cdot\nabla P_{\epsilon}+\mathcal{O}_{\mathbf{H}^{k-1}(\Omega_\epsilon)}(\epsilon).
\end{equation*}
Then, since
\begin{equation*}
(\mathcal{A}^\pm_1)(x_1)\cdot (n_{\Gamma_1}(x_1)-n_{\Gamma_\epsilon})=\mathcal{O}_{\mathbf{H}^{k-\frac{3}{2}}(\Gamma_\epsilon)}(\epsilon),    
\end{equation*}
it follows that 
\begin{equation*}
\begin{split}
\mathcal{G}^\pm_1(x_1)-\mathcal{G}^\pm_{\epsilon}\mp\epsilon\nabla_{B_{\epsilon}}(\nabla_n\tilde{\mathcal{B}}^\pm_{\epsilon})&=-a_{\epsilon}\nabla_n(W_1^\pm(x_1)-W^\pm_{\epsilon})\cdot n_{\Gamma_\epsilon}+ \mathcal{O}_{\mathbf{H}^{k-\frac{3}{2}}(\Gamma_\epsilon)}(\epsilon)
\\
&=\epsilon a_{\epsilon}\nabla_n\nabla P_{\epsilon}\cdot n_{\Gamma_\epsilon}\mp\epsilon \nabla_n\nabla_{B_{\epsilon}}\tilde{W}_{\epsilon}^\pm\cdot \nabla P_{\epsilon}+\mathcal{O}_{\mathbf{H}^{k-\frac{3}{2}}(\Gamma_\epsilon)}(\epsilon)
\\
&=\epsilon a_{\epsilon}\mathcal{N}_{\epsilon}\nabla P_{\epsilon}\cdot n_{\Gamma_\epsilon}\mp\epsilon\nabla_{B_{\epsilon}}(\nabla_n \tilde{W}_\epsilon^\pm\cdot \nabla P_{\epsilon})+\mathcal{O}_{\mathbf{H}^{k-\frac{3}{2}}(\Gamma_\epsilon)}(\epsilon),
\end{split}
\end{equation*}
where above we used the bounds $\|\nabla_{B_{\epsilon}}^2 P_{\epsilon}\|_{H^{k-\frac{1}{2}}(\Omega_{\epsilon})}+\|\nabla_{B_{\epsilon}}^2n_{\Gamma_{\epsilon}}\|_{H^{k-2}(\Gamma_{\epsilon})}\lesssim_M 1$. This gives
\begin{equation*}
\mathcal{G}^\pm_1(x_1)-\mathcal{G}^\pm_{\epsilon}\mp\epsilon\nabla_{B_{\epsilon}}\tilde{\mathcal{G}}^\pm_{\epsilon}=\epsilon a_{\epsilon}\mathcal{N}_{\epsilon}\nabla P_{\epsilon}\cdot n_{\Gamma_\epsilon}+\mathcal{O}_{\mathbf{H}^{k-\frac{3}{2}}(\Gamma_\epsilon)}(\epsilon).
\end{equation*}
Finally, noting that $\mathcal{N}_{\epsilon}n_{\Gamma_\epsilon}\cdot n_{\Gamma_\epsilon}$ is lower order, we have, thanks to the Leibniz rule for $\mathcal{N}_{\epsilon}$,
\begin{equation*}
\begin{split}
\epsilon a_{\epsilon}\mathcal{N}_{\epsilon}\nabla P_{\epsilon}\cdot n_{\Gamma_\epsilon} &=-\epsilon a_{\epsilon}\mathcal{N}_{\epsilon}(n_{\Gamma_\epsilon} a_{\epsilon})\cdot n_{\Gamma_\epsilon}=-\epsilon a_{\epsilon}\mathcal{N}_{\epsilon}a_{\epsilon}+\mathcal{O}_{\mathbf{H}^{k-\frac{3}{2}}(\Gamma_\epsilon)}(\epsilon).
\end{split}
\end{equation*}
Therefore,  the desired relation (v) for $\mathcal{G}^\pm$ holds. Relation (vi) follows from (v) and \Cref{gradbchain} as well as the Leibniz rule for $\mathcal{N}_{\epsilon}$.
\end{proof}
\textbf{Energy monotonicity}. To finish the proof of \Cref{Final transport +Euler}, it remains to establish energy monotonicity. The following lemma will allow us to more easily work with the relations in \Cref{goodvarrel}.
\begin{lemma}\label{changeofvar}
Define the ``pulled-back" energy $E^k_{*}(v_1, B_1,\Gamma_{1}):=\sum_{\alpha\in\{+,-\}}E^k_{*,\pm,1}+E^k_{*,\pm,2}+E^k_{*,\pm,3}$ by
\begin{equation*}
E^k_{*,\pm,1}=1+\|\omega_1^\pm(x_1)\|_{H^{k-1}(\Omega_{\epsilon})}^2+\|(\nabla_{B_1}\omega_1^\pm)(x_1)\|_{H^{k-\frac{3}{2}}(\Omega_{\epsilon})}^2+\|W^\pm_1(x_1)\|_{L^2(\Omega_{\epsilon})}^2
\end{equation*}
and
\begin{equation*}
E^k_{*,\pm,2}=\|a_1^{\frac{1}{2}}(x_1)\mathcal{N}_{\epsilon}^{k-1}(a_1(x_1))\|_{L^2(\Gamma_{\epsilon})}^2+\|\nabla\mathcal{H}_{\epsilon}\mathcal{N}_{\epsilon}^{k-2}(\mathcal{G}^{\pm}_1(x_1))\|_{L^2(\Omega_{\epsilon})}^2
\end{equation*}
and
\begin{equation*}
E^k_{*,\pm,3}=\|\nabla\mathcal{H}_{\epsilon}\mathcal{N}_{\epsilon}^{k-2}((\nabla_{B_1}a_1)(x_1))\|_{L^2(\Omega_{\epsilon})}^2+\|a_1^{-\frac{1}{2}}(x_1)\mathcal{N}_{\epsilon}^{k-2}(\nabla_{B_1}\mathcal{G}_1^{\pm})(x_1)\|_{L^2(\Gamma_{\epsilon})}^2.
\end{equation*}
Then we have the relation
\begin{equation*}
E^{k}(v_1,B_1,\Gamma_1)\leq E_{*}^{k}(v_1,B_1,\Gamma_{1})+C(M)\epsilon.
\end{equation*}
\end{lemma}
Before proving the above lemma, we show how it implies the desired energy monotonicity bound. In light of \Cref{changeofvar}, it suffices to establish the bound
\begin{equation*}
E^{k}_{*}(v_1,B_1,\Gamma_1)\leq (1+C(M)\epsilon)E^{k}(v_0,B_0,\Gamma_0)+C(M)\epsilon.
\end{equation*}
From the definition of $W_1$ and parts (iii)-(iv) of \Cref{rotregEmon}, we easily obtain
\begin{equation*}
E^k_{*,\pm,1}\leq (1+C(M)\epsilon)E^k_{\pm,1}(v_0,B_0,\Gamma_0)+C(M)\epsilon.
\end{equation*}
Next, we turn to the surface components of the energy. We need the following lemma to exploit the approximate orthogonality between $\mathcal{G}^\pm_{\epsilon}$ and $\nabla_{B_{\epsilon}}\tilde{\mathcal{G}}^\pm_{\epsilon}$. 
\begin{lemma}\label{orthogonality} There holds
\begin{equation*}
\left\lvert\sum_{\alpha\in\{+,-\}}\langle\mathcal{N}^{k-\frac{3}{2}}\mathcal{G}^{\alpha}_{\epsilon},\mathcal{N}^{k-\frac{3}{2}}\nabla_{B_{\epsilon}}\tilde{\mathcal{G}}^\alpha_{\epsilon}\rangle_{L^2(\Gamma_{\epsilon})}\right\rvert+\left\lvert\sum_{\alpha\in\{+,-\}}\langle\mathcal{N}^{k-2}\nabla_{B_{\epsilon}}\mathcal{G}^\alpha_{\epsilon},\mathcal{N}^{k-2}\nabla_{B_{\epsilon}}^2\tilde{\mathcal{G}}^\alpha_{\epsilon}\rangle_{L^2(\Gamma_{\epsilon})}\right\rvert\lesssim_M 1.
\end{equation*}
\end{lemma}
\begin{proof}
First, we observe that (up to simply changing $W^\pm_{\epsilon}$ to $\tilde{W}^\pm_{\epsilon}$ in the proof) \Cref{Gdecomp2} gives us the bound
\begin{equation*}
\|\nabla_{B_{\epsilon}}(\tilde{\mathcal{G}}^\pm_{\epsilon}-\mathcal{N}\tilde{W}_{\epsilon}^\pm\cdot\nabla P_{\epsilon})\|_{\mathbf{H}^{k-\frac{3}{2}}(\Gamma_{\epsilon})}\lesssim_M 1.
\end{equation*}
Moreover, by definition of $\mathcal{\tilde{G}^\pm}_\epsilon$ and the fact that $\Delta W_{\epsilon}^\pm=\Delta\tilde{W}_{\epsilon}^\pm$, we have
\begin{equation*}
\tilde{\mathcal{G}}_{\epsilon}^\pm-\mathcal{G}^\pm_{\epsilon}=\pm\nabla_n\tilde{B}_{\epsilon}^{ir}\cdot\nabla P_{\epsilon}+\mathcal{O}_{\mathbf{H}^{k-\frac{3}{2}}(\Gamma_\epsilon)}(\epsilon^{\frac{1}{2}})=\pm\mathcal{N}_{\epsilon}\tilde{B}^{ir}_\epsilon\cdot \nabla P_{\epsilon}+\mathcal{O}_{\mathbf{H}^{k-\frac{3}{2}}(\Gamma_\epsilon)}(\epsilon^{\frac{1}{2}})
\end{equation*}
and also from \Cref{Reg bound list} and \Cref{partitionofBa} we can verify (where importantly we note that this bound only holds for $\tilde{\mathcal{G}}_{\epsilon}^{\alpha}$ and not $\mathcal{G}_{\epsilon}^{\alpha}$)
\begin{equation*}
\|\nabla_{B_{\epsilon}}\tilde{\mathcal{G}}^\alpha_{\epsilon}\|_{\mathbf{H}^{k-\frac{3}{2}}(\Gamma_{\epsilon})}\lesssim_M \epsilon^{-\frac{1}{2}}.
\end{equation*}
From the above three estimates and the identity $\tilde{W}^+_{\epsilon}-\tilde{W}^-_{\epsilon}=2\tilde{B}_{\epsilon}$, it suffices to establish the bounds
\begin{equation*}\label{transportcancellation}
\mathcal{K}^j_{\epsilon}:=\left\lvert\langle\mathcal{N}^{k-\frac{j+3}{2}}\nabla_{B_{\epsilon}}^{j+1}(\mathcal{N}_{\epsilon}\tilde{B}_{\epsilon}\cdot \nabla P_{\epsilon}),\mathcal{N}^{k-\frac{j+3}{2}}\nabla_{B_{\epsilon}}^j(\mathcal{N}_{\epsilon}\tilde{B}^{ir}_\epsilon\cdot\nabla P_{\epsilon})\rangle_{L^2(\Gamma_{\epsilon})}\right\rvert\lesssim_M 1,\hspace{5mm}j\in\{0,1\}.
\end{equation*}
The key observation is that the two terms in the inner product above should be almost orthogonal. This is because, morally speaking, we should expect an approximation like
\begin{equation*}
\mathcal{N}_{\epsilon}\tilde{B}_{\epsilon}\cdot\nabla P_{\epsilon}\approx \mathcal{N}_{\epsilon}\tilde{B}^{ir}_{\epsilon}\cdot\nabla P_{\epsilon}. 
\end{equation*}
Due to the limited regularity of the free surface, this is not quite true (in the relevant norms). However, it turns out that we can still exhibit a strong enough cancellation by relying more heavily on the precise form of $\tilde{B}_{\epsilon}$. To see this, we begin as in the proof of \Cref{rotregEmon} by introducing the notation (here $B_{\epsilon}^h$ and $B^l$ appear in place of $v_{\epsilon}^h$ and $v^l$ and $B$ refers to the state one achieves after step 2 but before step 3)
\begin{equation*}
Q^l_j:=\mathcal{N}\nabla_{B_{\epsilon}^l}^jB^l\cdot\nabla P_{\epsilon},\hspace{5mm}Q^h_j:=\mathcal{N}\nabla_{B_{\epsilon}}^jB_{\epsilon}^h\cdot\nabla P_{\epsilon},\hspace{5mm} 0\leq j\leq 5.
\end{equation*}
Using the definition of $\tilde{B}_{\epsilon}$ and the bound $\|\nabla\Delta^{-1}\nabla\cdot \tilde{B}_{\epsilon}\|_{\mathbf{H}^{k}(\Omega_{\epsilon})}\lesssim_M\epsilon^{\frac{1}{2}}$ (see \Cref{divergenceremarkB}), we have
\begin{equation*}
\tilde{B}_{\epsilon}^{ir}\cdot n_\epsilon=-\epsilon^2\nabla_{B_{\epsilon}}^4B_{\epsilon}^h\cdot n_\epsilon+\mathcal{O}_{\mathbf{H}^{k-\frac{1}{2}}(\Gamma_\epsilon)}(\epsilon^{\frac{1}{2}}).
\end{equation*}
Consequently, for $j=0,1$, we may write
\begin{equation*}
\nabla_{B_{\epsilon}}^j\mathcal{N}_{\epsilon}(\tilde{B}_{\epsilon}^{ir}\cdot \nabla P_{\epsilon})=-\epsilon^2Q_{4+j}^h+\mathcal{O}_{H^{k-\frac{3+j}{2}}(\Gamma_\epsilon)}(\epsilon^{\frac{1}{2}}).
\end{equation*}
We then observe that
\begin{equation*}
\nabla_{B_{\epsilon}}^{j+1}(\mathcal{N}_{\epsilon}\tilde{B}_{\epsilon}\cdot \nabla P_{\epsilon})=(\mathcal{N}_{\epsilon}\nabla_{B_{\epsilon}}^{j+1}\tilde{B}_{\epsilon}\cdot \nabla P_{\epsilon})+\mathcal{O}_{H^{k-\frac{3+j}{2}}(\Gamma_\epsilon)}(1)=Q_{1+j}^l+Q_{1+j}^h+\mathcal{O}_{H^{k-\frac{3+j}{2}}(\Gamma_\epsilon)}(1).
\end{equation*}
From the above, we conclude that 
\begin{equation*}
\mathcal{K}_{\epsilon}^j\lesssim_M 1+\epsilon^2\left\lvert\langle \mathcal{N}^{k-\frac{j+3}{2}}Q_{j+1}^l,\mathcal{N}^{k-\frac{j+3}{2}}Q_{4+j}^h\rangle\right\rvert+\epsilon^2\left\lvert\langle \mathcal{N}^{k-\frac{j+3}{2}}Q_{j+1}^h,\mathcal{N}^{k-\frac{j+3}{2}}Q_{4+j}^h\rangle\right\rvert.
\end{equation*}
Arguing similarly to the proof of property (v) in \Cref{rotregEmon} (though the proof here is quite a bit simpler), we obtain the desired almost orthogonality type bound for $\mathcal{K}_{\epsilon}^j$. We omit the details  of this verification since they are straightforward modifications of the aforementioned argument.
\end{proof}
We now return to establishing the energy monotonicity for the surface components of the energy. Using the definition of $\tilde{\mathcal{G}}^\pm_{\epsilon}$, the identity 
\begin{equation*}
D_ta_{\epsilon}=\mathcal{N}v_{\epsilon}\cdot\nabla P_{\epsilon}+\mathcal{O}_{H^{k-1}(\Gamma_\epsilon)}(1)
\end{equation*}
and the regularization bounds in \Cref{Reg bound list}, we deduce the regularization bounds 
\begin{equation}\label{uncorrectedbound}
\|\nabla_{B_{\epsilon}}\tilde{\mathcal{G}}^\pm_{\epsilon}\|_{H^{k-\frac{3}{2}}(\Gamma_\epsilon)}\lesssim_M\epsilon^{-\frac{1}{2}},\hspace{5mm}\|D_ta_{\epsilon}\|_{H^{k-1}(\Gamma_\epsilon)}\lesssim_M\epsilon^{-\frac{1}{2}},\hspace{5mm}\|a_{\epsilon}\|_{H^{k-\frac{1}{2}}(\Gamma_\epsilon)}\lesssim_M\epsilon^{-\frac{1}{2}}.
\end{equation}
Such bounds are essential for estimating quadratic error terms when comparing $E_{*,2}^k(v_1,B_1,\Gamma_1)$ with the energy $E_2^k(v_{\epsilon},B_{\epsilon},\Gamma_{\epsilon}):=E_{+,2}^k(v_{\epsilon},B_{\epsilon},\Gamma_{\epsilon})+E_{-,2}^k(v_{\epsilon},B_{\epsilon},\Gamma_{\epsilon})$. Indeed, by combining \eqref{uncorrectedbound} with \Cref{goodvarrel} and \Cref{orthogonality}, we have 
\begin{equation*}
\begin{split}
E^k_{*,2}(v_1,B_1,\Gamma_1)&\leq E^k_{2}(v_{\epsilon},B_{\epsilon},\Gamma_{\epsilon})+\sum_{\alpha\in\{+,-\}}\epsilon\int_{\Gamma_{\epsilon}}a_{\epsilon}\mathcal{N}_{\epsilon}^{k-1}a_{\epsilon}\mathcal{N}^{k-1}D_t^{\alpha}a_{\epsilon}dS
\\
&-\sum_{\alpha\in\{+,-\}}\epsilon\int_{\Gamma_{\epsilon}}a_{\epsilon}\mathcal{N}^{k-\frac{3}{2}}\mathcal{G}^{\alpha}_{\epsilon}a\mathcal{N}^{k-\frac{1}{2}}a_{\epsilon}dS+C(M)\epsilon.
\end{split}
\end{equation*}
To estimate the second and third term above, we will simply carry out a discrete version of the corresponding energy estimate in \Cref{HEB}. By self-adjointness of $\mathcal{N}_{\epsilon}$ and the commutator estimate for $[\mathcal{N}^{\frac{1}{2}},a_{\epsilon}]$ from \Cref{Leibnizcom}, we see that
\begin{equation*}
\begin{split}
E^k_{*,2}(v_1,B_1,\Gamma_1)&\leq E^k_{2}(v_{\epsilon},B_{\epsilon},\Gamma_{\epsilon})+C(M)\epsilon+\sum_{\alpha\{+,-\}}C(M)\epsilon\|\mathcal{G}_{\epsilon}^{\alpha}-D_t^{\alpha}a_{\epsilon}\|_{H^{k-1}(\Gamma_\epsilon)}
\\
&\leq E^k_{2}(v_{\epsilon},B_{\epsilon},\Gamma_{\epsilon})+C(M)\epsilon.
\end{split}
\end{equation*}
It remains to prove the energy monotonicity bound for $E_{*,3}^k$. Before proceeding, we must first collect the analogues of \eqref{uncorrectedbound} which will aid us in controlling the resulting quadratic error terms. We observe that thanks to \Cref{goodvarrel}, \Cref{Reg bound list} and  the definition of $\mathcal{\tilde{G}}^\pm_{\epsilon}$,  we have
\begin{equation*}
\|\nabla_{B_{\epsilon}}^2\tilde{\mathcal{G}}^\pm_{\epsilon}\|_{H^{k-2}(\Gamma_\epsilon)}\lesssim_M\epsilon^{-\frac{1}{2}},\hspace{5mm} \|\nabla_{B_{\epsilon}}(\epsilon^{-1}F-D_ta_{\epsilon})\|_{H^{k-\frac{3}{2}}(\Gamma_\epsilon)}\lesssim_M\epsilon^{-\frac{1}{2}}.
\end{equation*}
Moreover, arguing as in the proof of \Cref{technicalprop} and using the irrotational regularization bound $\|\mathcal{N}_{\epsilon}B_{\epsilon}\cdot n_\epsilon\|_{H^{k-1}(\Gamma_\epsilon)}\lesssim_M\epsilon^{-\frac{1}{2}}$ from \Cref{Reg bound list}, we conclude that
\begin{equation*}
\|\nabla_{B_{\epsilon}}a_{\epsilon}\|_{H^{k-1}(\Gamma_\epsilon)}\lesssim_M\epsilon^{-\frac{1}{2}}.
\end{equation*}
Therefore, arguing as in the above bound for $E^k_{*,2}$, we see that
\begin{equation*}
\begin{split}
E^k_{*,3}(v_1,B_1,\Gamma_1)&\leq E^k_3(v_{\epsilon},B_{\epsilon},\Gamma_{\epsilon})+\sum_{\alpha\in\{+,-\}}\epsilon\int_{\Gamma_{\epsilon}}\mathcal{N}^{k-\frac{3}{2}}\nabla_{B_{\epsilon}}a_{\epsilon}\mathcal{N}^{k-\frac{3}{2}}\nabla_{B_{\epsilon}}(D_t^{\alpha}a_{\epsilon}-\mathcal{G}_{\epsilon}^{\alpha}-\epsilon^{-1}F)dS+C(M)\epsilon.
\end{split}
\end{equation*}
By using the identity 
\begin{equation*}
\langle\mathcal{N}_{\epsilon}^{\frac{1}{2}}f,\mathcal{N}_{\epsilon}^{\frac{1}{2}}f\rangle_{L^2(\Gamma_{\epsilon})}=\langle\nabla\mathcal{H}_{\epsilon}f,\nabla\mathcal{H}_{\epsilon}f\rangle_{L^2(\Omega_{\epsilon})}
\end{equation*}
and commuting $\nabla_{B_{\epsilon}}$ with $\nabla\mathcal{H}_{\epsilon}\mathcal{N}_{\epsilon}^{k-2}$ and integrating by parts as in \Cref{RB2}, we have
\begin{equation*}
\begin{split}
\epsilon\int_{\Gamma_{\epsilon}}\mathcal{N}^{k-\frac{3}{2}}\nabla_{B_{\epsilon}}a_{\epsilon}\mathcal{N}^{k-\frac{3}{2}}\nabla_{B_{\epsilon}}(D_t^{\alpha}a_{\epsilon}-&\mathcal{G}_{\epsilon}^{\alpha}-\epsilon^{-1} F)dS
\\
&\leq C(M)\epsilon+\epsilon\int_{\Gamma_{\epsilon}}\mathcal{N}^{k-2}\nabla_{B_{\epsilon}}^2a_{\epsilon}\mathcal{N}^{k-1}(D_t^{\alpha}a_{\epsilon}-\mathcal{G}_{\epsilon}^{\alpha}-\epsilon^{-1} F)dS.
\end{split}
\end{equation*}
By Cauchy-Schwarz and \Cref{partitionofBa}, we obtain
\begin{equation*}
\begin{split}
\epsilon\int_{\Gamma_{\epsilon}}\mathcal{N}^{k-\frac{3}{2}}\nabla_{B_{\epsilon}}a_{\epsilon}\mathcal{N}^{k-\frac{3}{2}}\nabla_{B_{\epsilon}}(D_t^{\alpha}a_{\epsilon}-\mathcal{G}_{\epsilon}^{\alpha}-\epsilon^{-1} F)dS\leq C(M)\epsilon.
\end{split}
\end{equation*}
This establishes the energy monotonicity bound. It finally now remains to establish \Cref{changeofvar}.
\begin{proof}[Proof of \Cref{changeofvar}]
By a simple change of variables, it is clear that 
\begin{equation*}
E^k_{\pm,1}(v_1,B_1,\Gamma_1)\leq E^k_{*,\pm,1}(v_1,B_1,\Gamma_1)+C(M)\epsilon.
\end{equation*}
The main difficulty is in dealing with the surface components of the energy. For this, we need the following proposition which was proved in \cite{Euler}.
\begin{proposition}\label{changeofvar2}
Let $-\frac{1}{2}\leq s\leq k-2$ and let $f\in H^{s+1}(\Gamma_1)$. Then we have the following bound on $\Gamma_{\epsilon}$\emph{:}
\begin{equation*}
\|(\mathcal{N}_{1}f)(x_1)-\mathcal{N}_{\epsilon}(f(x_1))\|_{H^s(\Gamma_{\epsilon})}\lesssim_M \epsilon\|f\|_{H^{s+1}(\Gamma_1)}.   
\end{equation*}
\end{proposition}
Returning to the proof of \Cref{changeofvar}, we note that
\begin{equation*}
\begin{split}
\|(\mathcal{N}^{k-1}_1a_1)(x_1)-\mathcal{N}^{k-1}_{\epsilon}(a_1(x_1))\|_{L^2(\Gamma_{\epsilon})}&\lesssim \|\mathcal{N}_{\epsilon}(\mathcal{N}^{k-2}_1a_1)(x_1)-\mathcal{N}^{k-1}_{\epsilon}(a_1(x_1))\|_{L^2(\Gamma_{\epsilon})} 
\\
&+\|\mathcal{N}_{\epsilon}(\mathcal{N}^{k-2}_1a_1)(x_1)-(\mathcal{N}^{k-1}_1a_1)(x_1)\|_{L^2(\Gamma_{\epsilon})}.
\end{split}
\end{equation*}
Applying \Cref{changeofvar2} to the term in the second line and using the $H^1\to L^2$ bound for $\mathcal{N}$, we have
\begin{equation*}
\|(\mathcal{N}^{k-1}_1a_1)(x_1)-\mathcal{N}^{k-1}_{\epsilon}(a_1(x_1))\|_{L^2(\Gamma_{\epsilon})}\lesssim_M \|(\mathcal{N}^{k-2}_1a_1)(x_1)-\mathcal{N}^{k-2}_{\epsilon}(a_1(x_1))\|_{H^1(\Gamma_{\epsilon})}+\mathcal{O}_M(\epsilon).     
\end{equation*}
Iterating this procedure and applying \Cref{changeofvar2} $k-2$ times, we see that 
\begin{equation*}
\|(\mathcal{N}^{k-1}_1a_1)(x_1)-\mathcal{N}^{k-1}_{\epsilon}(a_1(x_1))\|_{L^2(\Gamma_{\epsilon})}\lesssim_M\epsilon   . 
\end{equation*}
It follows from the above and a change of variables that we have
\begin{equation*}
\|a_1^{\frac{1}{2}}\mathcal{N}_1^{k-1}a_1\|_{L^2(\Gamma_1)}^2\leq \|a_1^{\frac{1}{2}}(x_1)\mathcal{N}_{\epsilon}^{k-1}(a_1(x_1))\|_{L^2(\Gamma_{\epsilon})}^2+\mathcal{O}_M(\epsilon).
\end{equation*}
A similar argument can be used to show that
\begin{equation*}
\|a^{-\frac{1}{2}}_1(x_1)(\mathcal{N}_1^{k-2}\nabla_{B_{1}}\mathcal{G}_1^\pm)(x_1)-a_{\epsilon}^{-\frac{1}{2}}\mathcal{N}_{\epsilon}^{k-2}\nabla_{B_{\epsilon}}\mathcal{G}_{\epsilon}^\pm\|_{L^2(\Gamma_\epsilon)}\lesssim_M\epsilon.   
\end{equation*}
To conclude the proof of \Cref{changeofvar}, we also need to show that
\begin{equation*}\label{Dtachangebound}
\|\nabla\mathcal{H}_1\mathcal{N}_1^{k-2}\mathcal{G}^\pm_1\|_{L^2(\Omega_{1})}^2\leq \|\nabla\mathcal{H}_{\epsilon}\mathcal{N}_{\epsilon}^{k-2}(\mathcal{G}_1^\pm(x_1))\|_{L^2(\Omega_{\epsilon})}^2+\mathcal{O}_M(\epsilon).
\end{equation*}
From a change of variables, we see that
\begin{equation*}
\|\nabla\mathcal{H}_1(\mathcal{N}_1^{k-2}\mathcal{G}_1^\pm)\|_{L^2(\Omega_{1})}^2- \|\nabla\mathcal{H}_{\epsilon}\mathcal{N}_{\epsilon}^{k-2}(\mathcal{G}_1^\pm(x_1))\|_{L^2(\Omega_{\epsilon})}^2\lesssim_M \mathcal{J}+\mathcal{O}_M(\epsilon),
\end{equation*}
where
\begin{equation*}
\mathcal{J}:=\|(\nabla\mathcal{H}_1\mathcal{N}_1^{k-2}\mathcal{G}_1^\pm)(x_1)-\nabla\mathcal{H}_{\epsilon}\mathcal{N}_{\epsilon}^{k-2}(\mathcal{G}_1^\pm(x_1))\|_{L^2(\Omega_{\epsilon})}.  
\end{equation*}
By elliptic regularity, it is easy to verify the bound
\begin{equation*}
\mathcal{J}\lesssim_M \|(\mathcal{N}_1^{k-2}\mathcal{G}_1^\pm)(x_1)-\mathcal{N}_{\epsilon}^{k-2}(\mathcal{G}_1^\pm(x_1))\|_{H^{\frac{1}{2}}(\Gamma_{\epsilon})}+\mathcal{O}_M(\epsilon).    
\end{equation*}
From here, we use \Cref{changeofvar2} similarly to the other surface term in the energy to estimate
\begin{equation*}
\|(\mathcal{N}_1^{k-2}(\mathcal{G}_1^\pm))(x_1)-\mathcal{N}_{\epsilon}^{k-2}(\mathcal{G}_1^\pm(x_1))\|_{H^{\frac{1}{2}}(\Gamma_{\epsilon})}\lesssim_M\epsilon.    
\end{equation*}
A similar argument can be used to deal with the remaining energy component. This completes the proof of \Cref{Final transport +Euler}, and, therefore,  the proof of \Cref{onestepiteration}.
\end{proof}

\subsection{Convergence of the iteration scheme}\label{COTS} 
We are now ready to use  Theorem~\ref{onestepiteration} 
 to prove the existence of regular solutions.
\begin{theorem}\label{t:existence}
Let $k$ be a sufficiently large integer and let $M > 0$.  There exists  a time 
$T = T(M)$ such that for all initial data $(v_0, B_0,\Gamma_0)\in\mathbf{H}^k$  satisfying  $\|(v_0, B_0,\Gamma_0)\|_{\mathbf{H}^k}\leq M$ there exists a unique solution $(v, B, \Gamma)$ to the free boundary MHD equations on the time interval $[0,T]$ with the given initial data and the following regularity properties:
\begin{equation*}
(v, B,\Gamma) \in L^\infty([0,T]; \mathbf H^k) \cap  C([0,T]; \mathbf H^{k-1})   
\end{equation*}
with the uniform bound
\begin{equation*}
\|(v, B, \Gamma)(t)\|_{ \mathbf H^k} \lesssim_M 1, \qquad t \in [0,T].  
\end{equation*}
\end{theorem}
\begin{remark}
Note that the uniqueness of the  solution in \Cref{t:existence} is an immediate consequence of 
\Cref{t:unique}.  The continuity of the solution in the $\mathbf H^k$ topology  will be established in \Cref{RS}.
\end{remark}
\begin{proof}
Given initial data $(v_0, B_0, \Gamma_0) \in \mathbf H^k$ with 
$\Gamma_0 \in \Lambda_*:= \Lambda(\Gamma_*,\epsilon_0,\delta)$,
for each small time step $\epsilon$ we will construct a discrete
approximate solution $(v_\epsilon, B_{\epsilon},\Gamma_\epsilon)$
which is defined at times $t = 0, \epsilon, 2\epsilon, \dots, j\epsilon$ with $j\approx_M \epsilon^{-1}$. We will proceed as follows:
\begin{enumerate}
\item We construct an initialization $(v_\epsilon(0), B_{\epsilon}(0), \Gamma_\epsilon(0))$ by using  \Cref{envbounds}. 

\item We define the approximate solutions $(v_\epsilon(j \epsilon), B_{\epsilon}(j\epsilon), \Gamma_\epsilon(j\epsilon))$ by applying
 the iteration step in Theorem~\ref{onestepiteration} inductively.
\end{enumerate}
We will rely on the energy monotonicity relation and the coercivity property in Theorem~\ref{Energy est. thm}  to control the growth of the $\mathbf H^k$ norms of $(v_\epsilon, B_{\epsilon}, \Gamma_\epsilon)$.
At time $t = 0$, we may use the energy coercivity property to  bound
\[
E^k(v_\epsilon(0), B_{\epsilon}(0), \Gamma_\epsilon(0)) \leq C_1(M). 
\]
We then iterate for as long as
\begin{equation}\label{stop-iteration}
\begin{aligned}
& E^k(v_\epsilon(j\epsilon), B_{\epsilon}(j\epsilon), \Gamma_\epsilon(j\epsilon)) \leq 2 C_1(M),
\\ 
& \Gamma_\epsilon(j\epsilon) \in 2\Lambda_*:= \Lambda(\Gamma_*,\epsilon_0,2\delta).
\end{aligned}
\end{equation}
Under the above conditions,  we may invoke the energy coercivity inequality in the opposite direction to conclude that
\[
\| (v_\epsilon(j\epsilon), B_{\epsilon}(j\epsilon), \Gamma_\epsilon(j\epsilon))\|_{\mathbf H^k} \leq C_2(M). 
\]
By the energy monotonicity bound \eqref{EMBITT} we  have
\[
E^k(v_\epsilon(j\epsilon), B_{\epsilon}(j\epsilon), \Gamma_\epsilon(j\epsilon))
\leq (1+ C(C_2(M))\epsilon )^jC_1(M)
\leq e^{ C(C_2(M))\epsilon j}C_1(M).
\]
Hence, the cutoff in the first inequality in \eqref{stop-iteration} cannot be reached until at least time
\[
t = \epsilon j<T(M): = (C(C_2(M)))^{-1},
\]
which is a bound independent of $\epsilon$. For the second requirement in \eqref{stop-iteration} we note that \eqref{approx-sln} ensures that at each time step the boundary  moves by at most $\mathcal{O}(\epsilon)$. Hence, by step $j$, it moves by at most  $\mathcal{O}(j \epsilon)$, which leads to a similar constraint on the number of iterations. Using \eqref{regboundprop} and similar reasoning, it is easy to see that the bounds on the vorticity do not significantly deteriorate on the above time-scale. In other words, we retain the bound
\begin{equation*}
\|\omega^\pm_{\epsilon}(j\epsilon)\|_{H^k}\lesssim_M\epsilon^{-\frac{3}{2}},\hspace{5mm}j\epsilon<T.
\end{equation*}
Therefore,  the discrete approximate solutions $(v_\epsilon, B_{\epsilon}, \Gamma_\epsilon)$
are all defined up to the time $T(M)$ above and maintain the uniform bound
\begin{equation*}
\|   (v_\epsilon, B_{\epsilon}, \Gamma_\epsilon)\|_{\mathbf H^k}
\lesssim_M 1 \qquad \text{in}\hspace{2mm} [0,T],
\end{equation*}
with $\Gamma_\epsilon \in 2\Lambda_*$.
Since $k$ is large, we may apply Sobolev embeddings to conclude the uniform bounds
\begin{equation}\label{c3}
\| v_\epsilon\|_{C^3} +\|B_{\epsilon}\|_{C^3}+ \|\eta_\epsilon\|_{C^3} \lesssim_M 1 \qquad  \text{in}\hspace{2mm} [0,T],
\end{equation}
where we use $\eta_\epsilon:=\eta_{\Gamma_\epsilon}$ to denote the defining function for $\Gamma_\epsilon \in 2\Lambda_*$.
\medskip


The other information that we have about $(v_\epsilon,B_{\epsilon})$ is \eqref{approx-sln}, which we wish to iterate  over multiple time steps. 
We first note that \eqref{approx-sln} implies that
\[
|(v_\epsilon,B_{\epsilon})(t,x) - (v_\epsilon,B_{\epsilon})(s,y)|+
|\nabla (v_\epsilon,B_{\epsilon})(t,x) - \nabla (v_\epsilon,B_{\epsilon})(s,y)| \lesssim_M |t-s| + |x-y|, \qquad t-s = \epsilon.
\]
Iterating this, we see that
\begin{equation}\label{Lip-ve}
|(v_\epsilon,B_{\epsilon})(t,x) - (v_\epsilon,B_{\epsilon})(s,y)|+
|\nabla (v_\epsilon,B_{\epsilon})(t,x) - \nabla (v_\epsilon,B_{\epsilon})(s,y)| \lesssim_M |t-s| + |x-y|, \ \ \ \ \ t,s \in \epsilon \N \cap [0,T].
\end{equation}
Using similar reasoning, the last equation in \eqref{approx-sln} tells us that 
\begin{equation}\label{Lip-etae}
   \|\eta_\epsilon(t) - \eta_\epsilon(s)\|_{C^1} \lesssim_M |t-s|, \qquad t,s \in \epsilon \N \cap [0,T].
\end{equation}
As a consequence of \Cref{goodvarrel} and the elliptic estimate $\|D_tP_{\epsilon}\|_{H^k}\lesssim_M 1$ at each time, we may also bound the  pressure difference by
\begin{equation}\label{Lip-pe}
|\nabla P_\epsilon(t,x) - \nabla P_\epsilon(s,y)| \lesssim_M |t-s|+|x-y|, \qquad t,s \in \epsilon \N \cap [0,T].
\end{equation}
We now return to 
\eqref{approx-sln}, making use of the last three Lipschitz bounds 
in time to reiterate and obtain
 second order information.
A direct iteration
using the bounds \eqref{Lip-ve} and \eqref{Lip-pe} to compare the expressions on the right at different times in the uniform norm yields 
\begin{equation}\label{Euler-app}
v_\epsilon(t) = v_\epsilon(s) -(t-s) (v_\epsilon(s)\cdot\nabla v_\epsilon(s)-B_{\epsilon}(s)\cdot\nabla B_{\epsilon}(s)+\nabla P_\epsilon(s))+\mathcal{O}((t-s)^2),
\qquad t,s \in \epsilon \N \cap [0,T]
\end{equation}
and
\begin{equation*}\label{Euler-app2}
B_\epsilon(t) = B_\epsilon(s)-(t-s)\left(v_\epsilon(s)\cdot\nabla B_\epsilon(s)-B_{\epsilon}(s)\cdot\nabla v_{\epsilon}(s)\right)+\mathcal{O}((t-s)^2),
\qquad t,s \in \epsilon \N \cap [0,T].
\end{equation*}
The same strategy applied to the last component of \eqref{approx-sln} gives the relation
\begin{equation}\label{kinenatic-app}
    \Omega_\epsilon(t) = (I + (t-s)v_\epsilon(s))\Omega_\epsilon(s) + \mathcal{O}((t-s)^2),  \qquad t,s \in \epsilon \N \cap [0,T].
\end{equation}
Having established the above properties of our approximate solutions $(v_\epsilon, B_{\epsilon}, \Gamma_\epsilon)$, we now aim to pass to the limit on a subsequence as $\epsilon \to 0$ to obtain the desired solution $(v, B, \Gamma)$. For convenience, we will let $\epsilon$ be of the form $\epsilon = 2^{-m}$ and send $m \to \infty$.
This  ensures that the time domains of the corresponding approximate solutions  $(v_m,B_m)$ are nested.
\medskip

Utilizing the Lipschitz bounds \eqref{Lip-ve}, \eqref{Lip-etae} and \eqref{Lip-pe},
a careful application of the Arzela-Ascoli theorem yields uniformly  convergent subsequences 
\begin{equation}\label{AA}
\eta_m \to \eta, \qquad (v_m,B_m) \to (v,B), \qquad \nabla (v_m,B_m) \to \nabla (v,B),   \qquad \nabla P_m \to \nabla P,
\end{equation}
with limits still satisfying the bounds \eqref{Lip-ve}, \eqref{Lip-etae} and \eqref{Lip-pe}. It remains to show that $(v, B, \Gamma)$ is a solution
to the free boundary MHD equations, with 
$\Gamma$ defined by $\eta$ and $P$, where $P$  is the associated pressure.
\medskip

We first upgrade the spatial regularity 
of $v$, $B$ and $\eta$. We observe that for each time $t \in 2^{-j} \N \cap [0,T]$ we may pass to the limit  as 
$m \to \infty$ in \eqref{c3} to obtain the uniform bound 
\begin{equation*}\label{c3-lim}
\| v\|_{C^3} +\|B\|_{C^3}+ \|\eta\|_{C^3} \lesssim_M 1 .
\end{equation*}
Since  $v$, $B$ and $\eta$ are all Lipschitz continuous in $t$, this extends routinely to all $t \in [0,T]$. Similar arguments apply to the $\mathbf H^k$ norm of $(v, B, \Gamma)$. 
\medskip

To show that $(v, B, \Gamma)$ solves the free boundary MHD equations, we proceed
 in several steps:
\medskip

\emph{i) The initial data.} The fact that at the initial time we have $(v(0), B(0), \Gamma(0)) = (v_0, B_0, \Gamma_0)$ is a direct consequence of the construction of $(v_\epsilon(0), B_{\epsilon}(0), \Gamma_\epsilon(0))$.

\medskip

\emph{ii) The pressure equation.} To check that 
$P$ is the pressure associated to $v$, $B$ and $\Gamma$ 
we use the uniform convergence of $\nabla (v_m,B_m)$,
$\eta_m$ and $\nabla P_m$  to pass to the limit in the pressure equation \eqref{Euler-pressure}.

\medskip
\emph{iii) The incompressible MHD equations.} 
We directly use the uniform convergence \eqref{AA}  to pass to the limit in \eqref{Euler-app}. This guarantees that $v$ and $B$ are differentiable in time and that the incompressible MHD equations are verified.

\medskip
\emph{iv) The kinematic boundary condition.}
We directly use the uniform convergence \eqref{AA} to pass to the limit in \eqref{kinenatic-app} and then argue as above to verify the kinematic boundary condition.

\medskip
\emph{v) The boundary condition for $B$}. This follows from the fact that $B_m\cdot n_m=0$ and passing to the limit.
\medskip

Lastly, we remark that the $C(\mathbf H^{k-1})$ regularity of $(v, B, \Gamma)$ follows directly from the incompressible MHD equations and the kinematic boundary condition. 
\end{proof}
\subsection{Commutator estimates used in \Cref{Existence section}}\label{existencecommutators}
Here we collect the various commutator estimates needed in the previous subsections.  Throughout this subsection, $k$ will be some sufficiently large, dimension-dependent integer, and $\Gamma$ will be some smooth hypersurface belonging to a suitable collar neighborhood with reference hypersurface $\Gamma_*$ and with uniform $H^k$ bound $\|\Gamma\|_{H^k}\lesssim_M 1$.
\begin{theorem}[Theorem A.8 in \cite{MR2388661}]\label{SZ1} For $s'\in [2-k,k-1]$, we have
\begin{equation*}
\|(-\Delta_{\Gamma})^{\frac{1}{2}}-\mathcal{N}\|_{H^{s'}(\Gamma)\to H^{s'}(\Gamma)}\lesssim_M 1.
\end{equation*}
\end{theorem}
\begin{remark}\label{replacewithpositive}
We remark that the same bound holds with $(1-\Delta_{\Gamma})^{\frac{1}{2}}$ and $1+\mathcal{N}$ in place of $(-\Delta_{\Gamma})^{\frac{1}{2}}$ and $\mathcal{N}$, respectively, as the errors are bounded from $H^{s'}\to H^{s'}$.
\end{remark}
We will need the following abstract result to slightly refine the above.
\begin{proposition}[Proposition A.7 in \cite{MR2388661}]\label{SZ2}
Let $X$ be a Hilbert space and let $A$ and $B$ be (possibly unbounded) self-adjoint
positive operators on $X$ so that $A^{-1}B$ and $AB^{-1}$ are bounded. Suppose that $K:=A^2-B^2$
is such that $KB^{-\alpha}$ is bounded for $\alpha\in [0,2)$. Then $(A-B)B^{1-\alpha}$ is bounded as well.
\end{proposition}
\begin{corollary}\label{fractionalpowerapprox} For $s'\in [\frac{1}{2}-k,k-\frac{3}{2}]$, we have
\begin{equation*}
\|(-\Delta_{\Gamma})^{\frac{1}{4}}-\mathcal{N}^{\frac{1}{2}}\|_{H^{s'}(\Gamma)\to H^{s'+\frac{1}{2}}(\Gamma)}\lesssim_M 1.
\end{equation*}
\end{corollary}
\begin{proof}
We first observe that $(1-\Delta_{\Gamma})^{\frac{1}{4}}-(1+\mathcal{N})^{\frac{1}{2}}$ maps $H^{s'}(\Gamma)\to H^{s'+\frac{1}{2}}(\Gamma)$. This follows by taking $A=(1+\mathcal{N})^{\frac{1}{2}}$ and $B=(1-\Delta_{\Gamma})^{\frac{1}{4}}$ in \Cref{SZ2} and applying \Cref{replacewithpositive}. The desired bound then follows easily since $(1+\mathcal{N})^{\frac{1}{2}}-\mathcal{N}^{\frac{1}{2}}$ and $(1-\Delta_{\Gamma})^{\frac{1}{4}}-(-\Delta_{\Gamma})^{\frac{1}{4}}$ also map $H^{s'}\to H^{s'+\frac{1}{2}}$.
\end{proof}
\begin{lemma}\label{coordinatescommutator} 
The following estimates hold: 
\begin{enumerate}
\item For $s\in\frac{1}{2}\mathbb{N}$ and $s\leq k-2$, 
\begin{equation*}
\|(\Delta_{\Gamma_*}(\mathcal{N}^{s}f)_*)^*-\mathcal{N}^{s}(\Delta_{\Gamma_*}f_*)^*\|_{L^2(\Gamma)}\lesssim_M \|f\|_{H^{s+1}(\Gamma)}.
\end{equation*}
\item For $s\geq 0$ and some fixed dimension-dependent $s_0\geq 0$, there holds 
\begin{equation*}
\|(\Delta_{\Gamma_*}(\mathcal{N}^2f)_*)^*-\mathcal{N}^2(\Delta_{\Gamma_*}f_*)^*\|_{H^{s}(\Gamma)}\lesssim_{M, s_0} \|\Gamma\|_{H^{s+4}}\|f\|_{H^{s_0}(\Gamma)}+\|f\|_{H^{s+3}(\Gamma)}.
\end{equation*}
\end{enumerate}
\end{lemma}
\begin{proof}
We start with (i). It follows from \Cref{SZ1} and \Cref{fractionalpowerapprox} that 
\begin{equation*}
\begin{split}
\|(\Delta_{\Gamma_*}(\mathcal{N}^{s}f)_*)^*-\mathcal{N}^{s}(\Delta_{\Gamma_*}f_*)^*\|_{L^2(\Gamma)}\lesssim_M \|(\Delta_{\Gamma_*}((-\Delta_{\Gamma})^{\frac{s}{2}}f)_*)^*-(-\Delta_{\Gamma})^{\frac{s}{2}}(\Delta_{\Gamma_*}f_*)^*\|_{L^2(\Gamma)}+\|f\|_{H^{s+1}(\Gamma)}.
\end{split}
\end{equation*}
As a preliminary step, we note that 
\begin{equation}\label{Delta-com}
\|(\Delta_{\Gamma_*}(\Delta_{\Gamma}f)_*)^*-\Delta_{\Gamma}(\Delta_{\Gamma_*}f_*)^*\|_{H^{\sigma}(\Gamma)}\lesssim_M\|f\|_{H^{\sigma +3}(\Gamma)}
\end{equation}
for $1-k \leq \sigma\leq k-4$. This follows from a computation in collar coordinates (more precisely, the fixed local coordinates for $\Gamma_*$), using the explicit formulas for $\Delta_{\Gamma}$ and $\Delta_{\Gamma_*}$. Therefore, if $s$ is even, we are done. Otherwise, we have reduced matters to proving that
\begin{equation*}
\|(\Delta_{\Gamma_*}((-\Delta_{\Gamma})^{\frac{\alpha}{4}}f)_*)^*-(-\Delta_{\Gamma})^{\frac{\alpha}{4}}(\Delta_{\Gamma_*}f_*)^*\|_{L^2(\Gamma)}\lesssim_M \|f\|_{H^{\frac{\alpha}{2}+1}(\Gamma)},\hspace{5mm}\alpha\in\{1,2,3\}.
\end{equation*}
Using duality, we can rewrite this in the form
\begin{equation*}
\langle (\Delta_{\Gamma_*}((-\Delta_{\Gamma})^{\frac{\alpha}{4}}f)_*)^*-(-\Delta_{\Gamma})^{\frac{\alpha}{4}}(\Delta_{\Gamma_*}f_*)^*),g\rangle \lesssim_M \|f\|_{H^{\frac{\alpha}{2}+1}(\Gamma)} \|g\|_{L^2(\Gamma)},\hspace{5mm}\alpha\in\{1,2,3\}.
\end{equation*}
To avoid the zero mode of $\Delta_{\Gamma}$
which correspond to constants, we remark that the last bound is straightforward if either $f$ or $g$ are constant, without needing the commutator structure. 
Hence, from here on we can assume that both $f$ and $g$ 
have zero average with respect to the $\Gamma$ surface measure.
\medskip

Our strategy will be to  reduce our desired bound to \eqref{Delta-com} using the resolvent representation of the fractional powers,
\[
(-\Delta_{\Gamma})^\frac{\alpha}4 = c_{\alpha}\int_0^\infty (-\Delta_{\Gamma}) (-\Delta_{\Gamma} + \lambda^2)^{-1} \lambda^{\frac{\alpha}2-1} \, d\lambda, \qquad \alpha \in (0,4),
\]
where $c_{\alpha}$ is some constant depending on $\alpha$. From the formal commutator identity $[A^{-1},B]=-A^{-1}[A,B]A^{-1}$, we obtain
\[
[\Delta_{\Gamma_*},(-\Delta_{\Gamma})^{\frac{\alpha}4}]
= c_{\alpha}\int_0^\infty(-\Delta_{\Gamma} + \lambda^2)^{-1} [\Delta_{\Gamma_*}, \Delta_{\Gamma}] (-\Delta_{\Gamma} + \lambda^2)^{-1} \lambda^{\frac{\alpha}2+1} \, d\lambda
\]
where by slight abuse of notation, for a smooth function $h$ on $\Gamma$, we interpret $[\Delta_{\Gamma_*},\Delta_{\Gamma}]h$ to mean
\begin{equation*}
[\Delta_{\Gamma_*},\Delta_{\Gamma}]h:=(\Delta_{\Gamma_*}(\Delta_{\Gamma}h)_*)^*-\Delta_{\Gamma}(\Delta_{\Gamma_*}h_*)^*.
\end{equation*}
Given two zero average test functions $f,g$ it then remains to show that 
\[
\int_0^\infty  \lambda^{\frac{\alpha}2+1}\langle [\Delta_{\Gamma_*}, \Delta_{\Gamma}] (-\Delta_{\Gamma} + \lambda^2)^{-1} f, (-\Delta_{\Gamma} + \lambda^2)^{-1} g \rangle \, d\lambda 
\lesssim_M \|g\|_{L^2(\Gamma)} \|f\|_{H^{1+\frac{\alpha}2}(\Gamma)}.
\]
Using \eqref{Delta-com} and Cauchy-Schwarz in $\lambda$,
we reduce this to
\begin{equation*}
\int_0^\infty \lambda^{\frac{\alpha}2 + 1} \|(-\Delta_{\Gamma} + \lambda^2)^{-1} f\|_{H^{2+\frac{\alpha}{4}}(\Gamma)}^2 \, d\lambda
\lesssim_M \| f\|_{H^{1+\frac{\alpha}{2}}(\Gamma)}^2.
\end{equation*}
Using powers of  $-\Delta_{\Gamma}$ this can be rewritten as an $L^2$ to $L^2$ bound, which then follows from the functional calculus for $-\Delta_{\Gamma}$.
\medskip

Now, we prove (ii). From \cite[Equation (A.13)]{MR2388661} we have the identity
\begin{equation*}
(-\Delta_{\Gamma}-\mathcal{N}^2)f=\kappa\mathcal{N}f+2\nabla_n\Delta^{-1}(\nabla\mathcal{H}n_\Gamma\cdot \nabla^2\mathcal{H}f)-\mathcal{N}n_\Gamma\cdot(\mathcal{N}fn_\Gamma+\nabla^{\top}f).
\end{equation*}
Therefore, using the bound $\|\kappa\|_{H^{s+2}(\Gamma)}+\|\mathcal{N}n_\Gamma\|_{H^{s+2}(\Gamma)}\lesssim_M \|\Gamma\|_{H^{s+4}}$ and the relevant elliptic estimates from \Cref{BEE}, it follows that 
\begin{equation*}
\|(\Delta_{\Gamma_*}(\mathcal{N}^2f)_*)^*-\mathcal{N}^2(\Delta_{\Gamma_*}f_*)^*\|_{H^{s}(\Gamma)}\lesssim_M \|\Gamma\|_{H^{s+4}}\|f\|_{H^{s_0}(\Gamma)}+\|(\Delta_{\Gamma_*}(\Delta_{\Gamma}f)_*)^*-\Delta_{\Gamma}(\Delta_{\Gamma_*}f_*)^*\|_{H^{s}(\Gamma)}.
\end{equation*}
A computation in collar coordinates then gives the desired estimate.
\end{proof}
We may use the above results to prove a commutator estimate for $[\mathcal{N}^{\frac{1}{2}},a]$.
\begin{proposition}\label{Leibnizcom} Let $\frac{1}{2}\leq s\leq k-1$ with $s\in \frac{1}{2}\mathbb{N}$. Moreover, assume that $a\in H^{k-1}(\Gamma)$ with $\|a\|_{H^{k-1}(\Gamma)}\lesssim_M 1$. There holds
\begin{equation}\label{Leibnizcomeq}
\|[\mathcal{N}^{s},a]\|_{H^{s-1}(\Gamma)\to L^2(\Gamma)}\lesssim_M 1.
\end{equation}
\end{proposition}
\begin{proof} If $s$ is an integer, this follows easily from the Leibniz rule in \Cref{Movingsurfid} and the fact that $\|a\|_{H^{k-1}(\Gamma)}\lesssim_M 1$. If $s$ is not an integer, then $s=\frac{1}{2}+m$ for some integer $0\leq m\leq k-2$. We then have
\begin{equation*}
[\mathcal{N}^s,a]=\mathcal{N}^{\frac{1}{2}}[\mathcal{N}^m,a]f+[\mathcal{N}^{\frac{1}{2}},a]\mathcal{N}^mf.
\end{equation*}
Therefore, it suffices to prove \eqref{Leibnizcomeq} in the case $s=\frac{1}{2}$. One can prove this by proceeding similarly to above using duality and the resolvent representation of $\mathcal{N}^{\frac{1}{2}}$. We omit the details.
\end{proof}
\section{Rough solutions and continuous dependence}\label{RS}
We now aim to construct solutions in the state space $\mathbf{H}^s$, $s>\frac{d}{2}+1$, as  limits of the regular solutions that we constructed in \Cref{Existence section}. The basic strategy is as follows:
\begin{enumerate}
    \item We regularize the initial data on the dyadic scale.
    \item\label{STEP2} We establish uniform bounds for the corresponding regularized solutions.
    \item We deduce convergence of the regularized solutions in a weaker topology.
    \item We use the difference estimates from \Cref{DF} and the uniform bounds from step \eqref{STEP2} to show that the regularized solutions converge in the $\mathbf{H}^s$ topology.
\end{enumerate}
For problems on $\mathbb{R}^d$, the above procedure is rather standard. However, in our setting it is more subtle as we must  compare functions defined on different domains. In particular, we must carefully employ regularization operators on $\mathbf{H}^s$ which produce states on extended domains. 
\medskip

When compared with our previous article \cite{Euler}, an additional complication arises from  the $\nabla_B$ terms in the definition of $\mathbf{H}^s$, as we will need to prove convergence in $H^{s-\frac{1}{2}}$ of terms involving regularizations of $\nabla_Bv$ and $\nabla_BB$. Although our difference estimates  from \Cref{DF} guarantee closeness of appropriate regularizations of $v$ and $B$, the same cannot be said for $\nabla_Bv$ nor $\nabla_BB$. To address this issue, we will prove  difference  estimates for these latter variables in a  regularized setting.
\medskip

To keep notation consistent, we will again work with the variables $W^\pm:=v\pm B$. We will also use the notation $W:=(W^+,W^-)$ when convenient.
\subsection{Initial data regularization}\label{IDR}
Let $(W_0,\Gamma_0)\in \mathbf{H}^s$ be an initial data (with domain $\Omega_0$) where $\Gamma_0$ lies within some collar neighborhood
$\Lambda_*=\Lambda(\Gamma_*,\epsilon,\delta)$ with $0<\delta \ll 1$. Our first aim is to regularize the initial data  at each sufficiently large dyadic scale $2^j$. We begin by  fixing a parameter $j_0$ sufficiently large depending only on $M_s(0):=\|(W_0,\Gamma_{0})\|_{\mathbf{H}^s}$, the lower bound $c_0$ in the Taylor sign condition and $\Lambda_*$. For $j\geq j_0$, we apply \Cref{envbounds} to obtain a regularized initial state $(W_{0,j},\Gamma_{0,j})$ (with domain $\Omega_{0,j}$) which satisfies the uniform bound
\begin{equation*}
\|(W_{0,j},\Gamma_{0,j})\|_{\mathbf{H}^s}\lesssim_{M_s(0)} \|(W_0,\Gamma_{0})\|_{\mathbf{H}^s},\hspace{5mm}j\geq j_0.
\end{equation*}
By Sobolev embeddings and the condition $s>\frac{d}{2}+1$, the regularized hypersurface $\Gamma_{0,j}$ in \Cref{envbounds}  satisfies the distance bound
\begin{equation*}
|\eta_{0,j}-\eta_0|\lesssim_{M_s(0)} 2^{-(\frac{3}{2}+\delta)j},\hspace{5mm}j\geq j_0,
\end{equation*}
for some $\delta>0$. Therefore, for such $j$, the domains $\Omega_{0,j}$ and $\Omega_0$ are within distance $2^{-(\frac{3}{2}+\delta)j}$ of each other and, moreover, as long as $j_0$ is large enough, our regularized variables stay in the collar and maintain a uniform lower bound on the Taylor term. In the sequel, we will also make use of the remaining properties of the regularization operators in \Cref{envbounds}, but these will be recalled as we need them.
\subsection{Uniform bounds and the lifespan of regular solutions}  By \Cref{t:existence}, the regularized data $(W_{0,j},\Gamma_{0,j})$ from \Cref{IDR} give rise to  smooth solutions $(W_j,\Gamma_j)$. Our objective now is to prove uniform bounds for these regular solutions and deduce a lifespan bound which depends solely on the size of the initial data $(W_0,\Gamma_0)$ in $\mathbf{H}^s$, the lower bound $c_0$ for the Taylor term and the collar.  
\medskip

Let us abbreviate $M:=M_s(0)$  and fix some large positive parameters $A_0$ and $A_0^*$ depending solely on the numerical constants for the data ($M$, $c_0$, etc.) such that $M_s(0)\ll A_0\ll A_0^*$. Our strategy will be to  perform  a bootstrap argument with the $\mathbf{H}^s$ norms of $(W_j,\Gamma_j)$.  To this end, we make the overarching assumption that
\begin{equation*}\label{boostrap}
\begin{split}
\|(W_j,\Gamma_j)(t)\|_{\mathbf{H}^s}\leq 2A_0^*,\hspace{5mm}\|(W_j,\Gamma_j)(t)\|_{\mathbf{H}^{s-\frac{1}{2}}}\leq 2A_0,\hspace{5mm} a_j(t)\geq \frac{c_0}{2},\hspace{5mm}\Gamma_j(t)\in 2\Lambda_*,\hspace{5mm}t\in [0,T],
\end{split}
\end{equation*}
where $j(M)=: j_0\leq j\leq j_1$
with $j(M)$ sufficiently large depending on $M$ and where $[0,T]$ is a time interval on which all of the $(W_j, \Gamma_j)$ are defined as smooth solutions evolving in the collar.
Here, $j_1\in \mathbb{N}$ is some arbitrarily large parameter, introduced for technical reasons to ensure that each application of the bootstrap involves only finitely many solutions.
Our goal will be to show that the constants in the bootstrap assumptions can be improved, as long as $T\leq T_0$ for some time $T_0>0$ which is independent of $j_1$.  
 \medskip
 
 For  large integers $k \geq s > \frac{d}{2}+1$, we may  view each
$(W_j, \Gamma_j)$ as a  solution to the free boundary MHD equations in $\mathbf H^k$. Due to  Theorems~\ref{Energy est. thm} and \ref{t:existence}, for each $j_1\geq j\geq j_0$, the solution $(W_j, \Gamma_j)$ can be continued past time $T$ in $\mathbf{H}^k$ (and hence in $\mathbf{H}^s$)  as long as the above bootstrap hypotheses are satisfied.  Roughly speaking, our choice for  $T_0$ will  be
\begin{equation*}
T_0\ll\frac{1}{\mathcal{P}(A_0^*)}
\end{equation*}
for some polynomial $\mathcal{P}$, though, in practice, $T_0$ will also depend on the collar and on $c_0$. Thanks to \Cref{Energy est. thm}, if  we could extend the bootstrap  to such a time $T_0$, we could deduce a uniform $\mathbf{H}^k$ bound for the solution $(W_j, \Gamma_j)$ in terms of its initial data in $\mathbf{H}^k$. We remark importantly, however, that  additional arguments are needed to establish the corresponding $\mathbf{H}^s$ bounds for $(W_j, \Gamma_j)$ when  $s>\frac{d}{2}+1$ is not an integer. 
\medskip

Let  $c_j$ be the $\mathbf{H}^s$ admissible frequency envelope for the initial data $(W_0,\Gamma_{0})$ given by \eqref{admissable} and  let $\alpha\geq 1$ be such that $k = s+\alpha$ is an integer. Using \Cref{envbounds}, the regularized data $(W_{0,j}, \Gamma_{0,j})$ satisfies the higher regularity bound
\begin{equation}\label{hiregbound0}
\|(W_{0,j}, \Gamma_{0,j})\|_{\mathbf{H}^{s+\alpha}}\lesssim_{A_0} 2^{\alpha j}c_j\|(W_0, \Gamma_0)\|_{\mathbf{H}^s}.
\end{equation}
Invoking  \Cref{Energy est. thm} and the bootstrap hypothesis, we conclude from \eqref{hiregbound0} and the definition of $c_j$ that
\begin{equation}\label{hiregbound}
\|(W_j, \Gamma_j)(t)\|_{\mathbf{H}^{s+\alpha}}\lesssim_{A_0} 2^{\alpha j}c_j(1+\|(W_0, \Gamma_0)\|_{\mathbf{H}^s}), \qquad t \in [0,T],
\end{equation}
when  $T\leq T_0\ll\frac{1}{\mathcal{P}(A_0^*)}$. One may view \eqref{hiregbound} as a high frequency bound which controls frequencies $\gtrsim 2^j$ in the solution $(W_j, \Gamma_j)$. Note that we have suppressed the implicit dependence on the Taylor term and the collar in the above estimates. We will do this throughout the section, except when these terms are of primary importance. 
\medskip

To control low frequencies, we use the difference estimates. We begin by noting that at time zero we have the  bound 
\begin{equation}\label{regularizeddiffbound}
D_j(0):=D((W_{0,j},\Gamma_{0,j}),(W_{0,j+1}, \Gamma_{0,j+1}))\lesssim_{A_0} 2^{-2js}c_j^2\|(W_0,\Gamma_0)\|_{\mathbf{H}^s}^2    .
\end{equation}
Indeed, for the first terms in \eqref{diff functional candidate pm} this bound follows immediately from \Cref{envbounds}. To control the surface integrals, recall that on $\tilde{\Gamma}_{0,j}:=\partial (\Omega_{0,j}\cap\Omega_{0,j+1})$ the pressure difference $P_{0,j}-P_{0,j+1}$ is proportional  to the distance between $\Gamma_{0,j}$ and $\Gamma_{0,j+1}$. Thus, by a change of variables, we have  
\begin{equation*}
\int_{\tilde{\Gamma}_{0,j}}|P_{0,j}-P_{0,j+1}|^2\,dS\approx_{A_0} \|\eta_{0,j+1}-\eta_{0,j}\|_{L^2(\Gamma_*)}^2\lesssim_{A_0} 2^{-2js}c_j^2\|(W_0,\Gamma_0)\|_{\mathbf{H}^s}^2,   
\end{equation*}
from which \eqref{regularizeddiffbound} readily follows. By \Cref{Difference}, we may propagate the difference bound \eqref{regularizeddiffbound} and  conclude that
\begin{equation*}\label{propagateddbound}
D_j(t)\lesssim_{A_0} 2^{-2js}c_j^2\|(W_0, \Gamma_0)\|_{\mathbf{H}^s}^2,   \qquad t \in [0,T],
\end{equation*}
when $T\leq T_0\ll\frac{1}{\mathcal{P}(A_0^*)}$. Arguing similarly to the above, we see that
\begin{equation}\label{diffbounds}
\|W_{j+1}-W_j\|_{L^2(\Omega_j\cap\Omega_{j+1})},\hspace{3mm}\|\eta_{j+1}-\eta_j\|_{L^2(\Gamma_*)}\lesssim_{A_0} 2^{-js}c_j\|(W_0,\Gamma_0)\|_{\mathbf{H}^s}.  
\end{equation}
The objective now is to  obtain a uniform $\mathbf{H}^s$ bound of the form 
\begin{equation*}\label{unifbound1}
\|(W_j,\Gamma_j)\|_{\mathbf{H}^s}\lesssim_{A_0}1+ \|(W_0, \Gamma_0)\|_{\mathbf{H}^s}
\end{equation*}
for $T\leq T_0$ by combining the high frequency bound \eqref{hiregbound} and the 
$L^2$ difference bound \eqref{diffbounds}, along with a supplementary low frequency difference bound for the variables $\nabla_{B_j}W^\pm_j$ to be established below. To obtain such a bound for the $\Gamma_j$ component of the norm, we consider the telescoping series  
\begin{equation}\label{etadecomp}
\eta_j=\eta_{j_0}+\sum_{j_0\leq l\leq j-1}(\eta_{l+1}-\eta_l)   \hspace{5mm} \text{on}\ \Gamma_*.
\end{equation}
  For each $j_0\leq l\leq j-1$ it follows trivially from \eqref{hiregbound} that
\begin{equation}\label{etaboundreg}
\|\eta_{l+1}-\eta_{l}\|_{H^{s+\alpha}(\Gamma_*)}\lesssim_{A_0} 2^{l\alpha}c_l(1+\|(W_0, \Gamma_0)\|_{\mathbf{H}^s}).    
\end{equation}
Then, using the telescoping sum and interpolation, it is routine  to verify from \eqref{diffbounds} and  \eqref{etaboundreg} that for each $k\geq 0$ we have the bound
\begin{equation}\label{P-keta}
\|P_k\eta_j\|_{H^s(\Gamma_*)}\lesssim_{A_0} c_k(1+\|(W_0, \Gamma_0)\|_{\mathbf{H}^s}).    
\end{equation}
Hence, by almost orthogonality, we have
\begin{equation}\label{surfenvbound}
\|\Gamma_j\|_{H^s}\lesssim_{A_0}1+ \|(W_0, \Gamma_0)\|_{\mathbf{H}^s}.    
\end{equation}
We now focus on  the bound for $W_j$. We first observe that a decomposition analogous to \eqref{etadecomp}  for $W_j$  fails, as $W_l$ and $W_{l+1}$ are defined on different domains. We therefore perform an additional regularization  $W_l\mapsto \Psi_{\leq l}W_l$ to obtain a function defined on a $2^{-l}$ enlargement of $\Omega_l$. By interpolating \eqref{diffbounds} and \eqref{surfenvbound} it follows that
 \begin{equation}\label{surfdistance}
 \|\eta_{j+1}-\eta_j\|_{L^{\infty}(\Gamma_*)}\lesssim_{A_0} 2^{-(\frac{3}{2}+\delta)j}, \hspace{5mm} \|\eta_{j+1}-\eta_j\|_{C^{1,\frac{1}{2}}(\Gamma_*)}\lesssim_{A_0} 2^{-\delta j}
 \end{equation}
 for some $\delta>0$. Hence,  $\Psi_{\leq l}W_l$ is defined on  $\Omega_j$. We may therefore consider the decomposition 
\begin{equation}\label{vtelescope}
W_j=\Psi_{\leq j_0}W_{j_0}+\sum_{j_0\leq l\leq j-1}\Psi_{\leq l+1}W_{l+1}-\Psi_{\leq l}W_l+(I-\Psi_{\leq j})W_j \hspace{5mm} \text{on}\ \Omega_j,
\end{equation}
 as well as the analogous decomposition 
\begin{equation}\label{Bgradtelescope}
\nabla_{B_j}W_j=\Phi_{\leq j_0}(\nabla_{B_{j_0}}W_{j_0})+\sum_{j_0\leq l\leq j-1}\Phi_{\leq l+1}(\nabla_{B_{l+1}}W_{l+1})-\Phi_{\leq l}(\nabla_{B_l}W_l)+(I-\Phi_{\leq j})(\nabla_{B_j}W_j).
\end{equation}
 The first term in each of the above decompositions is easy to control;  we focus on the remaining terms. For $l\geq j_0$, we define the domain
\begin{equation*}
\tilde{\Omega}_{l}=\bigcap_{k=l}^{j}\Omega_k.   
\end{equation*}
Thanks to \eqref{surfdistance}, for $j_0$ large enough (depending on the data parameters but not on $j$) we may ensure that each regularization operator $\Psi_{\leq l}$ is bounded from $H^s(\tilde{\Omega}_l)$ to $H^s(\tilde{\Omega}_l')$ where $\tilde{\Omega}_l'$ is a $2^{-l}$ enlargement of the union of all of the $\Omega_k$ for $k\geq l$. This fact will be used to establish bounds for the intermediate terms in 
\eqref{vtelescope} and \eqref{Bgradtelescope}. We first treat \eqref{vtelescope}, for which we will need the following simple lemma.
\begin{lemma}\label{L^2regboundseasier}
Let $j_0\leq l\leq j-1$, where $j_0$ is some large universal parameter depending only on the constants for the data. Let $\alpha\geq 1$ be such that $s+\alpha\in\mathbb{N}$. There holds
\begin{equation}\label{Ebound1}
\|\Psi_{\leq l+1}W_{l+1}-\Psi_{\leq l}W_l\|_{H^{s+\alpha}(\Omega_j)}\lesssim_{A_0} 2^{l\alpha}c_l(1+\|(W_0,\Gamma_0)\|_{\mathbf{H}^s}),   
\end{equation}
\begin{equation}\label{diffbound1}
\|\Psi_{\leq l+1}W_{l+1}-\Psi_{\leq l}W_l\|_{L^2(\Omega_j)}\lesssim_{A_0} 2^{-sl}c_l(1+\|(W_0,\Gamma_0)\|_{\mathbf{H}^s}).
\end{equation}
\end{lemma}
\begin{proof}
The bound \eqref{Ebound1} is clear from the $H^{s+\alpha}$ boundedness of $\Psi_{\leq l}$ and \eqref{hiregbound}. To establish the bound \eqref{diffbound1} we perform the splitting
\begin{equation*}
\Psi_{\leq l+1}W_{l+1}-\Psi_{\leq l}W_l=(\Psi_{\leq l+1}-\Psi_{\leq l})W_{l+1}+\Psi_{\leq l}(W_{l+1}-W_l).    
\end{equation*}
Then, using \Cref{c reg bounds} and \eqref{hiregbound} we estimate
\begin{equation*}
\|(\Psi_{\leq l+1}-\Psi_{\leq l})W_{l+1}\|_{L^2(\Omega_j)}\lesssim_{A_0} 2^{-ls}c_l(1+\|(W_0,\Gamma_0)\|_{\mathbf{H}^s}).    
\end{equation*}
For the remaining term, we use the difference estimates and the $L^2$ boundedness of $\Psi_{\leq l}$ to obtain
\begin{equation*}
\|\Psi_{\leq l}(W_{l+1}-W_l)\|_{L^2(\Omega_j)}\lesssim_{A_0} D((W_{l},\Gamma_l),(W_{l+1},\Gamma_{l+1}))^{\frac{1}{2}}\lesssim_{A_0} 2^{-ls}c_l\|(W_0,\Gamma_0)\|_{\mathbf{H}^s}.    
\end{equation*} 
\end{proof}
Using \Cref{L^2regboundseasier} and the corresponding bounds for $(I-\Psi_{\leq j})W_j$, we may conclude by a similar argument  to    \eqref{P-keta} that
\begin{equation}\label{Dyadic unf bound}
\|P_k\mathcal{E}_{\Omega_j}W_j\|_{H^s(\mathbb{R}^{d})}\lesssim_{A_0} c_k(1+\|(W_0,\Gamma_0)\|_{\mathbf{H}^s}),\hspace{5mm}k\geq 0.
\end{equation}
 Here, $\mathcal{E}_{\Omega_j}$ is the Stein extension operator on $\Omega_j$ from \Cref{Stein}, which we use to ensure that the implicit constants in the $H^s\to H^s$ bound for the extension depend only on the $C^{1}$ norm of $\Gamma_j$. Our  desired uniform bound
\begin{equation}\label{preliminaryimprovement}
\|(W_j,\Gamma_j)(t)\|_{H^s}\lesssim_{A_0} 1+\|(W_0,\Gamma_0)\|_{\mathbf{H}^s}  
\end{equation}
for $t \in [0,T_0]$  follows by combining  \eqref{surfenvbound} and \eqref{Dyadic unf bound}. In particular, if the constant $A_0^*$ is taken to be large relative to $A_0$ and the data size, the bootstrap assumption for $\|(W_j,\Gamma_j)\|_{H^s}$ improves. It remains to establish bounds for $\nabla_{B_j}W_j$ in $H^{s-\frac{1}{2}}$ by estimating the terms in \eqref{Bgradtelescope}. For this, we have the following analogue of \Cref{L^2regboundseasier}.
\begin{theorem}\label{L^2regbounds}Let $j_0\leq l\leq j-1$, where $j_0$ is some large universal parameter depending only on the constants for the data. Let $\alpha\geq 1$ be such that $\alpha+s\in\mathbb{N}$. Moreover, let $\sigma\in\mathbb{N}$ be such that $s-2\leq \sigma\leq s-1$ (since $s-2>\frac{d}{2}-1$, this implies that $\sigma\geq 1$). There holds
\begin{equation}\label{diffbound2}
\|\Phi_{\leq l+1}(\nabla_{B_{l+1}}W_{l+1})-\Phi_{\leq l}(\nabla_{B_l}W_l)\|_{H^{\sigma}(\Omega_j)}\lesssim_{A_0} 2^{-(s-\frac{1}{2}-\sigma)}c_l(1+\|(W_0,\Gamma_0)\|_{\mathbf{H}^s}),
\end{equation}
\begin{equation}\label{Ebound2}
\|\Phi_{\leq l+1}(\nabla_{B_{l+1}}W_{l+1})-\Phi_{\leq l}(\nabla_{B_l}W_l)\|_{H^{s+\alpha-\frac{1}{2}}(\Omega_j)}\lesssim_{A_0} 2^{l\alpha}c_l(1+\|(W_0,\Gamma_0)\|_{\mathbf{H}^s}).
\end{equation}
\end{theorem}
\begin{proof}
The proof of the bound \eqref{Ebound2} is entirely analogous to \eqref{Ebound1}, so we focus on establishing the bound \eqref{diffbound2}. Unlike with \eqref{diffbound1}, we cannot rely on the difference estimate in \Cref{Difference} as it does not propagate bounds for  $\nabla_{B_{l+1}}W_{l+1}-\nabla_{B_{l}}W_l$. Nevertheless, we can obtain a less general difference type bound in this setting by making the following observation: For each $j_0\leq l\leq j-1$, there exists a smooth domain $\tilde{\Omega}_l$ with boundary $\tilde{\Gamma}_l\in\Lambda_*$ such that $\|\Phi_{\leq l}\|_{H^{\sigma}(\tilde{\Omega}_l)\to H^{\sigma}(\Omega_j)}\lesssim_{A_0}1$,  we have the inclusion
\begin{equation}\label{smoothsubdomain}
\tilde{\Omega}_l\subset \bigcap_{l\leq m\leq j}\Omega_m,
\end{equation}
and $\tilde{\Gamma}_l$ retains the regularity of $\Gamma_l$, 
\begin{equation}\label{smoothsubdombounds}
\|\tilde{\Gamma}_l\|_{H^s}\lesssim_{A_0}1+\|(W_0,\Gamma_0)\|_{\mathbf{H}^s},\hspace{5mm}\|\tilde{\Gamma}_l\|_{H^{s+\alpha}}\lesssim_{A_0} 2^{l\alpha}c_l(1+\|(W_0,\Gamma_0)\|_{\mathbf{H}^s}).
\end{equation}
To see this,  simply define $\tilde{\Gamma}_l$ through the collar coordinate parameterization
\begin{equation*}
\tilde{\eta}_l:=\eta_l-C2^{-(\frac{3}{2}+\delta)l}
\end{equation*}
for some $\delta,\, C>0$. Thanks to the bound \eqref{surfdistance}, if $\delta$ is small enough and $C$ is large enough,  the domain $\tilde{\Omega}_l$ associated to $\tilde{\Gamma}_l$ will satisfy \eqref{smoothsubdomain}. With this definition and the bounds for $\Gamma_l$ from earlier, we have \eqref{smoothsubdombounds}. Moreover, since $\Phi_{\leq l}$ maps $H^{s-1}(\tilde{\Omega}_l)$ to $H^{s-1}(\tilde{\Omega}_j)$ where $\tilde{\Omega}_j$ is some $2^{-l}$ enlargement of $\Omega_j$ (by virtue of \eqref{surfdistance}), the bound $\|\Phi_{\leq l}\|_{H^{\sigma}(\tilde{\Omega}_l)\to H^{\sigma}(\Omega_j)}\lesssim_{A_0}1$ follows. We also importantly remark  that, by definition, $\tilde{\Gamma}_l$ is within distance $2^{(-\frac{3}{2}+\delta)l}$ of $\Gamma_l$. 
\medskip

Returning to \eqref{diffbound2}, we may argue as in \eqref{diffbound1} to obtain the bound
\begin{equation*}
\begin{split}
\|\Phi_{\leq l+1}(\nabla_{B_{l+1}}W_{l+1})-\Phi_{\leq l}(\nabla_{B_l}W_l)\|_{H^{\sigma}(\Omega_j)}\lesssim_{A_0}&\|\Phi_{\leq l}(\nabla_{B_{l}}W_{l}-\nabla_{B_{l+1}}W_{l+1})\|_{H^{\sigma}(\Omega_j)}
\\
+&2^{-l(s-\frac{1}{2}-\sigma)}c_l(1+\|(W_0,\Gamma_0)\|_{\mathbf{H}^s}).
\end{split}
\end{equation*}
To estimate the first term on the right, we note that from the above discussion, we have
\begin{equation*}
\|\Phi_{\leq l}(\nabla_{B_{l}}W_{l}-\nabla_{B_{l+1}}W_{l+1})\|_{H^{\sigma}(\Omega_j)}\lesssim_{A_0}\|\nabla_{B_{l}}W_{l}-\nabla_{B_{l+1}}W_{l+1}\|_{H^{\sigma}(\tilde{\Omega}_l)}.
\end{equation*}
It therefore suffices to directly estimate the difference $\nabla_{B_{l}}W_{l}-\nabla_{B_{l+1}}W_{l+1}$ in $\tilde{\Omega}_l$ (which is well-defined thanks to \eqref{smoothsubdomain}). This will turn out to be a bit easier than the general difference bound in \Cref{Difference} since we can perform the analysis on a smooth (as opposed to Lipschitz) domain  and allow for stronger implicit constants to appear in the estimates. The downside is that due to the correction of $\eta_l$ by the factor $C2^{-(\frac{3}{2}+\delta)l}$, we have to carry out the estimate in a suitable high regularity Sobolev norm, as opposed to $L^2$. This is why the above restriction on $\sigma$ is needed.  
\medskip

It is slightly awkward (although possible) to perform the analysis directly on the shrunken domain $\tilde{\Omega}_l$. So instead, we  mildly correct $W_{l+1}^+$ so that it is defined on $\Omega_l$. Indeed, for any smooth function $f$ defined on $\Omega_{l+1}$, we can define
\begin{equation*}
\tilde{f}(x):=f(x-C'2^{-(\frac{3}{2}+\delta')l}\nu(x)),\hspace{5mm}x\in\Omega_l,
\end{equation*}
where $C'$ and $\delta'$ are chosen so that $\tilde{f}$ is defined on $\tilde{\Omega}_l$ and $\nu$ is a smooth extension of the transversal vector field in \Cref{Collarcoords}. By the chain rule, the fundamental theorem of calculus (after possibly replacing $f$ with its Stein extension $\mathcal{E}_{\Omega_{l+1}}f$ in \Cref{Stein}) and interpolation, we have the bounds
\begin{equation}\label{correctionproperties}
\|\tilde{f}-f\|_{H^{r}(\tilde{\Omega}_l)}\lesssim 2^{-(\frac{3}{2}+\delta)l}\|f\|_{H^{r+1}(\Omega_{l+1})},\hspace{5mm}\|\tilde{f}\|_{H^r(\Omega_l)}\lesssim \|f\|_{H^r(\Omega_{l+1})},\hspace{5mm}r\geq 0.
\end{equation}
Moreover, we have the obvious bound $\|\tilde{f}\|_{H^r(\tilde{\Omega}_l)}\leq \|\tilde{f}\|_{H^r(\Omega_l)}$ by definition of $\tilde{f}$ and the inclusion $\tilde{\Omega}_l\subset \Omega_l$. 
\medskip

Now, we move to proving the requisite difference bound. From here on, we focus on proving the bound for $\nabla_{B_{l}}W_{l}^+-\nabla_{B_{l+1}}W_{l+1}^+$. The bound for $\nabla_{B_{l}}W_{l}^--\nabla_{B_{l+1}}W_{l+1}^-$ is completely analogous. 
\medskip

Our first observation is that in light of \eqref{preliminaryimprovement}, Sobolev embeddings and \Cref{L^2regboundseasier}, we have 
\begin{equation*}
\begin{split}
\|(B_{l+1}-B_l)\cdot\nabla W_{l+1}\|_{H^{\sigma}(\tilde{\Omega}_l)}&\lesssim_{A_0} \|B_{l+1}-B_l\|_{H^{\sigma}(\tilde{\Omega}_l)}\|W_{l+1}\|_{H^{s}(\tilde{\Omega}_l)}\lesssim_{A_0} 2^{-l(s-\frac{1}{2}-\sigma)}c_l.
\end{split}
\end{equation*}
Moreover, thanks to \eqref{correctionproperties}, simple product estimates and Sobolev embeddings, we have
\begin{equation*}
\|\nabla_{B_l}(W_{l+1}-\tilde{W}_{l+1})\|_{H^{\sigma}(\tilde{\Omega}_l)}\lesssim_{A_0} 2^{-l(\frac{3}{2}+\delta)}\|W_{l+1}\|_{H^{\sigma+2}(\Omega_{l+1})}.
\end{equation*}
Since $\sigma+2\geq s$, we have from \eqref{hiregbound}, \eqref{preliminaryimprovement} and by taking $2^{-l\delta}\leq c_l$,
\begin{equation}\label{tildecorrection}
\|\nabla_{B_l}(W_{l+1}-\tilde{W}_{l+1})\|_{H^{\sigma}(\tilde{\Omega}_l)}\lesssim_{A_0} 2^{-l(\frac{3}{2}+\delta)}2^{l(\sigma+2-s)}\lesssim_{A_0} c_l2^{-l(s-\frac{1}{2}-\sigma)}.
\end{equation}
Consequently, since $\tilde{\Omega}_l\subset \Omega_l$, we have  reduced matters to proving the bound
\begin{equation}\label{sufficientbound}
\|\nabla_{B_l}(\tilde{W}^+_{l+1}-W^+_l)\|_{H^{\sigma}(\Omega_l)}\lesssim_{A_0} 2^{-l(s-\frac{1}{2}-\sigma)}c_l,
\end{equation}
which is an estimate that can be carried out on $\Omega_l$. From here on, we will simplify notation somewhat and define
\begin{equation*}
U_l^\pm:=\nabla_{B_l}(\tilde{W}_{l+1}^\pm-W_l^\pm),\hspace{5mm} Q_l:=\nabla_{B_l}(\tilde{P}_{l+1}-P_l),\hspace{5mm}D_l^\pm:=\partial_t+W_l^\pm\cdot\nabla.
\end{equation*}
We will also write $D_{l+1}^\pm:=\partial_t+W_{l+1}^\pm\cdot \nabla$ and $\tilde{D}_{l+1}^\pm:=\partial_t+\tilde{W}_{l+1}^\pm\cdot \nabla$. Moreover, we will simply abbreviate $\Omega_l$ by $\Omega$ and $\Gamma_l$ by $\Gamma$. Additionally, we will use $\mathcal{N}$ and $\mathcal{H}$ to refer to the Dirichlet-to-Neumann operator and Harmonic extension operator on $\Gamma_l$. 
\medskip 

Our next goal will be to show that $U_l^+$ solves the following equation on $\Omega$:
\begin{equation}\label{eqnfordiffroughsol}
\begin{split}
D_{l}^-U_l^+&=-\nabla\mathcal{H}Q_l+R_l
\end{split}
\end{equation}
where $R_l$ denotes a generic error term which satisfies the estimate
\begin{equation}\label{Rlbound}
\|R_l\|_{L_T^1H^{\sigma}(\Omega)}\lesssim_{A_0} 2^{-l(s-\frac{1}{2}-\sigma)}c_l+C(A_0^*)\|(U^+_l,U^-_l)\|_{L_T^1H^{\sigma}(\Omega)}=:\delta_l^{\sigma}.
\end{equation}
Our starting point is to observe that by direct computation using the MHD equations and the fact that $D_l^-$ and $\nabla_{B_l}$ commute, we have
\begin{equation}\label{DUequation}
D_{l}^-U_l^+=-\nabla Q_l+\nabla_{B_l}((\tilde{W}_{l+1}^--W_l^-)\cdot\nabla \tilde{W}_{l+1}^+)-\nabla B_l\cdot\nabla (\tilde{P}_{l+1}-P_l)+F_l
\end{equation}
where
\begin{equation*}
F_l:=\nabla_{B_l}(\tilde{D}_{l+1}^-\tilde{W}_{l+1}^+-{(D_{l+1}^-W_{l+1}^+)}^\sim)+\nabla_{B_l}(\nabla \tilde{P}_{l+1}-{(\nabla P_{l+1})}^\sim).
\end{equation*}
By the chain rule, a  similar analysis to \eqref{tildecorrection} and \Cref{direst}, we obtain $\|F_l\|_{H^{\sigma}(\Omega)}\lesssim_{A_0} 2^{-l(s-\frac{1}{2}-\sigma)}c_l$. In other words, $F_l$ has the form of the remainder $R_l$. Our next aim is to show that the latter two terms on the right-hand side of \eqref{DUequation} are of the form $R_l$. In view of the restrictions on $\sigma$, we may use simple Sobolev product estimates and the bound $\|\tilde{W}_{l+1}^+\|_{H^s(\Omega)}\lesssim_{A_0} 1$ to obtain
\begin{equation*}
\|\nabla_{B_l}(\tilde{W}_{l+1}^--W_l^-)\cdot\nabla \tilde{W}_{l+1}^+\|_{H^{\sigma}(\Omega)}\lesssim_{A_0} \|\nabla_{B_l}(\tilde{W}_{l+1}^--W_l^-)\|_{H^{\sigma}(\Omega)}.
\end{equation*}
Moreover, from product estimates, interpolation, \eqref{hiregbound} and the bootstrap hypothesis, we have
\begin{equation*}
\begin{split}
\|(\tilde{W}_{l+1}^--W_l^-)\cdot\nabla_{B_l}\nabla \tilde{W}_{l+1}^+\|_{H^{\sigma}(\Omega)}&\lesssim_{A_0}\|\tilde{W}_{l+1}^--W_l^-\|_{H^{\sigma}(\Omega)}\|\nabla_{B_l}\nabla \tilde{W}_{l+1}^+\|_{H^{s-1}(\Omega)}
\\
&\lesssim_{A_0} C(A_0^*) 2^{-l(s-\frac{1}{2}-\sigma)}c_l.
\end{split}
\end{equation*}
Therefore, if $T$ is small enough, we have
\begin{equation*}
\|(\tilde{W}_{l+1}^--W_l^-)\cdot\nabla_{B_l}\nabla \tilde{W}_{l+1}^+\|_{L_T^1H^{\sigma}(\Omega)}\lesssim_{A_0} \delta_l^{\sigma}.
\end{equation*}
Combining the above, we find that
\begin{equation*}
\nabla_{B_l}((\tilde{W}_{l+1}^--W_l^-)\cdot\nabla \tilde{W}_{l+1}^+)=R_l
\end{equation*}
where $R_l$ is a term satisfying \eqref{Rlbound}.
To estimate the third term on the right-hand side of \eqref{DUequation}, we may invoke \Cref{direst} to estimate
\begin{equation*}
\begin{split}
\|\nabla B_l\cdot\nabla (\tilde{P}_{l+1}-P_l)\|_{H^{\sigma}(\Omega)}\lesssim_{A_0}C(A_0^*)(\|\Delta (\tilde{P}_{l+1}-P_l)\|_{H^{\sigma-1}(\Omega)}+\|\tilde{P}_{l+1}\|_{H^{\sigma+\frac{1}{2}}(\Gamma)}),
\end{split}
\end{equation*}
where  we used that $P_{l}=0$ on $\Gamma$.
Invoking \Cref{Difference} and interpolating, it is clear that we have
\begin{equation*}
\|\Delta (\tilde{P}_{l+1}-P_l)\|_{H^{\sigma-1}(\Omega)}\lesssim_{A_0}C(A_0^*)2^{-l(s-\frac{1}{2}-\sigma)}c_l.
\end{equation*}
To estimate the boundary term, we first recall that $P_{l+1}$ vanishes on $\Gamma_{l+1}$, and, moreover, in collar coordinates, we have by definition that $\tilde{P}_{l+1}(x+\eta_l(x)\nu(x))=P_{l+1}(x+\eta_l(x)\nu(x)-C2^{-l(\frac{3}{2}+\delta)}\nu(x+\eta_l(x)\nu(x)))$. Hence, we  can use the fundamental theorem of calculus (after possibly replacing $P_{l+1}$ with its Stein extension $\mathcal{E}_{\Omega_{l+1}}P_{l+1}$ from \Cref{Stein}), the fact that $P_{l+1}(x+\eta_{l+1}(x)\nu(x))=0$, and product and trace estimates to bound
\begin{equation}\label{productestateforPl}
\begin{split}
\|\tilde{P}_{l+1}\|_{H^{\sigma+\frac{1}{2}}(\Gamma)}\lesssim_{A_0} &\|\eta_l-\eta_{l+1}\|_{H^{\sigma+\frac{1}{2}}(\Gamma_*)}\|P_{l+1}\|_{C^1(\Omega_{l+1})}
\\
+(&\|\eta_l-\eta_{l+1}\|_{L^{\infty}(\Gamma_*)}+2^{-l(\frac{3}{2}+\delta)})\|P_{l+1}\|_{H^{\sigma+2}(\Omega_{l+1})},
\end{split}
\end{equation}
which from Sobolev embeddings, \Cref{direst}, \eqref{hiregbound} and \eqref{diffbounds}, we deduce that
\begin{equation}\label{Pdiffboundsigma}
\|\tilde{P}_{l+1}\|_{H^{\sigma+\frac{1}{2}}(\Gamma)}\lesssim_{A_0} C(A_0^*)2^{-l(s-\frac{1}{2}-\sigma)}c_l.
\end{equation}
Here,  we used the hypothesis that $s-\frac{1}{2}-\sigma\leq \frac{3}{2}$. The above analysis yields in particular the bound
\begin{equation*}
\|\nabla B_l\cdot\nabla (\tilde{P}_{l+1}-P_l)\|_{L_T^1H^{\sigma}(\Omega)}\lesssim_{A_0}\delta_l^{\sigma}
\end{equation*}
if $T$ is small enough. The final step to obtain the decomposition \eqref{eqnfordiffroughsol} is to split $Q_l=\mathcal{H}Q_l+\Delta^{-1}\Delta Q_l$ and to estimate
\begin{equation*}
\begin{split}
\|\nabla\Delta^{-1}\Delta Q_l\|_{H^{\sigma}(\Omega)}&\lesssim_{A_0}C(A_0^*)\|\Delta Q_l\|_{H^{\sigma-1}(\Omega)}
\\
&\lesssim_{A_0} C(A_0^*)(\|\tilde{P}_{l+1}-P_l\|_{H^\sigma(\Omega)}+\|\tilde{W}_{l+1}-W_l\|_{H^{\sigma}(\Omega)}+\|(U_l^+,U_l^-)\|_{H^{\sigma}(\Omega)})
\end{split}
\end{equation*}
in order to deduce that
\begin{equation*}
\|\nabla\Delta^{-1}\Delta Q_l\|_{L_T^1H^{\sigma}(\Omega)}\lesssim_{A_0}\delta_l^{\sigma}.
\end{equation*}
This gives \eqref{eqnfordiffroughsol}. Now, we aim to propagate the bounds for $U^+_l$ from $t=0$ to $t<T$. Our first aim is to reduce matters to controlling the boundary value $U^+_l\cdot n_\Gamma$. For this, we use \Cref{Balanced div-curl} (or direct estimates for the rotational/irrotational decomposition when $\sigma=1$) to estimate 
\begin{equation}\label{divcurlUplus}
\|U^+_l\|_{H^{\sigma}(\Omega)}\lesssim_{A_0} \|U^+_l\|_{H^{\sigma-1}(\Omega)}+\|\nabla\cdot U^+_l\|_{H^{\sigma-1}(\Omega)}+\|\nabla\times U^+_l\|_{H^{\sigma-1}(\Omega)}+\|U^+_l\cdot n_\Gamma\|_{H^{\sigma-\frac{1}{2}}(\Gamma)}.
\end{equation}
From \eqref{eqnfordiffroughsol}, the restriction $\sigma\geq 1$ and \eqref{Pdiffboundsigma}, we get
\begin{equation}\label{highorderdiff1}
 \|D_l^-U^+_l\|_{L_T^1H^{\sigma-1}(\Omega)}+\|D_l^-\nabla\cdot U^+_l\|_{L_T^1H^{\sigma-1}(\Omega)}+\|D_l^-\nabla\times U^+_l\|_{L_T^1H^{\sigma-1}(\Omega)}\lesssim_{A_0}\delta_l^{\sigma}
\end{equation}
if $T$ is small enough. To estimate the boundary term $U_l^+\cdot n_\Gamma$, we use ellipticity of $\mathcal{N}$, Poincare's inequality, \Cref{direst} and the trace theorem to estimate 
\begin{equation*}
\|U_l^+\cdot n_\Gamma\|_{H^{\sigma-\frac{1}{2}}(\Gamma)}\lesssim_{A_0} \|\nabla\mathcal{H}\mathcal{N}^{\sigma-1}(U_l^+\cdot n_\Gamma)\|_{L^2(\Omega)}+\|\nabla\cdot U^+_l\|_{H^{\sigma-1}(\Omega)}.
\end{equation*}
To propagate the bounds for the term on the right-hand side of the above estimate, we take our cue from \Cref{Difference} (or the linearized equation) and define the higher-order distance functional
\begin{equation*}
\mathcal{D}^{\sigma-1}:=\frac{1}{2}\int_{\Omega}|\nabla\mathcal{H}\mathcal{N}^{\sigma-1}(U_l^+\cdot n_\Gamma)|^2dx+\frac{1}{2}\int_{\Gamma}a_l^{-1}|\mathcal{N}^{\sigma}Q_l|^2dS.
\end{equation*}
Here, $a_l$ is the Taylor term corresponding to the solution $(W_l,\Gamma_l)$. Our goal now will be to simply carry out a $L^2$-based  energy estimate for $\mathcal{D}^{\sigma-1}$. One can compare the distance functional here to the one in \Cref{Difference}. The requisite energy estimate in this case, however, will be far simpler to carry out since all of the analysis can be performed on a smooth domain (rather than Lipschitz). Moreover, the implicit constants in the estimates here are allowed to be much stronger. 
\medskip

To estimate the main errors that arise, we will need the following lemma.
\begin{lemma}\label{Dsigmaapprox}
There holds 
\begin{equation*}
\|D_l^-\nabla\mathcal{H}\mathcal{N}^{\sigma-1}(U_l^+\cdot n_\Gamma)+\nabla\mathcal{H}\mathcal{N}^{\sigma}Q_l\|_{L_T^1L_x^2(\Omega)}+\|D_l^+\mathcal{N}^{\sigma}Q_l- a_l\mathcal{N}^{\sigma}(U_l^+\cdot n_\Gamma)\|_{L_T^1L_x^2(\Gamma)}\lesssim_{A_0} \delta_l^{\sigma}.
\end{equation*}
\end{lemma}
\begin{proof}
The estimate for the first term simply follows from \eqref{eqnfordiffroughsol} and the estimates
\begin{equation*}
\|[D_l^-,\nabla\mathcal{H}\mathcal{N}^{\sigma-1}](U_l^+\cdot n_\Gamma)\|_{L_T^1L_x^2(\Omega)}+\|U^+_l\cdot D_l^- n_\Gamma\|_{L_T^1H^{\sigma-\frac{1}{2}}(\Gamma)}\lesssim_{A_0}\delta_l^{\sigma},
\end{equation*}
which can be proved using the identities in \Cref{Movingsurfid}. The second estimate is more difficult. First, by \Cref{Movingsurfid}, the fact that $Q_l=\nabla_{B_l}\tilde{P}_{l+1}$ on $\Gamma$ and \Cref{L^2regboundseasier},  we have 
\begin{equation*}
\|[D_l^+,\mathcal{N}^{\sigma}]Q_l\|_{L_T^1L_x^2(\Gamma)}\lesssim_{A_0}\|\nabla_{B_l}\tilde{P}_{l+1}\|_{L_T^1H^{\sigma}(\Gamma)}\lesssim_{A_0} \delta_l^{\sigma}+\|(\nabla_{B_{l+1}}P_{l+1})^\sim\|_{L_T^1H^{\sigma}(\Gamma)}.
\end{equation*}
Using that $\nabla_{B_{l+1}}P_{l+1}=0$ on $\Gamma_{l+1}$ and a similar argument to \eqref{productestateforPl}, we have
\begin{equation*}
\|[D_l^+,\mathcal{N}^{\sigma}]Q_l\|_{L_T^1L_x^2(\Gamma)}\lesssim_{A_0} \delta_l^{\sigma}.
\end{equation*}
Next, we observe that on $\Gamma$, we may write
\begin{equation}\label{Qleqn}
\begin{split}
D_l^+Q_l&=\nabla_{B_l} D_l^+\tilde{P}_{l+1}
\\
&=-U_l^+\cdot\nabla \tilde{P}_{l+1}-(\tilde{W}_{l+1}^+-W_l^+)\cdot\nabla_{B_l}\nabla\tilde{P}_{l+1}-(B_l-\tilde{B}_{l+1})\cdot\nabla \tilde{D}_{l+1}^+\tilde{P}_{l+1}+\nabla_{\tilde{B}_{l+1}}\tilde{D}_{l+1}^+\tilde{P}_{l+1}.
\end{split}
\end{equation}
We can re-arrange the first term above as
\begin{equation*}
-U_l^+\cdot\nabla\tilde{P}_{l+1}=a_lU_l^+\cdot n_\Gamma-U_l^+\cdot\nabla (\tilde{P}_{l+1}-P_l).
\end{equation*}
Using the identities in \Cref{Movingsurfid}, it is then straightforward to obtain the bound
\begin{equation*}
\|\mathcal{N}^{\sigma}(U_l^+\cdot\nabla\tilde{P}_{l+1})+a_l\mathcal{N}^{\sigma}(U_l^+\cdot n_\Gamma)\|_{L_T^1L_x^2(\Gamma)}\lesssim_{A_0} \delta_l^{\sigma}.
\end{equation*}
It remains to estimate the latter three terms on the right-hand side of \eqref{Qleqn}. Using \eqref{correctionproperties} and similar analysis as above, the second and third terms can be estimated in $L_T^1H^{\sigma}(\Gamma)$ by $\delta_l^{\sigma}$. For the fourth term, we may use \eqref{correctionproperties} and that  $\nabla_{B_{l+1}}D_{l+1}^+P_{l+1}=0$ on $\Gamma_{l+1}$ to estimate 
\begin{equation*}
\begin{split}
\|\nabla_{\tilde{B}_{l+1}}\tilde{D}_{l+1}^+\tilde{P}_{l+1}\|_{L_T^1H^{\sigma}(\Gamma)}&\lesssim C(A_0^*)2^{-l(s-\frac{1}{2}-\sigma)}c_l\|\nabla_{B_{l+1}}D_{l+1}P_{l+1}\|_{L_T^1H^{s-\frac{1}{2}}(\Omega_{l+1})}
\\
&+2^{-(\frac{3}{2}+\delta)l}\|\nabla_{B_{l+1}}D_{l+1}P_{l+1}\|_{L_T^1H^{\sigma+\frac{3}{2}}(\Omega)}+\delta_l^{\sigma}
\\
&\lesssim \delta_l^{\sigma}.
\end{split}
\end{equation*}
This concludes the proof of \Cref{Dsigmaapprox}. 
\end{proof}
Now, we return to finish the proof of \Cref{L^2regbounds}. By \Cref{Leibniz} and  \Cref{Dsigmaapprox} we immediately obtain the energy estimate
\begin{equation*}
D^{\sigma-1}(t)\lesssim_{A_0} D^{\sigma-1}(0)+(\delta_l^{\sigma})^2\lesssim_{A_0}(\delta_l^{\sigma})^2.
\end{equation*}
Thanks to \eqref{divcurlUplus} and \eqref{highorderdiff1}, we have
\begin{equation*}
\|U^+\|_{L_T^{\infty}H^{\sigma}(\Omega)}\lesssim_{A_0} \delta_l^{\sigma}.
\end{equation*}
Carrying out the same analysis above for $U^-$ and then taking $T$ small enough finally yields
\begin{equation*}
\|(U^+_l,U^-_l)\|_{L_T^{\infty}H^{\sigma}(\Omega)}\lesssim_{A_0} c_l2^{-l(s-\frac{1}{2}-\sigma)}.
\end{equation*}
This  concludes the proof of \Cref{L^2regbounds}. 
\end{proof}
Arguing as in \eqref{preliminaryimprovement}, we may now obtain the uniform bound
\begin{equation*}
\|(W_j,\Gamma_j)\|_{\mathbf{H}^s}\lesssim_{A_0} 1+\|(W_0,\Gamma_0)\|_{\mathbf{H}^s}.
\end{equation*}
It therefore remains to improve the bootstrap assumptions on $\|(W_j,\Gamma_j)\|_{\mathbf{H}^{s-\frac{1}{2}}}$, the Taylor term and the collar neighborhood size. For this, we take  inspiration from  the previous works \cite{MR3925531,Euler,MR2388661}. We define the Lagrangian flow map $u_j(t,\cdot):\Omega_{0,j}\to \Omega_j(t)$  as the solution to the ODE 
\begin{equation*}
\partial_tu_j(t,y)=v_j(t,u_j(t,y)),\hspace{5mm} y\in \Omega_{0,j},\hspace{5mm}u_j(0)=I.   
\end{equation*}
Since $s>\frac{d}{2}+1$,  for any $0\leq t\leq T\leq T_0$ we have the bound
\begin{equation*}
\begin{split}
\|u_j(t,\cdot)-I\|_{H^{s}(\Omega_{0,j})}&\lesssim\int_{0}^{t}\|v_j(t',\cdot)\|_{H^{s}(\Omega_j(t'))}\|u_j(t',\cdot)\|_{H^s(\Omega_{0,j})}^s\,dt'    
\\
&\lesssim_{A_0} t\|(v_0,\Gamma_0)\|_{\mathbf{H}^s},
\end{split}
\end{equation*}
as long as  $T_0>0$ is sufficiently small. This  easily implies that
\begin{equation*}
\Gamma_j(t)\in \frac{3}{2}\Lambda_*,\hspace{5mm}\|\Gamma_{j}(t)\|_{H^{s-\frac{1}{2}}}\ll A_0,
\end{equation*}
as long as $A_0$ is large relative to the data size. The first condition above improves the bootstrap for the collar. By performing a similar analysis with $u_t$ in place of $u$ which utilizes the equations 
\begin{equation*}
\partial_t^2u_j(t,y)=\partial_t(v_j(t,u_j(t,y)))=-(\nabla P_j-B_j\cdot\nabla B_j)(t,u_j(t,y))   , 
\end{equation*}
\begin{equation*}
\partial_t(B_j(t,u_j(t,y)))=(\nabla_{B_j}v_j)(t,u_j(t,y)),
\end{equation*}
and the elliptic estimates for the pressure, we may further conclude that
\begin{equation*}
\|W_j(t)\|_{H^{s-\frac{1}{2}}(\Omega_t)}\ll A_0.   
\end{equation*}
It remains to prove that
\begin{equation*}
\|\nabla_{B_j}W_j(t)\|_{H^{s-1}(\Omega_t)}\ll A_0.
\end{equation*}
For this, we use the equation in Eulerian coordinates,
\begin{equation*}
D_t^\mp \nabla_{B}W^\pm=\nabla B\cdot\nabla P-\nabla\nabla_B P.
\end{equation*}
From the uniform bound $\|(W_j,\Gamma_j)\|_{\mathbf{H}^s}\lesssim_{A_0} 1+\|(W_0,\Gamma_0)\|_{\mathbf{H}^s}$ established above and \Cref{direst}, we have
\begin{equation*}
\|\nabla B_j\cdot\nabla P_j-\nabla\nabla_{B_j} P_j\|_{L_T^1H^{s-1}(\Omega_j)}\lesssim_{A_0} C(A_0^*)T.
\end{equation*}
Therefore, by taking $T$ small enough and carrying out a straightforward energy estimate, we have
\begin{equation*}
\|(W_j,\Gamma_j)\|_{\mathbf{H}^{s-\frac{1}{2}}}\ll A_0.
\end{equation*}
This improves the bootstrap assumption for $\|(W_j,\Gamma_j)\|_{\mathbf{H}^{s-\frac{1}{2}}}$.   To improve the bootstrap assumption for $a_j$ one may employ a similar argument with the pressure gradient which utilizes the $H^s$ bounds for $D_tP$ and the smallness of  $T_0$ relative to $M$ and $c_0$. We leave the details to the reader.
\subsection{The limiting solution} We now show that for $T\leq T_0$,
\begin{equation*}
(W,\Gamma)=\lim_{j\to\infty}(W_j,\Gamma_j)    \ \ \text{in}\ \ C([0,T];\mathbf{H}^s).
\end{equation*}
We begin by establishing the  domain convergence in $H^s$, which is more straightforward. Indeed, from \eqref{surfdistance} it is easy to see that the limiting domain $\Omega$ exists and has Lipschitz boundary $\Gamma$. For $j\geq j_0$, we may consider the telescoping sum
\begin{equation*}
\eta-\eta_j=\sum_{l=j}^{\infty}\eta_{l+1}-\eta_l. 
\end{equation*}
The difference bounds, the higher energy bounds and an analysis similar to the previous subsection yields 
\begin{equation}\label{limitdistance}
\|\eta-\eta_j\|_{L^{\infty}(\Gamma_*)}\lesssim_{A_0} 2^{-(\frac{3}{2}+\delta)j}    
\end{equation}
and
\begin{equation*}
\|\eta-\eta_j\|_{C([0,T];H^s(\Gamma_*))}\lesssim_{A_0} \|c_{\geq j}\|_{l^2}(1+\|(W_0,\Gamma_0)\|_{\mathbf{H}^s}).
\end{equation*}
Consequently, $\Gamma_j\to\Gamma$ in $C([0,T];H^s(\Gamma_*))$. We now establish the convergence $W_j\to W$ in $C([0,T];\mathbf{H}^s)$. To begin, we formally define $W$ through the  telescoping sum
\begin{equation*}
W=\Psi_{\leq j_0}W_{j_0}+\sum_{l\geq j_0}\Psi_{\leq {l+1}}W_{l+1}-\Psi_{\leq l}W_l  ,  
\end{equation*}
where $j_0$ is chosen so that all of the terms in the sum are defined on $\Omega$.  This is possible, thanks to \eqref{limitdistance}. As a preliminary step, we show that $\Psi_{\leq j}W_j\to W$ in $H^s(\Omega_t)$ uniformly in $t$, which is again unambiguous thanks to \eqref{limitdistance}. Note that
\begin{equation*}
W-\Psi_{\leq j}W_j=\sum_{l\geq j}\Psi_{\leq {l+1}}W_{l+1}-\Psi_{\leq l}W_l.    
\end{equation*}
By arguing similarly to the proof of the uniform lifespan bound (by slightly modifying \Cref{L^2regbounds} where necessary) we have
\begin{equation*}
\|W-\Psi_{\leq j}W_j\|_{H^s(\Omega_t)}\lesssim_{A_0} \|c_{\geq j}\|_{l^2}(1+\|(W_0,\Gamma_0)\|_{\mathbf{H}^s}).
\end{equation*}
Hence,  the desired uniform convergence in $H^s(\Omega_t)$ holds. We now aim to show that $\Phi_{\leq j}(\nabla_{B_j}W_j)$ converges to $\nabla_B W$ in $H^{s-\frac{1}{2}}(\Omega_t)$ uniformly in $t$. We begin by observing that $\Phi_{\leq j}(\nabla_{B_j}W_j)\to \nabla_B W$ in $L^2(\Omega_t)$. Indeed, we have
\begin{equation}\label{grad B limiting}
\Phi_{\leq j}(\nabla_{B_j}W_j)=\Phi_{\leq j}(\nabla_{B}\Psi_{\leq j}W_j)+\Phi_{\leq j}(\nabla_{B}\Psi_{> j}W_j)+\Phi_{\leq j}(\nabla_{B_j-B}W_j).
\end{equation}
The second term in \eqref{grad B limiting} converges to zero in $L^2$, and the third term also does, thanks to the difference bounds. Moreover, $(I-\Phi_{\leq j})(\nabla_BW)\to 0$ in $L^2(\Omega_t)$. Therefore, we have
\begin{equation}\label{Grad B limit 2}
\|\nabla_BW-\Phi_{\leq j}(\nabla_{B_j}W_j)\|_{L^2(\Omega_t)}\lesssim_{A_0} \|W-\Psi_{\leq j}W_j\|_{H^1(\Omega_t)}+o_{L^2(\Omega_t)}(1).
\end{equation}
The first term on the right-hand side of \eqref{Grad B limit 2} goes to zero thanks to the $H^s(\Omega_t)$ convergence established above. It remains to establish $H^{s-\frac{1}{2}}$ convergence of $\nabla_BW$. Thanks to the $L^2$ convergence, we need only show that the following formal telescoping decomposition converges:
\begin{equation*}
\nabla_B W-\Phi_{\leq j}(\nabla_{B_j}W_j)=\sum_{l\geq j}\Phi_{\leq l+1}(\nabla_{B_{l+1}}W_{l+1})-\Phi_{\leq l}(\nabla_{B_l}W_l).
\end{equation*}
This, again, follows by arguing similarly to the uniform lifespan bounds (after slightly adapting \Cref{L^2regbounds} to the present situation). Indeed, such arguments lead to the bound
\begin{equation*}
\|\nabla_B W-\Phi_{\leq j}(\nabla_{B_j}W_j)\|_{H^{s-\frac{1}{2}}(\Omega_t)}\lesssim_{A_0}\|c_{\geq j}\|_{l^2}(1+\|(W_0,\Gamma_0)\|_{\mathbf{H}^s}),
\end{equation*}
which is more than sufficient. To show convergence of $W_j$ and $\nabla_{B_j}W_j$ in the sense of \Cref{Def of convergence}, we consider the regularization $\tilde{W}=\Psi_{\leq m}W_m$. As above,  we have
\begin{equation*}
\|W-\Psi_{\leq m}W_m\|_{H^s(\Omega)}\lesssim_{A_0} \|c_{\geq m}\|_{l^2}(1+\|(W_0,\Gamma_0)\|_{\mathbf{H}^s})  ,
\end{equation*}
which goes to $0$ as $m\to\infty$. On the other hand, for $j>m$, it is easy to see that
\begin{equation*}
\begin{split}
\|W_j-\Psi_{\leq m}W_m\|_{H^s(\Omega_j)}\lesssim_{A_0} &\|(1-\Psi_{\leq j})W_j\|_{H^s(\Omega_j)}+\|\Psi_{\leq j}(W_j-W)\|_{H^s(\Omega_j)}+\|\Psi_{\leq m}(W_m-W)\|_{H^s(\Omega_j)}
\\
&+\|\Psi_{\leq j}W-\Psi_{\leq m}W\|_{H^s(\Omega_j)}.
\end{split}
\end{equation*}
Using \eqref{hiregbound} for the first term and the difference bounds for $D((W_j,\Gamma_j),(W,\Gamma)),$ $D((W_m,\Gamma_m),(W,\Gamma))$ for the second and third terms, respectively, we obtain 
\begin{equation*}
\|W_j-\Psi_{\leq m}W_m\|_{H^s(\Omega_j)}\lesssim_{A_0} \|c_{\geq m}\|_{l^2}(1+\|(W_0,\Gamma_0)\|_{\mathbf{H}^s})+\|\Psi_{\leq j}W-\Psi_{\leq m}W\|_{H^s(\Omega_j)}.    
\end{equation*}
To estimate the remaining term, we note that
\begin{equation*}
\begin{split}
\|\Psi_{\leq j}W-\Psi_{\leq m}W\|_{H^s(\Omega_j)}&\lesssim_{A_0} \|(\Psi_{\leq j}-\Psi_{\leq m})(W-\Psi_{\leq m}W_m)\|_{H^s(\Omega_j)}+\|(\Psi_{\leq j}-\Psi_{\leq m})\Psi_{\leq m}W_m\|_{H^s(\Omega_j)}
\\
&\lesssim_{A_0} \|W-\Psi_{\leq m}W_m\|_{H^{s}(\Omega)}+2^{-m\alpha}\|W_m\|_{H^{s+\alpha}(\Omega_m)}
\\
&\lesssim_{A_0} \|c_{\geq m}\|_{l^2}(1+\|(W_0,\Gamma_0)\|_{\mathbf{H}^s}),   
\end{split}
\end{equation*}
where we used \eqref{hiregbound} to estimate the second term in the final inequality. A similar argument (utilizing a difference type bound as in \Cref{L^2regbounds}) yields
\begin{equation*}
\|\nabla_BW-\Phi_{\leq m}(\nabla_{B_m}W_m)\|_{H^{s-\frac{1}{2}}(\Omega)}\lesssim_{A_0}\|c_{\geq m}\|_{l^2}(1+\|(W_0,\Gamma_0)\|_{\mathbf{H}^s})
\end{equation*}
as well as 
\begin{equation*}
\|\nabla_{B_j}W_j-\Phi_{\leq m}(\nabla_{B_m}W_m)\|_{H^{s-\frac{1}{2}}(\Omega_j)}\lesssim_{A_0}\|c_{\geq m}\|_{l^2}(1+\|(W_0,\Gamma_0)\|_{\mathbf{H}^s}), \hspace{5mm}  j>m.
\end{equation*}
 When combined, the above estimates establish strong convergence of $W_j$ to $W$ in $\mathbf{H}^s$; similar arguments yield continuity of $W$  with values in $\mathbf{H}^s$. As a final step, the reader may check that the limiting solution $(W,\Gamma)$ solves the free boundary MHD equations.
\subsection{Continuous dependence}\label{Cont dep section} Let $(W_0^n,\Gamma_0^n)\in\mathbf{H}^s$ be a sequence of initial data such that $(W_0^n,\Gamma_0^n)\to (W_0,\Gamma_0)$. Let $(W^n,\Gamma^n)$  and $(W,\Gamma)$ denote the respective solutions.  From the data convergence, we obtain  a uniform in $n$ lifespan in $\mathbf{H}^s$ for all of these solutions. Hence, on some compact time interval $[0,T]$, we have $\|(W^n,\Gamma^n)\|_{\mathbf{H}^s}+\|(W,\Gamma)\|_{\mathbf{H}^s}\lesssim_M 1$.  The objective of this subsection is to establish the convergence of the solutions $(W^n,\Gamma^n)\to (W,\Gamma)$ in $C([0,T];\mathbf{H}^s)$. 
\medskip

 We denote by $c_j^n$ and $c_j$ the admissible frequency envelopes for the data $(W_0^n,\Gamma_0^n)$ and $(W_0,\Gamma_0)$, respectively, and fix  $\epsilon>0$. We  let $\delta=\delta(\epsilon)>0$ be a small positive constant and  $n_0=n_0(\epsilon)$ be a  large positive integer which we will specify more precisely below.  By definition of convergence in $\mathbf{H}^s$, we may find   a divergence-free function $W_0^{\delta}\in H^s(\Omega_0^{\delta})$  on an enlarged domain $\Omega_0^{\delta}$ such that 
\begin{equation*}\label{v_0approx}
\|W_0-W_0^{\delta}\|_{H^s(\Omega_0)}+\limsup_{n\to\infty}\|W_0^n-W_0^{\delta}\|_{H^s(\Omega_0^n)}<\delta.
\end{equation*}
In particular, for $n$ sufficiently large depending only on $\delta$, the function $W_0^{\delta}$ is defined on a neighborhood of $\Omega_0$ and $\Omega_{0}^n$. For the sake of the argument below, we may assume that $W_0^{\delta}$ belongs to $H^s(\mathbb{R}^d)$. Indeed, we observe that for some $\delta'\ll\delta$, $W_0^{\delta}$ is defined on the domain $\Omega_0'$ defined by taking $\eta_0'=\eta_0+\delta'$. By  \Cref{continuosext}, we may  extend $W_0^{\delta}$ from this domain to $\mathbb{R}^d$. Notice, however, that $W_0^{\delta}$ is not necessarily divergence-free on $\mathbb{R}^d$, but is divergence-free on an enlargement of $\Omega_0$ and $\Omega_0^n$ when $n$ is sufficiently large. 
\medskip

Let $c_j^{\delta}$  denote the admissible frequency envelope for $(W_0^{\delta},\Gamma_0)$. Here, we emphasize that we are using the same domain $\Omega_0$ as $W_0$ for the frequency envelope $c_j^{\delta}$. We  remark that if $\delta>0$ is sufficiently small, the Taylor sign condition holds for this state. Hence, we may let $(W^{\delta},\Gamma^{\delta})$ be the corresponding $\mathbf{H}^s$ solution, which  has a lifespan comparable to $W$ and $W^n$ for $n$ sufficiently large. We  choose $j=j(\epsilon)$ sufficiently large so that
\begin{equation}\label{dataenvsize}
\|c_{\geq j}\|_{l^2}<\epsilon.   
\end{equation}
We then  choose $\delta(\epsilon)$ and  $n_0(\delta)$ so that for $n\geq n_0$,
\begin{equation}\label{envcontrol}
\|c_{\geq j}^n\|_{l^2}\lesssim_M \epsilon +\|c_{\geq j}\|_{l^2}\lesssim_M \epsilon.
\end{equation}
To establish that such a choice is possible,  we use  the definition of admissible frequency envelopes in \eqref{admissable} and square summing to estimate the error which occurs when comparing terms in $c^{\delta}_j$ and $c_j^n$ as well as the error which occurs when comparing terms in $c^{\delta}_j$ and $c_j$. The error in the first comparison is mainly comprised of two parts. The first  involves the error between $\eta_0^n$ and $\eta_0$. To control this,  note that if $\delta>0$ is sufficiently small and $n$ is sufficiently large we have
\begin{equation*}
\|\eta^n_0-\eta_0\|_{H^s(\Gamma_*)}<\delta<\epsilon.    
\end{equation*}
The second source of error comes from the extensions of the variables $W$ and $\nabla_B W$. That is,
\begin{equation*}
\|E_{\Omega_0^n}W_0^n-E_{\Omega_0}W_0^{\delta}\|_{H^s(\mathbb{R}^d)}\leq \|E_{\Omega_0^n}W_0^{\delta}-E_{\Omega_0}W_0^{\delta}\|_{H^s(\mathbb{R}^d)}+\|E_{\Omega_0^n}(W_0^n-W_0^{\delta})\|_{H^s(\mathbb{R}^d)}   
\end{equation*}
and
\begin{equation*}
\begin{split}
\|E_{\Omega_0^n}(\nabla_{B_0^n}W_0^n)-E_{\Omega_0}(\nabla_{B_0^{\delta}}W_0^{\delta})\|_{H^{s-\frac{1}{2}}(\mathbb{R}^d)}&\leq \|E_{\Omega_0^n}(\nabla_{B_0^{\delta}}W_0^{\delta})-E_{\Omega_0}(\nabla_{B_0^{\delta}}W_0^{\delta})\|_{H^{s-\frac{1}{2}}(\mathbb{R}^d)}
\\
&+\|E_{\Omega_0^n}(\nabla_{B_0^n}W_0^n-\nabla_{B_0^{\delta}}W_0^{\delta})\|_{H^{s-\frac{1}{2}}(\mathbb{R}^d)}.
\end{split}
\end{equation*}
To handle these errors, note that if $\delta\ll_M\epsilon$ then the uniform in $n$ boundedness of $E_{\Omega_0^n}$ and the definition of $W_0^{\delta}$ guarantees that the second term on the right-hand side of each of the above estimates is $\mathcal{O}(\epsilon)$. By the continuity property of the family $E_{\Omega_0^n}$ in \Cref{continuosext}, the first term on the right-hand side of each of the above estimates is also $\mathcal{O}(\epsilon)$ if $n$ is large enough relative to $\delta$. Finally, to establish \eqref{envcontrol} it is left to compare $c_j$ and $c_j^{\delta}$. This is easier, as it essentially involves  controlling  the error terms $\|E_{\Omega_0}(W_0^{\delta}-W_0)\|_{H^s(\mathbb{R}^d)}$ and $\|E_{\Omega_0}(\nabla_{B_0^{\delta}}W_0^{\delta}-\nabla_{B_0}W_0)\|_{H^{s-\frac{1}{2}}(\mathbb{R}^d)}$. We leave the details to the reader.
\medskip

Having established uniform smallness of the initial data frequency envelopes, the next step is to compare the corresponding solutions. Using the difference estimates, we observe that for sufficiently large $n$, the hypersurfaces $\Gamma^n$ and $\Gamma^{\delta}$ are within distance $\ll 2^{-j}$ as long as $\delta>0$ is chosen small enough relative to the integer $j$  which was chosen to ensure \eqref{dataenvsize}. Indeed, using the uniform $\mathbf{H}^s$ bounds and interpolation, we have
\begin{equation*}
\|\eta^n-\eta^{\delta}\|_{L^{\infty}(\Gamma_*)}\lesssim_M D((W^n,\Gamma^n),(W^{\delta},\Gamma^\delta))^{\frac{3}{4s}}\lesssim_{M} \delta^{\frac{3}{2s}},    
\end{equation*}
which ensures that we may compare $\Psi_{\leq j}W^{\delta}$ with $W^n$. Letting $(W^n_j,\Gamma^n_j)$  denote the regular solution corresponding to the regularized data $(W^n_{0,j},\Gamma_{0,j}^n)$ from the previous section, we have 
\begin{equation*}
\begin{split}
\|\Psi_{\leq j}W^{\delta}-W^n\|_{H^s(\Omega^n)}&\lesssim \|\Psi_{\leq j}(W^{\delta}-W^n)\|_{H^s(\Omega^n)}+\|\Psi_{\leq j}(W^n-W^n_j)\|_{H^s(\Omega^n)}+\|W^n-\Psi_{\leq j}W^n_j\|_{H^s(\Omega^n)}
\\
&\lesssim_M \|c_{\geq  j}^n\|_{l^2}+2^{js}D((W^n,\Gamma^n),(W^n_j,\Gamma^n_j))^{\frac{1}{2}}+2^{js}D((W^n,\Gamma^n),(W^{\delta},\Gamma^\delta))^{\frac{1}{2}}
\\
&\lesssim_M \|c_{\geq j}^n\|_{l^2}+2^{js}D((W^n,\Gamma^n),(W^{\delta},\Gamma^{\delta}))^{\frac{1}{2}},
\end{split}
\end{equation*}
which for sufficiently small $\delta>0$  gives
\begin{equation*}
\|\Psi_{\leq j}W^{\delta}-W^n\|_{H^s(\Omega^n)}\lesssim_M \epsilon.    
\end{equation*}
By employing a suitable analogue of \Cref{L^2regbounds} and arguing similarly to the above, we also have
\begin{equation*}
\|\Phi_{\leq j}(\nabla_{B^{\delta}}W^{\delta})-\nabla_{B^n}W^n\|_{H^{s-\frac{1}{2}}(\Omega^n)}\lesssim_M\epsilon.
\end{equation*}
Similarly, one may show that
\begin{equation*}
\|\eta^n-\eta\|_{H^s(\Gamma_*)}\lesssim_M \epsilon    
\end{equation*}
and
\begin{equation*}
\|\Psi_{\leq j}W^{\delta}-W\|_{H^s(\Omega)}\lesssim_M \epsilon,\hspace{5mm}\|\Phi_{\leq j}(\nabla_{B^{\delta}}W^{\delta})-\nabla_{B}W\|_{H^{s-\frac{1}{2}}(\Omega)}\lesssim_M\epsilon.
\end{equation*}
This completes the proof of continuous dependence.

\appendix

\section{Elliptic estimates, regularization operators and function space theory}\label{BEE}
The purpose of this appendix is to collect all of the auxiliary tools from \cite{Euler} which will be needed for our well-posedness proof. This includes extension and regularization operators, Littlewood-Paley projections, nonlinear estimates and elliptic theory. 
\medskip

The notation in this appendix is consistent with \Cref{AOMD}. In particular, $\Omega$  denotes a bounded, connected domain with boundary $\Gamma\in\Lambda_*:=\Lambda(\Gamma_*,\epsilon_0,\delta)$ for some  universal constants $0<\epsilon_0,\delta\ll 1$.  Since  the bounds in this appendix  rarely make reference to the MHD equations,  implicit constants will  usually only  depend   on the surface component of the control parameter $A$; namely, $A_{\Gamma}:=\|\Gamma\|_{C^{1,\epsilon_0}}$. Hence, for the purposes of  this appendix, by the relation $X\lesssim_A Y$, we mean $X\leq C(A_{\Gamma})Y$ for some constant $C$ depending exclusively on $A_{\Gamma}$ and the domain volume. The main exception to this rule  occurs in \Cref{Movingsurfid},  where we consider a family of moving domains $\Omega_t$ flowing with velocity $v$ (not assumed to solve the free boundary MHD equations) and  estimate commutators of  elliptic operators with material derivatives. In this case, we will specify the dependence of our control parameters on $v$ and $\Gamma$ explicitly.  We  remark that  allowing implicit constants to depend on the domain volume is completely harmless, as the volume is  conserved by the MHD flow and is uniform  in  the collar. 
\subsection{Extension operators and product  estimates on \texorpdfstring{$\Omega$}{}}
Our first objective is to construct an extension operator which is bounded from $H^s(\Omega)\to H^s(\mathbb{R}^d)$ for $s\geq 0$ as well as from $C^{k,\alpha}(\Omega)\to C^{k,\alpha}(\mathbb{R}^d)$ for a  range of $k$ and $\alpha$, with norm bounds depending solely on the control parameter $A$. This will allow us to transfer much of the standard theory of function spaces on $\mathbb{R}^d$ to  $\Omega$.
\medskip

 Let $U$ be a non-empty open set in $\mathbb{R}^d$. For $1\leq p\leq\infty$ and an integer $k\geq 0$, we let $W^{k,p}(U)$ denote the usual Sobolev space consisting of distributions whose derivatives up to order $k$ belong to $L^p(U)$. Given  a Lipschitz function $\varphi:\mathbb{R}^{d-1}\to \mathbb{R}$ with Lipschitz constant $M$, we  define the open set  $U_\varphi:=\{(x,y)\in \mathbb{R}^d: y>\varphi(x)\}$. A classical result of Stein \cite[Theorem 5', p.~181]{MR0290095} asserts that there exists a linear operator $\mathcal{E}:=\mathcal{E}_\varphi$ mapping functions on $U:=U_\varphi$ to functions on $\mathbb{R}^d$ with the property that $\mathcal{E}: W^{k,p}(U)\to W^{k,p}(\mathbb{R}^d)$ is well-defined and continuous for all $1\leq p\leq \infty$ and integers $k$. Moreover, the norm of $\mathcal{E} : W^{k,p}(U)\to W^{k,p}(\mathbb{R}^d)$ depends exclusively on the dimension $d$, the order of differentiability $k$ and the Lipschitz constant $M$.  As one may observe directly from its definition \cite[Equation (24), p.~182]{MR0290095}, $\mathcal{E}$ also maps $C^1(U)\to C^1(\mathbb{R}^d)$.
\medskip

Using a standard partition of unity argument one may construct an extension operator $\mathcal{E}:=\mathcal{E}_\Omega$ on any Lipschitz domain $\Omega$, with $W^{k,p}$ operator norm depending solely on $d,k,p$, the number and size of the balls needed to cover the boundary, and the Lipschitz constant of the defining function on each ball. The operator $\mathcal{E}$ is called \emph{Stein's extension operator}. Since for a tight enough collar $ \Lambda_*$ one may use the same balls to cover each $\Gamma\in \Lambda_*$, it is easy to see that $\mathcal{E}$ has operator bounds  which are uniform in the collar. 

\begin{remark}
In the last two paragraphs, the $W^{k,p}(\Omega)$ norm was defined in terms of weak-derivatives. However, in \Cref{AOMD} we  defined the $H^s(\Omega)$ norm of a function $f$ as the infimum of the $H^s(\mathbb{R}^d)$ norms of all possible extensions of $f$ to $\mathbb{R}^d$. Clearly, $\|\cdot\|_{W^{k,2}(\Omega)}\lesssim\|\cdot\|_{H^{k}(\Omega)}$ with universal implicit constant. On the other hand, using Stein's extension operator, it is easy to see that the reverse inequality holds for domains in the collar, with constant depending only on $A$.
\end{remark}

Recall  that for $s_0,s_1\in \mathbb{R}$ we have
\begin{equation}\label{Interpolation on domains}
    \left(H^{s_0}(\Omega),H^{s_1}(\Omega)\right)_{\theta,2}=H^s(\Omega), \ \text{where} \ s=(1-\theta)s_0+\theta s_1 \ \text{and} \ 0<\theta<1,
\end{equation}
with equivalent norms uniform in the collar. This allows us to extend the above mapping properties of $\mathcal{E}$ to fractional regularity spaces.
\begin{proposition}\label{Stein} 
 Let $\Omega$ be a bounded domain with boundary $\Gamma\in \Lambda_*$. Then for every $s\geq 0$ and $0\leq\alpha\leq 1+\epsilon_0$, Stein's extension operator $\mathcal{E}$ satisfies
\begin{equation*}
\|\mathcal{E}\|_{C^{\alpha}(\Omega)\to C^{\alpha}(\mathbb{R}^d)},\hspace{5mm}\|\mathcal{E}\|_{H^s(\Omega)\to H^s(\mathbb{R}^d)}\lesssim_A 1,
\end{equation*}
uniformly with respect to  $\Gamma  \in  \Lambda_*$. 
\end{proposition}
\begin{proof}
See \cite[Proposition 5.1]{Euler}.
\end{proof}
\subsection{Littlewood-Paley theory, paraproducts and bilinear estimates on \texorpdfstring{$\Omega$}{}} We now use Stein's extension operator to define  paraproducts  on  $\Omega$ and  prove bilinear estimates. 
\subsubsection{Littlewood-Paley decomposition}\label{LWP section} For a function $u$ on $\mathbb{R}^d$, we recall the standard Littlewood-Paley decomposition
\begin{equation*}
u=\sum_{k\geq 0}P_ku,  
\end{equation*}
where $P_0$ is a Fourier multiplier localized to the unit ball  and $P_k$, $k>0$, is a Fourier multiplier with a smooth symbol supported in the dyadic frequency region $|\xi|\approx 2^k$. The notation $P_{<k}$, $P_{\leq k}$, $P_{\geq k}$ and $P_{>k}$ will have the usual meaning. To define Littlewood-Paley projections when $u:\Omega\to \mathbb{R}$, we use Stein's extension operator. More specifically, we abuse notation and  define $P_ku:=P_k\mathcal{E}u$, with similar definitions for $P_{<k}$, $P_{\leq k}$, etc.~We also write $u_k$, $u_{<k}$, etc.~as shorthand for these operators applied to $u$.
\subsubsection{Paraproducts on $\Omega$} The above decomposition permits us to make use of certain aspects of the paradifferential calculus  on $\mathbb{R}^d$ for functions defined only on $\Omega$. For bilinear expressions in $f,g:\Omega\to \mathbb{R}$, we will make extensive use of the Littlewood-Paley trichotomy 
\begin{equation*}
f\cdot g=T_fg+T_gf+\Pi(f,g)  ,  
\end{equation*}
where the above three terms correspond to the respective ``low-high", ``high-low" and ``high-high" frequency interactions between $f$ and $g$. More precisely, $T_fg$ is defined as 
\begin{equation*}
T_fg:=\sum_{k}f_{<k-k_0}g_k    ,
\end{equation*}
where $k_0$ is some universal parameter independent of $k$. For most purposes, we will  take $k_0=4$.
\subsubsection{Multilinear estimates on $\Omega$}
As a consequence of the bounds for $\mathcal{E}$ and the corresponding inequality on $\mathbb{R}^d$, we have the  algebra property 
\begin{equation}\label{algebrapropproduct}
\|fg\|_{H^s(\Omega)}\lesssim_A \|f\|_{H^s(\Omega)}\|g\|_{L^{\infty}(\Omega)}+\|g\|_{H^s(\Omega)}\|f\|_{L^{\infty}(\Omega)},  
\end{equation}
when $s\geq 0$. In many of our estimates,  bilinear terms  will  appear in the form $\partial_if\partial_jg$ where $f:\mathbb{R}^d\to \mathbb{R}$  encodes the regularity of the domain and the desired uniform bound for $g$ is below $C^1$. In order to avoid negative regularity H\"older norms, we will need the following paraproduct type bound.
\begin{proposition}[Bilinear paraproduct type estimate on $\Omega$]\label{productestref}
Let either i) $s> 0$ and $\alpha_1,\alpha_2,\beta\in [0,1]$ or ii) $s=0$, $\alpha_1=\alpha_2=1$ and $\beta\in[0,1]$. Then we have for any $r\geq 0$,
\begin{equation*}
\begin{split}
\|\partial_if\partial_jg\|_{H^s(\Omega)}\lesssim_A&\|g\|_{H^{s+2-\alpha_1}(\Omega)}\|f\|_{C^{\alpha_1}(\Omega)}+\| f\|_{H^{s+r+1}(\Omega)}\sup_{k>0}2^{-k(r+\alpha_2-1)}\|g_k^1\|_{C^{\alpha_2}(\Omega)}
\\
+&\|f\|_{C^{1,2\epsilon}(\Omega)}\sup_{k>0}2^{k(s+\beta-\epsilon)}\|g_k^2\|_{H^{1-\beta}(\Omega)},
\end{split}
\end{equation*}
where $g=g_k^1+g_k^2$ is any sequence of partitions of $g$ in $C^{\alpha_2}(\Omega)+H^{1-\beta}(\Omega)$.
\end{proposition}
\begin{proof}
See \cite[Proposition 5.3]{Euler}.
\end{proof}
The following multilinear estimates will be used in our energy estimates to control product terms on $\Omega$.
\begin{proposition}[Multilinear Estimates]\label{Multilinearest} The following estimates hold.
\begin{enumerate}
\item (Bilinear estimate). Let $s$ and $\alpha_1,\alpha_2$ be as in  \Cref{productestref}. Assume that $f\in H^{s+2-\alpha_2}(\Omega)\cap C^{\alpha_1}(\Omega)$ and $g\in H^{s+2-\alpha_1}(\Omega)\cap C^{\alpha_2}(\Omega)$. Then we have
\begin{equation*}
 \|\partial_if\partial_jg\|_{H^s(\Omega)}\lesssim_A \|g\|_{H^{s+2-\alpha_1}(\Omega)}\|f\|_{C^{\alpha_1}(\Omega)}+\|f\|_{H^{s+2-\alpha_2}(\Omega)}\|g\|_{C^{\alpha_2}(\Omega)}.   
\end{equation*}    
\item (Trilinear estimate).  Let $\epsilon>0$ be a small parameter and let $\alpha_m,\beta_m,\gamma_m\in [0,1]$ for $m=1,2$. Then for every $s\geq 0$, there holds
\begin{equation*}
\begin{split}
\|\partial_if\partial_jg\partial_kh\|_{H^s(\Omega)}\lesssim_A &\|f\|_{H^{s+3-\alpha_1-\alpha_2}(\Omega)}\|g\|_{C^{\alpha_1+\epsilon}(\Omega)}\|h\|_{C^{\alpha_2+\epsilon}(\Omega)}+\|g\|_{H^{s+3-\beta_1-\beta_2}(\Omega)}\|f\|_{C^{\beta_1+\epsilon}(\Omega)}\|h\|_{C^{\beta_2+\epsilon}(\Omega)}
\\
+&\|h\|_{H^{s+3-\gamma_1-\gamma_2}(\Omega)}\|f\|_{C^{\gamma_1+\epsilon}(\Omega)}\|g\|_{C^{\gamma_2+\epsilon}(\Omega)}.
\end{split}
\end{equation*}
\end{enumerate}
\begin{proof}
The bilinear estimate is simply \cite[Corollary 5.4]{Euler}. The trilinear estimate follows from a standard application of the Littlewood-Paley trichotomy and a similar line of reasoning to the proof of \cite[Proposition 5.3]{Euler}. We omit the details.
\end{proof}
\end{proposition}
\subsection{Local coordinate parameterizations and Sobolev norms on hypersurfaces} Our next aim is to define Sobolev norms on $\Gamma\in \Lambda_*$ and recall certain refined versions of the trace and product estimates that will be used throughout the paper. 
\subsubsection{Local coordinates and a universal partition of unity}\label{Subsub local}
 We begin by designing a ``universal" set of coordinate neighborhoods  for $\Gamma_*$ which will enable us to flatten the boundaries of nearby hypersurfaces $\Gamma\in\Lambda_*$ in a uniform fashion. This will, in particular, allow us to consistently define Sobolev type norms on $\Gamma\in \Lambda_*$. Below we present a sketch of  the construction;  the reader is referred  to \cite[Section 5.3.1]{Euler} for additional details.
\medskip

For any set $S\subseteq \mathbb{R}^d$ and $\epsilon>0$, we let  $B(S,\epsilon)$ denote the $\epsilon$ neighborhood of $S$. Since $\Gamma_*$ is compact, for any $\sigma>0$ we may select $x_i\in \mathbb{R}^d$ and $r,r_i\in (0,\frac{1}{2}]$, $i=1,\dots,m$, such that
 \begin{enumerate}
     \item\label{prop1} $B(\Gamma_*,r)\subseteq\cup_{i=1}^m R_i(r_i)$ where $R_i(\cdot):=\tilde{R}_i(\cdot)\times I_i(\cdot)\subseteq\mathbb{R}^d$ is a rotated cylinder with perpendicular vertical segment centered at $x_i$ with the given equal radius and length.
     \item\label{prop2} For each $i$, there exists a function $f_{*i}:\tilde{R}_i(2r_i)\to I_i$ satisfying
     \begin{equation}\label{smallness}
         \|f_{*i}\|_{C^0}<\sigma r_i,\ \ \|Df_{*i}\|_{C^0}<\sigma \ \ \text{and} \ \Omega_*\cap R_i(2r_i)=\{z_d>f_{*i}(\tilde{z})\},
     \end{equation}
      where $z=(\tilde{z},z_d)$ denotes the standard Euclidean coordinates on $R_i$.
 \end{enumerate}
In the above setting, it is easy to see that when $\delta>0$ is sufficiently small,  \eqref{prop1} holds with $\Gamma_*$ replaced by any $\Gamma:=\partial\Omega\in \Lambda_*$. Moreover, there exist functions $f_i:\tilde{R}_i(2r_i)\to I_i$ satisfying \eqref{prop2}  with $\Omega_*$ replaced by $\Omega$ such that we may control the Sobolev and H\"older  norms of $f_i$ by the corresponding norms of $\Gamma$. More specifically, for  any $s\geq 0$, integer $k\geq 0$ and $\alpha\in [0,1)$, we have
\begin{equation*}
\|f_i\|_{H^{s}}\lesssim_A 1+\|\Gamma\|_{H^s},\hspace{10mm}\|f_i\|_{C^{k,\alpha}}\lesssim_A 1+\|\Gamma\|_{C^{k,\alpha}}. 
\end{equation*}
  Using these coordinate representations, we aim to design local coordinate maps on each $\tilde{R}_i(2r_i)$ which flatten $\Gamma$ and have  estimates which are uniform in $\Lambda_*$. 
\begin{remark}
    In some of the estimates below, we will abuse notation by writing $\|\Gamma\|$ instead of $1+\|\Gamma\|$. This convention is used to streamline the notation and will not affect any of the analysis of the free boundary MHD equations.
\end{remark}
Let $\overline{\gamma} : [0,\infty)\to [0,1]$ be a smooth cutoff which is equal to $1$ on $[0,\frac{5}{4}]$ and is supported on $[0,\frac{3}{2}]$. On each $\tilde{R}_i(2r_i)$ we define $\phi_i=\gamma_if_i$, where $\gamma_i(\tilde{z}):=\overline{\gamma}\left(\frac{|\tilde{z}|}{r_i}\right)$. We may extend $\phi_i$ to a function on $\mathbb{R}^d$ which  is bounded in suitable pointwise norms and gains half a degree of regularity in $H^s$ norms. Indeed, we may  define the extension $\Phi_i$ of $\phi_i$ by
    \begin{equation*}\label{fullspaceext}
       \Phi_i(z)=\int_{\mathbb{R}^{d-1}}\widehat{\phi_i}(\xi')e^{-(1+|\xi'|^2)z_d^2}e^{2\pi i\xi'\cdot \tilde{z}} d\xi' \ \text{for} \ z=(\tilde{z},z_d)\in \mathbb{R}^d.
    \end{equation*}
We claim that for each integer $k\geq 0$ and $\alpha\in [0,1)$,  we have the bound $\|\Phi_i\|_{C^{k,\alpha}(\mathbb{R}^d)}\lesssim_{k,\alpha} \|\phi_i\|_{C^{k,\alpha}(\mathbb{R}^{d-1})}$, with a similar  bound in $W^{k,\infty}$ for each $k\geq 0$. To see this, one must simply observe that $\Phi_i$ may be re-written as the convolution
\begin{equation*}
\Phi_i(z)=c_de^{-z_d^2}\int_{\mathbb{R}^{d-1}}\phi_i(\tilde{z}+z_dy)e^{-|y|^2}dy   ,
\end{equation*}
where $c_d$ is a dimensional constant. The gain of Sobolev regularity $\|\Phi_i\|_{H^{s+\frac{1}{2}}(\mathbb{R}^d)}\approx_s\|\phi_i\|_{H^s(\mathbb{R}^{d-1})}$ for  $s\geq 0$ follows from an inspection of the Fourier transform of $\Phi_i$. 
\medskip

Using the above properties, it is easy to see that when $\sigma>0$ from \eqref{smallness} is sufficiently small,  the map 
\begin{equation*}
    H_i(\tilde{z},z_d):=(\Tilde{z},z_d+\Phi_i(\tilde{z},z_d))
\end{equation*}
is a diffeomorphism from $\mathbb{R}^d\to\mathbb{R}^d$. Moreover,  for each $s\geq 0$, integer $k\geq 0$ and $\alpha\in [0,1)$, we have the bounds $\|H_i-Id\|_{C^{k,\alpha}}\lesssim_A \|\Gamma\|_{C^{k,\alpha}}$ and $\|H_i-Id\|_{H^{s+\frac{1}{2}}}\lesssim_A \|\Gamma\|_{H^s}$. Similar bounds hold for $G_i-Id$, where $G_i:=H_i^{-1}$ is the inverse function. Furthermore, the $d'$th component $g_i$ of $G_i$ satisfies the bounds $|\partial_{z_d}g_i|+|(\partial_{z_d}g_i)^{-1}|\lesssim_A 1$. Hence, if $\Lambda_*$ is a tight enough collar neighborhood and $\sigma>0$ is sufficiently small, we have
\begin{equation*}
\|H_i-Id\|_{C^1}+\|G_i-Id\|_{C^1}\lesssim_A \rho,
\end{equation*}
where $\rho>0$ is a positive constant which can be made as small as we wish by adjusting $\sigma$ and the width of the collar.  It follows, in particular, that for some uniform constant $\delta_*>0$ we have
\begin{equation*}
    \left(\tilde{R}_i\left(\frac{5}{4}r_i\right)\times I_i\left(\frac{5}{4}\delta_*r_i\right)\right)\cap \Omega=\left(\tilde{R}_i\left(\frac{5}{4}r_i\right)\times I_i\left(\frac{5}{4}\delta_*r_i\right)\right)\cap\{g_i>0\}.
\end{equation*}
\textbf{Partition of unity.} We now construct a partition of unity for $\Omega$ with bounds uniform in $\Lambda_*$. For this, we let $\gamma$ be a smooth cutoff defined on $[0,\infty)$ which is equal to $1$ on $[0,\frac{9}{8}]$, has  support in $[0,\frac{5}{4})$ and satisfies $0\leq \gamma\leq 1$. We also let $\zeta$ be a smooth function defined on $[0,\infty)$ which takes values in $[\frac{1}{3},\infty)$ and has the property that $\zeta=\frac{1}{3}$ on $[0,\frac{1}{3}]$ and $\zeta(x)=x$ for $x\geq \frac{2}{3}$. We then define
\begin{equation*}
 \tilde{\gamma}_{*i}(z):=\gamma(\frac{|\tilde{z}|}{r_i})\gamma(\frac{|z_d|}{\delta_*r_i}),\hspace{10mm}\eta:=\zeta\circ\sum_i(\tilde{\gamma}_{*i}\circ G_i) 
\end{equation*}
and
\begin{equation*}\label{Part of unity gamma*}
\gamma_{*i}:=\frac{\tilde{\gamma}_{*i}(G_i)}{\eta},\hspace{10mm}  \gamma_{*0}:=(1-\sum_i\gamma_{*i})\mathbbm{1}_{\Omega}.  
\end{equation*}
It is easy see that $\sum_{i\geq 0}\gamma_{*i}=1$  on $\Omega$ and $0\leq \gamma_{*i}\leq 1$ for each $i\geq 0$. Moreover, using  Moser and Sobolev product estimates, it may be verified that for each for $s\geq 0$ we have
\begin{equation*}
\|\gamma_{*i}\|_{H^{s+\frac{1}{2}}}\lesssim_A \|\Gamma\|_{H^s}. 
\end{equation*}
\subsubsection{Sobolev spaces on hypersurfaces in $\Lambda_*$} The above partition of unity may be used to consistently define $C^{k,\alpha}$ and $H^s$ spaces on hypersurfaces $\Gamma\in\Lambda_*$. Indeed, if $\Gamma$ is $C^1$ and in $H^s$, we may define the space $H^r(\Gamma)$ for $0\leq r\leq s$ through the inner product,
\begin{equation*}\label{def sob norm}
\langle f,g\rangle_{H^r(\Gamma)}:=\sum_{i\geq 1}\langle \phi_if_i,\phi_ig_i\rangle_{H^r(\mathbb{R}^{d-1})}   ,
\end{equation*}
where $\phi_i:=\gamma_{*i}\circ H_i(\tilde{z},0)$, $f_i:=f\circ H_i(\tilde{z},0)$ and $g_i:=g\circ H_i(\tilde{z},0)$. If $\Gamma$ is $C^{k,\alpha}$ we may also define
\begin{equation*}
\|f\|_{C^{k,\alpha}(\Gamma)}:=\sup_{i\geq 1}\|\phi_if_i\|_{C^{k,\alpha}(\mathbb{R}^{d-1})}.    
\end{equation*}
We remark, of course, that the $\phi_i$ in these definitions are not the same as the $\phi_i$ in \Cref{Subsub local}. 
\medskip

Using  the above framework and a general Moser estimate from \cite[Proposition 5.5]{Euler} one may prove the following refined product type estimate on the boundary $\Gamma$. 
\begin{proposition}[Product estimates on the boundary]\label{boundaryest}
Let $\Omega$ be a bounded domain with boundary $\Gamma\in\Lambda_*$. If $f,g$ are functions on $\Gamma$ and $g=g_j^1+g_j^2$ is any sequence of partitions, then for $s\geq 0$ and $r\geq 1$ we have
\begin{equation*}
\begin{split}
\|fg\|_{H^s(\Gamma)}\lesssim_A \|f\|_{L^{\infty}(\Gamma)}\|g\|_{H^s(\Gamma)}&+(\|f\|_{H^{s+r-1}(\Gamma)}+\|f\|_{L^{\infty}(\Gamma)}\|\Gamma\|_{H^{s+r}})\sup_{j>0}2^{-j(r-1)}\|g_j^1\|_{L^{\infty}(\Gamma)}
\\
&+(1+\|f\|_{C^{2\epsilon}(\Gamma)})\sup_{j>0}2^{j(s-\epsilon)}\|g_j^2\|_{L^2(\Gamma)}.
\end{split}
\end{equation*}
\end{proposition}
\begin{proof}
See \cite[Proposition 5.9]{Euler}.
\end{proof}
In a similar spirit to \Cref{boundaryest}, we have the following refined version of the trace theorem for $\Gamma$. 
\begin{proposition}[Balanced trace estimate]\label{baltrace} Let $\Omega$ be a bounded domain with boundary $\Gamma\in\Lambda_*$. For every $s>\frac{1}{2}$, $r\geq 0$, $\alpha,\beta\in [0,1]$ and every sequence of partitions $v=v_j^1+v_j^2$, we have
\begin{equation*}
    \|v_{|\Gamma}\|_{H^{s-\frac{1}{2}}(\Gamma)}\lesssim_A\|v\|_{H^s(\Omega)}+\|\Gamma\|_{H^{s+r-\frac{1}{2}}}\sup_{j>0}2^{-j(r+\alpha-1)}\|v_j^1\|_{C^{\alpha}(\Omega)}+\sup_{j>0}2^{j(s-1+\beta-\epsilon)}\|v_j^2\|_{H^{1-\beta}(\Omega)}.
\end{equation*}
\end{proposition}
\begin{proof}
See \cite[Proposition 5.11]{Euler}.
\end{proof}
Finally, we recall that we may control  the surface regularity in terms of the mean curvature $\kappa$. 
\begin{proposition}[Curvature estimate]\label{curvaturebound} Let $s\geq 2$. The following estimates for $\|\Gamma\|_{H^s}$ and the normal $n_\Gamma$ hold:
\begin{equation*}
\|\Gamma\|_{H^s}+\|n_\Gamma\|_{H^{s-1}(\Gamma)}\lesssim_A 1+\|\kappa\|_{H^{s-2}(\Gamma)}.
\end{equation*}
\end{proposition}
\begin{proof}
See \cite[Proposition 5.22]{Euler}.
\end{proof}
\subsubsection{An extension operator depending continuously on the domain}\label{cont of domain ext} To establish continuous dependence of the free boundary MHD flow, we will need a family of extension operators that depend continuously on the domain in a suitable sense. The existence of such a family of operators can be easily deduced from the above local coordinates.   
\begin{proposition}\label{continuosext}
Fix a collar neighborhood $\Lambda_*$ and let $s>\frac{d}{2}+\frac{1}{2}$. For each bounded domain $\Omega$ with $H^s$ boundary $\Gamma\in\Lambda_*$ there exists an extension operator $E_{\Omega}:H^s(\Omega)\to H^s(\mathbb{R}^d)$ such that for all $v\in H^s(\Omega)$, 
\begin{equation*}\label{otherextensionbounds}
\|E_{\Omega}v\|_{H^s(\mathbb{R}^d)}+\|\Gamma\|_{H^s}\approx_{A,\|v\|_{C^{\frac{1}{2}}(\Omega)}}\|(v,\Gamma)\|_{H^s},\hspace{5mm}\|E_{\Omega}v\|_{H^s(\mathbb{R}^d)}\lesssim_{A} \|\Gamma\|_{H^{s}}\|v\|_{H^s(\Omega)},  
\end{equation*}
where the dependence on $\|v\|_{C^{\frac{1}{2}}(\Omega)}$ is polynomial. Moreover, if $\Omega_n$ is a sequence of domains with $\Gamma_n\to \Gamma$ in $H^s$, then for every $v\in H^s(\mathbb{R}^d)$, there holds
\begin{equation*}\label{SoT}
\|E_{\Omega_n}v_{|\Omega_n}-E_{\Omega}v_{|\Omega}\|_{H^s(\mathbb{R}^d)}\to 0.    
\end{equation*}
\end{proposition}
\begin{proof}
This is a slightly cruder version of \cite[Proposition 5.12]{Euler}.
\end{proof}
\begin{remark}
We will only use the above extension operator in \Cref{Freq envelopes} when we define frequency envelopes and in \Cref{Cont dep section} when we establish continuity of the flow map for the free boundary MHD equations. In all other cases, our default extension operator will be that of Stein.
\end{remark}
\subsection{Balanced elliptic estimates}
We now collect the  refined elliptic estimates from \cite{Euler}.
\subsubsection{Pointwise elliptic estimates}
We begin by stating certain variants of the  $C^{1,\alpha}$ estimates for the Dirichlet problem which precisely quantify the regularity of the domain. We will mostly use these estimates when  $\alpha=\frac{1}{2}$ or $\alpha=\epsilon$.
\begin{proposition}[$C^{1,\alpha}$ estimates for the Dirichlet problem]\label{Gilbarg} Let $\Omega$ be a bounded $C^{1,\alpha}$ domain with $0<\alpha<1$ and with boundary $\Gamma\in\Lambda_*$. Consider the boundary value problem
\begin{equation*}\label{dirprob}
\begin{cases}
&\Delta v=\nabla\cdot g_1+g_2\hspace{5mm}\text{in }\Omega,
\\
&\hspace{3mm}v=\psi\hspace{10mm}\text{on }\partial\Omega.
\end{cases}
\end{equation*}
Then $v$ satisfies the estimate 
\begin{equation*}\label{C1a}
\|v\|_{C^{1,\alpha}(\Omega)}\lesssim_A \|\Gamma\|_{C^{1,\alpha}}(\|v\|_{W^{1,\infty}(\Omega)}+\|g_1\|_{L^{\infty}(\Omega)})+\|g_1\|_{C^{\alpha}(\Omega)}+\|g_2\|_{L^{\infty}(\Omega)}+\|\psi\|_{C^{1,\alpha}(\Gamma)}.
\end{equation*}
Interpolating and using the straightforward estimate 
\begin{equation*}
\|v\|_{L^{\infty}(\Omega)}\lesssim_A \|g_1\|_{L^{\infty}(\Omega)}+\|g_2\|_{L^{\infty}(\Omega)}+\|\psi\|_{L^{\infty}(\Gamma)},
\end{equation*}
 we deduce that
\begin{equation*}\label{C1eps}
\|v\|_{C^{1,\epsilon}(\Omega)}\lesssim_A \|g_1\|_{C^{\epsilon}(\Omega)}+\|g_2\|_{L^{\infty}(\Omega)}+\|\psi\|_{C^{1,\epsilon}(\Gamma)}
\end{equation*}
and
\begin{equation*}\label{C1est2}
\|v\|_{C^{1,\alpha}(\Omega)}\lesssim_A \|\Gamma\|_{C^{1,\alpha}}(\|g_1\|_{C^{\epsilon}(\Omega)}+\|g_2\|_{L^{\infty}(\Omega)}+\|\psi\|_{C^{1,\epsilon}(\Gamma)})+\|g_1\|_{C^{\alpha}(\Omega)}+\|g_2\|_{L^{\infty}(\Omega)}+\|\psi\|_{C^{1,\alpha}(\Gamma)}   . 
\end{equation*}
\end{proposition}
\begin{proof}
See \cite[Proposition 5.15]{Euler}.
\end{proof}
When $g_1$ and $g_2$ are both zero, we may use the maximum principle,  the $C^{1,\epsilon}$ bounds in \Cref{Gilbarg}  and interpolation to obtain $C^{\alpha}$ bounds for $\mathcal{H}$ with constant depending only on $A$.
\begin{corollary}\label{Hboundlow}
Let $0\leq \alpha<1$. The following low regularity bound for $\mathcal{H}$ holds uniformly for domains $\Omega$ with boundary $\Gamma\in\Lambda_*$,
\begin{equation*}
\|\mathcal{H}g\|_{C^{\alpha}(\Omega)}\lesssim_A \|g\|_{C^{\alpha}(\Gamma)}.
\end{equation*}
\end{corollary}
\begin{proof}
See \cite[Corollary 5.16]{Euler}.
\end{proof}
\subsubsection{\texorpdfstring{$H^s$}{} estimates for the Dirichlet problem}
We now move on to  proving $H^s$ type estimates for various elliptic problems.
We begin by analyzing the inhomogeneous Dirichlet problem,
\begin{equation*}\label{direst1}
\begin{cases}
&\Delta v=g\hspace{5.5mm}\text{in }\Omega,
\\
&\hspace{3mm}v=\psi\hspace{5mm}\text{on }\Gamma.
\end{cases}
\end{equation*}
We recall  a few baseline estimates. The first  is when $\psi=0$, in which case \cite[Equation (5.19)]{Euler} implies that  $v$ satisfies the $H^1$ estimate 
\begin{equation*}\label{H1base}
\|v\|_{H^1(\Omega)}\lesssim_A \|g\|_{H^{-1}(\Omega)}.    
\end{equation*}
The other case is when   $g=0$ and $\frac{1}{2}<s\leq 1$, where   \cite[Equation (5.20)]{Euler} yields the estimate
\begin{equation}\label{harmonicbase}
\|v\|_{H^s(\Omega)}\lesssim_A \|\psi\|_{H^{s-\frac{1}{2}}(\Gamma)}.    
\end{equation}
We further recall the following elliptic estimates which hold on $C^{1,\epsilon_0}$ domains.
\begin{proposition}\label{desiredelliptic}
For every $0<s<\frac{1}{2}+\epsilon_0$ we have
\begin{equation*}
\|\Delta^{-1}g\|_{H^{s+1}(\Omega)}\lesssim_A \|g\|_{H^{s-1}(\Omega)},\hspace{5mm}\|\mathcal{H}\psi\|_{H^{s+1}(\Omega)}\lesssim_A \|\psi\|_{H^{s+\frac{1}{2}}(\Gamma)}.
\end{equation*}
\end{proposition}
\begin{proof}
See \cite[Proposition 5.18]{Euler}.
\end{proof}
The balanced higher regularity estimates for the Dirichlet problem are as follows.
\begin{proposition}[Higher regularity bounds for the inhomogeneous Dirichlet problem]\label{direst}  Let $\Omega$ be a bounded domain with boundary $\Gamma\in\Lambda_*$. Suppose that $v$ solves the Dirichlet problem
\begin{equation*}
\begin{cases}
&\Delta v=g\hspace{5.5mm}\text{in }\Omega,
\\
&\hspace{3mm}v=\psi\hspace{5mm}\text{on }\partial\Omega,
\end{cases}
\end{equation*}
and let $s\geq 2$. Then for $r\geq 0$, $\alpha\in [0,1],\beta\in [0,1]$ and any sequence of partitions $v:=v_j^1+v_j^2$, we have
\begin{equation*}
    \begin{split}
        \|v\|_{H^s(\Omega)}\lesssim_A\|g\|_{H^{s-2}(\Omega)}&+\|\psi\|_{H^{s-\frac{1}{2}}(\Gamma)}
        \\
        &+\|\Gamma\|_{H^{s+r-\frac{1}{2}}}\sup_{j>0}2^{-j\left(\alpha-1+r\right)}\|v_j^1\|_{C^\alpha(\Omega)}
        +\sup_{j>0}2^{j(s-1+\beta-\epsilon)}\|v_j^2\|_{H^{1-\beta}(\Omega)}.
    \end{split}
\end{equation*}
\end{proposition}
\begin{proof}
See \cite[Proposition 5.19]{Euler}.
\end{proof}
We will also occasionally need the following consequence of the above estimate in the case when $1<s<2$ and $d=2,3$.
\begin{corollary}\label{direstlow} Let $\Omega$ be a bounded domain with boundary $\Gamma\in \Lambda_*$. Let $d=2$ or $3$, $1<s<2$ and $s_0>\frac{d}{2}+1$. Then for $f\in H^{s-2}(\Omega)$ we have
\begin{equation}\label{lowestimatedir}
\|\Delta^{-1}f\|_{H^s(\Omega)}\lesssim_A C(\|\Gamma\|_{H^{s_0-\frac{1}{2}}})\|f\|_{H^{s-2}(\Omega)}
\end{equation}
where $C$ is some constant depending on $s_0$ which is sub-polynomial in $\|\Gamma\|_{H^{s_0-\frac{1}{2}}}$.
\end{corollary}
\begin{proof}
We show the details for the case $d=2$ first. We observe that for $\epsilon>0$ sufficiently small (depending on $s_0$), the parameter $\sigma=2+2\epsilon$ is simultaneously such that $s_0>\sigma$ and, moreover, we have the embedding $H^{\sigma-\epsilon}(\Omega)\subset C^1(\Omega)$. Hence, by \Cref{direst}, we have
\begin{equation*}
\|\Delta^{-1}f\|_{H^{\sigma}(\Omega)}\lesssim_{A} \|f\|_{H^{\sigma-2}(\Omega)}+\|\Gamma\|_{H^{s_0-\frac{1}{2}}}\|\Delta^{-1}f\|_{H^{\sigma-\epsilon}(\Omega)}.
\end{equation*}
Interpolating and using the estimate $\|\Delta^{-1}f\|_{H^1(\Omega)}\lesssim_A\|f\|_{H^{-1}(\Omega)}\lesssim_A\|f\|_{H^{\sigma-2}(\Omega)}$, we have
\begin{equation*}
\|\Delta^{-1}f\|_{H^\sigma(\Omega)}\lesssim_A C(\|\Gamma\|_{H^{s_0-\frac{1}{2}}})\|f\|_{H^{\sigma-2}(\Omega)}.
\end{equation*}
Interpolating this with the bound $\|\Delta^{-1}f\|_{H^1(\Omega)}\lesssim_A\|f\|_{H^{-1}(\Omega)}$ gives \eqref{lowestimatedir} in the case $d=2$. If $d=3$, we instead take $\sigma=\frac{5}{2}+2\epsilon$ and perform a similar analysis.
\end{proof}
\subsubsection{Harmonic extension bounds}
In the special case $g=0$, \Cref{direst} yields the following corollary for the harmonic extension operator $\mathcal{H}$.
\begin{proposition}[Harmonic extension bounds]\label{Hbounds}Let $\Omega$ be a bounded domain with boundary $\Gamma\in\Lambda_*$. Then the following bound holds for the harmonic extension operator $\mathcal{H}$ when $s\geq 2$, $r\geq 0$, $\beta\in [0,\frac{1}{2})$ and $\alpha\in [0,1)$,
\begin{equation*}
 \|\mathcal{H}\psi\|_{H^s(\Omega)}\lesssim_A \|\psi\|_{H^{s-\frac{1}{2}}(\Gamma)}+\|\Gamma\|_{H^{s+r-\frac{1}{2}}}\sup_{j>0}2^{-j(\alpha-1+r)}\|\psi_j^1\|_{C^\alpha(\Gamma)}+\sup_{j>0}2^{j(s-1+\beta-\epsilon)}\|\psi_j^2\|_{H^{\frac{1}{2}-\beta}(\Gamma)}.
\end{equation*}
Here, $\psi=\psi_j^1+\psi_j^2$ is any sequence of partitions.
\end{proposition}
\begin{proof}
See \cite[Proposition 5.21]{Euler}.
\end{proof}
\subsubsection{Estimates for the Dirichlet-to-Neumann operator}
Recall that the \emph{Dirichlet-to-Neumann operator}  is the mapping $\mathcal{N}:=n_\Gamma\cdot (\nabla\mathcal{H})_{|\Gamma}$. This operator plays a fundamental role in free boundary fluid dynamics. Here we recall its  mapping properties, beginning with a baseline ellipticity estimate.
\begin{lemma}\label{baselineDN}
    The Dirichlet-to-Neumann map on $\Gamma$ satisfies 
    \begin{equation*}
        \|\psi\|_{H^1(\Gamma)}+\|\nabla^{\top}\psi\|_{L^2(\Gamma)}\lesssim_A \|\mathcal{N}\psi\|_{L^2(\Gamma)}+\|\psi\|_{L^2(\Gamma)}.
    \end{equation*}
\end{lemma}
\begin{proof}
See the proof of \cite[Lemma 5.23]{Euler}.
\end{proof}
The reverse inequality also holds, in the following sense.
\begin{lemma}\label{baselineDN2}
The Dirichlet-to-Neumann map on $\Gamma$ satisfies 
    \begin{equation*}
        \|\mathcal{N}\psi\|_{L^2(\Gamma)}\lesssim_A \|\psi\|_{H^1(\Gamma)}.
    \end{equation*}  
\end{lemma}
\begin{proof}
See \cite[Lemma 5.24]{Euler}.
\end{proof}
Next, we recall higher regularity versions of these bounds. We begin with a higher regularity version of \Cref{baselineDN}.
\begin{proposition}[Ellipticity for the Dirichlet-to-Neumann operator]\label{ellipticity} Let $s\geq \frac{1}{2}$ and let $k\geq 1$ be an integer. Then for $\alpha\in [0,1)$ and $\beta\in [0,\frac{1}{2})$, we have the bound
\begin{equation*}
\begin{split}
\|\psi\|_{H^{s+k}(\Gamma)}&\lesssim_A \|\psi\|_{L^2(\Gamma)}+\|\mathcal{N}^k\psi\|_{H^{s}(\Gamma)}+\|\Gamma\|_{H^{s+k+r}}\sup_{j>0}2^{-j(\alpha-1+r)}\|\psi_j^1\|_{C^{\alpha}(\Gamma)}
\\
&+\sup_{j>0}2^{j(s+k-\frac{1}{2}+\beta-\epsilon)}\|\psi_j^2\|_{H^{\frac{1}{2}-\beta}(\Gamma)}.
\end{split}
\end{equation*}
\end{proposition}
\begin{proof}
See \cite[Lemma 5.26]{Euler}.
\end{proof}
To complement the higher order ellipticity estimates for $\mathcal{N}$, we will need the reverse estimates which control powers of $\mathcal{N}$ applied to a function in terms of appropriate Sobolev norms of that function. The following two propositions  give such bounds -- they may be thought of as higher regularity analogues of \Cref{baselineDN2}.
\begin{proposition}[Dirichlet-to-Neumann operator bound I]\label{DNpower1} Let $s\geq \frac{1}{2}$, $r\geq 0$, $\alpha\in [0,1)$ and $\beta\in [0,\frac{1}{2})$. Then
\begin{equation*}
    \|\mathcal{N}\psi\|_{H^{s}(\Gamma)}\lesssim_A \|\psi\|_{H^{s+1}(\Gamma)}+\|\Gamma\|_{H^{s+1+r}}\sup_{j>0}2^{-j(r-1+\alpha)}\|\psi_j^1\|_{C^\alpha(\Gamma)}+\sup_{j>0}2^{j(s+\frac{1}{2}+\beta-\epsilon)}\|\psi_j^2\|_{H^{\frac{1}{2}-\beta}(\Gamma)}
\end{equation*}
for any sequence of partitions $\psi=\psi_j^1+\psi_j^2$.
\end{proposition}
\begin{proof}
See \cite[Proposition 5.29]{Euler}.
\end{proof}
\begin{proposition}[Dirichlet-to-Neumann operator bound II]\label{higherpowers} Let $m\geq 1$ be an integer, let $s\geq \frac{1}{2}$ and let $r\geq 0$, $\alpha\in [0,1)$ and $\beta\in [0,\frac{1}{2})$. Then we have the  bound
\begin{equation*}
\label{Iterated N}
    \|\mathcal{N}^m\psi\|_{H^{s}(\Gamma)}\lesssim_A \|\psi\|_{H^{s+m}(\Gamma)}+\|\Gamma\|_{H^{s+r+m}}\sup_{j>0}2^{-j(r+\alpha-1)}\|\psi_j^1\|_{C^\alpha(\Gamma)}+\sup_{j>0}2^{j(s-\frac{1}{2}+m+\beta-\epsilon)}\|\psi_j^2\|_{H^{\frac{1}{2}-\beta}(\Gamma)}
\end{equation*}
and the closely related bound when $s\geq\frac{3}{2}$,
\begin{equation*}\label{iteratedN2}
\|\mathcal{H}\mathcal{N}^m\psi\|_{H^{s+\frac{1}{2}}(\Omega)}\lesssim_A \|\psi\|_{H^{s+m}(\Gamma)}+\|\Gamma\|_{H^{s+r+m}}\sup_{j>0}2^{-j(r+\alpha-1)}\|\psi_j^1\|_{C^\alpha(\Gamma)}+\sup_{j>0}2^{j(s-\frac{1}{2}+m+\beta-\epsilon)}\|\psi_j^2\|_{H^{\frac{1}{2}-\beta}(\Gamma)}    
\end{equation*}
for any partition $\psi=\psi_j^1+\psi_j^2$.
\end{proposition}
\begin{proof}
See \cite[Proposition 5.30]{Euler}.
\end{proof}
Finally, we recall a few results about traces of normal and tangential derivatives on the boundary. For normal derivatives, we have the following estimates.
\begin{proposition}[Normal derivative trace bound]\label{L1bound} Let $s>0$, $r\geq 0$ and $\alpha,\beta\in [0,1]$. The normal trace operator $\nabla_n:=n_\Gamma\cdot (\nabla)_{|\Gamma}$ satisfies the bound
\begin{equation*}
    \|\nabla_nv\|_{H^{s}(\Gamma)}\lesssim_A\|v\|_{H^{s+\frac{3}{2}}(\Omega)}+\|\Gamma\|_{H^{s+r+1}}\sup_{j>0}2^{-j\left(r-1+\alpha\right)}\|v_j^1\|_{C^\alpha(\Omega)}+\sup_{j>0}2^{j(s+\beta+\frac{1}{2}-\epsilon)}\|v_j^2\|_{H^{1-\beta}(\Omega)}.
\end{equation*}
\end{proposition}
\begin{proof}
See \cite[Proposition 5.28]{Euler}.
\end{proof}
For the tangential derivative operator $\nabla^{\top}$, the  analogous result is as follows.
\begin{proposition}\label{tangradientbound}
Let $s\geq \frac{1}{2}$, $r\geq 0$, $\alpha\in [0,1)$ and $\beta\in [0,\frac{1}{2})$. Then
\begin{equation*}
\|\nabla^{\top}\psi\|_{H^s(\Gamma)}\lesssim_A \|\psi\|_{H^{s+1}(\Gamma)}+\|\Gamma\|_{H^{s+1+r}}\sup_{j>0}2^{-j(r-1+\alpha)}\|\psi_j^1\|_{C^{\alpha}(\Gamma)}+\sup_{j>0}2^{j(s+\frac{1}{2}+\beta-\epsilon)}\|\psi_j^2\|_{H^{\frac{1}{2}-\beta}(\Gamma)}    
\end{equation*}
for any sequence of partitions $\psi=\psi_j^1+\psi_j^2$.
\end{proposition}
\begin{proof}
See \cite[Proposition 5.31]{Euler}.
\end{proof}
Finally, we note a bound for $\mathcal{N}^m\nabla_n$ which will be used in \Cref{HEB}.
\begin{corollary}\label{EEcorollary} Let $\alpha,\beta\in [0,1]$, $s\geq \frac{1}{2}$ and $r\geq 0$. We have
\begin{equation*}
\|\mathcal{N}^m\nabla_nv\|_{H^{s}(\Gamma)}\lesssim_A \|v\|_{H^{s+m+\frac{3}{2}}(\Omega)}+\|\Gamma\|_{H^{s+1+m+r}}\sup_{j>0}2^{-j(r+\alpha-1)}\|v_j^1\|_{C^{\alpha}(\Omega)}+\sup_{j>0} 2^{j(s+\beta+\frac{1}{2}+m-\epsilon)}\|v_j^2\|_{H^{1-\beta}(\Omega)}
\end{equation*}
where $v=v_j^1+v_j^2$ is any sequence of partitions of $v$.
\end{corollary}
\begin{proof}
See \cite[Corollary 5.32]{Euler}.
\end{proof}

\subsubsection{Div-curl estimates}
At various points in the paper we will need good bounds for the following div-curl system.
\begin{proposition}[Div-curl estimate with Neumann type data]\label{Balanced div-curl} Let $v\in H^s(\Omega)$ be a vector field defined on $\Omega$ and let $s>\frac{3}{2}$, $\alpha,\beta\in [0,1]$. Let $v:=v_j^1+v_j^2$ be any partition of $v$. Moreover, let $\mathcal{B}v$ denote either the Neumann trace of $v$, $n_{\Gamma}\cdot \nabla v$ or the boundary value $\nabla^{\top}v\cdot n_\Gamma$. Then if $v$ solves the  div-curl system,
\begin{equation*}
\begin{cases}
&\nabla\cdot v=f,
\\
&\nabla\times v=\omega,
\\
&\mathcal{B}v=g,
\end{cases}
\end{equation*}
then $v$ satisfies the estimate
\begin{equation*}
\begin{split}
\|v\|_{H^{s}(\Omega)}&\lesssim_A\|f\|_{H^{s-1}(\Omega)}+\|\omega\|_{H^{s-1}(\Omega)}+\|g\|_{H^{s-\frac{3}{2}}(\Gamma)}+\|v\|_{L^2(\Omega)}+\|\Gamma\|_{H^{s+r-\frac{1}{2}}}\sup_{j>0}2^{-j(r+\alpha-1)}\|v_j^1\|_{C^{\alpha}(\Omega)}
\\
&\qquad +\sup_{j>0}2^{j(s-1+\beta-\epsilon)}\|v_j^2\|_{H^{1-\beta}(\Omega)}.
\end{split}
\end{equation*}
\end{proposition}
\begin{proof}
See \cite[Proposition 5.27]{Euler}.
\end{proof}
\begin{remark}
 Note that we do not claim that solutions to the above div-curl system always exist. We only claim that solutions to the above system satisfy the bounds in \Cref{Balanced div-curl}.
\end{remark}
\subsubsection{Rotational-irrotational decomposition}\label{RIR} Another variant of the above div-curl decomposition that will be useful in  instances where we  need estimates for vector fields in terms of their normal trace is the so-called rotational-irrotational decomposition. This is described in Appendix A of \cite{MR2388661} but we recall the relevant definitions here for convenience of the reader. We have the following two definitions.
\begin{definition}[Divergence-free projection]\label{divfreecorrec}
Given a vector field $v\in L^2(\Omega)$, we define its divergence-free projection $v^{div}$ by
\begin{equation*}
v^{div}:=v-\nabla\Delta^{-1}(\nabla\cdot v).
\end{equation*}
\end{definition}
\begin{definition}[Rotational-irrotational decomposition]\label{rot-irrot-decomp} Given a vector field $v\in L^2(\Omega)$, we define
\begin{equation*}
v^{ir}:=\nabla\mathcal{H}\mathcal{N}^{-1}(v^{div}\cdot n_\Gamma),\hspace{5mm} v^{rot}:=v^{div}-v^{ir}.
\end{equation*}
\end{definition}
In particular, if $v$ is divergence-free, it follows from \Cref{rot-irrot-decomp} that we have the following further decomposition of $v$ into two divergence-free parts:
\begin{equation*}
v=v^{rot}+v^{ir}.
\end{equation*} From Proposition A.5 in \cite{MR2388661}, we have the following basic $L^2(\Omega)$ bound.
\begin{proposition}\label{basicrotirrot}
For a vector field $v\in L^2(\Omega)$, there holds
\begin{equation*}
\|v^{ir}\|_{L^2(\Omega)}\lesssim_A \|v^{div}\cdot n_\Gamma\|_{H^{-\frac{1}{2}}(\Gamma)}.
\end{equation*}
\end{proposition}

\subsection{Moving surface identities and commutators}\label{Movingsurfid} We now assume that $\Omega_t$ is a one parameter family of domains with boundaries $\Gamma_t\in\Lambda_*$ flowing with a velocity vector field $v$ that is not necessarily divergence-free. The objective is to compile several identities and commutator estimates involving functions on $\Gamma_t$ and the material derivative $D_t:=\partial_t+v\cdot\nabla$. 
\begin{remark}\label{commutatorremark}
We importantly remark that if $w$ is a vector field on $\Omega$ such that $w$ is tangent to $\Gamma$ (i.e.~$w\cdot n_\Gamma=0$) then the identities (i) and (iii)-(vi) and the estimates below in \Cref{materialcom} (if $w$ is divergence-free) hold verbatim with $D_t$  replaced by $\nabla_w$ and $v$ replaced by $w$.
\end{remark}

The following algebraic identities were originally proven in \cite{MR2388661} and then collected in \cite[Section 5.6]{Euler}.
\begin{enumerate}
    \item (Material derivative of the normal). \begin{equation}\label{Moving normal}
    D_tn_{\Gamma_t}=-\left((\nabla v)^*(n_{\Gamma_{t}})\right)^\top.
\end{equation}
\item (Leibniz rule for $\mathcal{N}$).
\begin{equation}\label{DNLeibniz}
    \mathcal{N}(fg)=f\mathcal{N}g+g\mathcal{N}f-2\nabla_n\Delta^{-1}(\nabla \mathcal{H}f\cdot \nabla \mathcal{H}g).
\end{equation}
\item(Commutator with $\nabla$).
\begin{equation*}\label{Dt commutator grad}
    [D_t,\nabla]f=-(\nabla v)^*(\nabla f).
\end{equation*}
\item(Commutator with $\Delta^{-1}$).
\begin{equation*}\label{Dt commutator laplacian-1}
    [D_t,\Delta^{-1}]f=\Delta^{-1}\left(2\nabla v\cdot \nabla^2\Delta^{-1}f+\Delta v\cdot\nabla\Delta^{-1}f\right).
\end{equation*}
\item(Commutator with $\mathcal{H}$).
\begin{equation}\label{H-commutator}
\begin{split}
S_0f:=[D_t,\mathcal{H}]f=\Delta^{-1}(2\nabla v\cdot \nabla^2\mathcal{H}f+\nabla\mathcal{H}f\cdot\Delta v).
\end{split}
\end{equation}
\item(Commutator with $\mathcal{N}$).
\begin{equation}\label{N-commutator}
\begin{split}
S_1f:=[D_t,\mathcal{N}]f=D_t n_{\Gamma_t}\cdot \nabla \mathcal{H}f-n_{\Gamma_t}\cdot((\nabla v)^*(\nabla \mathcal{H}f))+n_{\Gamma_t}\cdot\nabla ([D_t,\mathcal{H}]f).
\end{split}    
\end{equation}
\end{enumerate}
We also recall the general Leibniz type formula,
\begin{equation}\label{Leibniz on moving hypersurfaces general}
\frac{d}{dt}\int_{\Gamma_{t}}fdS=\int_{\Gamma_{t}}D_{t}f+f(\mathcal{D}\cdot v^{\top}-\kappa v^{\perp})\,dS,
\end{equation} 
where $\mathcal{D}$ is the covariant derivative and the integration by parts formula
\begin{equation}\label{surfintbyparts}
\int_\Gamma \nabla_B gfdS=-\int_\Gamma \nabla_B fgdS+\int_\Gamma fgn_\Gamma\cdot \nabla B\cdot n_\Gamma dS,
\end{equation}
where $\Gamma=\partial\Omega$ and $B$ is a divergence-free function on $\Omega$ satisfying $B\cdot n_\Gamma=0$.
\subsubsection{Balanced commutator estimates}
The above identities can be used to establish refined estimates for commutators involving $D_t$ and  $\mathcal{N}$. 
Indeed, for divergence-free velocities $v$, it is straightforward to verify that $S_0\psi$ can be re-written in the form
\begin{equation}\label{Sid}
 S_0\psi=\Delta^{-1}\nabla\cdot\mathcal{B}(\nabla v,\nabla \mathcal{H}\psi)   ,
\end{equation}
where $\mathcal{B}$ is an $\mathbb{R}^d$-valued bilinear form.  By \eqref{N-commutator}, we can write the commutator $[D_t,\mathcal{N}]$ as
\begin{equation*}
S_1\psi:=[D_t,\mathcal{N}]\psi=\nabla_n S_0\psi-\nabla\mathcal{H}\psi\cdot (\nabla_n v)-\nabla^{\top}\psi\cdot\nabla v\cdot n_{\Gamma_t}    .
\end{equation*}
Higher order commutators $S_k$ can then be calculated using the identity
\begin{equation}\label{commutatorexpansion}
S_k\psi:=[D_t,\mathcal{N}^k]\psi=\sum_{l+m=k-1}\mathcal{N}^l[D_t,\mathcal{N}]\mathcal{N}^{m}\psi,  
\end{equation}
where $k\in\mathbb{N}$ and $l,m$ are non-negative integers. Exploiting these formulas, precise estimates for $S_k$ were proven in  \cite[Proposition 5.33]{Euler}  when $v$ is divergence-free and $s\geq \frac{1}{2}$. We will need the following variant for our analysis;  note that we have dropped the $t$ subscripts in the statements below.
\begin{proposition}\label{materialcom}
 Let $s_0>\frac{d}{2}+1$ and put $M_{s_0}:=\|v\|_{H^{s_0}(\Omega)}+\|\Gamma\|_{H^{s_0}}$. Suppose that the flow velocity $v$ is divergence-free and let $s\geq \frac{1}{2}$, $k\geq 1$. Then for any sequence of partitions $\psi=\psi_j^1+\psi_j^2$, $r\geq 0$ and $\alpha\in [0,1]$, there holds
 \begin{equation}\label{RHScom}
 \begin{split}
 \|S_k\psi\|_{H^s(\Gamma)}&\lesssim_{M_{s_0}}\|\psi\|_{H^{s+k}(\Gamma)}+(\|v\|_{H^{s+\frac{1}{2}+k+r}(\Omega)}+\|\Gamma\|_{H^{s+k+\frac{1}{2}+r}})\sup_{j>0}2^{-j(\alpha-1+r)}\|\psi_j^1\|_{C^{\alpha}(\Gamma)}
 \\
 +&\sup_{j>0} 2^{j(s+k-\epsilon)}\|\psi_j^2\|_{H^{\epsilon}(\Gamma)}.
 \end{split}
 \end{equation}
\end{proposition}
\begin{proof}
From \eqref{commutatorexpansion}, we need to prove the estimate above with the left-hand side replaced by $\mathcal{N}^l[D_t,\mathcal{N}]\mathcal{N}^m\psi$ where $l+m=k-1$. We will focus  on the term $\mathcal{N}^l(\nabla_n S_0\mathcal{N}^m\psi)$, as it is the most difficult. Let us define $G:=\mathcal{B}(\nabla v,\nabla \mathcal{H}\mathcal{N}^m\psi)$. We begin by applying \Cref{EEcorollary} and then \Cref{direst} to obtain, using the identity \eqref{Sid},
\begin{equation*}
\begin{split}
\|\mathcal{N}^l(\nabla_n S_0\mathcal{N}^m\psi)\|_{H^s(\Gamma)}&\lesssim_{M_{s_0}} \|G\|_{H^{s+l+\frac{1}{2}}(\Omega)}+\|\Gamma\|_{H^{s+\frac{1}{2}+k+r}}\sup_{j>0}2^{-j(m+\frac{1}{2}+r)}\|\Delta^{-1}\nabla\cdot G_j^1\|_{W^{1,\infty}(\Omega)}
\\
&\quad +\sup_{j>0} 2^{j(s+l+\frac{1}{2}-\epsilon)}\|\Delta^{-1}\nabla\cdot G_j^2\|_{H^1(\Omega)}  ,
\end{split}
\end{equation*}
where $G=G_j^1+G_j^2$ is a partition of $G$ defined by taking $G_j^1=\mathcal{B}(\nabla v,\nabla P_{<j}\mathcal{H}\mathcal{N}^m_{<j}\psi_j^1)$, where $\mathcal{N}_{<j}:=\nabla_n P_{<j}\mathcal{H}$. Using the $C^{1,\epsilon}$ estimate for $\Delta^{-1}$ and the maximum principle for $\mathcal{H}$, it is straightforward to estimate
\begin{equation*}
2^{-j(m+\frac{1}{2}+r)}\|\Delta^{-1}\nabla\cdot G_j^1\|_{W^{1,\infty}(\Omega)}\lesssim_{M_{s_0}}2^{-j(\frac{1}{2}+r-\epsilon)}\|\nabla P_{<j}\mathcal{H}\psi_j^1\|_{L^{\infty}(\Gamma)}\lesssim_{M_{s_0}} 2^{-j(\alpha-1+r)}\|\psi_j^1\|_{C^{\alpha}(\Gamma)},
\end{equation*}
where we used the crude bound $2^{-j(\frac{1}{2}-\epsilon)}\lesssim 1$, the bound $\|P_{<j}\|_{C^{\alpha}\to W^{1,\infty}}\lesssim_M 2^{j(1-\alpha)}$ as well as \Cref{Hboundlow}. Moreover, using the $H^{-1}\to H^1_0$ estimate for $\Delta^{-1}$, we can control the other term by
\begin{equation*}\label{twotechnicalterms}
\begin{split}
\|\Delta^{-1}\nabla\cdot G_j^2\|_{H^1(\Omega)}&\lesssim_{M_{s_0}} \|\nabla P_{<j}\mathcal{H}\mathcal{N}^m_{<j}\psi_j^2\|_{L^2(\Omega)}+\|\nabla P_{<j}\mathcal{H}\mathcal{N}_{<j}^m\psi-\nabla \mathcal{H}\mathcal{N}^m\psi\|_{L^2(\Omega)}=:I_1+I_2.
\end{split}
\end{equation*}
Using the $H^{\frac{1}{2}}\to H^1$ bound for the harmonic extension operator $\mathcal{H}$, Sobolev product estimates,  the bound $\|n_\Gamma\|_{C^{\frac{1}{2}}(\Gamma)}\lesssim_{M_{s_0}}\|\Gamma\|_{C^{1,\frac{1}{2}}}\lesssim_{M_{s_0}}\|\Gamma\|_{H^{s_0}}$ and the trace theorem, we have
\begin{equation*}
\begin{split}
\|\mathcal{H}\mathcal{N}^m_{<j}\psi_j^2\|_{H^1(\Omega)}\lesssim_{M_{s_0}} \|\mathcal{N}^m_{<j}\psi_j^2\|_{H^{\frac{1}{2}}(\Gamma)}&\lesssim_{M_{s_0}} \|n_\Gamma\|_{C^{\frac{1}{2}}(\Gamma)}\|\nabla P_{<j}\mathcal{H}\mathcal{N}^{m-1}_{<j}\psi_j^2\|_{H^1(\Omega)}
\\
&\lesssim_{M_{s_0}} 2^j\|P_{<j}\mathcal{H}\mathcal{N}^{m-1}_{<j}\psi_j^2\|_{H^1(\Omega)}.
\end{split}
\end{equation*}
Iterating this and using the $H^{\epsilon}\to H^{\frac{1}{2}+\epsilon}$ bound  \eqref{harmonicbase}, we arrive at 
\begin{equation*}
2^{j(s+l+\frac{1}{2}-\epsilon)}I_1\lesssim_{M_{s_0}} 2^{j(s+k-\epsilon)}\|\mathcal{H}\psi_j^2\|_{H^{\frac{1}{2}+\epsilon}(\Omega)}\lesssim_{M_{s_0}}2^{j(s+k-\epsilon)}\|\psi_j^2\|_{H^{\epsilon}(\Gamma)}.
\end{equation*}
On the other hand, we can estimate
\begin{equation*}
I_2\lesssim_{M_{s_0}} 2^{-j(s+l+\frac{1}{2})}\|\mathcal{H}\mathcal{N}^m\psi\|_{H^{s+l+\frac{3}{2}}(\Omega)}+\|\nabla\mathcal{H}(\mathcal{N}^m_{<j}\psi-\mathcal{N}^m\psi)\|_{L^2(\Omega)}=:I_2^1+I_2^2.
\end{equation*}
Using \Cref{Hbounds} and \Cref{higherpowers}, we can control $2^{j(s+l+\frac{1}{2}-\epsilon)}I_2^1$ by the right-hand side of \eqref{RHScom}. On the other hand, expanding $I_2^2$ and using the $H^{\frac{1}{2}}\to H^{1}$ bound for $\mathcal{H}$, \Cref{boundaryest}, the fact that $\|n_\Gamma\|_{C^{\frac{1}{2}}(\Gamma)}\lesssim_{M_{s_0}}1$  and the trace theorem, we can estimate
\begin{equation*}
\begin{split}
I_2^2&\lesssim_{M_{s_0}} \sum_{p=0}^{m-1}\|\nabla\mathcal{H}(\mathcal{N}_{<j}^p(\mathcal{N}-\mathcal{N}_{<j})\mathcal{N}^{m-1-p}\psi\|_{L^2(\Omega)}
\\
&\lesssim_{M_{s_0}}\sum_{p=0}^{m-1}2^{jp}\|\nabla P_{\geq j}\mathcal{H}\mathcal{N}^{m-1-p}\psi\|_{H^1(\Omega)}
\\
&\lesssim_{M_{s_0}} 2^{-j(s+l+\frac{1}{2})}\sum_{p=0}^{m-1}\|\mathcal{H}\mathcal{N}^{m-1-p}\psi\|_{H^{s+l+p+\frac{5}{2}}(\Omega)}.
\end{split}
\end{equation*}
Using \Cref{Hbounds}, we then control $2^{j(s+l+\frac{1}{2}-\epsilon)}I_2^2$ by the right-hand side of \eqref{RHScom}.
\medskip

It remains to estimate $\|G\|_{H^{s+l+\frac{1}{2}}(\Omega)}$. By applying \Cref{productestref} with the partition 
\begin{equation*}
\mathcal{H}\mathcal{N}^m\psi=(\mathcal{H}\mathcal{N}_{<j}^m\psi_j^1)+(\mathcal{H}\mathcal{N}_{<j}^m\psi_j^2+\mathcal{H}(\mathcal{N}^m-\mathcal{N}_{<j}^m)\psi)
\end{equation*}
we can estimate $G$ by the right-hand side of \eqref{RHScom} by using a similar analysis to the above, but now making important use of the term involving $\|v\|_{H^{s+\frac{1}{2}+k+r}(\Omega)}$ in \eqref{RHScom}. 
\end{proof}
\subsection{Regularization operators which extend the domain}\label{SSRO}
In order to construct solutions to the free boundary MHD equations, we will need a well-chosen family of regularization operators. Beyond the usual regularization properties in \Cref{c reg bounds not div free}, we  will require the following features.
\begin{enumerate}
\item\label{EP} (Extension property). There is a $\delta_0>0$ so that whenever $\Omega_j$ is a domain containing $\Omega$ with boundary $\Gamma_j\in\Lambda_*$ satisfying $\|\dist(x,\Omega)\|_{L^{\infty}(\Omega_j)}<\delta_0 2^{-j}$ we have that the regularization  $\Psi_{\leq j}v$ at the dyadic scale $2^j$  is defined on $\Omega_j$.
\item\label{RIDF} (Regularization is divergence-free). Given $\Omega_j$ as above and any divergence-free function $v$  on $\Omega$, the regularization $\Psi_{\leq j}v$ satisfies $\nabla\cdot \Psi_{\leq j}v=0$ on $\Omega_j$.
\end{enumerate}
Property \eqref{EP} will be needed  for comparing functions defined on different -- but sufficiently close -- domains. Property \eqref{RIDF} will allow us to  maintain the divergence-free condition when regularizing the velocity and the magnetic field.
\medskip

We begin by constructing  regularization operators which have the extension property \eqref{EP} but  not necessarily the divergence-free property \eqref{RIDF}. Since the explicit form of these operators will be needed in the existence scheme for estimating commutators with $\nabla_B$, we include a sketch of the proof.
\begin{proposition}\label{c reg bounds not div free}
    Fix $\alpha_0$ and let $v$, $\Omega$ and $\Omega_j$ be as above. There exists a regularization operator $\Phi_{\leq j}$ which is bounded from $H^s(\Omega)\to H^s(\Omega_{j})$ for every $s\geq 0$ with the following properties.
    \begin{enumerate}
        \item (Regularization bounds).
        \begin{equation*}
        \|\Phi_{\leq j}v\|_{H^{s+\alpha}(\Omega_{j})} \lesssim_{A} 2^{j\alpha }\|v\|_{H^s(\Omega)},\hspace{16mm}0\leq\alpha.
        \end{equation*}
        \item (Difference bounds).
        \begin{equation*}
        \|(\Phi_{\leq j+1}-\Phi_{\leq j})v\|_{H^{s-\alpha}(\Omega_{j+1})}\lesssim_{A} 2^{-j\alpha}\|v\|_{H^s(\Omega)},\hspace{7mm} 0\leq\alpha\leq \min\{s,\alpha_0\}.
        \end{equation*}
        \item (Error bounds).
        \begin{equation*}
        \|(I-\Phi_{\leq j})v\|_{H^{s-\alpha}(\Omega)}\lesssim_{A} 2^{-j\alpha}\|v\|_{H^s(\Omega)},\hspace{15mm} 0\leq\alpha\leq \min\{s,\alpha_0\}.
        \end{equation*}
    \end{enumerate}
\end{proposition}
\begin{proof}
Our objective is to define a suitable kernel $K^{j}$ so that
\begin{equation*}
\Phi_{\leq j}v(x)=\int_\Omega K^{j}(x,y)v(y)\, dy.
\end{equation*}
The kernel $K^j(x,y)$ will take the form
\begin{equation*}\label{Form of kern}
    K^j(x,y)=\sum_{k=0}^n K_k^j(x,y)\chi_k(x),
\end{equation*}
where $(\chi_k)_{k=0}^n$ is a partition of unity of a neighborhood of $\Omega$. More specifically, we choose $(\chi_k)_{k=0}^n$ to be subordinate to an open cover $\{U_k\}_{k=0}^n$ so that there are unit vectors $(e_k)_{k=1}^n$ with $e_k$  outward oriented and uniformly transversal to $\Gamma\cap U_k$. The remaining set $U_0$  is then selected to cover the portion  of $\Omega$ away from the boundary. It is easy to see that such a smooth partition of unity exists, with bounds depending solely on the properties of $\Lambda_*$. 
\medskip

Define $e_0:=0$ and take $e_k$ with $k\in \{1,\dots, n\}$ as in the previous paragraph. We select a smooth bump function $\phi_k$ so that
\begin{enumerate}
    \item The support of $\phi_k$ satisfies supp$\phi_k\subseteq B(e_k,\delta_1)$, $\delta_1\ll 1.$
    \item The average of $\phi_k$ is $1$, i.e., $\int_{\mathbb{R}^d}\phi_k(z)\,dz=1$.
    \item $\phi_k$ has zero moments up to some suitably large order $N$, i.e., $\int_{\mathbb{R}^d} z^\alpha \phi_k(z)dz=0$, $1\leq |\alpha|\leq N.$
\end{enumerate}
For each $j>0$, we  consider the regularizing kernel
\begin{equation*}
  K_{0,k}^j(z):=2^{jd}\phi_k(2^j z)
\end{equation*}
and define $K_k^j(x,y):=K_{0,k}^j(x-y)$ for  $y\in \Omega$. Note that for fixed $x\in U_k$, $K_k^j(x,y)$ is non-zero  only when $2^j(x-y)\in B(e_k,\delta_1)$, i.e.,  when $y$ is within distance $2^{-j}\delta_1$ of $x-2^{-j} e_k$. Hence, we may view our kernel $K^j$ not only for $x\in \Omega$ but also for $x$ in an $\mathcal{O}(2^{-j})$ enlargement of $\Omega$. With this observation in mind, one may check that  the kernels $K^j$ satisfy
\begin{enumerate}
    \item $K^j:\tilde{\Omega}_{j}\times \Omega\to \mathbb{R}$, where $\tilde{\Omega}_{j}:=\{x\in \mathbb{R}^d : d(x,\Omega)\leq c 2^{-j}\}$ and $c$ is some small universal constant.
    \item $|\partial_x^\alpha\partial_{y}^\beta K^j(x,y)|\lesssim 2^{j(d+|\alpha|+|\beta|)}$  for multi-indices $\alpha,\beta$.
    \item $\int_\Omega K^j(x,y)\,dy=1$.
    \item $\int_\Omega K^j(x,y)(x-y)^\alpha \,dy=0$ for $1\leq |\alpha|\leq N.$
\end{enumerate}
From the definition of $K^{j}$, we see that $\Phi_{\leq j}v$ is defined on a neighborhood of $\Omega_{j}$ if $\delta_0$ from the extension property \eqref{EP} is small enough. It is also directly verified  that $\Phi_{\leq j}$ satisfies the three properties in \Cref{c reg bounds  not div free} when $s$ and $\alpha$ are integers; the latter two bounds use the moment conditions with $N=N(\alpha_0)$. By interpolation, we obtain the bounds in \Cref{c reg bounds  not div free} for non-integer regularities as well. 
\end{proof}
Using the family of operators $\Phi_{\leq j}$, we may construct our desired regularization operators $\Psi_{\leq j}$.
\begin{proposition}\label{c reg bounds}
  For each $j$, there exists an operator $\Psi_{\leq j}$ which is bounded from $H^s_{div}(\Omega)\to H^s_{div}(\Omega_{j})$ for every $s\geq 0$ which satisfies the extension and regularization  properties in \Cref{c reg bounds not div free}.
\end{proposition}
\begin{proof}
The proof proceeds by correcting the operator in \Cref{c reg bounds not div free} by a gradient potential. See \cite[Proposition 6.2]{Euler} for details.
\end{proof}
We  also note the pointwise analogues of the above estimates. 
\begin{proposition}\label{pointwisereg} Let $0\leq\alpha<2$. Given the assumptions of \Cref{c reg bounds}, the regularization operator $\Psi_{\leq j}$ satisfies the pointwise bounds
\begin{equation*}
\|\Psi_{\leq j}v\|_{C^{\alpha}(\Omega_j)}\lesssim_A 2^{j\beta}\|v\|_{C^{\alpha-\beta}(\Omega)}
\end{equation*}
for $0\leq\beta\leq\alpha$ and
\begin{equation*}
\|(I-\Psi_{\leq j})v\|_{C^{\alpha}(\Omega)}+\|(\Psi_{\leq j+1}-\Psi_{\leq j})v\|_{C^{\alpha}(\Omega_{j+1})}\lesssim_A 2^{-j\beta}\|v\|_{C^{\alpha+\beta}(\Omega)} 
\end{equation*}
for $\beta\geq 0$. Similar bounds hold for the regularization operator $\Phi_{\leq j}$.
\end{proposition}
\begin{proof}
See \cite[Proposition 6.3]{Euler}.
\end{proof}
As in \cite{Euler}, we will use the above  operators to  form a crude paradifferential calculus. For any integer $l>0$, we let $\Phi_l:=\Phi_{\leq l+1}-\Phi_{\leq l}$ and $\Psi_{l}:=\Psi_{\leq l+1}-\Psi_{\leq l}$. We also define $\Phi_0:=\Phi_{\leq 0}$ and $\Psi_0:=\Psi_{\leq 0}$. For a vector or scalar valued function $f$ defined on $\Omega$, we write $f^l:=\Phi_l f$ and $f^{\leq l}:=\Phi_{\leq l}f$. If, in addition, $f$ is a divergence-free vector field, we instead use $f^l$ and $f^{\leq l}$ to respectively represent  $\Psi_lf$ and $\Psi_{\leq l}f$. This  ensures that the divergence-free structure of $f$ is preserved. By an abuse of notation, we define
\begin{equation*}
f^lg^{\leq l}:=\sum_{l\geq 0}\sum_{0\leq m\leq l}f^lg^m-\frac{1}{2}\sum_{l\geq 0}f^lg^l.
\end{equation*}
This definition ensures  that we have the  crude bilinear paraproduct  decomposition
\begin{equation}\label{bilpara}
fg=f^lg^{\leq l}+f^{\leq l}g^l,
\end{equation}
where $f^lg^{\leq l}$ selects the portion of $fg$ where $f$ is at higher or comparable frequency compared to $g$. Trilinear (or even higher order) expressions of the form $f^lg^{\leq l}h^{\leq l}$ may be similarly defined so that we have $fgh=f^lg^{\leq l}h^{\leq l}+f^{\leq l}g^{l}h^{\leq l}+f^{\leq l}g^{\leq l}h^{l}$. Such decompositions will be used frequently in \Cref{HEB}.
\subsubsection{Frequency envelopes}\label{Freq envelopes} We now define $\mathbf{H}^s$ frequency envelopes. The verification of the frequency envelope bounds presented below will be noticeably more involved than in our previous work \cite[Section 6.1]{Euler}, as we must preserve the boundary conditions in the state space $\mathbf{H}^s$.
\medskip

Let $\Gamma\in\Lambda_*$ and let $s>\frac{d}{2}+1$. Suppose that $f\in H^s(\Omega)$ and $\Gamma\in H^s$ is parameterized in collar coordinates by $x\mapsto x+\eta_\Gamma(x)\nu(x)$. Using the extension operator from \Cref{continuosext}, we have the following Littlewood-Paley decomposition for a function $f$ defined on $\Omega$:
\begin{equation}\label{LPonsurf}
f=\sum_{j\geq 0}P_jf    ,
\end{equation}
where  $P_jf$ is interpreted to mean $P_jE_{\Omega}f$ with $E_{\Omega}$  as in \Cref{continuosext}. 
\begin{remark}
Note that the use of $E_{\Omega}$ rather than Stein's extension operator $\mathcal{E}_{\Omega}$ in \eqref{LPonsurf} conflicts with our convention in \Cref{LWP section}.  The definition in \eqref{LPonsurf}  will only be used in this subsection; its purpose is to ensure that  frequency envelopes associated to different initial data can be suitably compared.
\end{remark}

We have a corresponding Littlewood-Paley type decomposition for functions on $\Gamma_*$. Indeed, we can write for $j>0$, $P_j:=\varphi (2^{-2j}\Delta_{\Gamma_*})-\varphi (2^{-2(j+1)}\Delta_{\Gamma_*})$ and $P_0:=\varphi(\Delta_{\Gamma_*})$ where $\varphi:\mathbb{R}\to\mathbb{R}$ with $\varphi =1$ on the unit ball and with support in $B_2(0)$. Correspondingly, we can define the multipliers $P_{\leq j}$ and $P_{>j}$.
\medskip

 From \Cref{continuosext}, Sobolev embeddings and  almost orthogonality, we have for any $\frac{d}{2}+1<s_0<s$,
\begin{equation*}\label{LWP}
\begin{split}
\|(v,B,\Gamma)\|_{\textbf{H}^s}^2\approx_{M_{s_0}} &\sum_{j\geq 0}2^{2js}\left(\|P_jv\|_{L^2(\mathbb{R}^d)}^2+\|P_jB\|_{L^2(\mathbb{R}^d)}^2+\|P_j\eta_{\Gamma}\|_{L^2(\Gamma_*)}^2\right)
\\
+&\sum_{j\geq 0}2^{2j(s-\frac{1}{2})}\left(\|P_j(\nabla_Bv)\|_{L^2(\mathbb{R}^d)}^2+\|P_j(\nabla_BB)\|_{L^2(\mathbb{R}^d)}^2\right),
\end{split}
\end{equation*}
where $M_{s_0}:=\|(v,B,\Gamma)\|_{\mathbf{H}^{s_0}}.$ The above equivalence will allow us to define $\mathbf{H}^s$ frequency envelopes for  states $(v,B,\Gamma)\in\mathbf{H}^s$ with the $l^2$ decay required to establish our continuous dependence result as well as the continuity of solutions with values in $\mathbf{H}^s$. 
\begin{definition}[Frequency envelopes] Let $s> \frac{d}{2}+1$, $\Gamma\in\Lambda_*$ and $(v,B,\Gamma)\in\mathbf{H}^s$. An \emph{$\mathbf{H}^s$ frequency envelope} for the triple $(v,B,\Gamma)$ is a positive sequence $c_j$ with $\|c_j\|_{l^2}\lesssim_{M_{s_0}} 1$  such that for each $j\geq 0$, 
\begin{equation*}
N_j\lesssim_{M_{s_0}} c_j\|(v,B,\Gamma)\|_{\mathbf{H}^s},
\end{equation*}
where 
\begin{equation*}
N_j:=\|P_jv\|_{H^s(\mathbb{R}^d)}+\|P_jB\|_{H^s(\mathbb{R}^d)}+\|P_j\eta_{\Gamma}\|_{H^s(\Gamma_*)}+\|P_j(\nabla_Bv)\|_{H^{s-\frac{1}{2}}(\mathbb{R}^d)}+\|P_j(\nabla_BB)\|_{H^{s-\frac{1}{2}}(\mathbb{R}^d)}.
\end{equation*}
We say that the sequence $(c_j)_j$ is admissible if $c_0\approx_{M_{s_0}} 1$ and it is slowly varying,
\begin{equation*}
c_j\leq 2^{\delta |j-k|}c_k,\hspace{10mm} j,k\geq 0,\hspace{10mm}0<\delta\ll 1.    
\end{equation*}
We can always define an admissible frequency envelope by the formula
\begin{equation}\label{admissable}
c_j=2^{-\delta j}+(1+\|(v,B,\Gamma)\|_{\mathbf{H}^{s}})^{-1}\max_k 2^{-\delta |j-k|}N_k.
\end{equation}
Unless otherwise stated, we will take this as our formula for $c_j$. 
The following proposition will be useful in our construction of rough solutions  as well as for proving continuity of the data-to-solution map.
\end{definition}
\begin{proposition}\label{envbounds}
Let $\Gamma\in\Lambda_*$ and let $s>\frac{d}{2}+1$. Suppose that $(v,B,\Gamma)\in\mathbf{H}^s$ and let $(c_j)_j$ be its associated admissible frequency envelope. Then there exists a family of regularized domains $\Omega_j$ with boundaries $\Gamma_j\in\Lambda_*$ and $\Gamma_j\in H^s$ along with associated divergence-free regularizations $v_j$ and $B_j$ defined on  $\Omega_j$  such that $B_j$ is tangent to $\Gamma_j$ and the following properties hold:
\begin{enumerate}
    \item\label{iGPA} (Good pointwise approximation).
    \begin{equation} \label{vbg-point}
 (v_j,B_j,\Gamma_j)\to (v,B,\Gamma)\hspace{5mm}in\hspace{2mm} C^{1}\times C^{1,\frac{1}{2}}\hspace{5mm}as\hspace{2mm}j\to\infty.       
    \end{equation}
\item\label{2UB} (Uniform bound).
\begin{equation}\label{vbg-s}
\|(v_j,B_j,\Gamma_j)\|_{\mathbf{H}^s}\lesssim_{M_{s_0}} \|(v,B,\Gamma)\|_{\mathbf{H}^s}.    
\end{equation}
\item\label{3HR} (Higher regularity).
\begin{equation}\label{vbg-reg}
 \|(v_j,B_j,\Gamma_j)\|_{\mathbf{H}^{s+\alpha}}\lesssim_{M_{s_0}} 2^{j\alpha}c_j\|(v,B,\Gamma)\|_{\mathbf{H}^s},\hspace{5mm}\alpha>1.  
\end{equation}
\item\label{4LFDB} (Low-frequency difference bounds I).
\begin{equation}\label{vbg-diff}
\|(v_j,B_j)-(v_{j+1},B_{j+1})\|_{L^2\times L^2(\Omega_j\cap\Omega_{j+1})}\lesssim_{M_{s_0}} 2^{-js}c_j\|(v,B,\Gamma)\|_{\mathbf{H}^s},
\end{equation}
\begin{equation*}\label{g-diff}
\|\eta_{\Gamma_j} -\eta_{\Gamma_{j+1}}\|_{L^2(\Gamma_*)}\lesssim_{M_{s_0}} 2^{-js}c_j\|(v,B,\Gamma)\|_{\mathbf{H}^s}.
\end{equation*}
\item\label{5LFDB} (Low-frequency difference bounds II).
\begin{equation}\label{vbg-diff2}
\|(\nabla_{B_j} v_j,\nabla_{B_j} B_j)-(\nabla_{B_{j+1}} v_{j+1},\nabla_{B_{j+1}} B_{j+1})\|_{H^1\times H^1(\Omega_j\cap\Omega_{j+1})}\lesssim_{M_{s_0}} 2^{-j(s-\frac32)}c_j\|(v,B,\Gamma)\|_{\mathbf{H}^s}.
\end{equation}

\end{enumerate}
\end{proposition}
\begin{remark}
The last bound \eqref{vbg-diff2} is carried out in $H^1$ instead of $L^2$ for technical convenience. This choice will avoid the need to work in negative regularity Sobolev spaces at several points in the proof.
\end{remark}
\begin{proof}
Below we will use $M_s:=\|(v,B,\Gamma)\|_{\mathbf{H}^s}$ as a shorthand. We define $\Gamma_j$ in collar coordinates through the regularization
\begin{equation*}
\eta_{j}:=P_{<j}\eta:=\varphi (2^{-2j}\Delta_{\Gamma_*})\eta.
\end{equation*}
Here, $\eta$ is the parameterization for $\Gamma$ on $\Gamma_*$ and $\varphi$ is a radial, unit-scale bump function adapted to the unit ball. We also define
\begin{equation*}
v_j:=(\Phi_{\leq j}v)^{div}=\Phi_{\leq j}v-\nabla\Delta^{-1}(\nabla\cdot \Phi_{\leq j} v)
\end{equation*}
where $\Delta^{-1}$ denotes the Dirichlet Laplacian corresponding to $\Gamma_j$. A large part of the proof repeats the arguments of 
\cite[Proposition 6.6]{Euler}, where the $v$ and $\Gamma$ regularizations are defined in a similar manner to the above. This takes care of all of the $\Gamma_j$ bounds in the theorem as well as all of the linear $v_j$ bounds, but excludes any
estimates involving $\nabla_{B_j} v_j$.
\medskip

It remains to consider the part of the theorem involving the $B$ regularizations $B_j$. We define our initial guess as
\[
\tilde B_{j} := (\Phi_{\leq j} B)^{div}=\Phi_{\leq j} B-\nabla\Delta^{-1}(\nabla\cdot \Phi_{\leq j}B). 
\]
This satisfies the same linear bounds as $v_j$, but fails to satisfy the requirement to be tangent to the boundary $\Gamma_j$.  We correct it to be purely rotational by setting 
\[
B_{j} = \tilde B_{j} - B_j^{c}, \qquad B_j^c := \nabla\mathcal H_j\mathcal{N}_j^{-1}(\tilde{B}_j\cdot n_j)=\tilde{B}_j^{ir},
\]
where $\mathcal H_j$ denotes the harmonic extension for 
the regularized domain $\Omega_j$ and $\mathcal{N}_j$ is the associated Dirichlet-to-Neumann operator.
This correction is divergence-free and satisfies the desired tangency condition, so we proceed to show that it satisfies the  bounds \eqref{vbg-point}-\eqref{vbg-diff2}. We will do this in two steps: (i) we prove that the bounds \eqref{vbg-point}-\eqref{vbg-diff2} hold for $\tilde B_j$, and (ii) we estimate 
perturbatively the contribution of the correction $B_j^c$.
\medskip

\textbf{(i). The estimates for $\tilde B_j$ and $v_j$.} 
Following \cite[Proposition 6.6]{Euler}, we already have the bounds in \eqref{vbg-point} 
and \eqref{vbg-diff}, as well as the linear 
$B_j$ component of \eqref{vbg-s} and \eqref{vbg-reg}. It remains to prove the $\nabla_{B_j}$
bounds in \eqref{vbg-s}, \eqref{vbg-reg} and \eqref{vbg-diff2}, where it is enough to consider the expressions $\nabla_{\tilde B_j} v_j$ and $\nabla_{\tilde{B}_j}\tilde{B}_j$. We will carry out the analysis for $\nabla_{\tilde{B}_j}v_j$ as the other term is handled identically. In all three bounds, it suffices to compare $\nabla_{\tilde{B}_j}v_j$ with $\Phi_{\leq j} ( \nabla_{B} v)$, as the latter term is easily estimated in all cases.
So, our remaining task is to prove the appropriate bounds for the difference
\[
D_j := \nabla_{\tilde B_j} v_j - \Phi_{\leq j} ( \nabla_{B} v).
\]
We expand the above quantity to obtain a commutator structure:
\[
D_j = [ \nabla_{\tilde B_j}, \Phi_{\leq j}] v + \Phi_{\leq j} ( \nabla_{ \tilde B_j -B} v)-\nabla_{\tilde{B}_j}\nabla\Delta^{-1}(\nabla\cdot\Phi_{\leq j} v). 
\]
The above expression is well-defined on $\Omega_j$ because of the mapping properties of $\Phi_{\leq j}$ and the fact that (by Sobolev embeddings) $\eta_j=\eta+\mathcal{O}_{L^{\infty}(\Gamma_*)}(2^{-j(\frac{3}{2}+\delta)})$. The second term is easy to estimate in all Sobolev norms. Precisely, we have the $L^2$ bound
\[
\| \nabla_{ \tilde B_j -B} v\|_{L^2(\Omega_j\cap\Omega)}
\lesssim \| \tilde B_j -B\|_{L^2(\Omega_j\cap\Omega)} \|\nabla v\|_{L^\infty(\Omega)} \lesssim_{M_{s_0}} 2^{-sj} c_j M_s,
\]
after which applying $\Phi_{\leq  j}$ yields bounds at any regularity,
\[
\| \Phi_{\leq  j}(\nabla_{ \tilde B_j - B} v)\|_{H^\sigma(\Omega_j)}
 \lesssim_{M_{s_0}} 2^{(\sigma-s)j} c_j M_s,\hspace{5mm}\sigma\geq 0.
\]
This suffices  for  \eqref{vbg-s}, \eqref{vbg-reg} and \eqref{vbg-diff2}. Next, we estimate the third term which involves the divergence-free correction. First, we observe the following bound:
\begin{equation*}
\|\Delta^{-1}(\nabla\cdot\Phi_{\leq j} v)\|_{H^{s+\frac{3}{2}+\alpha}(\Omega_j)}\lesssim_{M_{s_0}} \|\nabla\cdot\Phi_{\leq j} v\|_{H^{s-\frac{1}{2}+\alpha}(\Omega_j)}+\|\Gamma_j\|_{H^{s+1+\alpha}}\|\nabla\cdot\Phi_{\leq j} v\|_{H^{s_0-2}(\Omega_j)},\hspace{5mm}\alpha\geq 1,
\end{equation*}
which follows from \Cref{direst}, Sobolev embeddings and the assumption that $s_0>\frac{d}{2}+1$. Since $v$ is divergence-free, the operator $\nabla\cdot\Phi_{\leq j}$ acts as a commutator, which is essentially a mollifier that is localized at frequency $2^j$. This easily yields
\begin{equation*}
\|\nabla\cdot\Phi_{\leq j} v\|_{H^{s-\frac{1}{2}+\alpha}(\Omega_j)}\lesssim_{M_{s_0}} 2^{j(\alpha-\frac{1}{2})}M_s\lesssim_{M_{s_0}} 2^{j\alpha}c_jM_s
\end{equation*}
and
\begin{equation*}
\begin{split}
\|\Gamma_j\|_{H^{s+1+\alpha}}\|\nabla\cdot\Phi_{\leq j} v\|_{H^{s_0-2}(\Omega_j)}&\lesssim_{M_{s_0}} 2^{j(1+\alpha)}c_jM_s\|\nabla\cdot\Phi_{\leq j} v\|_{H^{s_0-2}(\Omega_j)}
\\
&\lesssim_{M_{s_0}} 2^{j\alpha}c_jM_s.
\end{split}
\end{equation*}
From this and simple product estimates, it is straightforward to estimate
\begin{equation}\label{gradBdivcorrection}
\begin{split}
\|\nabla_{\tilde{B}_j}\nabla\Delta^{-1}(\nabla\cdot\Phi_{\leq j} v)\|_{H^{s-\frac{1}{2}+\alpha}(\Omega_j)}&\lesssim_{M_{s_0}}\|\tilde{B}_j\|_{H^{s-\frac{1}{2}+\alpha}(\Omega_j)}+\|\Delta^{-1}(\nabla\cdot\Phi_{\leq j}v)\|_{H^{s+\frac{3}{2}+\alpha}(\Omega_j)}
\\
&\lesssim_{M_{s_0}} 2^{j\alpha}c_jM_s.
\end{split}
\end{equation}
Similarly, we have 
\begin{equation*}
\|\nabla_{\tilde{B}_j}\nabla\Delta^{-1}(\nabla\cdot\Phi_{\leq j}v)\|_{L^2(\Omega_j)}\lesssim_{M_{s_0}} \|\nabla\cdot\Phi_{\leq j}v\|_{L^2(\Omega_j)}\lesssim_{M_{s_0}} 2^{-js}c_jM_s,
\end{equation*}
which suffices. We now consider the commutator term in $D_j$, which we further expand as
\[
[ \nabla_{\tilde B_j}, \Phi_{\leq j}] v=
\tilde B_j [\nabla, \Phi_{\leq j}] v + 
[ {\tilde B_j}, \Phi_{\leq j}] \nabla v.
\]
The commutator in the first term is essentially a mollifier which is localized at frequency $2^j$ and satisfies a good $L^2$ bound exactly as above. However, the commutator in the second term is $2^{-j}$ times a mollifier 
at frequency $\leq 2^j$, so we no longer have a good $L^2$ bound for its output. Instead, we only obtain higher regularity bounds, 
\[
\| [ {\tilde B_j}, \Phi_{\leq j}] \nabla v\|_{H^\sigma(\Omega_j)} \lesssim_{M_{s_0}} c_j 2^{j(\sigma-s)}  M_s, \qquad \sigma > s-1.
\]
This suffices for \eqref{vbg-s} and \eqref{vbg-reg} but not for \eqref{vbg-diff2}. In this last case, we need to consider differences; namely,
\[
 [ \tilde B_{j+1}, \Phi_{\leq j+1}] \nabla v -  [ {\tilde B_j}, \Phi_{\leq j}] \nabla v = 
 [ \tilde B_{j+1} - \tilde B_j, \Phi_{\leq j+1}] \nabla v + [ {\tilde B_j}, \Phi_{\leq j+1} - \Phi_{\leq j}] \nabla v.
\]
For the first term we have a favorable $L^2$  bound (which suffices for the bound  \eqref{vbg-diff2} by interpolation),
\begin{equation*}
\begin{split}
\| [ \tilde B_{j+1} - \tilde B_j, \Phi_{\leq j+1}] \nabla v\|_{L^2(\Omega_j\cap\Omega_{j+1})} &\lesssim_{M_{s_0}}
\|  \tilde B_{j+1} - \tilde B_j\|_{L^2(\Omega_j\cap\Omega_{j+1})}\| \nabla v\|_{L^\infty(\Omega)} 
\\
&\lesssim_{M_{s_0}} c_j 2^{-js} \| B\|_{H^s(\Omega)} \|\nabla v\|_{L^\infty(\Omega)}.
\end{split}
\end{equation*}
For the second term, on the other hand, we have the mollifier $\Phi_{\leq j+1} - \Phi_{\leq j}$ which is localized at frequency $2^j$ with rapidly decreasing tails.
Therefore, we can estimate
\[
\| [ {\tilde B_j}, \Phi_{\leq j+1} - \Phi_{\leq j}] \nabla v\|_{L^2(\Omega_j\cap\Omega_{j+1})} \lesssim_{M_{s_0}} c_j 2^{-sj}M_s
\]
as needed. The corresponding bounds for $\nabla_{\tilde{B}_j}\tilde{B}_j$ are virtually identical.

\medskip

\textbf{(ii). The estimates involving the correction  $ B_j^c$.} The correction term $B_j^c$ is determined by the trace 
\[
f_{j}:= \tilde B_{j} \cdot n_{j}
\]
on the boundary $\Gamma_j$, so our first goal
here will be to prove bounds for $f_j$.
\medskip

We would like to estimate $f_j$ in $H^{s-\frac12}(\Gamma_j)$ as well as in both higher and lower norms with appropriate frequency factors. The difficulty is that $n_j$ only has regularity 
$H^{s-1}(\Gamma_j)$, which at frequency $2^j$ loses a $2^{\frac{j}2}$ factor. We will need to regain this factor by using the cancellation $B \cdot n_\Gamma = 0$ on $\Gamma$. This not only suffices, but even allows us to capture a gain in all lower norms above  $H^{-\frac12}$. Precisely, we claim that
\begin{equation}\label{fj}
\| f_{j}\|_{H^{\sigma}(\Gamma_j)} \lesssim_{M_{s_0}} 2^{(\sigma -s +\frac12) j}c_j M_s, \qquad \sigma \geq -\frac12.
\end{equation}
To prove this we work in the collar coordinates, which are denoted by $x \in \Gamma_*$, where the normal vector $n_j(x)$ is a smooth function
$n_j(x) = n(x, \eta_j(x), \nabla \eta_j(x))$. A small but very useful trick here is to observe 
that, by eliminating a harmless multiplicative factor arising from the normalization of $n_j$, 
we can assume that $n$ is linear in $\nabla \eta$.
This will be assumed from here on. Using the cancellation $B\cdot n_\Gamma = 0$ on $\Gamma$,
we represent $f_j$ in the form
\begin{equation*}
f_j = \tilde B_{j}(x,\eta_{j}(x)) \cdot n_{j} - P_{<j} ( B(x,\eta(x)) \cdot n).
\end{equation*}
Here, $P_{<j}$ is defined  by the functional calculus for $\Delta_{\Gamma_*}$ in order to be consistent with the definition of $\eta_j = P_{<j} \eta$. In local coordinates, this makes it 
a zeroth order pseudodifferential operator whose symbol is localized at frequency $\lesssim 2^j$ and equal to $1$ at frequency $\ll 2^j$ modulo Schwartz tails.
This is the only place where we use the cancellation; from here on, we take this as the expression for $f_j$ and prove the  bounds \eqref{fj} without making any further use of
the tangency condition. 
\medskip

Before we proceed with 
the proof of \eqref{fj}, we remark that this bound will suffice for the linear $B_j$ bounds 
in the proposition, but not for the bilinear 
bounds, precisely those concerning $\nabla_{\tilde B_j} B_j^c$. To prepare the ground for this
expression, we will split $f_j$ into a good 
and a bad component:
\begin{equation}\label{fgb}
f_j =  f_j^g + f_j^b,
\end{equation}
where the good component satisfies a better bound;
namely,
\begin{equation*}\label{fjg}
\| f_{j}^g\|_{H^{\sigma}(\Gamma_j)} \lesssim_{M_{s_0}} 2^{(\sigma -s) j}c_j M_s, \qquad \sigma \geq 0.
\end{equation*}
This will suffice in order to directly estimate the contribution of $f_j^g$ in $\nabla_{\tilde B_j} B_j^c$, and will in turn allow us to focus our attention on the contribution of the bad component $f_j^b$.
\medskip

We begin by processing a bit the second term in $f_j$, in several steps:
\medskip

a) Our first observation is that we can replace $B$ with $\Phi_{\leq \frac{3}{2}j}B$ since we have (owing to the fact that $s>\frac{d}{2}+1$),
\begin{equation*}
\|(\Phi_{\leq \frac{3}{2}j}B)(x,\eta)-B(x,\eta)\|_{L^2(\Gamma_*)}\lesssim_{M_{s_0}} 2^{-\frac{3}{2}(s-\frac{1}{2}-\epsilon)j}c_jM_s\lesssim_{M_{s_0}}2^{-sj}c_jM_s.
\end{equation*}
Moreover, thanks to the restriction $s>\frac{d}{2}+1$ there holds 
\begin{equation*}
\|\eta-\eta_j\|_{L^{\infty}(\Gamma_*)}\lesssim_{M_{s_0}} 2^{-(\frac{3}{2}+\epsilon)j},
\end{equation*}
 and thus, $\Phi_{\leq \frac{3}{2}j}B$ is defined on a $\mathcal{O}(2^{-\frac{3}{2}j})$-sized enlargement of $\Omega\cup\Omega_j$.
\medskip 

b) The second observation takes advantage 
of the Lipschitz bound on $B$ which,  combined with the $L^2$ bound $\|\eta - \eta_j\|_{L^2(\Gamma_*)}
\lesssim 2^{-js} c_j$, allows us to replace 
$(\Phi_{\leq \frac{3}{2}j}B)(x,\eta(x))$ with $(\Phi_{\leq \frac{3}{2}j}B)(x,\eta_j(x))$ modulo an $f_j^g$ contribution. Then again modulo an $f_j^g$ error, we may  replace $\Phi_{\leq \frac{3}{2}j}B$ by its divergence-free correction on $\Omega_j$,
\begin{equation*}
B':=\Phi_{\leq \frac{3}{2}j}B-\nabla\Delta_{\Omega_j}^{-1}(\nabla\cdot \Phi_{\leq \frac{3}{2}j}B).
\end{equation*}
\medskip

c) The third observation is that the high$\times$high contributions are also good, as 
by Sobolev embeddings, we have for some universal constant $C$,
\[
\| P_{> j-C} B'(x,\eta_j(x)) \cdot P_{> j-2C} n \|_{L^2(\Gamma_*)} \lesssim_{M_{s_0}} 2^{-j (2s-1 -\frac{d}2)} c_j^2M_s
\]
where $2s-\frac32 -\frac{d}2 > s-1$. This allows 
us to replace the second term in $f_j$ by
\[
P_{<j} (P_{< j-C} B'(x,\eta_j(x)) \cdot  n) +
P_{<j} (P_{> j-C} B'(x,\eta_j(x)) \cdot  P_{< j-2C}n).
\]

d) The fourth observation is that we can use 
commutator bounds to move the projector $P_{<j}$
onto the high frequency factor; namely,
\[
\|  [P_{<j}, P_{< j-C} B'(x,\eta_j(x))]  n \|_{L^2(\Gamma_*)}
\lesssim_{M_{s_0}} 2^{-sj} c_j \|  B'(x,\eta_j(x))\|_{C^1(\Gamma_*)} \| n\|_{H^{s-1}(\Gamma)},
\]
respectively
\[
\|  [P_{<j}, P_{< j-2C} n] P_{> j-C} B'(x,\eta_j(x)) \|_{L^2(\Gamma_*)}
\lesssim_{M_{s_0}} 2^{-sj} c_j \| B'(x,\eta_j(x))\|_{H^{s-\frac12}(\Gamma_*)}.
\]
At the  conclusion of this step, the second term in $f_j$ is replaced by 
\[
P_{< j-C} B'(x,\eta_j(x)) \cdot  P_{<j} n +
 P_{j-C < \cdot < j} B'(x,\eta_j(x)) \cdot  P_{< j-2C}n,
\]
and further, since the high$\times$high contributions are good, by 
\[
P_{<j} B'(x,\eta_j(x))   P_{<j} n.
\]

e) In order to better compare the two terms in $f_j$, it  would be very convenient to be able to compare $n_j$ with $P_{<j}n$. Here we take advantage of our earlier choice that $n$ depends linearly on $\nabla \eta$, which implies that we 
have the Moser/commutator type bound
\[
\| n_j - P_{<j} n\|_{L^2(\Gamma_*)} \lesssim_{M_{s_0}} 2^{-sj} c_jM_s. 
\]
This allows us to replace $P_{<j} n$ by $n_j$
modulo good contributions.
\medskip

This concludes our sequence of reductions for the second term in $f_j$. At this point we 
have obtained the representation \eqref{fgb} with 
\begin{equation}\label{fjb}
f_j^b = (\tilde B_j (x,\eta_j(x)) - P_{<j}B'(x,\eta_j(x))) \cdot n_j. 
\end{equation}
To complete the proof of \eqref{fj} it remains to 
show that $f_j^b$ satisfies \eqref{fj}. This is easier at higher regularity $\sigma > s$, where we can estimate the two terms separately using the Leibniz rule. So, we focus on the more delicate $H^{-\frac12}$ bound, where we have to show that 
\begin{equation}\label{delta B}
\|(\tilde B_j (x,\eta_j(x)) - P_{<j}B'(x,\eta_j(x)))\cdot n_j\|_{H^{-\frac12}(\Gamma_*)} \lesssim_{M_{s_0}} 2^{-sj} c_jM_s.
\end{equation}
To capture the cancellation here it is convenient to instead compare both terms with $B'(x,\eta_j(x))$. 
\medskip

For the first difference (from the definition of $B'$) we have the interior bound
\[
\| B' -\tilde B_j \|_{L^2 (\Omega_j)} \lesssim_{M_{s_0}} 2^{-sj} c_jM_s.   
\]
By itself this does not give an $H^{-\frac12}$ trace on $\Gamma_j$. Combining it with the 
divergence-free condition, however, it yields (by a standard duality argument) an $H^{-\frac12}$ trace for the normal component.
The second difference is easier to estimate 
directly using the definition of $B'$ and the trace theorem:
\[
\|P_{>j}B'(x,\eta_j(x))\|_{H^{-\frac12}(\Gamma_*)}
\lesssim_{M_{s_0}} 2^{-sj}c_j M_s.
\]
This completes the proof of \eqref{fjb} and thus of \eqref{fj}. To conclude, we only need to show the correction bound 
\begin{equation}\label{Bjc}
\| B_j^c \|_{H^\sigma(\Omega_j)} \lesssim 2^{(\sigma-s)j} c_jM_s, \qquad \sigma \geq 0. 
\end{equation}
This is clear when $\sigma=0$ thanks to the $H^{\frac{1}{2}}\to H^1$ bound for $\mathcal{H}_j$ and  (in view of the fact that $\tilde{B}_j$ is divergence-free) the estimate $\|\mathcal{N}^{-1}_j(\tilde{B}_j\cdot n_j)\|_{H^{\frac{1}{2}}(\Gamma_j)}\lesssim \|\tilde{B}_j\cdot n_j\|_{H^{-\frac{1}{2}}(\Gamma_j)}$. By straightforward interpolation, it remains to consider the case $\sigma\geq s$. For this, we use \Cref{Hbounds}, \Cref{ellipticity}, \eqref{fj} and Sobolev embeddings to obtain
\begin{equation}\label{Bjcest}
\begin{split}
\|B_j^c\|_{H^{\sigma}(\Omega_j)}&\lesssim_{M_{s_0}} \|\tilde{B}_j\cdot n_j\|_{H^{\sigma-\frac{1}{2}}(\Gamma_j)}+\|\Gamma_j\|_{H^{\sigma+\frac{1}{2}}}\|\tilde{B}_j\cdot n_j\|_{H^{s-\frac{3}{2}}(\Gamma_j)}
\\
&\lesssim_{M_{s_0}} (2^{j(\sigma-s)}+2^{j(\sigma+\frac{1}{2}-s)}2^{-j})c_jM_s,
\end{split}
\end{equation}
which yields \eqref{Bjc}.
This suffices for the bounds \eqref{vbg-point} and \eqref{vbg-diff} as well as for the linear 
$B_j$ bounds in \eqref{vbg-s} and \eqref{vbg-reg}.
\medskip

To complete the proof of the proposition it remains to consider all of the estimates involving 
$B_j^c$ in $\nabla_{B_j}$ terms. Here we can split
\[
\nabla_{B_j} = \nabla_{\tilde B_j} + \nabla_{B_j^c}.
\]
The contributions of  $\nabla_{B_j^c}$ are easy to estimate since \eqref{Bjc} implies the uniform bound
\[
\| B_j^c\|_{L^\infty(\Omega_j)} \lesssim_{M_{s_0}} 2^{-j(1+\delta)} c_j,
\]
as well as corresponding higher regularity bounds. This only leaves us with the bounds 
for the bilinear expression $\nabla_{\tilde B_j} B_j^c$ in \eqref{vbg-s}, \eqref{vbg-reg}
and \eqref{vbg-diff2}. More precisely, we have to prove that
\[
 \|\nabla_{B_j} \nabla\mathcal H_j\mathcal{N}_j^{-1}f_j \|_{H^\sigma(\Omega_j)} \lesssim_{M_{s_0}} 2^{j(\sigma - s +\frac12)} c_jM_s
\]
for $\sigma \geq 1$. We note that a direct estimate here fails by half a derivative, so a more careful analysis is needed. We begin by commuting $\nabla_{B_j}$ with $\nabla\mathcal{H}_j\mathcal{N}_j^{-1}$. By making use of \Cref{direst}, \Cref{Movingsurfid}, \Cref{commutatorremark} and  arguing similarly to the estimate \eqref{Bjcest}, we have
\begin{equation*}
\|[\nabla_{B_j},\nabla\mathcal{H}_j]\mathcal{N}_j^{-1}f_j\|_{H^{\sigma}(\Omega_j)}\lesssim_{M_{s_0}} 2^{(\sigma-s+\frac{1}{2})j}c_jM_s,
\end{equation*}
and also thanks to \Cref{ellipticity}, 
\begin{equation*}
\|\nabla\mathcal{H}_j\nabla_{B_j}\mathcal{N}_j^{-1}f_j\|_{H^{\sigma}(\Omega_j)}\lesssim_{M_{s_0}}\|\mathcal{N}_j\nabla_{B_j}\mathcal{N}_j^{-1}f_j\|_{H^{\sigma-\frac{1}{2}}(\Gamma_j)}+2^{(\sigma-s+\frac{1}{2})j}c_jM_s.
\end{equation*}
Moreover, 
\begin{equation*}
\|[\nabla_{B_j},\mathcal{N}_j]\mathcal{N}_j^{-1}f_j\|_{H^{\sigma-\frac{1}{2}}(\Gamma_j)}\lesssim_{M_{s_0}} 2^{(\sigma-s+\frac{1}{2})j}c_jM_s.
\end{equation*}
Therefore, it remains to show that
\begin{equation*}
\|\nabla_{B_j}f_j\|_{H^{\sigma-\frac{1}{2}}(\Gamma_j)}\lesssim_{M_{s_0}}2^{(\sigma-s+\frac{1}{2})j}c_jM_s.
\end{equation*}
A further simplification 
arises from the fact that the contributions of $f_j^g$ gain an additional half derivative and thus they are also directly perturbative (by making use of the identity $\nabla_{B_j}=B_j\cdot\nabla^{\top}$,  \Cref{boundaryest} and \Cref{tangradientbound}). So,
we are left with proving the bound
\begin{equation*}
\|\nabla_{B_j}f_j^b\|_{H^{\sigma-\frac{1}{2}}(\Gamma_j)}\lesssim_{M_{s_0}}2^{(\sigma-s+\frac{1}{2})j}c_jM_s
.
\end{equation*}
We consider first the case when $\sigma = 1$. Using the expression \eqref{fjb} for $f_j^b$, we expand
\[
\nabla_{B_j} f_j^b = \nabla_{B_j}(\tilde B_j (x,\eta_j(x)) - P_{<j}B'(x,\eta_j(x))) \cdot n_j +
(\tilde B_j (x,\eta_j(x)) - P_{<j}B'(x,\eta_j(x))) \cdot \nabla_{B_j} n_j.
\]
For the second term we use the simpler $H^{\frac{1}{2}}$ version of \eqref{delta B} to write
\[
\| (\tilde B_j (x,\eta_j(x)) - P_{<j}B'(x,\eta_j(x)))\|_{H^{\frac{1}{2}}(\Gamma_*)} \lesssim 2^{j( - s +1)} c_j,
\]
while from the identity $\nabla_{B_j}n_j=-\nabla^{\top}B_j\cdot n_j$ from \Cref{Movingsurfid}, we have
\[
\| \nabla_{B_j} n_j\|_{C^{\frac{1}{2}}(\Gamma_j)}  \lesssim_{M_{s_0}} 2^{\frac{j}{2}},
\]
which by simple Sobolev product estimates entirely suffices. For the first term, we harmlessly discard $n_j$ and we add and subtract a $B'$ to obtain (suppressing the evaluation at $(x,\eta_j(x))$)
\begin{equation*}
\begin{split}
\nabla_{B_j}(\tilde B_j - P_{<j}B') &= 
\nabla_{B_j}(\Phi_{\leq j} - I) B'  +
\nabla_{B_j} P_{>j}B'+\nabla_{B_j}(\Phi_{\leq j}(B-B'))-\nabla_{B_j}\nabla\cdot\Delta^{-1}(\nabla\cdot\Phi_{\leq j} B).
\end{split}
\end{equation*}
The last term on the right can be crudely estimated as in \eqref{gradBdivcorrection}. From the definition of $B'$ and by commuting $\nabla_{B_j}$ with $\Phi_{\leq j}$ and $\Phi_{\leq\frac{3}{2}j}$, the third term on the right is also easily dispensed with. For the remaining two terms, the commutator bounds are favorable, so we can replace the above expression (modulo acceptable errors) with
\[
(\Phi_{\leq j} - I) \nabla_{B_j} B' (x,\eta_j(x)) +
 P_{>j} \nabla_{B_j} B'(x,\eta_j(x)).
\]
Then, making use of  the trace theorem and the definition of $B'$, it 
suffices to prove an $(H^{s-\frac12} + 2^{-\frac{j}2} H^{s-1})(\Omega_j)$ bound for $\nabla_{B_j} B$.
The correction $B_j^c$ has a uniform bound of $2^{-\frac{j}2} c_j$ as well as corresponding higher regularity, so we can discard it and replace the above by 
\[
\nabla_{\tilde B_j} B = \nabla_B B - \nabla_{\tilde B_j - B} B,
\]
where the second term is also favorable since
$\tilde B_j - B$ has a uniform bound of $2^{-j} c_j$. On the other hand, $\nabla_B B$ belongs to $H^{s-\frac12}$, with frequency envelope $c_j$.
Thus, \eqref{fjb} follows. The proof of the corresponding higher regularity bounds is similar,
by appropriately using the Leibniz rule and then repeating the analysis above.
\end{proof}

\bibliographystyle{plain}
\bibliography{refs2.bib}

\end{document}